\documentclass[11pt]{amsart}

\allowdisplaybreaks
\usepackage{amsmath}
\usepackage{amsfonts}
\usepackage{amssymb}
\usepackage{amsthm}
\usepackage{graphicx}
\usepackage{pb-diagram}
\usepackage{epstopdf}
\usepackage{amscd}
\usepackage{pdfpages}
\usepackage{bbm}
\usepackage{multirow}
\usepackage[mathscr]{eucal}
\usepackage{epsfig,epsf,epic}
\usepackage{pdfsync}
\usepackage{enumerate}
\usepackage{etex}
\usepackage[all]{xy}
\usepackage{fancyref}
\usepackage{mathtools}
\usepackage{scalerel,stackengine}
\usepackage{comment}
\usepackage{hyperref}

\stackMath

\newtheorem{theorem}{Theorem}[section]
\newtheorem{prop}[theorem]{Proposition}

\newtheorem{corollary}[theorem]{Corollary}
\newtheorem{definition}[theorem]{Definition}
\newtheorem{example}[theorem]{Example}
\newtheorem{lemma}[theorem]{Lemma}
\newtheorem{remark}[theorem]{Remark}
\newtheorem{assum}[theorem]{Assumption}
\newtheorem{condition}[theorem]{Condition}
\newtheorem{notation}[theorem]{Notation}
\numberwithin{equation}{section}

\DeclareMathOperator{\Hom}{Hom}

\DeclareMathOperator{\trace}{tr}

\newcommand{\real}{\mathbb{R}}
\newcommand{\comp}{\mathbb{C}}

\newcommand{\inte}{\mathbb{Z}}

\newcommand{\pdb}{\bar{\partial}}
\newcommand{\dd}[1]{\frac{\partial}{\partial #1}}

\newcommand{\half}{\frac{1}{2}}

\newcommand{\reallywidehat}[1]{%
	\savestack{\tmpbox}{\stretchto{%
			\scaleto{%
				\scalerel*[\widthof{\ensuremath{#1}}]{\kern-.6pt\bigwedge\kern-.6pt}%
				{\rule[-\textheight/2]{1ex}{\textheight}}
			}{\textheight}%
		}{0.5ex}}%
	\stackon[1pt]{#1}{\tmpbox}%
}

\newcommand{\reallywidetilde}[1]{%
	\savestack{\tmpbox}{\stretchto{%
			\scaleto{%
				\scalerel*[\widthof{\ensuremath{#1}}]{\kern-.6pt\sim\kern-.6pt}%
				{\rule[-\textheight/2]{1ex}{\textheight}}
			}{\textheight}%
		}{0.5ex}}%
	\stackon[1pt]{#1}{\tmpbox}%
}

\newcommand{\bchprod}{\odot}
\newcommand{\blat}{\mathbf{K}}
\newcommand{\cfr}{R}
\newcommand{\cfrk}[1]{\prescript{#1}{}{\cfr}}
\newcommand{\logs}{S^{\dagger}}
\newcommand{\logsk}[1]{\prescript{#1}{}{S}^{\dagger}}
\newcommand{\logsf}{\hat{S}^{\dagger}}
\newcommand{\logsdrk}[2]{\prescript{#1}{}{\Omega}^{#2}_{S^{\dagger}}}
\newcommand{\logsvfk}[1]{\prescript{#1}{}{\Theta}_{S^{\dagger}}}
\newcommand{\logsdrf}[1]{\hat{\Omega}^{#1}_{S^{\dagger}}}
\newcommand{\logsvff}{\hat{\Theta}_{S^{\dagger}}}
\newcommand{\bva}[1]{\prescript{#1}{}{\mathcal{G}}}
\newcommand{\tbva}[2]{\prescript{#1}{#2}{\mathcal{K}}}
\newcommand{\gmiso}[2]{\prescript{#1}{#2}{\sigma}}
\newcommand{\dpartial}[1]{\prescript{#1}{}{\partial}}
\newcommand{\volf}[1]{\prescript{#1}{}{\omega}}
\newcommand{\rdr}[2]{\prescript{#1}{#2 \parallel}{\mathcal{K}}}
\newcommand{\bvd}[1]{\prescript{#1}{}{\Delta}}
\newcommand{\gmc}[1]{\prescript{#1}{}{\nabla}}
\newcommand{\rest}[1]{\prescript{#1}{}{\flat}}
\newcommand{\patch}[1]{\prescript{#1}{}{\psi}}
\newcommand{\hpatch}[1]{\prescript{#1}{}{\hat{\psi}}}
\newcommand{\resta}[1]{\prescript{#1}{}{\mathfrak{b}}}
\newcommand{\patchij}[1]{\prescript{#1}{}{\mathfrak{p}}}
\newcommand{\cocyobs}[1]{\prescript{#1}{}{\mathfrak{o}}}
\newcommand{\bvdobs}[1]{\prescript{#1}{}{\mathfrak{w}}}
\newcommand{\lscript}[3]{\prescript{#1}{#2}{#3}}

\newcommand{\simplex}{\blacktriangle}
\newcommand{\simplexbdy}{\vartriangle}
\newcommand{\hsimplex}{\blacksquare}
\newcommand{\hsimplexbdy}{\square}
\newcommand{\twc}[1]{\prescript{#1}{}{TW}}
\newcommand{\ktwc}[2]{\prescript{#1}{#2}{\mathbf{A}}}
\newcommand{\restmap}{\mathfrak{r}}
\newcommand{\glue}[1]{\prescript{#1}{}{g}}
\newcommand{\iauto}[1]{\prescript{#1}{}{\mathfrak{a}}}
\newcommand{\sauto}[1]{\prescript{#1}{}{\vartheta}}
\newcommand{\autoij}[1]{\prescript{#1}{}{\phi}}
\newcommand{\gluehom}[1]{\prescript{#1}{}{h}}
\newcommand{\thmcocyo}[1]{\prescript{#1}{}{\mathsf{O}}}
\newcommand{\hsauto}[1]{\prescript{#1}{}{\varkappa}}
\newcommand{\hiauto}[1]{\prescript{#1}{}{\mathsf{a}}}
\newcommand{\hautoij}[1]{\prescript{#1}{}{\varrho}}
\newcommand{\autozeroone}[1]{\prescript{#1}{}{\mathtt{p}}}
\newcommand{\cech}[1]{\prescript{#1}{}{\check{\mathcal{C}}}}
\newcommand{\cechd}[1]{\prescript{#1}{}{\delta}}
\newcommand{\dbtwist}[1]{\prescript{#1}{}{\mathfrak{d}}}
\newcommand{\volftwist}[1]{\prescript{#1}{}{\mathfrak{f}}}
\newcommand{\drglue}[1]{\prescript{#1}{}{\hat{g}}}
\newcommand{\polyv}[1]{\prescript{#1}{}{PV}}
\newcommand{\totaldr}[2]{\prescript{#1}{#2}{\mathcal{A}}}
\newcommand{\dbobs}[1]{\prescript{#1}{}{\mathfrak{l}}}
\newcommand{\dvolfobs}[1]{\prescript{#1}{}{\mathfrak{y}}}
\newcommand{\drd}[2]{\prescript{#1}{#2}{\mathbf{d}}}
\newcommand{\ptd}[1]{\prescript{#1}{}{\breve{\mathbf{d}}}}

\newcommand{\ecfr}[1]{\prescript{#1}{}{T}}

\newcommand{\logsc}{S_{T}^{\dagger}}
\newcommand{\logsck}[1]{\prescript{#1}{}{S}_{T}^{\dagger}}
\newcommand{\logscf}{\hat{S}_{T}^{\dagger}}
\newcommand{\logscdrk}[2]{\prescript{#1}{}{\Omega}^{#2}_{S_T^{\dagger}}}
\newcommand{\logscvfk}[1]{\prescript{#1}{}{\Theta}_{S_T^{\dagger}}}
\newcommand{\logscdrf}[1]{\hat{\Omega}^{#1}_{S_T^{\dagger}}}
\newcommand{\logscvff}{\hat{\Theta}_{S_T^{\dagger}}}

\begin{document}

\title[Geometry of the MC equation near degenerate CY]{Geometry of the Maurer-Cartan equation near degenerate Calabi-Yau varieties}

\author[Chan]{Kwokwai Chan}
\address{Department of Mathematics\\ The Chinese University of Hong Kong\\ Shatin\\ Hong Kong\\}
\email{kwchan@math.cuhk.edu.hk}

\author[Leung]{Naichung Conan Leung}
\address{The Institute of Mathematical Sciences and Department of Mathematics\\ The Chinese University of Hong Kong\\ Shatin\\ Hong Kong\\}
\email{leung@math.cuhk.edu.hk}

\author[Ma]{Ziming Nikolas Ma}
\address{Department of Mathematics\\ Southern University of Science and Technology\\ Nanshan District \\ Shenzhen\\ China\\}
\email{nikolasming@outlook.com}

\begin{abstract}
	Given a degenerate Calabi-Yau variety $X$ equipped with local deformation data, we construct an {\em almost} differential graded Batalin-Vilkovisky algebra $\polyv{}^{*,*}(X)$, producing a singular version of the extended Kodaira-Spencer differential graded Lie algebra in the Calabi-Yau setting.
	Assuming Hodge-to-de Rham degeneracy and a local condition that guarantees freeness of the Hodge bundle, we prove a Bogomolov-Tian-Todorov--type unobstructedness theorem for smoothing of singular Calabi-Yau varieties. In particular, this provides a unified proof for the existence of smoothing of both $d$-semistable log smooth Calabi-Yau varieties (as studied by Friedman \cite{Friedman83} and Kawamata-Namikawa \cite{kawamata1994logarithmic}) and maximally degenerate Calabi-Yau varieties (as studied by Kontsevich-Soibelman \cite{kontsevich-soibelman04} and Gross-Siebert \cite{gross2011real}).
	We also demonstrate how our construction yields a logarithmic Frobenius manifold structure on a formal neighborhood of $X$ in the extended moduli space by applying the technique of Barannikov-Kontsevich \cite{barannikov-kontsevich98, Barannikov99}.
\end{abstract}

\maketitle

\section{Introduction}\label{sec:introduction}

\subsection{Background}\label{sec:background_for_intro}

Deformation theory plays an indispensable role in algebraic geometry, mirror symmetry and mathematical physics. Two major approaches are the \v{C}ech approach \cite{Sernesi_book} and the Kodaira-Spencer differential-geometric approach \cite{Morrow-Kodaira_book}. The latter is particularly powerful in Calabi-Yau geometry \cite{Calabi, Yau77, Yau78}.
Given a Calabi-Yau manifold $X$, the Kodaira-Spencer differential graded Lie algebra (abbrev. dgLa) $(\Omega^{0,*}(X,T_{X}^{1,0}), \bar{\partial}, [\cdot,\cdot])$ can be upgraded to a {\em differential graded Batalin-Vilkovisky (abbrev. dgBV) algebra} $(\Omega^{0,*}(X,\wedge^*T^{1,0}),\bar{\partial},\Delta,\wedge)$. The famous Bogomolov-Tian-Todorov Theorem \cite{bogomolov1978hamiltonian, tian1987smoothness, todorov1989weil} then yields unobstructedness of deformations, or equivalently, local smoothness of the moduli space (see also the recent works \cite{Liu-Rao-Yang15, Liu-Rao-Wan19}), as well as a local {\em Frobenius manifold} structure on the {\em extended} moduli space by the work of Barannikov-Kontsevich \cite{barannikov-kontsevich98, Barannikov99}.

A degeneration of Calabi-Yau manifolds $\{X_q\}$ gives a singular Calabi-Yau variety equipped with a natural {\em log structure} in the sense of Fontaine-Illusie and Kato \cite{kato1989logarithmic}. So, conversely, one should study smoothing of singular Calabi-Yau varieties via log geometry. This is essential for understanding compactifications of moduli spaces of Calabi-Yau manifolds.
Two fundamental results in this direction are the existence of smoothing of $d$-semistable log smooth Calabi-Yau varieties due to Friedman \cite{Friedman83} and Kawamata-Namikawa \cite{kawamata1994logarithmic} and that of maximally degenerate Calabi-Yau varieties due to Kontsevich-Soibelman \cite{kontsevich-soibelman04} and Gross-Siebert \cite{gross2011real}. In both cases, existence of smoothing was proved by an order-by-order construction of thickenings, but the methods were entirely different: in the log smooth case, it was done by proving Hodge-to-de Rham degeneracy and then applying the $T^1$-lifting technique, while in the maximally degenerate case, it was done using consistent scattering diagrams (which actually yields {\em explicit} thickenings).

From Quillen, Deligne and Drinfeld (see e.g. \cite{KS_deformation_theory}), we learned the general philosophy that any deformation problem should be governed by a dgLa (or more generally, an $L_{\infty}$-algebra); see e.g. \cite{fiorenza2007structures, fiorenza2012differential,  fiorenza2012formality, fiorenza2012cosimplicial, hinich2001dg, iacono2007differential, lurie2011derived, manetti2005differential, manetti2007lie, manetti2015some, pridham2010unifying}. Employing this framework, Bogomolov-Tian-Todorov--type theorems have been formulated and proven for smooth (log) Calabi-Yau manifolds, e.g. in \cite{kontsevichgeneralized, KKP08, katzarkov2017bogomolov, iacono2010algebraic, iacono2015, iacono2017abstract, Liu-Rao-Wan19}.
We therefore want to know if there exists a dgLa (or better, a dgBV algebra) controlling the smoothing of singular Calabi-Yau varieties, which in particular can lead to a unified proof of the above smoothing results.

A hint was given in our earlier work \cite{kwchan-leung-ma, kwchan-ma-p2}, where we implemented (part of) the proposal by Kontsevich-Soibelman \cite{kontsevich00} and Fukaya \cite{fukaya05} and demonstrated how asymptotic expansions of the Maurer-Cartan elements of the Kodaira-Spencer dgLa $\Omega^{0,*}(X,\wedge^{*} T_{X}^{1,0})$ (where $X$ is a torus bundle over a base $B$) gave rise to consistent scattering diagrams and tropical disk counts as the torus fibers shrink. Besides deepening our understanding of the Strominger-Yau-Zaslow proposal \cite{syz96}, this shows that consistent scattering diagrams and tropical counts are encoded in dgLa's and their Maurer-Cartan elements.

To construct the desired algebraic structure that governs the smoothing of singular Calabi-Yau varieties, one major difficulty arises from {\em non-trivial topology change} dictated by the residue of the Gauss-Manin connection in a degeneration of Calabi-Yau manifolds. In particular, the trivial deformation {\em cannot} be a smoothing -- this is in sharp contrast with deformations in the smooth case.
Hence one cannot expect to get an ordinary dgLa or dgBV algebra, because otherwise $\varphi = 0$, which corresponds to the trivial deformation, would be a Maurer-Cartan solution.
What we discovered here is that, instead, a so-called {\em almost} dgBV algebra $\polyv{}^{*,*}(X)$ can naturally be constructed from a singular Calabi-Yau variety $X$ equipped with suitable local smoothing models. This gives a singular analogue of the Kodaira-Spencer dgBV algebra, (the classical part of) whose associated Maurer-Cartan equation indeed governs geometric smoothings of such an $X$.

More precisely, partially motivated by the {\em divisorial log deformation theory} of Gross-Siebert \cite{Gross-Siebert-logI, Gross-Siebert-logII}, we develop an abstract algebraic framework to construct an almost dgBV algebra $\polyv{}^{*,*}(X)$ from suitable local deformation (or thickening) data attached to $X$ via a local-to-global \v{C}ech-de Rham--type gluing procedure (see Theorem \ref{thm:main_thm_1_intro}).
Such a simplicial construction of $\polyv{}^{*,*}(X)$ can capture the aforementioned nontrivial topology change because the local thickenings are not locally trivial. It also allows us to directly link the smoothing of singular Calabi-Yau varieties with the Hodge theory developed in \cite{barannikov-kontsevich98, Barannikov99, kontsevichgeneralized, KKP08, li2013variation}. This will play an important role in future study of Calabi-Yau moduli and higher genus $B$-models \cite{costello2012quantum}.

In this paper, we give two immediate applications of Theorem \ref{thm:main_thm_1_intro}. First, assuming Hodge-to-de Rham degeneracy and a local condition that guarantees freeness of the Hodge bundle, we prove a Bogomolov-Tian-Todorov--type unobstructedness theorem for the smoothing of the singular Calabi-Yau variety $X$ (see Theorem \ref{thm:theorem_1_intro}). Let us emphasize that, once we obtain the correct algebraic structure, the proof is simply by applying standard techniques in BV algebras \cite{kontsevichgeneralized, KKP08, terilla2008smoothness}. This is in line with the framework developed by Katzarkov-Kontsevich-Pantev \cite{kontsevichgeneralized, KKP08}. In particular, this produces a unified proof for existence of smoothing in both the log smooth and maximally degenerate cases.
Second, under two further assumptions, we are able to construct a logarithmic Frobenius manifold structure on a formal neighborhood of $X$ in the extended moduli space by applying the technique of Barannikov-Kontsevich \cite{barannikov-kontsevich98, Barannikov99} (see Theorem \ref{thm:intro_thm_2}).

Very recently, Felten, Filip and Ruddat \cite{Felten-Filip-Ruddat} proved that the Hodge-to-de Rham degeneracy assumption and the local condition in Theorem \ref{thm:theorem_1_intro} both hold for so-called {\em toroidal crossing spaces} -- a very general class of spaces that includes both log smooth and maximally degenerate Calabi-Yau varieties. So Theorems \ref{thm:main_thm_1_intro} and \ref{thm:theorem_1_intro} imply the existence of smoothing for all these spaces (see \cite{Felten-Filip-Ruddat} for more details). In particular, the compactified moduli space of Calabi-Yau manifolds should include such spaces, over which the moduli space is also smooth.
In \cite{Felten}, Felten further showed that the almost dgLa (called, perhaps more appropriately, a {\em pre-dgLa} in \cite{Felten}) we constructed here indeed governs the log smooth deformation functor. He also gave a nice explanation why ordinary dgLa's are not sufficient in controlling deformations of a log smooth morphism. See also \cite{Sano21, Barrott-Doran21} for even more recent applications of our results.

\subsection{Main results}\label{sec:main_result_intro}

To describe our main results, we need to fix a monoid $Q$. Also let $\comp[Q]$ be the universal coefficient ring equipped with the monomial ideal $\mathbf{m} = \langle Q\setminus \{0\}\rangle$. 

\subsubsection{A singular analogue of the Kodaira-Spencer dgBV algebra (\S \ref{sec:abstract_framework} \& \S \ref{sec:abstract_construction})}

Consider a complex analytic space $(X,\mathcal{O}_X)$ equipped with a covering $\mathcal{V} = \{V_\alpha\}_{\alpha}$ by {\em Stein} open subsets, together with {\em local deformation or thickening data} attached to each $V_{\alpha}$, which consist of a $k^{\text{th}}$-order coherent sheaf of BV algebras $(\bva{k}_{\alpha}^*, \wedge, \bvd{k}_\alpha)$ over $\cfrk{k} := \comp[Q]/\mathbf{m}^{k+1}$ that acts on a $k^{\text{th}}$-order coherent sheaf of de Rham modules $(\tbva{k}{}_{\alpha}^*,\wedge,\dpartial{k}_{\alpha})$ (see Definitions \ref{def:higher_order_data} \& \ref{def:higher_order_data_module} in \S \ref{sec:abstract_framework}) for each $k \in \inte_{\geq 0}$. 

Then fix another Stein covering $\mathcal{U} = \{U_i\}_{i \in \mathbb{N}}$ of $X$ equipped with {\em higher order local patching data}, namely, isomorphisms $\patch{k}_{\alpha\beta,i} : \bva{k}_{\alpha}|_{U_i} \rightarrow \bva{k}_{\beta}|_{U_i}$ of sheaves satisfying certain conditions (see Definition \ref{def:higher_order_patching}). In geometric situations, these patching isomorphisms always come from {\em local uniqueness} of the local thickening data, but they are {\em not compatible} directly. Fortunately, the differences between these data are captured by Lie bracket with local sections from $\bva{k}^{*}_{\alpha}$ (see Definition \ref{def:higher_order_patching}).

The ordinary \v{C}ech approach to deformation theory is done by solving for compatible gluings $\glue{k}_{\alpha\beta} : \bva{k}^*_{\alpha} \rightarrow \bva{k}^*_\beta$ and understanding the obstructions in doing so. In our situation, instead of gluing directly, we first take a dg resolution of the sheaf $\bva{k}^*_\alpha$, given as a sheaf of dgBV algebras $\polyv{k}^{*,*}_{\alpha}$
determined by the Thom-Whitney construction \cite{whitney2012geometric, dupont1976simplicial}. Then we solve for gluing morphisms $\glue{k}_{\alpha\beta}: \polyv{k}^{*,*}_\alpha \rightarrow \polyv{k}^{*,*}_\beta$ which satisfy the cocycle condition and are compatible for different orders $k$. Morally speaking, this works because the local sheaves $\polyv{k}^{*,*}_{\alpha}$'s are ``softer'' than the $\bva{k}^*_\alpha$'s.\footnote{In the maximally degenerate case, the local thickenings do not glue (which was why Gross-Siebert \cite{gross2011real} needed consistent scattering diagrams to correct the gluings) but here we observe that the local BV algebras, after taking dg resolutions, do glue.} Now the upshot is that this gives rise to an {\em almost} dgBV algebra (meaning that the derivation is a differential only in the $0$-th order), instead of a genuine dgBV algebra, and it is sufficient for deducing unobstructedness.
Our first main result is the following:

\begin{theorem}[=Theorem \ref{lem:existence_compatible_gluing_data} + Proposition \ref{lem:homotopy_compatible_gluing_morphism} + Theorem \ref{thm:construction_of_differentials}]\label{thm:main_thm_1_intro}
There exists an almost differential graded Batalin-Vilkovisky (abbrev. dgBV) algebra of polyvector fields\footnote{We are using the Thom-Whitney resolution but the operator is written as $\pdb$ because it plays the role of the Dolbeault operator in the differential-geometric approach, and similarly, we write $\polyv{}^{*,*}(X)$ because it plays the role of the complex of polyvector fields.}
$$(\polyv{}^{*,*}(X) , \pdb, \bvd{}, \wedge)$$
over $\comp[[Q]]$ meaning that it satisfies all the required identities for a dgBV algebra except that $\pdb^2 = \pdb \Delta + \Delta \pdb = 0$
hold only in the $0$-th order, i.e. when restricted to $\polyv{0}^{*,*}(X) = \polyv{}^{*,*}(X) \otimes_{\comp[[Q]]} (\comp[[Q]]/\langle Q\setminus \{0\}\rangle )$.
\end{theorem}

In geometric situations such as the log smooth case \cite{Friedman83, kawamata1994logarithmic} or the maximally degenerate case \cite{kontsevich-soibelman04, Gross-Siebert-logII}, this theorem provides the correct analogue of the Kodaira-Spencer dgBV algebra which controls smoothings of the singular Calabi-Yau variety $X$.

\subsubsection{Unobstructedness (\S \ref{sec:differentials} \& \S \ref{sec:abstract_theorem})}

Now we consider the extended Maurer-Cartan equation
\begin{equation}\label{eqn:maurer_cartan_introduction}
(\bar{\partial} + t \Delta + [\varphi,\cdot])^2 = 0
\end{equation}
for $\varphi \in \polyv{}^{*,*}(X)[[t]]$, where $t$ is the descendant parameter as in \cite{Barannikov99}.
Using standard techniques in the theory of BV algebras \cite{kontsevichgeneralized, KKP08, terilla2008smoothness}, we prove an unobstructedness theorem under two assumptions:
\begin{itemize}
\item
the {\em Hodge-to-de Rham degeneracy assumption} which says that the cohomology $H^*(\polyv{0}(X)[[t]],\pdb+t\bvd{})$ is a free $\comp[[t]]$ module (Assumption \ref{assum:Hodge_de_rham_degeneracy}), and
\item
a local condition which guarantees {\em freeness of the Hodge bundle} (Assumption \ref{assum:local_assumption_for_triviality_of_hodge_bundle}).
\end{itemize}

\begin{theorem}[=Theorem \ref{thm:unobstructedness_of_MC_equation} + Lemma \ref{lem:from_descendant_maurer_cartan_to_classical_maurer_cartan} + Proposition \ref{prop:Maurer_Cartan_give_consistent_gluing}]\label{thm:theorem_1_intro}
	Under Assumptions \ref{assum:local_assumption_for_triviality_of_hodge_bundle} and \ref{assum:Hodge_de_rham_degeneracy}, the extended Maurer-Cartan equation \eqref{eqn:maurer_cartan_introduction} can be solved order by order for $\varphi \in \polyv{}^{*,*}(X) \otimes \comp[[t]]$. In particular, geometric smoothings of $X$ over the formal scheme $\text{Spf}(\comp[[Q]])$ are unobstructed.
\end{theorem}

This can be viewed as a singular version of the famous Bogomolov-Tian-Todorov (BTT) theorem \cite{bogomolov1978hamiltonian, tian1987smoothness, todorov1989weil}. It also extends the framework put forth by Katzarkov-Kontsevich-Pantev \cite{kontsevichgeneralized, KKP08} to the singular case.

We remark that Assumption \ref{assum:local_assumption_for_triviality_of_hodge_bundle} depends on how good the local smoothing models are. If they are good enough, Theorem \ref{thm:theorem_1_intro} essentially reduces smoothability of $X$ to validity of the Hodge-to-de Rham degeneracy (i.e. Assumption \ref{assum:Hodge_de_rham_degeneracy}).

\subsubsection{Log Frobenius manifold structure (\S \ref{sec:abstract_semi_infinite_VHS})}

From the almost dgBV algebra $\polyv{}^{*,*}(X)$, we can construct a {\em logarithmic Frobenius manifold structure} on a formal neighborhood of $X$ in the extended moduli space by directly adapting the techniques of Barannikov-Kontsevich \cite{barannikov-kontsevich98, Barannikov99}. For this purpose, we need to suitably enlarge our coefficient ring $\comp[Q]$ to include all the extended moduli parameters.

Firstly, from the log structure on $X$, we have the complex of log de Rham differential forms $\Omega^*$ and we can construct the residue action $N_{\nu}$ of the Gauss-Manin connection $\gmc{}$ acting on the cohomology $\mathbb{H}^*(X,\Omega^*)$ for each constant vector field on $\text{Spec}(\comp[Q])$ given by $\nu \in (Q^{gp})^{\vee} \otimes _{\inte} \real$. We assume the following:
\begin{itemize}
	\item
	the existence of a {\em weight filtration} of the form
	$$
	\{0\} \subset \mathcal{W}_{\leq 0 } \subset \cdots  \subset \mathcal{W}_{\leq r} \subset \cdots \subset \mathcal{W}_{d} = \mathbb{H}^*(X,\Omega^*)
	$$
	indexed by half-integer weights $r \in \half \inte$, which is {\em opposite} to the Hodge filtration $\mathcal{F}^{\bullet} (\mathbb{H}^*(X,\Omega^*))$ (Assumption \ref{assum:weighted_filtration_assumption}), and
	\item
	the existence of a compatible trace map $\trace: \mathbb{H}^*(X,\Omega^*) \rightarrow \comp$ such that the associated pairing $\prescript{0}{}{\mathsf{p}}(\alpha,\beta):= \trace(\alpha \wedge \beta)$ is non-degenerate (Assumption \ref{assum:poincare_pairing_assumption}).
\end{itemize}

Given a versal solution $\varphi$ of the Maurer-Cartan equation \eqref{eqn:maurer_cartan_introduction}, we consider the $\reallywidehat{\comp[Q][[t]]}$ module
$$\mathcal{H}_+ :=\varprojlim_k H^*(\polyv{k}^{*,*}(X)[[t]], \pdb + t \bvd{} + [\varphi,\cdot]),$$
which is equipped with the Gauss-Manin connection $\gmc{}$ together with a $\gmc{}$-flat pairing $\langle \cdot, \cdot \rangle$ constructed from $\prescript{0}{}{\mathsf{p}}$. 
	
\begin{theorem}[=Theorem \ref{thm:construction_of_Frobenius_manifold}]\label{thm:intro_thm_2}
The triple $(\mathcal{H}_+,\gmc{},\langle \cdot,\cdot \rangle)$ is a semi-infinite log variation of Hodge structures in the sense of Definition \ref{def:log_semi_infinite_VHS}.
Under Assumption \ref{assum:weighted_filtration_assumption}, we can construct an opposite filtration $\mathcal{H}_-$ to the Hodge bundle $\mathcal{H}_+$ in the sense of Definition \ref{def:opposite_filtration}. Furthermore, there exists a versal solution to the Maurer-Cartan equation \eqref{eqn:maurer_cartan_introduction} such that $ e^{\varphi /t}$ gives a miniversal section of the Hodge bundle in the sense of Definition \ref{def:miniversal_section}.
Finally, under Assumption \ref{assum:poincare_pairing_assumption}, there exists a structure of logarithmic Frobenius manifold on the formal neighborhood $\text{Spf}(\comp[[Q]])$ of $X$ in the extended moduli space constructed from these data.
\end{theorem}

\subsubsection{Geometric applications (Examples \& \S \ref{sec:application_to_gross_siebert_program})}
In a running example throughout the paper and the last section, we explain how to apply our results to the geometric settings studied by Friedman \cite{Friedman83} and Kawamata-Namikawa \cite{kawamata1994logarithmic} (the $d$-semistable log smooth case in a running example starting from Example \ref{ex:log-smooth-I}) and Kontsevich-Soibelman \cite{kontsevich-soibelman04} and Gross-Siebert \cite{gross2011real} (the maximally degenerate case in \S \ref{sec:application_to_gross_siebert_program}). In both cases, there is a Stein cover $(V_\alpha)_\alpha$ of $X$ together with a local {\em toric} thickening $\mathbf{V}_\alpha$ of each $V_\alpha$ over $\text{Spec}(\comp[Q])$:
$$
\xymatrix@1{ V_{\alpha} \  \ar@{^{(}->}[rr] \ar[d] & & \mathbf{V}_{\alpha} \ar[d]^{\pi}\\
\text{Spec}(\comp) \  \ar@{^{(}->}[rr] & & \text{Spec}(\comp[Q])
}
$$
These serve as local models for the smoothing of $X$.

Let $Z \subset X$ be the codimension $2$ singular locus of the log-structure of $X$ and write the inclusion of the smooth locus as $j: X\setminus Z \rightarrow X$. 

We take $\bva{k}^*_{\alpha}$ and $\tbva{k}{}^*_{\alpha}$ as the push-forwards by $j$ of the sheaf of relative log polyvector fields and the sheaf of total log holomorphic de Rham complex respectively. The higher order patching data $\patch{k}_{\alpha\beta,i}$ come from uniqueness of the local models near a point in $X$. These data fit into our algebraic framework. Also, both freeness of the Hodge bundle (Assumption \ref{assum:local_assumption_for_triviality_of_hodge_bundle}) and Hodge-to-de Rham degeneracy (Assumption \ref{assum:Hodge_de_rham_degeneracy}) hold: see \cite[Lemma 4.1]{kawamata1994logarithmic} for the log smooth case and \cite[Theorems 3.26 \& 4.1]{Gross-Siebert-logII} for the maximally degenerate case.\footnote{\label{fn:gap-gross-siebert} There was a gap in the proof of \cite[Theorem 4.1]{Gross-Siebert-logII} discovered by Felten-Filip-Ruddat, but it was filled by them in \cite{Felten-Filip-Ruddat}; see \cite[Theorem 1.10]{Felten-Filip-Ruddat} for details.} Therefore, we obtain the following corollary.
	
\begin{corollary}[see Corollaries \ref{cor:Kawamata-Namikawa} and \ref{prop:gs_unobstructedness}]
In both the log smooth and maximally degenerate cases, the complex analytic space $(X,\mathcal{O}_X)$ is smoothable, i.e. there exists a $k^{\text{th}}$-order thickening $(\prescript{k}{}{X},\prescript{k}{}{\mathcal{O}})$ over $\logsk{k}$ locally modeled on $\prescript{k}{}{\mathbf{V}}_{\alpha}$ for each $k \in \inte_{\geq 0}$, and these thickenings are compatible.
\end{corollary}

For the construction of log Frobenius manifold structures on the extended moduli spaces in these cases, see Corollaries \ref{cor:Kawamata-Namikawa-2} and \ref{cor:Gross-Siebert_log_Frobenius}.


As mentioned above, the recent work \cite{Felten-Filip-Ruddat} by Felten-Filip-Ruddat has shown that our results are applicable to the much more general class of toriodal crossing spaces. See also \cite{Chan-Ma19, Chan-Ma-Suen20} for extension of this algebraic framework to smoothing of pairs.




\section*{Acknowledgement}
We would like to thank Mark Gross, Donatella Iacono, Si Li, Grigory Mikhalkin, Helge Ruddat, Taro Sano, Mao Sheng, Chin-Lung Wang and Jeng-Daw Yu and the anonymous referees for very useful discussions, comments and suggestions. In particular, we thank Mark Gross for informing us the recent work \cite{Felten-Filip-Ruddat} and connecting us to the authors of that paper. We are also grateful to Simon Felten, Matej Filip and Helge Ruddat for supportive communications, and for very carefully reading an earlier draft of this paper and providing us a list of constructive suggestions and comments.

K. Chan was supported by a grant of the Hong Kong Research Grants Council (Project No. CUHK14302617) and direct grants from CUHK.
N. C. Leung was supported by grants of the Hong Kong Research Grants Council (Project No. CUHK14303516 \& CUHK14301117) and direct grants from CUHK.
Z. N. Ma was partially supported by the Institute of Mathematical Sciences (IMS) and the Department of Mathematics at The Chinese University of Hong Kong.
\section*{Notation Summary}

\begin{notation}\label{not:universal_monoid}
	We fix a rank $s$ lattice $\blat$ together with a strictly convex $s$-dimensional rational polyhedral cone $Q_\real \subset \blat_\real:= \blat\otimes_\inte \real$. We let $Q := Q_\real \cap \blat$ and call it the {\em universal monoid}. We consider the ring $\cfr:=\comp[Q]$ and write a monomial element as $q^m \in \cfr$ for $m \in Q$, and consider the maximal ideal given by $\mathbf{m}:= \comp[Q\setminus \{0\}]$. 
	We consider the Artinian ring $\cfrk{k}:= \cfr / \mathbf{m}^{k+1}$ and the completion $\hat{\cfr}:= \varprojlim_{k} \cfrk{k}$ of $\cfr$. We further equip $\cfr$, $\cfrk{k}$ and $\hat{\cfr}$ with the natural monoid homomorphism $Q \rightarrow \cfr$, $m \mapsto q^m$, giving them the structure of a {\em log ring} (see \cite[Definition 2.11]{gross2011real}); the corresponding log spaces will be denoted as $\logs$, $\logsk{k}$ and $\logsf$ respectively.
	
	Furthermore, we let $\logsdrk{}{*} := \cfr \otimes_{\comp} \bigwedge^*\blat_{\comp}$, $\logsdrk{k}{*}:=   \cfrk{k} \otimes_{\comp} \bigwedge^*\blat_{\comp}$ and $\logsdrf{*} := \hat{\cfr} \otimes_{\comp} \bigwedge^*\blat_{\comp}$ (here $\blat_\comp = \blat \otimes_\inte \comp$) be the spaces of log de Rham differentials on $\logs$, $\logsk{k}$ and $\logsf$ respectively. We write $1 \otimes m$ as $d \log q^m$ for $m \in \blat$, and these spaces are equipped with the de Rham differential $\partial$ satisfying $\partial(q^m) = q^m d\log q^m$. We also denote by $\logsvfk{}:= \cfr \otimes_{\comp} \blat_{\comp}^{\vee}$, $\logsvfk{}$ and $\logsvff$, respectively, the spaces of log derivations, which are equipped with a natural Lie bracket $[\cdot,\cdot]$. We write $1\otimes n$ as $\partial_n$ with action $\partial_n (q^m) = (m,n) q^m$, where $(m,n)$ is the natural pairing between $\blat_\comp$ and $\blat^{\vee}_\comp$.
\end{notation}

For a $\inte^2$-graded vector space $V^{*,*} = \bigoplus_{p,q}V^{p,q}$, we write $V^{k} = \bigoplus_{p+q= k} V^{p,q}$, and $V^{*} = \bigoplus_{k} V^{k}$ if we only care about the total degree. We also simply write $V$ if we do not need the grading.

Throughout this paper, we are dealing with two \v{C}ech covers $\mathcal{V} = (V_\alpha)_{\alpha}$,  $\mathcal{U} = (U_i)_{i \in \inte_+}$ and also $k^{\text{th}}$-order thickenings at the same time, so we will adapt the following (rather unusual) notational convention:
The top left corner in $\prescript{k}{}{\spadesuit}$ refers to the order of $\spadesuit$. 
The bottom left corner in $\prescript{}{\bullet}{\spadesuit}$ stands for something constructed from the Koszul filtration on $\tbva{}{\bullet}_{\alpha}$'s (as in Definitions \ref{def:0_order_data} and \ref{def:higher_order_data_module}), where $\bullet$ can be $r$, $r_1:r_2$ or $\parallel$ (meaning relative forms). 
The bottom right corner is reserved for the \v{C}ech indices; we write $\spadesuit_{\alpha_0 \cdots \alpha_{\ell}}$ for the \v{C}ech indices of $\mathcal{V}$ and $\spadesuit_{i_0 \cdots i_l}$ for the \v{C}ech indices of $\mathcal{U}$, and if they appear at the same time, we write $\spadesuit_{\alpha_0 \cdots \alpha_{\ell},i_0 \cdots i_l}$. 
\section{The abstract algebraic setup}\label{sec:abstract_framework}

\subsection{BV algebras and modules}\label{sec:abstract_bv_algebra_1}
\begin{definition}\label{def:abstract_BV_algebra}
	A {\em graded Batalin-Vilkovisky (abbrev. BV) algebra} is a unital $\inte$-graded commutative $\comp$-algebra $(V^{*},\wedge)$ together with a degree $1$ operator $\Delta$ such that $\Delta(1)=0$, $\Delta^2=0$ and the operator $\delta_v : V^{*} \rightarrow V^{*+|v|+1}$ defined by
	$
	\delta_v(w) : = \Delta(v\wedge w ) - \Delta(v) \wedge w -(-1)^{|v|} v\wedge \Delta(w)
	$
	is a derivation of degree $|v|+1$ for any homogeneous element $v \in V^*$ (here $|v|$ denotes the degree of $v$).
\end{definition}

\begin{definition}\label{def:dgBV}
	A {\em differential graded Batalin-Vilkovisky (abbrev. dgBV) algebra} is a graded BV algebra $(V^*,\wedge,\Delta)$ together with a degree $1$ operator $\pdb$ satisfying
	$$
	\pdb(\alpha \wedge \beta) = (\pdb \alpha) \wedge \beta + (-1)^{|\alpha|} \alpha \wedge (\pdb \beta) ,\quad
	\pdb^2 = \pdb \bvd{} + \bvd{} \pdb = 0. 
	$$
\end{definition}

\begin{definition}\label{def:dgla}
	
	A {\em differential graded Lie algebra (abbrev. dgLa)} is a triple $(L^* , d , [\cdot,\cdot])$, where $L = \bigoplus_{i \in \inte} L^i$, $[\cdot,\cdot] : L^{*} \otimes L^* \rightarrow L^*$ is a graded skew-symmetric pairing satisfying the Jacobi identity $[a,[b,c]] + (-1)^{|a||b|+|a||c|} [b,[c,a]] + (-1)^{|a||c| + |b| |c|}[ c,[a,b]] = 0$ for homogeneous elements $a,b,c \in L^*$, and $d: L^{*} \rightarrow L^{*+1}$ is a degree $1$ differential satisfying $d^2 =0$ and the Leibniz rule
	$
	d[a,b] = [da,b] + (-1)^{|a|} [a,db]
	$
	for homogeneous elements $a, b \in L^*$.
\end{definition}

Given a BV algebra $(V^*,\wedge,\Delta)$, the map $[\cdot,\cdot] : V \otimes V \rightarrow V$ defined by $[v,w] = (-1)^{|v|}\delta_v(w)$ is called the {\em associated Lie bracket}.\footnote{For polyvector fields on a Calabi-Yau manifold, we have $[\cdot,\cdot] = -[\cdot,\cdot]_{sn}$, where $[\cdot,\cdot]_{sn}$ is the Schouten-Nijenhuis bracket; see e.g. \cite[\S 6.A]{huybrechts2006complex}.} Using this bracket, the triple $(V^*[-1], \Delta, [\cdot,\cdot])$ forms a dgLa.

\begin{notation}\label{not:exponential_group}
	Given a {\em nilpotent} graded Lie algebra $L^*$, we define a product $\bchprod$ by the Baker-Campbell-Hausdorff formula:
	$
	v \bchprod w := v  + w + \half[v,w] + \cdots
	$
	for $v,w \in V^*$. The pair $(L^*, \bchprod)$ is called the {\em  exponential group of $L^*$} and is denoted by $\exp(L^*)$.
\end{notation}

\begin{lemma}[see e.g. \S 1 in \cite{manetti2005differential}]\label{lem:gauge_action_on_differential}
	For a dgLa $(L^*,d , [\cdot,\cdot])$, we consider the endomorphism $\text{ad}_{\vartheta}:= [\vartheta,\cdot]$ for an element $\vartheta \in L^0$ such that $\text{ad}_{\vartheta}$ is nilpotent. Then we have the formula 
	$$
	e^{\text{ad}_{\vartheta}} (d + [\xi,-]) e^{- \text{ad}_{\vartheta}} = d  + [e^{ \text{ad}_{\vartheta}}(\xi), \cdot]- [\frac{e^{\text{ad}_\vartheta} -1}{\text{ad}_{\vartheta}} (d\vartheta), \cdot]
	$$
	for $\xi \in L^*$.
	For a nilpotent element $\vartheta \in L^0$, we define the gauge action
	$$
	\exp(\vartheta) \star \xi:=  e^{\text{ad}_{\vartheta}}(\xi)- \frac{e^{\text{ad}_\vartheta} -1}{\text{ad}_{\vartheta}} (d\vartheta)
	$$
	for $\xi \in L^*$.
	Then we have $\exp(\vartheta_1)  \star \left( \exp(\vartheta_2) \star \xi \right) = \exp(\vartheta_1 \bchprod \vartheta_2) \star \xi$, where $\bchprod$ is the Baker-Campbell-Hausdorff product as in Notation \ref{not:exponential_group}.
\end{lemma}

\begin{definition}[see e.g. \cite{kowalzig2015gerstenhaber}]\label{def:abstract_BV_module}
	A {BV module} $(M^*,\partial)$ over a BV algebra $(V^{*},\wedge, \Delta)$ is a complex of $\comp$-vector spaces equipped with a degree $1$ differential $\partial$ and a graded action by $(V^*,\wedge)$, denoted as $\iota_v=v \lrcorner : M^* \rightarrow M^{*+|v|}$ (for a homogeneous element $v\in V^*$) and called the {\em interior multiplication or contraction by $v$}, 
		such that if we let 
		$
		(-1)^{|v|}\mathcal{L}_v := [\partial, v \lrcorner] := \partial \circ (v\lrcorner )- (-1)^{|v|} (v\lrcorner) \circ \partial
		$, 
		where $[\cdot,\cdot]$ is the graded commutator for operators, then $[\mathcal{L}_{v_1}, v_2 \lrcorner] = [v_1,v_2] \lrcorner$. 
\end{definition}

Given a BV module $(M^*, \partial)$ over $(V^{*},\wedge, \Delta)$, we have $[\partial, \mathcal{L}_v]  = 0$ and $\mathcal{L}_{[v_1,v_2]} = \{ \mathcal{L}_{v_1}, \mathcal{L}_{v_2}\}$, where $\{\cdot,\cdot\}$ stands for graded commutator for operators, and 
$
\mathcal{L}_{v_1 \wedge v_2}  = (-1)^{|v_2|} \mathcal{L}_{v_1} \circ (v_2 \lrcorner)+(v_1 \lrcorner) \circ \mathcal{L}_{v_2}.
$


\begin{definition}\label{def:abstract_derham_module}
	A BV module $(M^*, \partial)$ over $(V^{*},\wedge, \Delta)$ is called a {\em de Rham module} if there is a unital differential graded algebra (abbrev. dga) structure $(M^*, \wedge, \partial)$ such that
	$v \lrcorner (w_1 \wedge w_2) = (v \lrcorner w_1) \wedge w_2 + (-1)^{|w_1|} w_1 \wedge (v \lrcorner w_2)$ for $v \in V^{-1}$ (i.e. $v \lrcorner $ acts as a derivation)
	and $v \lrcorner (w_1 \wedge w_2) = (v \lrcorner w_1) \wedge w_2 = w_1 \wedge (v \lrcorner w_2)$ for $v \in V^{0}$.
	If in addition there is a finite decreasing filtration of BV submodules $ \{0\} = \lscript{}{N}{M}^*\subset  \cdots \subset \lscript{}{r}{M}^* \subset \cdots\subset \lscript{}{0}{M}^* = M^*$, then we call it a {\em filtered de Rham module}.\footnote{This is motivated by the de Rham complex equipped with the Koszul filtration associated to a family of varieties; see e.g. \cite[Chapter 10.4]{peters2008mixed}.}
\end{definition}

Given a de Rham module $(M^*, \partial)$ over $(V^{*},\wedge, \Delta)$, it is easy to check that for $v \in V^{-1}$, $\mathcal{L}_v$ acts as a derivation, i.e.
$\mathcal{L}_v (w_1 \wedge w_2) = (\mathcal{L}_v w_1) \wedge w_2 + w_1 \wedge (\mathcal{L}_v w_2)$.

\begin{lemma}\label{lem:BV_gauge_action_computation}
	Given a BV algebra $(V^*, \wedge, \Delta)$ acting on a BV module $(M^*,\partial)$, both with bounded degree, together with 
	an element $v \in V^{-1}$ such that the operator $v \wedge$ is nilpotent and 
	an element $\volf{}$ such that $\partial \volf{}= 0$ and satisfying $\bvd{}(\alpha) \lrcorner \volf{} = \partial(\alpha \lrcorner \volf{})$, we have the following identities
	\begin{align*}
	\exp([\Delta, v\wedge]) (1) &= \exp\Big(\sum_{k=0}^{\infty} \frac{\delta_v^k }{(k+1)!}(\Delta v)\Big),\\
	\exp([\partial, v \lrcorner]) \volf{}  &= \exp\Big(\sum_{k=0}^{\infty} \frac{\delta_v^k }{(k+1)!}(\Delta v)\Big) \lrcorner \volf{},
	\end{align*}
	where $\delta_v$ is the operator defined in Definition \ref{def:abstract_BV_algebra}. 
	\end{lemma}

\begin{proof}
	To prove the first identity, notice that $[\Delta, v\wedge ] = \delta_v + (\Delta v )\wedge $ and
	\begin{multline*}
	\exp\left(\sum_{k=0}^{\infty} \frac{\delta_v^k }{(k+1)!}(\Delta v)\right) 
	=\\
	1 +\sum_{m  \geq 1} \sum_{\substack{0\leq k_1 < \cdots < k_m,\\ s_1, \cdots, s_m > 0}} \frac{(\frac{\delta_v^{k_1} }{(k_1+1)!} (\Delta v))^{s_1} \cdots (\frac{\delta_v^{k_m} }{(k_m+1)!} (\Delta v))^{s_m} }{(s_1!) \cdots (s_m !)}
	\end{multline*}
	So it suffices to establish the equality
	\begin{multline*}
	\frac{(\delta_v + (\Delta v )\wedge)^L}{L!} (1) =\\
	 \sum_{\substack{0\leq k_1 < \cdots < k_m,\ s_1, \cdots, s_m > 0: \\ (k_1+1)s_1 + \cdots (k_m +1) s_m = L}} \frac{\left(\frac{\delta_v^{k_1} }{(k_1+1)!} (\Delta v)\right)^{s_1} \cdots \left(\frac{\delta_v^{k_m} }{(k_m+1)!} (\Delta v)\right)^{s_m} }{(s_1!) \cdots (s_m!)},
	\end{multline*}
	which can be proven by induction on $L$. Essentially the same proof gives the second identity.
\end{proof}

\subsection{The $0^{\text{th}}$-order data}\label{sec:0_order_data}

Let $(X,\mathcal{O}_X)$ be a $d$-dimensional compact complex analytic space. 


\begin{definition}\label{def:0_order_data}
	A {\em $0^{\text{th}}$-order datum over $X$} consists of:
	\begin{itemize}
		\item  a coherent sheaf of graded BV algebras $(\bva{0}^{*},[\cdot,\cdot],\wedge, \bvd{0})$ over $X$ (with $ -d \leq * \leq 0$), called the {\em $0^{\text{th}}$-order complex of polyvector fields}, such that $\bva{0}^0  = \mathcal{O}_X$ and the natural Lie algebra morphism $\bva{0}^{-1} \rightarrow \text{Der}(\mathcal{O}_X)$, $v \mapsto [v,\cdot]$ is injective,
		
		 \item  a coherent sheaf of dga's $(\tbva{0}{}^*, \wedge, \dpartial{0})$ over $X$ (with $0 \leq * \leq d+s$) endowed with a dg module structure over the dga $\logsdrk{0}{*}$, called the {\em $0^{\text{th}}$-order de Rham complex}, and equipped with the natural filtration $\tbva{0}{\bullet}^*$ defined by $\tbva{0}{s}^*:= \logsdrk{0}{\geq s} \wedge \tbva{0}{}^*$ (here $\wedge$ denotes the dga action),
		 
		 
		 \item a de Rham module structure on $\tbva{0}{}^*$ over $\bva{0}^*$ such that $[\varphi \lrcorner, \alpha\wedge] = 0$ for any $\varphi \in \bva{0}^*$ and $\alpha \in \logsdrk{0}{*}$, and 
		 
		 \item an element $\volf{0} \in \Gamma(X,\tbva{0}{0}^d/ \tbva{0}{1}^d)$ with $\dpartial{0}(\volf{0}) = 0$, called the {\em $0^{\text{th}}$-order volume element}
	\end{itemize}
	such that
	\begin{enumerate}
		\item the map $ \lrcorner \volf{0}: (\bva{0}^{*}[-d], \bvd{0}) \rightarrow (\tbva{0}{0}^*/ \tbva{0}{1}^*, \dpartial{0})$ is an isomorphism, and
		\item the map $\gmiso{0}{r}^{-1} :\logsdrk{0}{r} \otimes_\comp (\tbva{0}{0}^*/ \tbva{0}{1}^*[-r]) \rightarrow \tbva{0}{r}^*/ \tbva{0}{r+1}^*$, given by taking wedge product by $\logsdrk{0}{r}$ (here $[-r]$ is the upshift of the complex by degree $r$), is also an isomorphism.
	\end{enumerate}
\end{definition}

Note that $(\tbva{0}{}^*, \wedge, \dpartial{0})$ is a filtered de Rham module over $\bva{0}^*$ using the filtration $\tbva{0}{\bullet}^*$, and the map $\gmiso{0}{r}^{-1}$ is an isomorphism of BV modules. We write $(\rdr{0}{}^* , \dpartial{0}) := (\tbva{0}{0}^*/ \tbva{0}{1}^*, \dpartial{0})$ and $\gmiso{0}{r} := \left(\gmiso{0}{r}^{-1} \right)^{-1}$.

\begin{example}\label{ex:log-smooth-I}
	We will use the $d$-semistable log smooth case \cite{kawamata1994logarithmic} (or simply, the log smooth case) as the running example throughout this paper.
	Following \cite[\S 2]{kawamata1994logarithmic}, we take a projective $d$-dimensional simple normal crossing variety $(X,\mathcal{O}_X)$.\footnote{In a discussion with T. Sano, we realised that the projectivity assumption was not necessary for studying smoothing of $X$, as indicated in \cite{felten2020logarithmic, Sano21}.} Let $Q = \mathbb{N}^s$ and write $\cfr = \comp[[t_1,\dots,t_s]]$, where $s$ is the number of connected components of $D = \cup_{i=1}^s D_i := \text{Sing}(X)$. There exists a log structure on $X$ over the $Q$-log point $\logsk{0}$, making it a $d$-semistable log variety $X^\dagger$ over $\logsk{0}$; we further require it to be log Calabi-Yau. 
	In this case, 
	\begin{itemize}
		\item
		the $0^{\text{th}}$-order complex of polyvector fields is given by the analytic sheaf of relative log polyvector fields $\bva{0}^* := \bigwedge^{-*} \Theta_{X^{\dagger}/\logsk{0}}$ equipped with the natural product structure;\footnote{In \cite{kawamata1994logarithmic}, the sheaf of relative log derivations was denoted as $T_{\mathscr{X}/\mathscr{A}}(\log)$.}
		
		\item 
		the $0^{\text{th}}$-order de Rham complex is given by the analytic sheaf of total log differential forms $\tbva{0}{}^* := \Omega^*_{X^{\dagger}/\comp}$,\footnote{In \cite{kawamata1994logarithmic}, the sheaf of total log differential forms was denoted as $\Omega^*_{\mathscr{X}/\mathbb{C}}(\log)$.} which is a locally free sheaf (in particular coherent) of dga's, and equipped with a natural dga structure over $\logsdrk{0}{*}$ inducing the filtration in Definition \ref{def:0_order_data};
		
		\item
		the volume element $\volf{0}$ is given via the trivialization $\Omega^d_{X^{\dagger}/\logsk{0}} \cong \mathcal{O}_X$ coming from the Calabi-Yau condition, and then $\bva{0}^*$ is equipped with the BV operator defined by $\bvd{0}(\varphi) \lrcorner \volf{0} := \dpartial{0} (\varphi \lrcorner \volf{0})$.
	\end{itemize}
	These data satisfies all the conditions in Definition \ref{def:0_order_data}. For example, the map $\gmiso{0}{r} : \logsdrk{0}{r} \otimes_\comp (\tbva{0}{0}^*/ \tbva{0}{1}^*[-r] ) \rightarrow  \tbva{0}{r}^*/ \tbva{0}{r+1}^*$ given by taking wedge product in $\Omega^*_{X^{\dagger}/\comp}$ is an isomorphism of sheaves of BV modules.
\end{example}

Definition \ref{def:0_order_data} is actually an extraction of the abstract properties of the sheaves in Example \ref{ex:log-smooth-I} that are necessary for our construction and such that they hold even beyond the case of log smooth Calabi-Yau varieties; see \S \ref{sec:application_to_gross_siebert_program} for definition of the $0^{\text{th}}$-order datum in the maximally degenerate case.

Also note that our filtration $\tbva{0}{\bullet}^*$ is motivated by the Koszul filtration from the variation of Hodge structures (see e.g. \cite{peters2008mixed}) so that $\tbva{0}{\parallel}^*$ is the sheaf $\tbva{0}{}^* := \Omega^*_{X^{\dagger}/\logsk{0}}$ of relative log differential forms over $\logsk{0}$.

Now we consider the hypercohomology $\mathbb{H}^*(\rdr{0}{}^*,\dpartial{0})$ of the complex of sheaves $(\rdr{0}{}^*,\dpartial{0})$.

\begin{definition}\label{def:hodge_filtration}
	For each $r \in \frac{1}{2}\mathbb{Z}$, let $\mathcal{F}^{\geq r}\mathbb{H}^l$ be the image of the linear map $ \mathbb{H}^l(\rdr{0}{}^{\geq p},\dpartial{0}) \rightarrow \mathbb{H}^l(\rdr{0}{}^*,\dpartial{0})$, where $p$ is the smallest integer such that $2p \geq 2r + l - d$.\footnote{We follow Barannikov \cite{Barannikov99} for the convention on the index $r$ of the Hodge filtration, which differs from the usual one by a shift. Also, we usually write $\mathcal{F}^{\geq r}$, instead of $\mathcal{F}^{\geq r} \mathbb{H}^*$, when there is no confusion.}
	Then
	$$
	0 \subset \mathcal{F}^{\geq d} \subset \mathcal{F}^{\geq d-\half} \subset \cdots \subset \mathcal{F}^{\geq r} \subset \cdots \subset \mathcal{F}^{\geq 0} = \mathbb{H}^*(\rdr{0}{}^*,\dpartial{0})
	$$
	is called the {\em Hodge filtration}.
\end{definition}

We have the following exact sequence of sheaves from Definition \ref{def:0_order_data}
\begin{equation}\label{eqn:0_order_exact_sequence}
0 \rightarrow \logsdrk{0}{1} \otimes \rdr{0}{}^*[-1] \cong \tbva{0}{1}^*/\tbva{0}{2}^* \rightarrow \tbva{0}{0}^*/\tbva{0}{2}^* \rightarrow \rdr{0}{}^* \cong   \tbva{0}{0}^*/\tbva{0}{1}^*  \rightarrow 0
\end{equation}

\begin{definition}\label{def:0_th_order_GM_connection}
Take the long exact sequence associated to the hypercohomology of \eqref{eqn:0_order_exact_sequence}, we obtain the {\em $0^{\text{th}}$-order Gauss-Manin (abbrev. GM) connection}:
\begin{equation}\label{eqn:GM_connecting_homomorphism}
\gmc{0}: \mathbb{H}^*(\rdr{0}{}^*,\dpartial{0}) \rightarrow \logsdrk{0}{1} \otimes \mathbb{H}^*(\rdr{0}{}^*,\dpartial{0}).
\end{equation}
\end{definition}
Note that the $0^{\text{th}}$-order GM connection is actually the residue of the usual GM connection.

\begin{prop}\label{prop:0_th_order_griffith_transversality}
	Griffith's transversality holds for $\gmc{0}$, i.e. 
	$$\gmc{0} (\mathcal{F}^{\geq r}) \subset \logsdrk{0}{1} \otimes \mathcal{F}^{\geq r-1}.$$ 
	\end{prop}
The proof of this is standard; see e.g., \cite[Corollary 10.31]{peters2008mixed}. 
With $[\volf{0}] \in \mathcal{F}^{\geq d}\mathbb{H}^0$ and Griffith's transversality, we obtain the {\em $0^\text{th}$-order Kodaira-Spencer map} $\gmc{0}([\volf{0}]): \logsdrk{0}{1} \rightarrow \mathcal{F}^{\geq d-1}\mathbb{H}^0$.

\subsection{The higher order data}\label{sec:higher_order_data}

We fix an open cover $\mathcal{V}$ of $X$ which consists of Stein open subsets $V_{\alpha} \subset X$. We consider the following local thickening datum in terms of local thickening of the sheaf $\bva{0}^*$ as $\bva{k}^*_{\alpha}$ on $V_{\alpha}$. 

\begin{definition}\label{def:higher_order_data}
	A {\em local thickening datum of the complex of polyvector fields} (with respect to $\mathcal{V}$) consists of, for each $k \in \inte_{\geq 0}$ and $V_\alpha \in \mathcal{V}$,
	\begin{itemize}
		\item a coherent sheaf of BV algebras $(\bva{k}_{\alpha}^*,[\cdot,\cdot], \wedge, \bvd{k}_{\alpha})$
		over $V_\alpha$ such that $\bva{k}_{\alpha}^*$ is also a sheaf of algebras over $\cfrk{k}$ so that $[\cdot,\cdot], \wedge$ are $\cfrk{k}$-bilinear and $\bvd{k}_{\alpha}$ is $\cfrk{k}$-linear, and
		
		\item a surjective morphism of sheaves of BV algebras $\rest{k+1,k}_{\alpha}:\bva{k+1}_{\alpha}^* \rightarrow \bva{k}_{\alpha}^*$ which is $\cfrk{k+1}$-linear and induces a sheaf isomorphism upon tensoring with $\cfrk{k}$
	\end{itemize}
	satisfying the following conditions
	\begin{enumerate}
		\item $(\bva{0}_{\alpha}^*,[\cdot,\cdot], \wedge, \bvd{0}_{\alpha}) = (\bva{0}^*,[\cdot,\cdot], \wedge, \bvd{0})|_{V_\alpha}$,
		
		\item $\bva{k}_{\alpha}^*$ is flat over $\cfrk{k}$, i.e. the stalk $(\bva{k}_{\alpha}^*)_x$ is flat over $\cfrk{k}$ for any $x \in V_\alpha$, and
		
		\item the natural Lie algebra morphism $\bva{k}_{\alpha}^{-1} \rightarrow \text{Der}(\bva{k}_{\alpha}^0)$ is injective.
	\end{enumerate} 
\end{definition}

We write $\rest{k,l}_{\alpha}:= \rest{l+1,l}_{\alpha}\circ\cdots\circ\rest{k,k-1}_{\alpha} : \bva{k}_{\alpha}^* \rightarrow \bva{l}_{\alpha}^*$ for every $k>l$, and $\rest{k,k}_{\alpha} \equiv \text{id}$. We also introduce the following notation: Given two elements $\mathfrak{a} \in \bva{k_1}_{\alpha}^*$, $\mathfrak{b} \in \bva{k_2}_{\alpha}^*$ and $l \leq \text{min}\{k_1,k_2\}$, we say that $\mathfrak{a} = \mathfrak{b} \ \text{(mod $\mathbf{m}^{l+1}$)}$ if and only if $\rest{k_1,l}_{\alpha}(\mathfrak{a}) = \rest{k_2,l}_{\alpha}(\mathfrak{b})$.

\begin{example}\label{ex:log-smooth-II}
	Continuing the log smooth Example \ref{ex:log-smooth-I}, every point $\bar{x} \in X^{\dagger}$ is covered by a log chart $V$ biholomorphic to an open neighborhood of $(0,\dots,0)$ in  $\{z_0 \cdots z_r = 0 \mid (z_0 ,\dots ,z_d) \in \comp^{d+1} \}$ \cite[\S 1]{kawamata1994logarithmic}. From log deformation theory \cite[\S 2]{kawamata1994logarithmic}, we obtain a smoothing $\mathbf{V}^{\dagger}$ of $V$ given by a neighborhood of $(0,\dots,0)$ in $\{z_0 \cdots z_r = s_i \mid (z_0 ,\dots ,z_d) \in \comp^{d+1} \}$ if $V \cap D_i \neq \emptyset$. We choose a Stein cover $\mathcal{V} =\{V_\alpha\}_{\alpha}$ of $X$ by such log charts together with a local smoothing $\mathbf{V}_{\alpha}^{\dagger}$ of each $V_\alpha$, and denote by $\prescript{k}{}{\mathbf{V}}_\alpha^{\dagger}$ the $k^{\text{th}}$ order thickening of the local model $V_{\alpha}$. Then the sheaf of $k^{\text{th}}$-order polyvector fields in Definition \ref{def:higher_order_data} is given by $\bva{k}^*_{\alpha} := \bigwedge^{-*} \Theta_{\prescript{k}{}{\mathbf{V}}_\alpha^{\dagger}/\logsk{k}}$ equipped with the natural product structure.
\end{example}

\begin{notation}\label{not:open_stein_covers}
	We fix, once and for all, another cover $\mathcal{U}$ of $X$ which consists of a countable collection of Stein open subsets $\mathcal{U} = \{U_i\}_{i \in \inte_+}$ forming a basis of topology (we refer readers to \cite[Chapter IX Theorem 2.13]{demailly1997complex} for the existence of such a cover). Note that an arbitrary finite intersection of Stein open subsets remains Stein. 
\end{notation}

We require two different local thickenings on $V_{\alpha}$ and $V_{\beta}$ isomorphic on some small enough $U_i \subset V_{\alpha} \cap V_{\beta}$ via a non-unique isomorphism $\patch{k}_{\alpha\beta,i}$ as follows. 

\begin{definition}\label{def:higher_order_patching}
	A {\em patching datum of the complex of polyvector fields} (with respect to $\mathcal{U}, \mathcal{V}$) consists of, for each $k \in \inte_{\geq 0}$ and triple $(U_i; V_\alpha, V_\beta)$ with $U_i \subset V_{\alpha\beta} := V_{\alpha} \cap V_{\beta}$,
	a sheaf isomorphism $\patch{k}_{\alpha\beta,i}:\bva{k}_{\alpha}^*|_{U_i} \to \bva{k}_{\beta}^*|_{U_i}$ over $\cfrk{k}$ preserving the structures $[\cdot,\cdot], \wedge$ and fitting into the diagram
	$$
	\xymatrix@1{ \bva{k}_{\alpha}^*|_{U_i} \ar[r]^{\patch{k}_{\alpha\beta,i}} \ar[d]^{\rest{k,0}_{\alpha}} &   \bva{k}_{\beta}^*|_{U_i} \ar[d]^{\rest{k,0}_{\beta}}\\
		\bva{0}^*|_{U_i} \ar@{=}[r] &  \bva{0}^*|_{U_i}, 
	}
	$$
	and an element $\bvdobs{k}_{\alpha\beta,i} \in \bva{k}_{\alpha}^0(U_i)$ with $\bvdobs{k}_{\alpha\beta,i} = 0 \ \text{(mod $\mathbf{m}$)}$ such that
	\begin{equation}\label{eqn:higher_order_bvd_different}
	\patch{k}_{\beta\alpha,i} \circ \bvd{k}_{\beta} \circ \patch{k}_{\alpha\beta,i} - \bvd{k}_{\alpha} = [\bvdobs{k}_{\alpha\beta,i}, \cdot]
	\end{equation}
	satisfying the following conditions: 
	\begin{enumerate}
		\item $\patch{k}_{\beta\alpha,i} = \patch{k}_{\alpha \beta,i}^{-1}$, $\patch{0}_{\alpha \beta,i} \equiv \text{id}$;
		\item for $k>l$ and $U_i \subset V_{\alpha\beta}$, there exists $\resta{k,l}_{\alpha\beta,i} \in \bva{l}_{\alpha}^{-1}(U_i)$  with $\resta{k,l}_{\alpha\beta,i} = 0 \ \text{(mod $\mathbf{m}$)}$ such that 
		\begin{equation}\label{eqn:bva_different_order_comparison}
		\patch{l}_{\beta\alpha,i} \circ \rest{k,l}_{\beta} \circ \patch{k}_{\alpha\beta,i} = \exp\left([\resta{k,l}_{\alpha\beta,i},\cdot]\right) \circ \rest{k,l}_{\alpha};
		\end{equation}
		
		\item for $k\in \inte_{\geq 0}$ and $U_i, U_j \subset V_{\alpha \beta}$, there exists $\patchij{k}_{\alpha\beta,ij} \in \bva{k}_{\alpha}^{-1}(U_i\cap U_j)$ with $\patchij{k}_{\alpha\beta,ij} = 0 \ \text{(mod $\mathbf{m}$)}$ such that
		\begin{equation}\label{eqn:higher_order_U_i_U_j_different}
		\left(\patch{k}_{ \beta \alpha,j}|_{U_i \cap U_j}\right)\circ \left(\patch{k}_{\alpha \beta,i}|_{U_i \cap U_j}\right) = \exp\left([\patchij{k}_{\alpha\beta,ij}, \cdot]\right); and
		\end{equation}
		
		\item for $k\in \inte_{\geq 0}$ and $U_i \subset V_{\alpha\beta\gamma}:=V_\alpha \cap V_\beta \cap V_\gamma$, there exists $\cocyobs{k}_{\alpha\beta\gamma,i} \in \bva{k}_{\alpha}^{-1}(U_i)$  with $\cocyobs{k}_{\alpha\beta\gamma,i} = 0 \ \text{(mod $\mathbf{m}$)}$ such that 
		\begin{equation}\label{eqn:higher_order_cocycle_different}
		\left(\patch{k}_{\gamma \alpha,i}|_{U_i}\right) \circ \left(\patch{k}_{\beta \gamma,i}|_{U_i}\right) \circ \left(\patch{k}_{\alpha \beta ,i}|_{U_i}\right) = \exp\left([\cocyobs{k}_{\alpha\beta\gamma,i},\cdot]\right).
		\end{equation}
	\end{enumerate}
	
\end{definition}

\begin{example}\label{ex:log-smooth-III}
	In geometric situations, the above patching datum is always induced from the local uniqueness of local thickening $\prescript{k}{}{\mathbf{V}}_{\alpha}$'s of $X^{\dagger}$.
	For example, in the log smooth case (continuing Examples \ref{ex:log-smooth-I} and \ref{ex:log-smooth-II}), \cite[Theorem 2.2]{kawamata1994logarithmic} says that for each $k \in \inte_{\geq 0}$ and triple $(U_i; V_\alpha, V_\beta)$ with $U_i \subset V_{\alpha\beta} := V_{\alpha} \cap V_{\beta}$, the two log deformations $\prescript{k}{}{\mathbf{V}}_{\alpha}$ and $\prescript{k}{}{\mathbf{V}}_{\beta}$ are isomorphic over $U_i$ via $\prescript{k}{}{\Psi}_{\alpha\beta,i}$. This induces the patching isomorphisms $\patch{k}_{\alpha\beta,i}:\bva{k}_{\alpha}^*|_{U_i} \to \bva{k}_{\beta}^*|_{U_i}$ in Definition \ref{def:higher_order_patching}.
	The existence of the log vector fields $\patchij{k}_{\alpha\beta,i}$'s, $\cocyobs{k}_{\alpha\beta\gamma,i}$'s and $\resta{k,l}_{\alpha\beta,i}$'s, which measure the incompatibilities of the patching isomorphisms $\patch{k}_{\alpha\beta,i}$'s, follows from the fact that any automorphism of a log deformation over $U_{i}$ or $U_{ij}$ comes from exponential action of vector fields. See \S \ref{sec:higher_order_patching_data_from_gross_siebert} for the maximally degenerate case.
\end{example}

\begin{lemma}\label{lem:patching_uniqueness_lemma}
	The elements $\resta{k,l}_{\alpha\beta,i}$'s, $\patchij{k}_{\alpha\beta,ij}$'s and $\cocyobs{k}_{\alpha\beta\gamma,i}$'s are uniquely determined by the patching isomorphisms $\patch{k}_{\alpha\beta,i}$'s.
\end{lemma}

\begin{proof}
	We just prove the statement for the elements $\patchij{k}_{\alpha\beta,ij}$'s as the other cases are similar. Suppose we have another set of elements $\reallywidetilde{\patchij{k}_{\alpha\beta,ij}}$'s satisfying \eqref{eqn:higher_order_U_i_U_j_different}, then we have
	$\exp([\patchij{k}_{\alpha\beta,ij} - \reallywidetilde{\patchij{k}_{\alpha\beta,ij}},\cdot]) \equiv \text{id}$
	as actions on $\bva{k}^0_\alpha(U_{ij})$ where $U_{ij} = U_i \cap U_j$. The result then follows from an order-by-order argument using the assumptions that $\rest{k,0}_\alpha(\patchij{k}_{\alpha\beta,ij} - \reallywidetilde{\patchij{k}_{\alpha\beta,ij}}) = 0$ and that the map $\bva{k}^{-1}_\alpha \to \text{Der}(\bva{k}^0_\alpha)$ is injective. 
\end{proof}

\begin{definition}\label{def:higher_order_data_module}
	A {\em local thickening datum of the de Rham complex} (with respect to $\mathcal{V}$) consists of, for each $k \in \inte_{\geq 0}$ and $V_\alpha \in \mathcal{V}$,
		\begin{itemize}
			\item a coherent sheaf of dgas $(\tbva{k}{}_{\alpha}^*,\wedge , \dpartial{k}_{\alpha})$ with a dg module structure over $\logsdrk{k}{*}$ equipped with the natural filtration  $\tbva{k}{\bullet}_{\alpha}^*$ defined by $\tbva{k}{s}_\alpha^*:= \logsdrk{k}{\geq s} \wedge \tbva{k}{}_\alpha^*[s]$,
			
			\item a de Rham module structure on $\tbva{k}{}_{\alpha}^*$ over $\bva{k}_{\alpha}^*$ such that $[\varphi \lrcorner , \alpha \wedge] = 0$ for any $\varphi \in \bva{k}^*_\alpha$ and $\alpha \in \logsdrk{k}{*}$,
			
			\item a surjective $\cfrk{k+1}$-linear morphism $\rest{k+1,k}_{\alpha}:\tbva{k+1}{}_{\alpha}^* \rightarrow \tbva{k}{}_{\alpha}^*$ inducing an isomorphism upon tensoring with $\cfrk{k}$ which is compatible with both $\rest{k+1,k}: \bva{k+1}_\alpha^* \rightarrow \bva{k}^*_\alpha$ and $\logsdrk{k+1}{*} \rightarrow \logsdrk{k}{*}$ under the contraction and dg actions respectively,\footnote{Here we abuse notations and use $\rest{k+1,k}_{\alpha}$ for both $\tbva{k+1}{}_{\alpha}^*$ and $\bva{k+1}_{\alpha}^*$.} and
			
			\item an element $\volf{k}_\alpha \in \Gamma(V_{\alpha},\tbva{k}{0}_{\alpha}^d/\tbva{k}{1}_{\alpha}^d)$ satisfying $\dpartial{k}_{\alpha}(\volf{k}_{\alpha}) = 0$ called {\em the local $k^{\text{th}}$-order volume element}
		\end{itemize}
	such that
		\begin{enumerate}
			\item $\tbva{k}{r}_{\alpha}^*$ is flat over $\cfrk{k}$ for $0\leq r\leq s$; 
			
			\item $\rest{k+1,k}(\volf{k+1}_\alpha) = \volf{k}_\alpha$;
			
			\item $(\tbva{0}{}_{\alpha}^*, \tbva{0}{\bullet}_\alpha^*,\wedge, \dpartial{0}_{\alpha}) = (\tbva{0}{}^*,\tbva{0}{\bullet}^*, \wedge, \dpartial{0})|_{V_\alpha}$ and $\volf{0}_\alpha = \volf{0}|_{V_\alpha}$;
			
			\item the map $\lrcorner \volf{k}_\alpha: (\bva{k}^{*}_\alpha[-d], \bvd{k}_\alpha) \rightarrow (\tbva{k}{0}^*_\alpha/ \tbva{k}{1}^*_\alpha , \dpartial{k}_\alpha)$ is an isomorphism, and
			
			\item the map 
			$\gmiso{k}{r}_\alpha^{-1} : \logsdrk{k}{r} \otimes_{\cfrk{k}} (\tbva{k}{0}^*_\alpha/ \tbva{k}{1}^*_\alpha[-r]) \rightarrow \tbva{k}{r}^*_\alpha/ \tbva{k}{r+1}^*_\alpha$, given by taking wedge product by $\logsdrk{k}{r}$, is also an isomorphism.
			
		\end{enumerate}
\end{definition}

Note that $\tbva{k}{}_{\alpha}^*$ is a filtered de Rham module over $\bva{k}_{\alpha}^*$ using the filtration $\tbva{k}{\bullet}_{\alpha}^*$. We write $\rdr{k}{}^*_\alpha :=  \tbva{k}{0}^*_\alpha/ \tbva{k}{1}^*_\alpha$ and $\gmiso{k}{r}_\alpha = (\gmiso{k}{r}_\alpha^{-1})^{-1}$.

We also write $\rest{k,l}_{\alpha}:= \rest{l+1,l}_{\alpha}\circ\cdots\circ\rest{k,k-1}_{\alpha}$ for every $k>l$ and $\rest{k,k}_{\alpha} \equiv \text{id}$, and introduce the following notation: 
Given two elements $\mathfrak{a} \in \tbva{k_1}{\bullet}_{\alpha}^*$, $\mathfrak{b} \in \tbva{k_2}{\bullet}_{\alpha}^*$ and $l \leq \text{min}\{k_1,k_2\}$, we say that $\mathfrak{a} = \mathfrak{b} \ \text{(mod $\mathbf{m}^{l+1}$)}$ if and only if $\rest{k_1,l}_{\alpha}(\mathfrak{a}) = \rest{k_2,l}_{\alpha}(\mathfrak{b})$.

From Definition \ref{def:higher_order_data_module}, we have the following diagram of BV modules
\begin{equation}
\xymatrix@1{ 
	0 \ar[r]& \blat \otimes_\inte \rdr{k+1}{}^*_\alpha[-1] \ar[r] \ar[d]^{\text{id}\otimes \rest{k+1,k}_\alpha } & \tbva{k+1}{0}^*_\alpha / \tbva{k+1}{2}^*_\alpha \ar[r] \ar[d]^{\rest{k+1,k}_\alpha}& \rdr{k+1}{}^*_\alpha \ar[r]  \ar[d]^{\rest{k+1,k}_\alpha}& 0\\
	0 \ar[r]& \blat \otimes_\inte \rdr{k}{}^*_\alpha[-1] \ar[r] & \tbva{k}{0}^*_\alpha / \tbva{k}{2}^*_\alpha \ar[r]& \rdr{k+1}{}^*_\alpha \ar[r]& 0.
}
\end{equation}

In general geometric situations, the sheaves $\tbva{k}{}^*_{\alpha}$'s are taken to be suitable sheaves of total logarithmic differential forms on $\prescript{k}{}{\mathbf{V}}_{\alpha}^*$ with a natural action by $\bva{k}^*_{\alpha}$ via contraction, and $\volf{k}_{\alpha}$ is taken to be a local lifting of the relative volume form $\volf{0}$ over $\logsk{k}$.
The local sheaves $\tbva{k}{}^*_{\alpha}$'s of differential forms are locally identified via the isomorphisms $\hpatch{k}_{\alpha\beta,i}$'s induced by the corresponding local uniqueness of local thickening $\prescript{k}{}{\mathbf{V}}_{\alpha}$'s of $V_{\alpha}$'s as in Example \ref{ex:log-smooth-III}. Therefore it is natural to require the compatibility between $\hpatch{k}_{\alpha\beta,i}$'s and the data $\patch{k}_{\alpha\beta,i}$'s, $\bvdobs{k}_{\alpha\beta,i}$'s, $\resta{k,l}_{\alpha\beta,i}$'s, $\patchij{k}_{\alpha\beta,ij}$'s and $\cocyobs{k}_{\alpha\beta\gamma,i}$'s as in the following Definition \ref{def:higher_order_module_patching}.

\begin{example}\label{ex:log-smooth-IV}
	In the log smooth case (continuing Examples \ref{ex:log-smooth-I}, \ref{ex:log-smooth-II} and \ref{ex:log-smooth-III}), the datum in Definition \ref{def:higher_order_data_module} can be chosen as follows: for each $k \in \inte_{\geq 0}$,
	\begin{itemize}
		\item
		the $k^{\text{th}}$-order de Rham complex is given by $\tbva{k}{}^*_{\alpha} := \Omega^*_{\prescript{k}{}{\mathbf{V}}_\alpha^{\dagger}/\comp}$;
		
		\item 
		the local $k^{\text{th}}$-order volume element is given by a lifting $\volf{}_{\alpha}$ of $\volf{0}$ as an element in $\Omega^*_{\mathbf{V}_\alpha^{\dagger}/\logs}$ and taking $\volf{k}_{\alpha} = \volf{}_{\alpha} \ (\text{mod $\mathbf{m}^{k+1}$})$, and then the BV operator $\bvd{k}_{\alpha}$ on $\bva{k}^*_{\alpha}$ is induced by the volume form $\volf{k}_{\alpha}$;
		
		\item
		the isomorphism $\gmiso{k}{r}$ of sheaves of BV modules is induced by taking wedge product as in Example \ref{ex:log-smooth-I}.
	\end{itemize}
	See \S \ref{sec:higher_order_deformation_data_from_Gross_siebert} for the maximally degenerate case.
\end{example}

\begin{definition}\label{def:higher_order_module_patching}
	A {\em patching datum of the de Rham complex} (with respect to $\mathcal{U}, \mathcal{V}$) consists of, for each $k \in \inte_{\geq 0}$ and triple $(U_i; V_\alpha, V_\beta)$ with $U_i \subset V_{\alpha\beta}$, a sheaf isomorphism $\hpatch{k}_{\alpha\beta,i}$ of dg modules over $\logsdrk{k}{*}$ such that it fits into the diagram
	$$
	\xymatrix@1{ \tbva{k}{\bullet}_{\alpha}^*|_{U_i} \ar[rr]^{\hpatch{k}_{\alpha\beta,i}} \ar[d]^{\rest{k,0}_{\alpha}}  & &   \tbva{k}{\bullet}_{\beta}^*|_{U_i} \ar[d]^{\rest{k,0}_{\beta}}\\
		\tbva{0}{\bullet}^*|_{U_i} \ar@{=}[rr] & &  \tbva{0}{\bullet}^*|_{U_i}, 
	}
	$$
	and satisfying the following conditions:
	\begin{enumerate}
		\item $\hpatch{k}_{\alpha\beta,i}$ is an isomorphism of de Rham modules meaning that the diagram
		\begin{equation}
		\xymatrix@1{ \bva{k}_\alpha^*|_{U_i} \ar@/^1pc/[r]^{\lrcorner} \ar[d]_{\patch{k}_{\alpha\beta,i}} &     \tbva{k}{\bullet}_{\alpha}^*|_{U_i} \ar[d]^{\hpatch{k}_{\alpha\beta,i}}\\
			\bva{k}_\beta^*|_{U_i}\ar@/^1pc/[r]^{\lrcorner}  &  \tbva{k}{\bullet}_{\beta}^*|_{U_i},
		}
		\end{equation}
		is commutative;
		
		\item $\hpatch{k}_{\beta\alpha,i} = \hpatch{k}_{\alpha \beta,i}^{-1}$, $\hpatch{0}_{\alpha \beta,i} \equiv \text{id}$;
		
		\item we have
		\begin{equation}\label{eqn:higher_order_volume_form_different}
		 \hpatch{k}_{\beta\alpha,i} (\volf{k}_\beta |_{U_i}) = \exp(\bvdobs{k}_{\alpha\beta,i} \lrcorner)(\volf{k}_\alpha|_{U_i}),
		\end{equation}
		where the elements $\bvdobs{k}_{\alpha\beta,i}$'s are as in Definition \ref{def:higher_order_patching};
		
		\item for $k>l$ and $U_i \subset V_{\alpha\beta}$, we have
		\begin{equation}
		\hpatch{l}_{\beta\alpha,i} \circ \rest{k,l}_{\beta} \circ \hpatch{k}_{\alpha\beta,i} = \exp\left(\mathcal{L}_{\resta{k,l}_{\alpha\beta,i}}\right) \circ \rest{k,l}_{\alpha},
		\end{equation}
		 where the elements $\resta{k,l}_{\alpha\beta,i}$'s are as in \eqref{eqn:bva_different_order_comparison};
		
		\item for $k\in \inte_{\geq 0}$ and $U_i, U_j \subset V_{\alpha \beta}$, we have
		\begin{equation}\label{eqn:module_higher_order_U_i_U_j_difference}
		\left(\hpatch{k}_{ \beta \alpha ,j}|_{U_i \cap U_j}\right)\circ \left(\hpatch{k}_{\alpha \beta,i}|_{U_i \cap U_j}\right) = \exp\left(\mathcal{L}_{\patchij{k}_{\alpha\beta,ij}}\right),
		\end{equation}
		where the elements $\patchij{k}_{\alpha\beta,ij}$'s are as in \eqref{eqn:higher_order_U_i_U_j_different}; and
		
		\item for $k\in \inte_{\geq 0}$ and $U_i \subset V_{\alpha\beta\gamma}$, we have
		\begin{equation}
		\left(\hpatch{k}_{\gamma \alpha,i}|_{U_i}\right) \circ \left(\hpatch{k}_{\beta \gamma,i}|_{U_i}\right) \circ \left(\hpatch{k}_{\alpha \beta,i }|_{U_i}\right) = \exp\left(\mathcal{L}_{\cocyobs{k}_{\alpha\beta\gamma,i}}\right),
		\end{equation}
		where the elements $\cocyobs{k}_{\alpha\beta\gamma,i}$'s are as in \eqref{eqn:higher_order_cocycle_different}.
	\end{enumerate}
	
\end{definition}

Since both $\volf{k}_{\alpha}$ and $\volf{k}_{\beta}$ are nowhere-vanishing, $\hpatch{k}_{\alpha\beta,i}$ actually determines $\bvdobs{k}_{\alpha\beta,i}$.
Also observe that for every $k\in \inte_{\geq 0}$ and any $U_i \subset V_{\alpha\beta}$, we have a commutative diagram
$$
\xymatrix@1{ 
	0 \ar[r]& \blat \otimes_\inte \left(\rdr{k}{}_{\alpha}^{*}[-1]\right)|_{U_i} \ar[r] \ar[d]^{\text{id}\otimes \hpatch{k}_{\alpha\beta,i} } & \left(\tbva{k}{0}_{\alpha}^*/\tbva{k}{2}_{\alpha}^*\right)|_{U_i} \ar[r] \ar[d]^{\hpatch{k}_{\alpha\beta,i}}& \rdr{k}{}_{\alpha}^{*}|_{U_i} \ar[r] \ar[d]^{\hpatch{k}_{\alpha\beta,i}}& 0\\
	0 \ar[r]& \blat \otimes_\inte \left(\rdr{k}{}_{\beta}^{*}[-1]\right)|_{U_i} \ar[r] & \left(\tbva{k}{0}_{\beta}^*/\tbva{k}{2}_{\beta}^*\right)|_{U_i} \ar[r]& \rdr{k}{}_{\beta}^{*}|_{U_i} \ar[r]& 0.
}
$$

\begin{example}\label{ex:log-smooth-V}
	In the log smooth case (continuing Examples \ref{ex:log-smooth-I}, \ref{ex:log-smooth-II}, \ref{ex:log-smooth-III} and \ref{ex:log-smooth-IV}), the isomorphism between two log deformations $\prescript{k}{}{\mathbf{V}}_{\alpha}$ and $\prescript{k}{}{\mathbf{V}}_{\beta}$ over $U_i$ via $\prescript{k}{}{\Psi}_{\alpha\beta,i}$ induces the patching isomorphisms $\hpatch{k}_{\alpha\beta,i}: \tbva{k}{\bullet}_{\alpha}^*|_{U_i} \to \tbva{k}{\bullet}_{\beta}^*|_{U_i}$ in Definition \ref{def:higher_order_module_patching}.
	The difference between volume elements is compared by $\prescript{k}{}{\Psi}_{\alpha\beta,i}^*(\volf{k}_{\beta})= \exp(\bvdobs{k}_{\alpha\beta,i} \lrcorner )\volf{k}_{\alpha}$ for some holomorphic function $\bvdobs{k}_{\alpha\beta,i}$.
	See \S \ref{sec:higher_order_deformation_data_from_Gross_siebert} for the maximally degenerate case.
\end{example}

\begin{remark}
We can deduce \eqref{eqn:higher_order_bvd_different} from \eqref{eqn:higher_order_volume_form_different} as follows: From \eqref{eqn:higher_order_volume_form_different}, we have
$(\hpatch{k}_{\beta\alpha,i} \circ ( \lrcorner \volf{k}_{\beta}|_{U_i}) \circ \patch{k}_{\alpha\beta,i})(\gamma) =   ( \gamma \wedge \exp(\bvdobs{k}_{\alpha\beta,i}) ) \lrcorner (\volf{k}_{\alpha}|_{U_i}),$
so
\begin{multline*}
(\patch{k}_{\beta\alpha,i} \circ \bvd{k}_{\beta} \circ \patch{k}_{\alpha\beta,i}) (\gamma) \wedge \exp(\bvdobs{k}_{\alpha\beta, i })\stackrel{(\ref{eqn:higher_order_volume_form_different})}{=}\\
\bvd{k}_{\alpha} ( \gamma \wedge \exp(\bvdobs{k}_{\alpha\beta,i}) ) =
 (\bvd{k}_\alpha (\gamma) + [\bvdobs{k}_{\alpha\beta,i}, \gamma]) \wedge \exp(\bvdobs{k}_{\alpha\beta, i }) 
\end{multline*}
for any $\gamma \in \bva{k}^*_\alpha(U_i)$, which gives $\bvd{k}_{\alpha}(\gamma) + [\bvdobs{k}_{\alpha\beta,i},\gamma] = (\patch{k}_{\beta\alpha,i} \circ \bvd{k}_{\beta} \circ \patch{k}_{\alpha\beta,i}) (\gamma)$. 
\end{remark}

\section{Abstract construction of the \v{C}ech-Thom-Whitney complex}\label{sec:abstract_construction}

\subsection{The simplicial set $\mathcal{A}^*(\simplex_{\bullet})$}\label{sec:simplicial_set_general}

In this subsection, we recall some notations and facts about the simplicial sets $\mathcal{A}_\Bbbk^*(\simplex_{\bullet})$ of polynomial differential forms with coefficient $\Bbbk = \mathbb{Q}, \real,\comp$; we will simply write $\mathcal{A}^*(\simplex_\bullet)$ when $\Bbbk = \comp$, which will be the case for all other parts of this paper.

\begin{notation}\label{not:simplicial_set}
	We let $\text{Mon}$ (resp. $\text{sMon}$) be the category of finite ordinals $[n] = \{0,1,\dots,n\}$ in which morphisms are non-decreasing maps (resp. strictly increasing maps).
	We denote by $\mathtt{d}_{i,n} : [n-1] \rightarrow [n]$ the unique strictly increasing map which skips the $i$-th element, and by $\mathtt{e}_{i,n} : [n+1] \rightarrow [n]$ the unique non-decreasing map sending both $i$ and $i+1$ to the same element $i$. 
\end{notation}

Note that every morphism in $\text{Mon}$ can be decomposed as a composition of the maps $\mathtt{d}_{i,n}$'s and $\mathtt{e}_{i,n}$'s, and any morphism in $\text{sMon}$ can be decomposed as a composition of the maps $\mathtt{d}_{i,n}$'s. 

\begin{definition}[\cite{weibel1995introduction}]\label{def:simplicial_and_cosimplicial_set}
	Let $\mathtt{C}$ be a category. A {\em (semi-)simplicial object in $\mathtt{C}$} is a contravariant functor $\mathtt{A}(\bullet) : \text{Mon} \rightarrow  \mathtt{C}$ (resp. $\mathtt{A}(\bullet) : \text{sMon} \rightarrow  \mathtt{C}$), and a {\em (semi-)cosimplicial object in $\mathtt{C}$} is a covariant function $\mathtt{A}(\bullet) : \text{Mon} \rightarrow \mathtt{C}$ (resp. $\mathtt{A}(\bullet) : \text{sMon} \rightarrow \mathtt{C}$).
\end{definition}


\begin{definition}[\cite{griffiths1981rational}]\label{def:simplicial_de_rham}
	Let $\Bbbk$ be a field which is either $\mathbb{Q}$, $\real$ or $\comp$. Consider the dga
	$$
	\mathcal{A}_{\Bbbk}^*(\simplex_n) := \frac{\Bbbk[x_0,\dots,x_n,dx_0,\dots,dx_n]}{(\sum_{i=0}^n x_i -1, \sum_{i=0}^n dx_i)},
	$$
	with $\deg(x_i) = 0$, $\deg(dx_i) =1$, and equipped with the degree $1$ differential $d$ defined by $d(x_i) = dx_i$ and the Leibniz rule.\footnote{In the case $\Bbbk = \real $, this can be thought of as the space of polynomial differential forms on $\real^{n+1}$ restricted to the $n$-simplex $\simplex_n$.}
	Given $a : [n] \rightarrow [m]$ in $\text{Mon}$, we let $a^*:=\mathcal{A}_\Bbbk(a) : \mathcal{A}_{\Bbbk}^*(\simplex_m) \rightarrow \mathcal{A}_{\Bbbk}^*(\simplex_n)$ be the unique dga morphism satisfying $a^*(x_j) = \sum_{i \in [n]: a(i) = j} x_i$ and $a^*(x_{j}) = 0$ if $j \neq a(i)$ for any $i \in [n]$. From this we obtain a simplicial object in the category of dga's, which we denote as $\mathcal{A}_\Bbbk^*(\simplex_\bullet)$.
\end{definition}

\begin{notation}\label{not:de_rham_form_on_simplex_boundary}
	We denote by $\simplexbdy_n$ the boundary of $\simplex_n$, and let 
	\begin{multline}
	\mathcal{A}_{\Bbbk}^*(\simplexbdy_n) : =Big\{(\alpha_0,\dots,\alpha_n) \mid \alpha_{i}  \in \mathcal{A}_{\Bbbk}^*(\simplex_{n-1}), \\ \mathtt{d}_{i,n-1}^* (\alpha_j) =\mathtt{d}_{j-1,n-1}^*(\alpha_i) \ \text{for $0 \leq i<j \leq n$} \Big\}
	\end{multline}
	be the space of polynomial differential forms on $\simplexbdy_n$. There is a natural restriction map defined by $\beta|_{ \simplexbdy_n} := (\mathtt{d}_{0,n}^*(\beta),\dots,\mathtt{d}^*_{n,n}(\beta))$ for $\beta \in \mathcal{A}^*_\Bbbk(\simplex_n)$. 
\end{notation}

The following extension lemma will be frequently used in subsequent constructions:

\begin{lemma}[Lemma 9.4 in \cite{griffiths1981rational}]\label{lem:simplicial_de_rham_extension}
	For any $\vec{\alpha} = (\alpha_0,\dots,\alpha_n) \in \mathcal{A}_{\Bbbk}^*(\simplexbdy_n)$, there exists $\beta \in \mathcal{A}^*_\Bbbk(\simplex_n)$ such that $\beta|_{\simplexbdy_n} = \vec{\alpha}$. 
\end{lemma}

\begin{notation}\label{not:homotopy_simplex_notation}
	We let $\hsimplex_n := \simplex_1 \times \simplex_n$, where $\simplex_1 := \{ (\mathtt{t}_0,\mathtt{t}_1) \mid 0\leq \mathtt{t}_i \leq 1, \ \mathtt{t}_0 + \mathtt{t}_1 = 1 \}$, and 
	\begin{multline}\label{eqn:homotopy_simplex_de_rham}
	\mathcal{A}^*_\Bbbk(\hsimplex_n) := \mathcal{A}^*_\Bbbk(\simplex_1) \otimes_{\Bbbk} \mathcal{A}^*_\Bbbk(\simplex_n) =\\
	 \frac{\Bbbk[x_0,\dots,x_n,dx_0,\dots,dx_n;\mathtt{t}_0,\mathtt{t}_1,d\mathtt{t}_0,d\mathtt{t}_1]}{(\sum_{i=0}^n x_i -1, \sum_{i=0}^n dx_i, \mathtt{t}_0 +\mathtt{t}_1 -1, d\mathtt{t}_0 + d\mathtt{t}_1)}.
	\end{multline}
	Besides the restriction maps $\mathtt{d}_{j,n}^* : \mathcal{A}^*_{\Bbbk}(\hsimplex_n) \rightarrow \mathcal{A}^*_\Bbbk(\hsimplex_{n-1})$ induced from that on $\simplex_n$, we also have the maps $\mathtt{r}_j^* : \mathcal{A}^*_\Bbbk(\hsimplex_n)  \rightarrow \mathcal{A}^*_\Bbbk(\simplex_n) $ defined by putting $\mathtt{t}_j = 1$ (and $t_{1-j} = 0$).  
\end{notation}

\begin{notation}\label{not:de_rham_form_on_homotopy_simplex_boundary}
	We denote by $\hsimplexbdy_n$ the boundary of $\hsimplex_n$, and let
	\begin{equation*}
	\mathcal{A}_{\Bbbk}^*(\hsimplexbdy_n) := \left\{(\alpha_0,\dots,\alpha_n, \beta_0 ,\beta_1) \ | \ \displaystyle{\substack{\alpha_{i} \in \mathcal{A}_{\Bbbk}^*(\hsimplex_{n-1}),\ \beta_{i} \in \mathcal{A}^*(\simplex_n),\\ \mathtt{d}_{i,n-1}^* (\alpha_j) = \mathtt{d}_{j-1,n-1}^*(\alpha_i) \text{ for $0 \leq i<j \leq n$} \\
	\mathtt{r}_i^* (\alpha_j) = \mathtt{d}_{j,n}^*(\beta_i) \text{ for $i=0,1$ and $0\leq j \leq n$}}}\right\}
	\end{equation*}
	be the space of polynomial differential forms on $\hsimplexbdy_n$. There is a natural restriction map defined by $\gamma|_{ \hsimplexbdy_n} := (\mathtt{d}_{0,n}^*(\gamma),\dots,\mathtt{d}^*_{n,n}(\gamma), \mathtt{r}_0^*(\gamma),\mathtt{r}_1^*(\gamma))$ for $\gamma \in \mathcal{A}^*_{\Bbbk}(\hsimplex_n)$. 
\end{notation}
\begin{lemma}\label{lem:homotopy_simplicial_de_rham_extension}
	For any $(\alpha_0,\dots,\alpha_n, \beta_0,\beta_1) \in \mathcal{A}_{\Bbbk}^*(\hsimplexbdy_n)$, there exists $\gamma \in \mathcal{A}^*_\Bbbk(\hsimplex_n)$ such that $\gamma|_{\hsimplexbdy_n} = (\alpha_0,\dots,\alpha_n, \beta_0,\beta_1)$. 
\end{lemma}
This variation of Lemma \ref{lem:simplicial_de_rham_extension} can be proven by the same technique as in \cite[Lemma 9.4]{griffiths1981rational}.

\subsection{Local Thom-Whitney complexes}\label{sec:local_thom_whitney_complex}


Consider a sheaf of BV algebras $(\bva{}^* , \wedge ,\bvd{})$ on a topological space $V$,\footnote{Readers may assume that $\bva{}^*$ is a bounded complex for the purpose of this paper.} together with an acyclic cover $\mathcal{U} = \{U_i\}_{i \in \inte_+}$ of $V$ such that $H^{>0}(U_{i_0\cdots i_l},\bva{}^j) = 0$ for all $j$ and all finite intersections $U_{i_0 \cdots i_l}:= U_{i_0}\cap \cdots \cap U_{i_l}$. In particular, this allows us to compute the sheaf cohomology $H^*(V,\bva{}^j)$ and the hypercohomology $\mathbb{H}^*(V,\bva{}^*)$ using the \v{C}ech complex $\check{\mathcal{C}}^*(\mathcal{U},\bva{}^j)$ and the total complex of $\check{\mathcal{C}}^*(\mathcal{U},\bva{}^*)$ respectively. 

Let $\mathcal{I} = \{ (i_0,\dots,i_l) \mid i_j \in \inte_+, \ i_0 < i_1 < \cdots < i_l\}$ be the index set.
Let $\simplex_l$ be the standard $l$-simplex in $\real^{l+1}$ and $\mathcal{A}^q(\simplex_l)$ be the space of $\comp$-valued polynomial differential $q$-forms on $\simplex_l$. Also let $\mathtt{d}_{j,l} : \simplex_{l-1} \rightarrow \simplex_l $ be the inclusion of the $j^{\text{th}}$-facet in $\simplex_l$ and let $\mathtt{d}_{j,l}^*$ be the pullback map. See Definition \ref{def:simplicial_de_rham} and Notation \ref{not:simplicial_set} in \S \ref{sec:simplicial_set_general} for details.

\begin{definition}[see e.g. \cite{whitney2012geometric, dupont1976simplicial, fiorenza2012differential}]\label{def:thom_whitney_general}
	The {\em Thom-Whitney complex} is defined as $\twc{}^{*,*}(\bva{}) := \bigoplus_{p,q} \twc{}^{p,q}(\bva{})$ where
	\begin{multline*}
	\twc{}^{p,q}(\bva{}):= \Big\{ (\varphi_{i_0 \cdots i_l})_{(i_0,\dots,  i_l) \in \mathcal{I}} \mid \varphi_{i_0 \cdots i_l} \in \mathcal{A}^q(\simplex_l) \otimes_\comp \bva{}^p(U_{i_0\cdots i_l}), \\
	\mathtt{d}_{j,l}^* (\varphi_{i_0 \cdots i_l}) = \varphi_{i_0 \cdots \hat{i}_j \cdots i_l} |_{U_{i_0 \cdots i_l}} \Big\}.
	\end{multline*}
	It is equipped with the structures $(\wedge,\pdb, \bvd{})$ defined componentwise by
	\begin{align*}
	(\alpha_{I} \otimes v_I) \wedge ( \beta_I \otimes w_I) &:= (-1)^{|v_I||\beta_I|} (\alpha_I \wedge \beta_I) \otimes (v_I \wedge w_I),\\
	\pdb (\alpha_I \otimes v_I) := (d\alpha_I) \otimes v_I ,&\quad 
	\bvd{}(\alpha_I \otimes v_I) := (-1)^{|\alpha_I|} \alpha_I \otimes (\bvd{} v_I) ,
	\end{align*}
	for $\alpha_I, \beta_I \in \mathcal{A}^*(\simplex_l)$ and $v_I, w_I \in \bva{}^*(U_I),$ where $I = (i_0,\dots,i_l) \in \mathcal{I}$ and $l = |I| - 1$.
\end{definition}

\begin{remark}
	We use the notation $\pdb$ since it plays the role of the Dolbeault operator in the classical deformation theory of smooth Calabi-Yau manifolds.
\end{remark}

$(\twc{}^{*,*}(\bva{}), \bar{\partial}, \bvd{}, \wedge)$ forms a dgBV algebra 
in the sense of Definition \ref{def:dgBV}. 
From Definitions \ref{def:dgBV} and \ref{def:thom_whitney_general}, the Lie bracket on the Thom-Whitney complex is determined componentwise by the formula
\begin{equation}\label{eqn:Thom_whitney_Lie_bracket_formula}
[\alpha_I \otimes v_I , \beta_I \otimes w_I ] := (-1)^{(|v_I|+1) |\beta_I|} (\alpha_I \wedge \beta_I) \otimes [v_I,w_I],
\end{equation}
for $\alpha_I, \beta_I \in \mathcal{A}^*(\simplex_l)$ and $v_I, w_I \in \bva{}^*(U_I)$ where $l = |I| - 1$.

We consider the integration map $ \mathtt{I} : \twc{}^{p,q}(\bva{}) \rightarrow \check{\mathcal{C}}^q(\mathcal{U},\bva{}^p)$ defined by
\begin{equation*}
\mathtt{I}(\alpha_{i_0\dots i_l}) := \left(\int_{\simplex_q} \otimes \text{id}\right) \left(\alpha_{i_0 \dots i_l} \right)
\end{equation*}
for each component $\alpha_{i_0\dots i_l} \in \mathcal{A}^q(\simplex_l) \otimes \bva{}^p(U_{i_0 \dots i_l})$ of $(\alpha_{i_0 \dots i_l})_{(i_0 \dots i_l) \in \mathcal{I}} \in \twc{}^{p,q}(\bva{})$.
Notice that $\mathtt{I}$ is a chain morphism from $(\twc{}^{p,*}(\bva{}), \pdb)$ to $(\check{\mathcal{C}}^*(\mathcal{U},\bva{}^p), \delta)$, where $\delta$ is the \v{C}ech differential. Taking the total complexes gives a chain morphism from $(\twc{}^{*,*}(\bva{}), \pdb \pm  \bvd{})$ to $\check{\mathcal{C}}^*(\mathcal{U},\bva{}^*)$, which is equipped with the total \v{C}ech differential $\delta \pm \bvd{}$. 

\begin{lemma}[\cite{whitney2012geometric}]\label{lem:thom_whitney_construction_quasi_iso}
	The maps $\mathtt{I}:\twc{}^{p,*}(\bva{}) \to \check{\mathcal{C}}^*(\mathcal{U},\bva{}^p)$ and $I:\twc{}^{*,*}(\bva{}) \to \check{\mathcal{C}}^*(\mathcal{U},\bva{}^*)$ are quasi-isomorphisms.
\end{lemma}

\begin{remark}
	Comparing to the standard construction of the Thom-Whitney complex in e.g. \cite{fiorenza2012differential} where one considers $(\varphi_{i_0\cdots i_l})_{(i_0, \ldots, i_l) \in \mathcal{I}} \in \prod_{l \geq 0} \left( \mathcal{A}^*(\simplex_l) \otimes_{\comp} \prod_{i_0< \cdots <i_l} \bva{}^{p}(U_{i_0\cdots i_l}) \right)$, we are taking a bigger complex in Definition \ref{def:thom_whitney_general} for the purpose of later constructions. However, the original proof of Lemma \ref{lem:thom_whitney_construction_quasi_iso} works in exactly the same way for this bigger complex, and hence $\twc{}^{p,*}(\bva{})$ also serves as a resolution of $\bva{}^p$. 
\end{remark}

\begin{definition}\label{def:thom_whitney_0th_order_higher_order}
	Given the $0^{\text{th}}$-order complex of polyvector fields $(\bva{0}^*,\wedge, \bvd{0})$ over $X$ (Definition \ref{def:0_order_data}), we then use the Stein cover $\mathcal{U} = \{U_i\}_{i \in \inte_+}$ in Notation \ref{not:open_stein_covers} to define the {\em $0^{\text{th}}$-order Thom-Whitney complex} $(\twc{}^{*,*}(\bva{0}), \pdb,\bvd{0},  \wedge)$.
	To simplify notations, we write $\twc{0}^{*,*}$ for $\twc{}^{*,*}(\bva{0})$.
	
	Given a finite intersection of open subsets $V_{\alpha_0 \cdots \alpha_\ell} := V_{\alpha_0} \cap \cdots \cap V_{\alpha_\ell}$ of the cover $\mathcal{V}$,
	and local thickenings of the complex of polyvector fields $(\bva{k}^*_{\alpha_i},\wedge,\bvd{k}_{\alpha_i})$ over $V_{\alpha_i}$ for each $k \in \inte_{\geq 0}$ (Definition \ref{def:higher_order_data}), we use the cover $\mathcal{U}_{\alpha_0 \cdots \alpha_\ell} := \{U \in \mathcal{U} \mid U \subset V_{\alpha_0 \cdots \alpha_\ell} \}$ to define the {\em local Thom-Whitney complex} $(\twc{}^{*,*}(\bva{k}_{\alpha_i}|_{V_{\alpha_0 \cdots \alpha_\ell}}), \pdb, \bvd{k}_{\alpha_i}, \wedge)$ over $V_{\alpha_0 \cdots \alpha_\ell}$. To simplify notations, we write $\twc{k}^{*,*}_{\alpha_i;\alpha_0 \cdots \alpha_\ell}$ for $\twc{}^{*,*}(\bva{k}_{\alpha_i}|_{V_{\alpha_0 \cdots \alpha_\ell}})$.
\end{definition}

The covers $\mathcal{U}$ and $\mathcal{U}_{\alpha_0 \cdots \alpha_\ell}$ satisfy the acyclic assumption at the beginning of this section because  $\bva{0}^*$ and $\bva{k}^*_\alpha$ are coherent sheaves and all the open sets in these covers are Stein:

\begin{theorem}[Cartan's Theorem B \cite{cartan1957varietes}; see e.g. Chapter IX Corollary 4.11 in \cite{demailly1997complex}]\label{thm:cartan_theorem_B}
	For a coherent sheaf $\mathcal{F}$ over a Stein space $U$, we have
	$
	H^{>0}(U,\mathcal{F}) = 0.
	$
\end{theorem}

\subsection{The gluing morphisms}\label{sec:solving_for_gluing}

\subsubsection{Existence of a set of compatible gluing morphisms} \label{sec:proof_of_compatible_gluing_data}

The aim of this subsection is to construct, for each $k \in \inte_{\geq 0}$ and any pair $V_\alpha, V_\beta \in \mathcal{V}$, an {\em isomorphism} 
\begin{equation}\label{eqn:g_alpha_beta_map}
\glue{k}_{\alpha \beta } : \twc{k}^{*,*}_{\alpha;\alpha\beta } \rightarrow \twc{k}^{*,*}_{\beta ; \alpha \beta},
\end{equation}
as a collection of maps $(\glue{k}_{\alpha\beta, I})_{I\in \mathcal{I}}$ so that for each $\varphi =(\varphi_I)_{I\in \mathcal{I}} \in \twc{k}^{*,*}_{\alpha;\alpha \beta}$ with $\varphi_I \in \mathcal{A}^*(\simplex_l) \otimes \bva{k}^*_\alpha(U_I)$ we have $\left( \glue{k}_{\alpha\beta}(\varphi) \right)_I = \glue{k}_{\alpha\beta, I}(\varphi_I)$, which preserves the algebraic structures $[\cdot,\cdot], \wedge$ and satisfies the following condition:
\begin{condition}\label{assum:induction_hypothesis}
	\begin{enumerate}
		\item for $U_i \subset V_\alpha \cap V_\beta$, we have 
		\begin{equation}\label{eqn:alpha_beta_gluing_explicit_form_at_point}
		\glue{k}_{\alpha\beta, i} = \exp([ \iauto{k}_{\alpha\beta,i}, \cdot ]) \circ \patch{k}_{\alpha\beta,i}
		\end{equation}
		for some element $\iauto{k}_{\alpha\beta,i} \in \bva{k}^{-1}_{\beta}(U_i)$ with $\iauto{k}_{\alpha\beta,i} = 0 \ \text{(mod $\mathbf{m}$)}$;
		
		\item for $U_{i_0}, \dots, U_{i_l} \subset V_\alpha \cap V_{\beta}$, we have 
		\begin{equation}\label{eqn:alpha_beta_gluing_explicit_form_on_simplex}
		\glue{k}_{\alpha\beta,i_0\cdots i_l} = \exp([\sauto{k}_{\alpha\beta, i_0 \cdots i_l},\cdot]) \circ \left(\glue{k}_{\alpha\beta,i_0}|_{U_{i_0 \cdots i_l}}\right),
		\end{equation}
		for some element $\sauto{k}_{\alpha\beta, i_0 \cdots i_l} \in \mathcal{A}^0(\simplex_l) \otimes \bva{k}^{-1}_\beta(U_{i_0 \cdots i_l})$ with $\sauto{k}_{\alpha\beta, i_0 \cdots i_l} = 0 \ \text{(mod $\mathbf{m}$)}$; and
		
		\item the elements $\sauto{k}_{\alpha\beta, i_0 \cdots i_l}$'s satisfy the relation:\footnote{Here $\mathtt{d}_{j,l}^*$ is induced by the corresponding map $\mathtt{d}_{j,l}^* : \mathcal{A}^*(\simplex_l) \rightarrow \mathcal{A}^*(\simplex_{l-1})$ on the simplicial set $\mathcal{A}^*(\simplex_\bullet)$ introduced in Definition \ref{def:simplicial_de_rham} and $\bchprod$ is the Baker-Campbell-Hausdorff product in Notation \ref{not:exponential_group}.}
		\begin{equation}\label{eqn:alpha_beta_gluing_explicit_form_on_simplex_boundary_relation}
		\mathtt{d}_{j,l}^*( \sauto{k}_{\alpha\beta, i_0 \cdots i_l}) =
		\begin{dcases}
		\sauto{k}_{\alpha\beta, i_0 \cdots \widehat{i_j} \cdots i_l} & \text{for $j>0$},\\
		\sauto{k}_{\alpha\beta, \widehat{i_0}  \cdots i_l} \bchprod \autoij{k}_{\alpha\beta, i_0 i_1}  & \text{for $j=0$}, 
		\end{dcases}
		\end{equation}
		where $\autoij{k}_{\alpha\beta, i_0 i_1}\in \bva{k}_{\beta}^{-1}(U_{i_0 i_1})$ is the unique element such that
		\begin{equation}\label{eqn:alpha_beta_gluing_relation_between_i_0_and_i_1}
		\exp([\autoij{k}_{\alpha\beta,i_0i_1}, \cdot]) \glue{k}_{\alpha \beta, i_0} = \glue{k}_{\alpha \beta, i_1}.
		\end{equation}
	\end{enumerate}
\end{condition}

\begin{lemma}\label{lem:preserving_thom_whitney_condition}
	Suppose that the morphisms $\glue{k}_{\alpha\beta}$'s, each of which is a collection of maps $(\glue{k}_{\alpha\beta, I})_{I\in \mathcal{I}}$, all satisfy Condition \ref{assum:induction_hypothesis}. For any $\varphi =(\varphi_I)_{I\in \mathcal{I}} \in \twc{k}^{*,*}_{\alpha;\alpha \beta}$, we have $\left( \glue{k}_{\alpha\beta,I}(\varphi_I) \right)_{I\in \mathcal{I}} \in \twc{k}^{*,*}_{\beta;\alpha\beta}$.
\end{lemma}

\begin{proof}
	Suppose that we have $(\varphi_I)_{I\in \mathcal{I}} \in \twc{k}^{*,q}_{\alpha;\alpha\beta}$ such that $\varphi_{i_0 \cdots i_l } \in \mathcal{A}^q(\simplex_l) \otimes \bva{k}^*_\alpha(U_{i_0\cdots i_l})$ and $\varphi_{i_0 \cdots \widehat{i_j} \cdots i_l} = \mathtt{d}_{j,l}^*(\varphi_{i_0\cdots i_l})$. Letting
	$
	(\glue{k}_{\alpha\beta}\varphi)_{i_0\cdots i_l} := (\exp([\sauto{k}_{\alpha\beta, i_0 \cdots i_l},\cdot]) \circ \glue{k}_{\alpha\beta,i_0})(\varphi_{i_0\cdots i_l}),
	$
	we have 
	\begin{align*}
	(\glue{k}_{\alpha\beta}\varphi)_{i_0\cdots \widehat{i_j} \cdots i_l} 
	& = (\exp([\sauto{k}_{\alpha\beta, i_0 \cdots \widehat{i_j} \cdots  i_l},\cdot]) \circ \glue{k}_{\alpha\beta,i_0})(\varphi_{i_0\cdots \widehat{i_j} \cdots i_l})\\
	& = (\exp([\mathtt{d}_{j,n}^*(\sauto{k}_{\alpha\beta, i_0  \cdots  i_l}),\cdot]) \circ \glue{k}_{\alpha\beta,i_0})(\mathtt{d}_{j,n}^*(\varphi_{i_0 \cdots i_l}))\\
	& = \mathtt{d}_{j,n}^*((\exp([\sauto{k}_{\alpha\beta, i_0  \cdots  i_l},\cdot]) \circ \glue{k}_{\alpha\beta,i_0})(\varphi_{i_0 \cdots i_l})),
	\end{align*}
	and 
	\begin{align*}
	(\glue{k}_{\alpha\beta}\varphi)_{\widehat{i_0} \cdots i_l} 
	& = (\exp([\sauto{k}_{\alpha\beta, \widehat{i_0} \cdots  i_l},\cdot]) \circ \glue{k}_{\alpha\beta,i_1})(\varphi_{ \widehat{i_0} \cdots i_l})\\
	& = (\exp([\sauto{k}_{\alpha\beta, \widehat{i_0} \cdots  i_l},\cdot]) \circ \exp([\autoij{k}_{\alpha\beta,i_0 i_1},\cdot]) \circ \glue{k}_{\alpha\beta,i_0})(\varphi_{ \widehat{i_0} \cdots i_l})\\
	& = (\exp([\mathtt{d}_{0,n}^*(\sauto{k}_{\alpha\beta, i_0  \cdots  i_l}),\cdot]) \circ \glue{k}_{\alpha\beta,i_0})(\mathtt{d}_{0,n}^*(\varphi_{i_0 \cdots i_l}))\\
	& = \mathtt{d}_{0,n}^*((\exp([\sauto{k}_{\alpha\beta, i_0  \cdots  i_l},\cdot]) \circ \glue{k}_{\alpha\beta,i_0})(\varphi_{i_0 \cdots i_l})),
	\end{align*}
	which are the required conditions for $(\glue{k}_{\alpha\beta}\varphi)_{I} \in \twc{k}^{*,*}_{\beta;\alpha\beta}$. 
\end{proof}

Given a multi-index $(\alpha_0 \cdots \alpha_\ell)$, we have, for each $j=0,\dots,\ell$, a natural restriction map
\begin{equation}\label{eqn:thom_whitney_restriction_map}
\restmap_{\alpha_j} : \twc{k}^{*,*}_{\alpha_i; \alpha_0 \cdots \widehat{\alpha_j} \cdots \alpha_\ell} \rightarrow  \twc{k}^{*,*}_{\alpha_i; \alpha_0 \cdots \alpha_\ell} ,
\end{equation}
defined componentwise by
$$
\restmap_{\alpha_j} \Big( (\varphi_{I})_{I \in \mathcal{I}} \Big) = (\varphi_I)_{I \in \mathcal{I}'}
$$
for $(\varphi_I)_{I\in \mathcal{I}} \in \twc{k}^{*,*}_{\alpha_i; \alpha_0 \cdots \widehat{\alpha_j} \cdots \alpha_\ell}$, where $\mathcal{I}' = \{ (i_0,\dots,i_l) \in \mathcal{I} \mid U_{i_j} \subset V_{\alpha_0\cdots \alpha_{\ell}} \}$. The map $\restmap_{\alpha_j}$ is a morphism of dgBV algebras.

Now for a triple $V_\alpha, V_\beta, V_\gamma \in \mathcal{V}$, we define the {\em restriction of $\glue{k}_{\alpha \beta}$ to $\twc{k}^{*,*}_{\alpha;\alpha\beta\gamma}$} as the unique map $\glue{k}_{\alpha \beta}:\twc{k}^{*,*}_{\alpha;\alpha\beta\gamma} \to \twc{k}^{*,*}_{\beta;\alpha\beta\gamma}$ that fits into the diagram
$$
\xymatrix@1{ \twc{k}^{*,*}_{\alpha;\alpha\beta} \ar[r]^{\restmap_\gamma} \ar[d]^{\glue{k}_{\alpha \beta}} & \twc{k}^{*,*}_{\alpha;\alpha\beta\gamma} \ar[d]^{\glue{k}_{\alpha \beta}}\\
	\twc{k}^{*,*}_{\beta;\alpha\beta} \ar[r]^{\restmap_\gamma} & \twc{k}^{*,*}_{\beta;\alpha\beta\gamma}.}
$$

\begin{definition}\label{def:compatible_gluing_morphism}
	The morphisms $\{\glue{k}_{\alpha \beta}\}$ satisfying Condition \ref{assum:induction_hypothesis} are said to form {\em a set of compatible gluing morphisms} if in addition the following conditions are satisfied:
	\begin{enumerate}
		\item $\glue{0}_{\alpha\beta} = \text{id}$ for all $\alpha, \beta$;
		
		\item (compatibility between different orders) for each $k \in \inte_{\geq 0}$ and any pair $V_\alpha, V_\beta \in \mathcal{V}$,
		\begin{equation}\label{eqn:g_alpha_beta_order_compatibility}
		\glue{k}_{\alpha \beta} \circ \rest{k+1,k}_{\alpha} = \rest{k+1,k}_{\beta} \circ \glue{k+1}_{\alpha \beta};
		\end{equation}
		
		\item (cocycle condition) for each $k \in \inte_{\geq 0}$ and any triple $V_\alpha, V_\beta, V_\gamma \in \mathcal{V}$,
		\begin{equation}\label{eqn:g_alpha_beta_cocycle}
		\glue{k}_{\gamma \alpha} \circ \glue{k}_{\beta \gamma} \circ \glue{k}_{\alpha \beta} = \text{id}
		\end{equation}
		when $\glue{k}_{\alpha \beta}$, $\glue{k}_{\beta \gamma}$ and $\glue{k}_{\gamma \alpha}$ are restricted to $\twc{k}^{*,*}_{\alpha;\alpha \beta \gamma}$, $\twc{k}^{*,*}_{\beta;\alpha \beta \gamma}$ and $\twc{k}^{*,*}_{\gamma;\alpha \beta \gamma}$ respectively. 
	\end{enumerate}
\end{definition}

\begin{theorem}\label{lem:existence_compatible_gluing_data}
	There exists a set of compatible gluing morphisms $\{\glue{k}_{\alpha \beta}\}$.
\end{theorem}

We will construct the gluing morphisms $\glue{k}_{\alpha \beta}$'s inductively. To do so, we need a couple of lemmas.

\begin{lemma}\label{lem:cech_complex_of_a_point}
	Fixing $U_{i_0},\dots,U_{i_l} \in \mathcal{U}$ and $-d \leq j \leq 0$, we consider the index set $I_{i_0\cdots i_l }:= \{ \alpha \mid U_{i_r} \subset V_\alpha\ \text{for all $0 \leq r \leq l$} \}$ and the following trivial \v{C}ech complex $\check{\mathcal{C}}^*(I_{i_0 \cdots i_l},\bva{0}^j)$ associated to the vector space $\bva{0}^j(U_{i_0 \cdots i_l})$ given by
	$$
	\prod_{\alpha \in I_{i_0 \cdots i_l}} \bva{0}^j(U_{i_0 \cdots i_l}) 
	\rightarrow \prod_{\alpha, \beta \in I_{i_0 \cdots i_l}} \bva{0}^j(U_{i_0 \cdots i_l}) \rightarrow \prod_{\alpha,\beta,\gamma \in I_{i_0 \cdots i_l}}\bva{0}^j(U_{i_0 \cdots i_l})  \cdots ,
	$$
	where each arrow is the \v{C}ech differential associated to the index set $I_{i_0 \cdots i_l}$. We have 
	$$H^{>0}(\check{\mathcal{C}}(I_{i_0 \cdots i_l},\bva{0}^j)) = 0\text{ and } H^{0}(\check{\mathcal{C}}(I_{i_0 \cdots i_l},\bva{0}^j)) = \bva{0}^{j}(U_{i_0 \cdots i_l});$$
	the same holds for $\bva{0}^j \otimes \mathbb{V}$ for any vector space $\mathbb{V}$.
\end{lemma}

\begin{proof}
	We consider the topological space $\text{pt}$ consisting of a single point and an indexed cover $(\hat{V}_\alpha)_{\alpha \in I_{i_0 \cdots i_l}}$ such that $\hat{V}_\alpha = \text{pt}$ for each $\alpha$. Then we take a constant sheaf $\digamma$ over $\text{pt}$ with $\digamma(\text{pt}) = \bva{0}^j(U_{i_0 \cdots i_l})$. Since $ \bva{0}^j(U_{i_0 \cdots i_l}) = \digamma(V_{\alpha_0 \cdots \alpha_\ell})$ for any $\alpha_0 ,\dots ,\alpha_\ell \in I_{i_0 \cdots i_l}$, we have a natural isomorphism $\check{\mathcal{C}}^*(I_{i_0 \cdots i_l},\bva{0}^j) \cong \check{\mathcal{C}}^*(I_{i_0 \cdots i_l},\digamma)$. The result then follows by considering the \v{C}ech cohomology of $\text{pt}$.
\end{proof}

\begin{lemma}[Lifting Lemma]\label{lem:simplex_lifting_lemma}
	Let $\rest{} : \mathcal{F} \rightarrow \mathcal{H}$ be a surjective morphism of sheaves over $V := V_{\alpha_0 \cdots \alpha_{\ell}}$. For a Stein open subset $U := U_{i_0 \cdots i_l} \subset V$, let $\mathtt{w} \in \mathcal{A}^{q}(\simplex_l) \otimes \mathcal{H}(U)$ and $\partial (\mathtt{v}) \in \mathcal{A}^{q}(\simplexbdy_{l}) \otimes \mathcal{F}(U)$ such that $\rest{}(\partial (\mathtt{v})) = \mathtt{w}|_{\simplexbdy_{l}}$. Then there exists $\mathtt{v} \in \mathcal{A}^{q}(\simplex_l) \otimes \mathcal{F}(U)$  such that 
	$\mathtt{v}|_{\simplexbdy_{l}} =\partial( \mathtt{v})$ and $ \rest{}(\mathtt{v}) = \mathtt{w}$. The same holds if $\mathcal{A}^{q}(\simplex_l)$ and $\mathcal{A}^q(\simplexbdy_l)$ are replaced by $\mathcal{A}^q(\hsimplex_l)$ and $\mathcal{A}^q(\hsimplexbdy_l)$ respectively.
	\end{lemma}

\begin{proof}
	By Lemma \ref{lem:simplicial_de_rham_extension}, there is a lifting $\widehat{\mathtt{v}} \in \mathcal{A}^q(\simplex_l) \otimes \mathcal{F}(U)$ such that $\widehat{\mathtt{v}}|_{\simplexbdy_l} = \partial(\mathtt{v})$. Let
	$\mathsf{u} := \mathtt{w} - \rest{}(\widehat{\mathtt{v}}) \in \mathcal{A}^q_0(\simplex_l) \otimes \mathcal{H}(U)$, where $\mathcal{A}^q_0(\simplex_l)$ is the space of differential $q$-forms {\em whose restriction to $\simplexbdy_l$ is $0$}. Since $U$ is Stein, the map $\rest{}: \mathcal{F}(U) \rightarrow \mathcal{H}(U)$ is surjective. So we have a lifting $\widetilde{\mathsf{u}}$ of $\mathsf{u}$ to $\mathcal{A}^q_0(\simplex_l) \otimes \mathcal{F}(U)$. Now the element $\mathtt{v} := \widehat{\mathtt{v}}+ \widetilde{\mathsf{u}}$ satisfies the desired properties. The same proof applies to the case involving $\hsimplex_l$ and $\hsimplexbdy_l$.
\end{proof}

\begin{lemma}[Key Lemma]\label{lem:compatible_gluing_data}
	Suppose we are given a set of gluing morphisms $\{\glue{k}_{\alpha \beta}\}$ for some $k\geq 0$ satisfying Condition \ref{assum:induction_hypothesis} and the cocycle condition \eqref{eqn:g_alpha_beta_cocycle}. Then there exists a set of $\{\glue{k+1}_{\alpha \beta}\}$ satisfying Condition \ref{assum:induction_hypothesis} , the compatibility condition \eqref{eqn:g_alpha_beta_order_compatibility} as well as the cocycle condition \eqref{eqn:g_alpha_beta_cocycle}.
\end{lemma}

\begin{proof}
	We will prove by induction on $l$ where $l = |I| - 1$ for a multi-index $I = (i_0, \ldots, i_l) \in \mathcal{I}$.
	
	For $l = 0$, we fix $i = i_0$. From \eqref{eqn:alpha_beta_gluing_explicit_form_at_point} in Condition \ref{assum:induction_hypothesis}, we have $\glue{k}_{\alpha \beta, i} = \exp([ \iauto{k}_{\alpha\beta,i}, \cdot ]) \circ \patch{k}_{\alpha\beta,i}$ for some $\iauto{k}_{\alpha\beta,i}  \in \bva{k}^{-1}_{\beta}(U_i)$. Also, from \eqref{eqn:bva_different_order_comparison} in Definition \ref{def:higher_order_patching}, there exist elements $\resta{k+1,k}_{\alpha\beta,i} \in \bva{k}_{\alpha}^{-1}(U_i)$ such that
	\begin{multline*}
	\rest{k+1,k}_{\beta} \circ \patch{k+1}_{\alpha\beta,i} \stackrel{(\ref{eqn:bva_different_order_comparison})}{=} \patch{k}_{\alpha \beta,i} \circ \exp([\resta{k+1,k}_{\alpha \beta,i},\cdot]) \circ \rest{k+1,k}_{\alpha}\\
	= \exp([\patch{k}_{\alpha \beta,i}(\resta{k+1,k}_{\alpha \beta,i}),\cdot] ) \circ \patch{k}_{\alpha \beta,i} \circ \rest{k+1,k}_{\alpha},
	\end{multline*}
	where we use the fact that $\patch{k}_{\alpha \beta,i}$ is an isomorphism preserving the Lie bracket $[\cdot,\cdot]$.
	Therefore we have
	$
	\glue{k}_{\alpha\beta,i} \circ \rest{k+1,k}_{\alpha} = \exp([ \iauto{k}_{\alpha\beta,i}, \cdot ]) \circ \exp(-[(\patch{k}_{\alpha \beta,i})(\resta{k+1,k}_{\alpha \beta,i}),\cdot]) \circ  \rest{k+1,k}_{\beta} \circ \patch{k+1}_{\alpha\beta,i}.
	$
	Taking a lifting $\Upsilon_{\alpha \beta , i}$ of the term $\iauto{k}_{\alpha\beta,i} \bchprod \left(\patch{k}_{\alpha \beta,i}(-\resta{k+1,k}_{\alpha \beta,i})\right)$ from $\bva{k}^{-1}_\beta(U_i)$ to $\bva{k+1}^{-1}_\beta(U_i)$ in the above equation (using the surjectivity of the map $\rest{k+1,k}_{\alpha}:\bva{k+1}_{\alpha}^* \rightarrow \bva{k}_{\alpha}^*$), we define a lifting of $\glue{k}_{\alpha \beta,i}$:
	$$\widetilde{\glue{k+1}_{\alpha \beta, i}} := \exp([\Upsilon_{\alpha \beta,i} ,\cdot ]) \circ \patch{k+1}_{\alpha\beta,i} : \bva{k+1}^*_{\alpha}|_{U_i} \rightarrow \bva{k+1}^*_{\beta}|_{U_i}.$$

	As endomorphisms of $\bva{k+1}^*_\alpha$, we have
	$
	\widetilde{\glue{k+1}_{\gamma \alpha, i}} \circ \widetilde{\glue{k+1}_{ \beta \gamma, i}} \circ \widetilde{\glue{k+1}_{\alpha \beta, i}} = \exp\left([\thmcocyo{k+1}_{\alpha \beta\gamma,i}, \cdot ]\right) 
	$
	for some $\thmcocyo{k+1}_{\alpha \beta\gamma,i} \in \bva{k+1}^{-1}_{\alpha}(U_i)$. Now $\widetilde{\glue{k+1}_{\gamma \alpha, i}} \circ \widetilde{\glue{k+1}_{ \beta \gamma, i}} \circ \widetilde{\glue{k+1}_{\alpha \beta, i}} = \text{id} \  \text{(mod $\mathbf{m}^{k+1}$)}$, so we have $\exp([\thmcocyo{k+1}_{\alpha \beta\gamma,i}, \cdot ])(v) = v \ \text{(mod $\mathbf{m}^{k+1}$)}$ for all $v \in \bva{k}^*_{\alpha}(U_i)$. Therefore $\thmcocyo{k+1}_{\alpha \beta\gamma,i} = 0 \ \text{(mod $\mathbf{m}^{k+1}$)}$ as the map $\bva{k}^{-1}_\alpha \rightarrow \text{Der}(\bva{k}^0_\alpha)$ is injective (see Definition \ref{def:higher_order_data}). 
	Since every stalk $(\bva{k+1}^*_\alpha)_x$ is a free $\cfrk{k+1}$-module and $\rest{k+1,0}_{\alpha}$ induces a sheaf isomorphism upon tensoring with the residue field $\cfrk{0} \cong \comp$ (over $\cfrk{k+1}$), we have a sheaf isomorphism of $\cfrk{k+1}$ modules
	\begin{equation}\label{eqn:locally_free_identification}
	\bva{k+1}^*_\alpha \cong  \bva{0}^* \oplus \bigoplus_{j=1}^{k+1} \left((\mathbf{m}^{j}/\mathbf{m}^{j+1}) \otimes_{\comp} \bva{0}^*\right).
	\end{equation}
	Hence we have $\thmcocyo{k+1}_{\alpha \beta\gamma,i}  \in (\mathbf{m}^{k+1}/\mathbf{m}^{k+2}) \otimes_{\comp} \bva{0}^{-1}(U_i)$. 
	
	Now we consider the \v{C}ech complex $\check{C}^*(I_{i},\bva{0}^{-1}) \otimes_{\comp} (\mathbf{m}^{k+1}/\mathbf{m}^{k+2})$ as in Lemma \ref{lem:cech_complex_of_a_point}. The collection $(\thmcocyo{k+1}_{\alpha \beta\gamma,i})_{\alpha, \beta , \gamma \in I_{i}}$ is a $2$-cocycle in $\check{C}^2(I_{i},\bva{0}^{-1}) \otimes_{\comp} (\mathbf{m}^{k+1}/\mathbf{m}^{k+2})$. By Lemma \ref{lem:cech_complex_of_a_point} for the case $l=0$, there exists $(\prescript{k+1}{}{\mathsf{c}}_{\alpha \beta,i})_{\alpha \beta} \in \check{\mathcal{C}}^1(I_{i},\bva{0}^{-1}) \otimes_{\comp} (\mathbf{m}^{k+1}/\mathbf{m}^{k+2})$ whose image under the \v{C}ech differential is precisely $(\thmcocyo{k+1}_{\alpha \beta\gamma,i})_{\alpha\beta\gamma}$. By the identification \eqref{eqn:locally_free_identification}, we can regard $\prescript{k+1}{}{\mathsf{c}}_{\alpha \beta,i}$ as an element in $\bva{k+1}^{-1}_{\beta}(U_i)$ such that $\prescript{k+1}{}{\mathsf{c}}_{\alpha \beta,i} = 0 \ \text{(mod $\mathbf{m}^{k+1}$)}$. Therefore letting
	$
	\glue{k+1}_{\alpha\beta, i } : = \exp([\prescript{k+1}{}{\mathsf{c}}_{\alpha \beta,i},\cdot]) \circ \widetilde{\glue{k+1}_{\alpha \beta, i}},
	$
	we have the cocycle condition $\glue{k+1}_{\gamma\alpha,i} \circ \glue{k+1}_{\beta\gamma,i} \circ \glue{k+1}_{\alpha \beta,i} = \text{id}$.
	
	For the induction step, we assume that maps $\glue{k+1}_{\alpha\beta,i_0\cdots i_j}$ satisfying all the required conditions have been constructed for each multi-index $(i_0 \cdots i_j)$ with $j \leq l-1$. We shall construct $\glue{k+1}_{\alpha\beta,i_0\cdots i_l}$ for any multi-index $(i_0, \ldots, i_l)$. In view of Condition \ref{assum:induction_hypothesis}, what we need are elements $\sauto{k+1}_{\alpha\beta,i_0\cdots i_l} \in \mathcal{A}^0(\simplex_l) \otimes \bva{k+1}^{-1}_\beta(U_{i_0 \cdots i_l})$ satisfying \eqref{eqn:alpha_beta_gluing_explicit_form_on_simplex_boundary_relation} and the cocycle condition \eqref{eqn:g_alpha_beta_cocycle}, the latter of which can be written explicitly as
	\begin{multline*}
	\exp([\sauto{k+1}_{\gamma\alpha,i_0 \cdots i_l}, \cdot]) \circ  \glue{k+1}_{\gamma \alpha, i_0} \circ \exp([\sauto{k+1}_{\beta \gamma,i_0 \cdots i_l},\cdot]) 
	\circ\\ \glue{k+1}_{\beta \gamma, i_0} \circ   \exp([ \sauto{k+1}_{\alpha \beta,i_0 \cdots i_l},\cdot])   \circ \glue{k+1}_{\beta \alpha, i_0}\\
	= \exp([\sauto{k+1}_{\gamma\alpha,i_0 \cdots i_l},\cdot]) \circ \exp([\glue{k+1}_{\gamma \alpha, i_0}(\sauto{k+1}_{\beta \gamma,i_0 \cdots i_l}), \cdot])
	\\ \circ \exp([\glue{k+1}_{\beta \alpha, i_0}(\sauto{k+1}_{\alpha \beta,i_0 \cdots i_l}), \cdot]) = \text{id}.
	\end{multline*}

	Using the $\sauto{k+1}_{\alpha \beta,i_0 \cdots i_{l-1}}$'s that were defined previously, we want to define a lifting $\reallywidetilde{\sauto{k+1}_{\alpha \beta, i_0 \cdots i_l}}$ of the element $\sauto{k}_{\alpha\beta,i_0 \cdots i_l}$. Before that, we first define its restriction to the boundry $\simplexbdy_l$ as 
	\begin{multline*}
	\partial (\reallywidetilde{\sauto{k+1}_{\alpha\beta, i_0 \cdots  i_l}})  :=\\ (\sauto{k+1}_{\alpha\beta,\widehat{ i_0} \cdots i_l} \bchprod \autoij{k+1}_{\alpha\beta, i_0 i_1}, \sauto{k+1}_{\alpha\beta,i_0 \widehat{ i_1} \cdots i_l}, \dots, \sauto{k+1}_{\alpha\beta,i_0\cdots  \widehat{ i_l} } ),
	\end{multline*}
	where $\autoij{k+1}_{\alpha\beta, i_0 i_1}$ is defined in Condition \ref{assum:induction_hypothesis}. For $ 0 \leq r_1<r_2 \leq l$, we have
	\begin{multline}
	\mathtt{d}_{r_1,l-1}^*(\sauto{k+1}_{\alpha\beta,i_0 \cdots \widehat{ i_{r_2}} \cdots i_l})  \\
	=\begin{dcases}
	\sauto{k+1}_{\alpha\beta,i_0 \cdots \widehat{i_{r_1}}\cdots \widehat{ i_{r_2}} \cdots i_l}    =\mathtt{d}_{r_2-1,l-1}^*(\sauto{k+1}_{\alpha\beta,i_0 \cdots \widehat{ i_{r_1}} \cdots i_l})\\
	 \text{\hspace{8cm} if $r_1 \neq 0$,}\\
	\sauto{k+1}_{\alpha\beta,\widehat{i_0} \cdots \widehat{ i_{r_2}} \cdots i_l} \bchprod \autoij{k+1}_{\alpha\beta, i_0 i_1}  =\mathtt{d}_{r_2-1,l-1}^*(\sauto{k+1}_{\alpha\beta,\widehat{i_0}  \cdots i_l} \bchprod \autoij{k+1}_{\alpha\beta, i_0 i_1})\\  
	\text{\hspace{8cm} if $r_1 = 0, r_2 \neq 1$,}\\
	\sauto{k+1}_{\alpha\beta,\widehat{i_0} \widehat{i_1} \cdots  i_l} \bchprod \autoij{k+1}_{\alpha\beta, i_0 i_2}  =\mathtt{d}_{0,l-1}^*(\sauto{k+1}_{\alpha\beta,\widehat{i_0}  \cdots i_l} \bchprod \autoij{k+1}_{\alpha\beta, i_0 i_1})\\
	\text{\hspace{8cm} if $r_1 = 0, r_2 = 1$,}
	\end{dcases}
	\end{multline}
	upon restricting to $U_{i_0\dots i_l}$, where the last case follows from the identity $\autoij{k+1}_{\alpha\beta, i_1 i_2} \bchprod \autoij{k+1}_{\alpha\beta, i_0 i_1}  = \autoij{k+1}_{\alpha\beta, i_0 i_2}$, which in turn follows from the definition of $\autoij{k+1}_{\alpha\beta, i_0 i_1}$ in Condition \ref{assum:induction_hypothesis}. Therefore we have checked that 
	$$\partial (\reallywidetilde{\sauto{k+1}_{\alpha\beta, i_0 \cdots  i_l}}) \in \mathcal{A}^0(\simplexbdy_l) \otimes \bva{k+1}^{-1}_\beta(U_{i_0 \cdots i_l}).$$ By Lemma \ref{lem:simplex_lifting_lemma}, we obtain $\reallywidetilde{\sauto{k+1}_{\alpha \beta, i_0 \cdots i_l}} \in \mathcal{A}^0(\simplex_l) \otimes \bva{k+1}^{-1}_{\beta}(U_{i_0 \cdots i_l})$ satisfying
	\begin{align*}(\reallywidetilde{\sauto{k+1}_{\alpha \beta , i_0 \cdots i_l}})|_{\simplexbdy_l}  &= \partial (\reallywidetilde{\sauto{k+1}_{\alpha\beta, i_0 \cdots  i_l}}),\\
	\reallywidetilde{\sauto{k+1}_{\alpha \beta, i_0 \cdots i_l}} &= \sauto{k}_{\alpha\beta,i_0 \cdots i_l} \ \  \text{(mod $\mathbf{m}^{k+1}$)}.
	\end{align*}

	Therefore, we have an obstruction term $\thmcocyo{k+1}_{\alpha\beta\gamma,i_0\cdots i_l} \in \mathcal{A}^0(\simplex_l) \otimes \bva{k+1}^{-1}_\alpha(U_{i_0 \cdots i_l})$ given by
	\begin{multline*}
	\thmcocyo{k+1}_{\alpha\beta\gamma,i_0\cdots i_l} = \\
	\reallywidetilde{\sauto{k+1}_{\gamma\alpha,i_0 \cdots i_l}}\bchprod \glue{k+1}_{\gamma \alpha, i_0}(\reallywidetilde{\sauto{k+1}_{\beta \gamma,i_0 \cdots i_l}}) \bchprod \glue{k+1}_{\beta \alpha, i_0}(\reallywidetilde{\sauto{k+1}_{\alpha \beta,i_0 \cdots i_l}}) ,
	\end{multline*}
	which satisfies $\thmcocyo{k+1}_{\alpha\beta\gamma,i_0\cdots i_l} = 0 \ \text{(mod $\mathbf{m}^{k+1}$)}$. Direct computation gives
	$
	\exp([\mathtt{d}_{r,l}^*(\thmcocyo{k+1}_{\alpha\beta\gamma,i_0\cdots i_l}), \cdot] ) = \text{id}
	$
	for all $r = 0,\dots , l$. Using injectivity of $\bva{k}^{-1}_\alpha \rightarrow \text{Der}(\bva{k}^0_\alpha)$ we deduce $(\thmcocyo{k+1}_{\alpha\beta\gamma,i_0\cdots i_l}) |_{\simplexbdy_l} = 0$.
	
	Via \eqref{eqn:locally_free_identification} again, we may regard the term $\thmcocyo{k+1}_{\alpha\beta\gamma,i_0\cdots i_l}$ as lying in $\mathcal{A}^0_0(\simplex_l) \otimes \bva{0}^{-1}(U_{i_0 \cdots i_l}) \otimes (\mathbf{m}^{k+1}/\mathbf{m}^{k+2})$. By a similar argument as in the $l = 0$ case, we obtain an element $(\prescript{k+1}{}{\mathsf{c}}_{\alpha \beta,i_0\cdots i_l})_{\alpha \beta}$ whose image under the \v{C}ech differential is precisely $(\thmcocyo{k+1}_{\alpha\beta\gamma,i_0\cdots i_l})_{\alpha\beta\gamma}$ and such that $(\prescript{k+1}{}{\mathsf{c}}_{\alpha \beta,i_0\cdots i_l})|_{\simplexbdy_l} = 0$. Therefore setting
	$
	\sauto{k+1}_{\alpha\beta,i_0 \cdots i_l} := \prescript{k+1}{}{\mathsf{c}}_{\alpha \beta,i_0\cdots i_l} \bchprod \reallywidetilde{\sauto{k+1}_{\alpha\beta,i_0 \cdots i_l}}
	$
	solves the required cocycle condition \eqref{eqn:g_alpha_beta_cocycle}. We also have $\sauto{k+1}_{\alpha\beta,i_0 \cdots i_l} = \sauto{k}_{\alpha \beta, i_0 \cdots i_l} \ \text{(mod $\mathbf{m}^{k+1}$)}$ since $\prescript{k+1}{}{\mathsf{c}}_{\alpha \beta,i_0\cdots i_l} = 0 \ \text{(mod $\mathbf{m}^{k+1}$)}$ by our construction, and $(\sauto{k+1}_{\alpha\beta,i_0 \cdots i_l})|_{\simplexbdy_l} = \partial(\reallywidetilde{\sauto{k+1}_{\alpha\beta ,i_0 \cdots i_l}})$ which is the required compatibility condition \eqref{eqn:alpha_beta_gluing_explicit_form_on_simplex_boundary_relation}. This completes the proof of the lemma.
\end{proof}

\begin{proof}[Proof of Theorem \ref{lem:existence_compatible_gluing_data}]
	We prove by induction on the order $k$. For the initial case $k=0$, as $\bva{0}^*$ is globally defined on $X$ with $\bva{0}^*_\alpha = \bva{0}^*|_{V_\alpha}$ (see Definition \ref{def:higher_order_data}), we can (and have to) set $\glue{0}_{\alpha \beta , i } = \patch{0}_{\alpha\beta,i} = \text{id}$ and $\sauto{0}_{\alpha \beta , i_0 \cdots i_l} = 0$. The induction step is proven in the Key Lemma (Lemma \ref{lem:compatible_gluing_data}). 
\end{proof}

\subsubsection{Homotopy between two sets of gluing morphisms}\label{sec:homotopy_between_gluing_morphism}
The set of compatible gluing morphisms $\{\glue{k}_{\alpha\beta}\}$ constructed in Theorem \ref{lem:existence_compatible_gluing_data} is not unique (except for $k=0$). To understand the relation between two sets of such data, say, $\{\glue{k}_{\alpha\beta}(0)\}$ and $\{\glue{k}_{\alpha\beta}(1)\}$, we need, for each $k \in \inte_{\geq 0}$ and any pair $V_\alpha, V_\beta \in \mathcal{V}$, an isomorphism 
\begin{equation}\label{eqn:homotopy_alpha_beta_map}
\gluehom{k}_{\alpha \beta} :  \twc{k}^{*,*}_{\alpha;\alpha\beta }(\simplex_1) \rightarrow  \twc{k}^{*,*}_{\beta ; \alpha \beta}(\simplex_1),
\end{equation}
as a collection of maps $(\gluehom{k}_{\alpha\beta, I})_{I\in \mathcal{I}}$, such that
\begin{itemize}
\item for each $\varphi =(\varphi_I)_{I \in \mathcal{I}} \in  \twc{k}^{*,*}_{\alpha;\alpha \beta}(\simplex_1)$ with $\varphi_I \in \mathcal{A}^*(\simplex_1) \otimes \mathcal{A}^*(\simplex_l) \otimes \bva{k}^*_\alpha(U_I)$, we have $\left( \gluehom{k}_{\alpha\beta}(\varphi) \right)_I = (\gluehom{k}_{\alpha\beta, I} )(\varphi_I)$,

\item it preserves the algebraic structures $[\cdot,\cdot], \wedge$ obtained via tensoring with the dga $\mathcal{A}^*(\simplex_1)$, and

\item fits into the following commutative diagram
\begin{equation}\label{eqn:homotopy_alpha_beta_meaning}
	\xymatrix@1{\twc{k}^{*,*}_{\alpha;\alpha\beta } \ar[d]^{\glue{k}_{\alpha \beta}(0)} & &\ar[ll]_{\mathtt{r}_0^*}   \twc{k}^{*,*}_{\alpha;\alpha\beta }(\simplex_1) \ar[d]^{\gluehom{k}_{\alpha\beta}} \ar[rr]^{\mathtt{r}_1^*} & &  \twc{k}^{*,*}_{\alpha;\alpha\beta } \ar[d]^{\glue{k}_{\alpha\beta}(1)}\\
	\twc{k}^{*,*}_{\beta;\alpha\beta } &	&\ar[ll]_{\mathtt{r}_0^*}  \twc{k}^{*,*}_{\beta;\alpha\beta }(\simplex_1) \ar[rr]^{\mathtt{r}_1^*}& & \twc{k}^{*,*}_{\beta;\alpha\beta }}
\end{equation}
where $\twc{k}^{p,*}_{\alpha;\alpha\beta }(\simplex_1)$ is the Thom-Whitney complex constructed from the sheaf $\mathcal{A}^*(\simplex_1) \otimes \bva{k}^p_{\alpha}|_{V_{\alpha\beta}}$;
\end{itemize}
here the degree $*$ in $\twc{k}^{p,*}_{\alpha;\alpha\beta }(\simplex_1)$ refers to the total degree on $\mathcal{A}^*(\hsimplex_l)=\mathcal{A}^*(\simplex_1) \otimes \mathcal{A}^*(\simplex_l)$,
and $\mathtt{r}_j^*:\mathcal{A}^*(\hsimplex_l) \to \mathcal{A}^*(\simplex_l)$ is induced by the evaluation $\mathcal{A}^*(\simplex_1) \rightarrow \comp$ at $\mathtt{t}_j = 1$ for $j=0,1$ as in Notation \ref{not:de_rham_form_on_homotopy_simplex_boundary}. 
The isomorphisms $\gluehom{k}_{\alpha \beta}$'s are said to constitute a {\em homotopy from $\{\glue{k}_{\alpha\beta}(0)\}$ to $\{\glue{k}_{\alpha\beta}(1)\}$} if they further satisfy the following condition (cf. Condition \ref{assum:induction_hypothesis}):

\begin{condition}\label{cond:induction_hypothesis_for_homotopy}
	\begin{enumerate}
		\item for $U_i \subset V_\alpha \cap V_\beta$, we have 
		\begin{equation}\label{eqn:homotopy_alpha_beta_gluing_explicit_form_at_point}
		\gluehom{k}_{\alpha\beta, i} = \exp([ \hiauto{k}_{\alpha\beta,i}, \cdot ]) \circ \patch{k}_{\alpha\beta,i}
		\end{equation}
		for some $\hiauto{k}_{\alpha\beta,i} \in \mathcal{A}^0(\simplex_1) \otimes \bva{k}^{-1}_{\beta}(U_i)$ with $\hiauto{k}_{\alpha\beta,i} = 0 \ \text{(mod $\mathbf{m}$)}$;

		\item for $U_{i_0}, \dots, U_{i_l} \subset V_\alpha \cap V_{\beta}$, we have 
		\begin{equation}\label{eqn:homotopy_alpha_beta_gluing_explicit_form_on_simplex}
		\gluehom{k}_{\alpha\beta,i_0\cdots i_l} = \exp([\hsauto{k}_{\alpha\beta, i_0 \cdots i_l},\cdot]) \circ \left(\gluehom{k}_{\alpha\beta,i_0}|_{U_{i_0 \cdots i_l}}\right),
		\end{equation}
		for some element $\hsauto{k}_{\alpha\beta, i_0 \cdots i_l} \in\mathcal{A}^0(\simplex_1) \otimes \mathcal{A}^0(\simplex_l) \otimes \bva{k}^{-1}_\beta(U_{i_0 \cdots i_l})$ with $\hsauto{k}_{\alpha\beta, i_0 \cdots i_l} = 0 \ \text{(mod $\mathbf{m}$)}$;
		
		\item the elements $\hiauto{k}_{\alpha\beta,i}$'s satisfy the relation 
		\begin{equation}
		\mathtt{r}_j^*(\hiauto{k}_{\alpha\beta,i}) = \begin{dcases}
		\iauto{k}_{\alpha\beta,i}(0)& \text{for $j=0$,} \\
		\iauto{k}_{\alpha \beta ,i}(1)& \text{for $j=1$,} \\
		\end{dcases}
		\end{equation}
		where $\iauto{k}_{\alpha\beta,i}(j)$ is the element associated to $\glue{k}_{\alpha\beta,i}(j)$ as in  \eqref{eqn:alpha_beta_gluing_explicit_form_at_point};
		
		\item the elements $\hsauto{k}_{\alpha\beta,i_0 \cdots i_l}$'s satisfy the relation (cf. \eqref{eqn:alpha_beta_gluing_explicit_form_on_simplex_boundary_relation}):
		\begin{equation}\label{eqn:homotopy_alpha_beta_gluing_explicit_form_on_simplex_boundary_relation}
		\mathtt{d}_{j,l}^*( \hsauto{k}_{\alpha\beta, i_0 \cdots i_l})  =	\begin{dcases}
		\hsauto{k}_{\alpha\beta, i_0 \cdots \widehat{i_j} \cdots i_l} & \text{for $j>0$,}\\
		\hsauto{k}_{\alpha\beta, \widehat{i_0}  \cdots i_l} \bchprod \hautoij{k}_{\alpha\beta, i_0 i_1} & \text{for $j=0$,}
		\end{dcases}
		\end{equation}
		where $\hautoij{k}_{\alpha\beta, i_0 i_1} \in \mathcal{A}^0(\simplex_1) \otimes \bva{k}_{\beta}^{-1}(U_{i_0 i_1})$ is the unique element such that
		$\exp([\hautoij{k}_{\alpha\beta,i_0i_1}, \cdot]) \circ \gluehom{k}_{\alpha \beta, i_0} = \gluehom{k}_{\alpha \beta, i_1},$
		and the relation
		\begin{equation}\label{eqn:homotopy_alpha_beta_gluing_explicit_form_relation_on_interval}
		\mathtt{r}_j^*(\hsauto{k}_{\alpha\beta, i_0 \cdots i_l}) = \begin{dcases}
		\sauto{k}_{\alpha\beta,i_0 \cdots i_l}(0)& \text{for $j=0$,} \\
		\sauto{k}_{\alpha \beta ,i_0 \cdots i_l}(1)& \text{for $j=1$,} \\
		\end{dcases}
		\end{equation}
		where $\sauto{k}_{\alpha\beta,i_0 \cdots i_l}(j) \in \mathcal{A}^0(\simplex_l) \otimes \bva{k}^{-1}_\beta(U_{i_0 \cdots i_l})$ is the element associated to $\glue{k}_{\alpha\beta,i_0 \cdots i_l}(j)$ as in \eqref{eqn:alpha_beta_gluing_explicit_form_on_simplex} for $j=0,1$.
	\end{enumerate}
\end{condition}

\begin{definition}\label{def:homotopy_compatible_gluing_morphism}
	A homotopy $\{\gluehom{k}_{\alpha \beta}\}$ from $\{\glue{k}_{\alpha\beta}(0)\}$ to $\{\glue{k}_{\alpha\beta}(1)\}$ is said to be {\em compatible} if in addition the following conditions are satisfied:
	\begin{enumerate}
		\item $\gluehom{0}_{\alpha\beta} = \text{id}$ for all $\alpha, \beta$;
		\item (compatibility between different orders) for each $k \in \inte_{\geq 0}$ and any pair $V_\alpha, V_\beta \in \mathcal{V}$,
		\begin{equation}\label{eqn:homotopy_g_alpha_beta_order_compatibility}
		\gluehom{k}_{\alpha \beta} \circ \rest{k+1,k}_{\alpha} = \rest{k+1,k}_{\beta} \circ \gluehom{k+1}_{\alpha \beta};
		\end{equation}
		
		\item (cocycle condition) for each $k \in \inte_{\geq 0}$ and any triple $V_\alpha, V_\beta, V_\gamma \in \mathcal{V}$,
		\begin{equation}\label{eqn:homotopy_g_alpha_beta_cocycle}
		\gluehom{k}_{\gamma \alpha} \circ \gluehom{k}_{\beta \gamma} \circ \gluehom{k}_{\alpha \beta} = \text{id}
		\end{equation}
		when $\gluehom{k}_{\alpha \beta}$, $\gluehom{k}_{\beta \gamma}$, $\gluehom{k}_{\gamma \alpha}$ are restricted respectively to $\twc{k}^{*,*}_{\alpha;\alpha \beta \gamma} (\simplex_1)$, $\twc{k}^{*,*}_{\beta;\alpha \beta \gamma} (\simplex_1)$, $\twc{k}^{*,*}_{\gamma;\alpha \beta \gamma} (\simplex_1)$.
	\end{enumerate}
\end{definition}

The same induction argument as in Theorem \ref{lem:existence_compatible_gluing_data} proves the following:

\begin{prop}\label{lem:homotopy_compatible_gluing_morphism}
	Given any two sets of compatible gluing morphisms $\{\glue{k}_{\alpha\beta}(0)\}$ and $\{\glue{k}_{\alpha\beta}(1)\}$, there exists a compatible homotopy $\{\gluehom{k}_{\alpha\beta}\}$ from $\{\glue{k}_{\alpha\beta}(0)\}$ to $\{\glue{k}_{\alpha\beta}(1)\}$.
\end{prop}



\subsection{The \v{C}ech-Thom-Whitney complex}\label{sec:the_Cech_thom_whitney_complex}

The goal of this subsection is to construct a \v{C}ech-Thom-Whitney complex $\cech{k}^*(\twc{}, \glue{})$ for each $k\in \inte_{\geq 0}$ from a given set $\glue{} = \{\glue{k}_{\alpha\beta}\}$ of compatible gluing morphisms.



\begin{definition}\label{def:cech_thom_whitney_complex}
 	For $\ell \in \inte_{\geq 0}$, we let 
 	$$\twc{k}^{*,*}_{\alpha_0 \cdots \alpha_{\ell}}(\glue{}) \subset \bigoplus_{i=0}^\ell \twc{k}^{*,*}_{\alpha_i;\alpha_0 \cdots \alpha_{\ell}}$$
 	be the set of elements $(\varphi_0,\cdots, \varphi_{\ell})$ such that $ \varphi_{j} = \glue{k}_{\alpha_i \alpha_j} (\varphi_i)$. Then the {\em $k^{\text{th}}$-order \v{C}ech-Thom-Whitney complex over $X$}, $\cech{k}^*(\twc{}, \glue{})$ is defined by setting
 	$\cech{k}^\ell(\twc{}^{p,q},\glue{}) := \prod_{\alpha_0 \cdots \alpha_\ell} \twc{k}^{p,q}_{\alpha_0 \cdots \alpha_{\ell}}(\glue{})$ and $\cech{k}^\ell(\twc{},\glue{}):= \bigoplus_{p,q} \cech{k}^\ell(\twc{}^{p,q},\glue{})$ 
 	for each $k \in \inte_{\geq 0}$.
 	
 	This is equipped with the {\em \v{C}ech differential} $\cechd{k}_\ell := \sum_{j =0}^{\ell+1} (-1)^j\restmap_{j,\ell+1}: \cech{k}^{\ell} (\twc{},\glue{}) \to \cech{k}^{\ell+1}(\twc{},\glue{})$, where $\restmap_{j,\ell} : \cech{k}^{\ell-1}(\twc{},\glue{}) \rightarrow \cech{k}^{\ell}(\twc{},\glue{})$ is the natural restriction map defined componentwise by the map $\restmap_{j,\ell}: \twc{k}^{*,*}_{\alpha_0 \cdots \widehat{\alpha_j} \cdots  \alpha_{\ell}}(\glue{}) \rightarrow \twc{k}^{*,*}_{\alpha_0 \cdots \alpha_{\ell}}(\glue{})$ which in turn comes from \eqref{eqn:thom_whitney_restriction_map}.
 
 	We define the {\em $k^{\text{th}}$-order complex of polyvector fields over $X$} by
 	$$\polyv{k}^{*,*}(\glue{}) := \text{Ker}(\cechd{k}_{0})$$
 	and denote the natural inclusion $\polyv{k}^{*,*}(\glue{}) \rightarrow \cech{k}^0(\twc{},\glue{})$ by $\cechd{k}_{-1}$, so we have the following sequence of maps
 	\begin{multline}\label{eqn:cech_thom_whitney_complex}
 	0 \rightarrow \polyv{k}^{p,q}(\glue{}) \rightarrow \cech{k}^0(\twc{}^{p,q},\glue{}) \rightarrow\\ \cech{k}^1(\twc{}^{p,q},\glue{}) \rightarrow \cdots \rightarrow \cech{k}^{\ell}(\twc{}^{p,q},\glue{}) \rightarrow \cdots.
 	\end{multline}
\end{definition}

For $\ell\in \inte_{\geq 0}$ and $k\geq l$, there is a natural map $\rest{k,l}: \cech{k}^\ell(\twc{}^{p,q},\glue{})  \rightarrow \cech{l}^{\ell}(\twc{}^{p,q},\glue{})$ defined componentwise by the map $\rest{k,l}_{\alpha_j}: \twc{k}^{p,q}_{\alpha_j;\alpha_0 \cdots \alpha_{\ell}} \rightarrow \twc{l}^{p,q}_{\alpha_j;\alpha_0 \cdots \alpha_{\ell}}$ obtained from $\rest{k,l}_{\alpha}: \bva{k}_{\alpha}^* \rightarrow \bva{l}_{\alpha}^*$ (see Definition \ref{def:higher_order_data}). Similarly, we have the natural maps $\rest{k,l}:\polyv{k}^{p,q}(\glue{}) \rightarrow \polyv{l}^{p,q}(\glue{})$.

\begin{definition}\label{def:inverse_limit_cech_thom_whitney_complex}
	The {\em \v{C}ech-Thom-Whitney complex} $\cech{}^{l}(\twc{},\glue{}) = \bigoplus_{p,q} \cech{}^\ell(\twc{}^{p,q},\glue{})$ is defined by taking inverse limit $\cech{}^\ell(\twc{}^{p,q},\glue{}) :=\varprojlim_k \cech{k}^{\ell}(\twc{}^{p,q},\glue{})$ along the sequence of maps $\rest{k+1,k}: \cech{k+1}^\ell(\twc{}^{p,q},\glue{})  \to \cech{k}^{\ell}(\twc{}^{p,q},\glue{})$.
	
	The {\em complex of polyvector fields} $\polyv{}^{*,*}(\glue{}) = \bigoplus_{p,q} \polyv{}^{p,q}(\glue{})$ is defined by taking the inverse limit $\polyv{}^{p,q}(\glue{}):= \varprojlim_k \polyv{k}^{p,q}(\glue{})$ along the maps $\rest{k+1,k}: \polyv{k+1}^{p,q}(\glue{}) \rightarrow \polyv{k}^{p,q}(\glue{})$
	
	The maps $\rest{k,l}$'s commute with the \v{C}ech differentials $\cechd{k}_{\ell}$'s and $\cechd{l}_{\ell}$'s, so we have the following sequence of maps 
	\begin{multline}\label{eqn:cech_thom_whitney_complex_all_order}
	0 \rightarrow \polyv{}^{p,q}(\glue{}) \rightarrow \cech{}^0(\twc{}^{p,q},\glue{}) \rightarrow\\
	 \cech{}^1(\twc{}^{p,q},\glue{}) \rightarrow \cdots \rightarrow \cech{}^{\ell}(\twc{}^{p,q},\glue{}) \rightarrow \cdots.
	\end{multline}
\end{definition}

\begin{lemma}\label{lem:exactness_of_cech_thom_whitney_complex}
	Given $\prescript{k+1}{}{\mathtt{w}} \in  \cech{k+1}^{\ell+1}(\twc{},\glue{})$ with $\cechd{k+1}_{\ell+1} (\prescript{k+1}{}{\mathtt{w}} ) = 0$ and $\prescript{k}{}{\mathtt{v}} \in \cech{k}^{\ell}(\twc{},\glue{})$ satisfying $\cechd{k}_{\ell}( \prescript{k}{}{\mathtt{v}}) = \prescript{k+1}{}{\mathtt{w}} \ \text{(mod $\mathbf{m}^{k+1}$)}$, there exists a lifting $\prescript{k+1}{}{\mathtt{v}} \in \cech{k+1}^{\ell}(\twc{},\glue{})$ such that $\cechd{k+1}_{\ell} (\prescript{k+1}{}{\mathtt{v}}) = \prescript{k+1}{}{\mathtt{w}}$. 
	As a consequence, both \eqref{eqn:cech_thom_whitney_complex} and \eqref{eqn:cech_thom_whitney_complex_all_order} are exact sequences.
\end{lemma}

\begin{proof}
	We only need to prove the first statement of the lemma because the second statement follows by induction on $k$ (note that the initial case for this induction is $k=-1$ where we take the trivial sequence whose terms are all zero).
	
	Without loss of generality, we assume that $\prescript{k+1}{}{\mathtt{w}} \in \cech{k+1}^{\ell+1}(\twc{}^{p,q}, \glue{})$ and $\prescript{k}{}{\mathtt{v}} \in \cech{k}^{\ell}(\twc{}^{p,q},\glue{})$ for some fixed $p$ and $q$. We need to construct $\prescript{k+1}{}{\mathtt{v}}_{\alpha_0 \cdots \alpha_\ell} \in \twc{k+1}^{p,q}_{\alpha_0 \cdots \alpha_{\ell}}(\glue{})$ for every multi-index $(\alpha_0, \dots , \alpha_\ell)$ which, by Definition \ref{def:cech_thom_whitney_complex}, can be written as
	$$\prescript{k+1}{}{\mathtt{v}}_{\alpha_0 \cdots \alpha_\ell} = \left(\prescript{k+1}{}{\mathtt{v}}_{\alpha_0;\alpha_0 \cdots \alpha_\ell},\cdots, \prescript{k+1}{}{\mathtt{v}}_{\alpha_\ell;\alpha_0 \cdots \alpha_\ell}\right)$$ 
	satisfying $\prescript{k+1}{}{\mathtt{v}}_{\alpha_j;\alpha_0 \cdots \alpha_\ell} = \glue{k+1}_{\alpha_i\alpha_j}(\prescript{k+1}{}{\mathtt{v}}_{ \alpha_i;\alpha_0 \cdots \alpha_\ell})$, and each component $\prescript{k+1}{}{\mathtt{v}}_{\alpha_j;\alpha_0 \cdots \alpha_\ell}$ is of the form 
	$\prescript{k+1}{}{\mathtt{v}}_{\alpha_j;\alpha_0 \cdots \alpha_\ell} = (\prescript{k+1}{}{\mathtt{v}}_{\alpha_j;\alpha_0 \cdots \alpha_\ell; i_0\cdots i_l})_{i_0\cdots i_l}$,
	where $U_{i_r} \subset V_{\alpha_0 \cdots \alpha_\ell}$ and $\prescript{k+1}{}{\mathtt{v}}_{\alpha_j;\alpha_0 \cdots \alpha_\ell; i_0\cdots i_l} \in \mathcal{A}^q(\simplex_l) \otimes \bva{k+1}^{p}_{\alpha_j}(U_{i_0 \cdots i_l})$, such that $\mathtt{d}_{j,l}^*\left(\prescript{k+1}{}{\mathtt{v}}_{\alpha_j;\alpha_0 \cdots \alpha_\ell; i_0\cdots i_l}\right) = \prescript{k+1}{}{\mathtt{v}}_{\alpha_j;\alpha_0 \cdots \alpha_\ell; i_0 \cdots \hat{i}_j \cdots i_l} |_{U_{i_0 \cdots i_l}}$ (see Definition \ref{def:thom_whitney_general}). We will use induction on $l$ to prove the existence of such an element.
	
	The initial case is $l = q$. We fix $U_{i_0\cdots i_q}$ and consider all the multi-indices $(\alpha_0,\cdots, \alpha_\ell)$ such that $U_{i_{r}} \subset V_{\alpha_0 \cdots \alpha_{\ell}}$ for $r = 0, \ldots, q$. Since $\bva{k+1}^{p}_{\alpha_0}$ is free over $\cfrk{k+1}$, we can take a lifting $\reallywidetilde{\prescript{k+1}{}{\mathtt{v}}_{\alpha_0;\alpha_0 \cdots \alpha_\ell; i_0\cdots i_l}} \in \mathcal{A}^q(\simplex_q) \otimes \bva{k+1}^{p}_{\alpha_0}(U_{i_0 \cdots i_q})$ of $\prescript{k}{}{\mathtt{v}}_{\alpha_0;\alpha_0 \cdots \alpha_\ell; i_0\cdots i_q}$. Then we let $$\reallywidetilde{\prescript{k+1}{}{\mathtt{v}}_{\alpha_j;\alpha_0 \cdots \alpha_\ell; i_0\cdots i_q}}:= \glue{k+1}_{\alpha_0 \alpha_j,i_0 \cdots i_q}( \reallywidetilde{\prescript{k+1}{}{\mathtt{v}}_{\alpha_0;\alpha_0 \cdots \alpha_\ell; i_0\cdots i_q}})$$
	for $j = 1, \ldots, \ell$ and set
	$$\reallywidetilde{\prescript{k+1}{}{\mathtt{v}}_{\alpha_0 \cdots \alpha_\ell;i_0 \cdots i_q}} := (\reallywidetilde{\prescript{k+1}{}{\mathtt{v}}_{\alpha_0;\alpha_0 \cdots \alpha_\ell;i_0 \cdots i_q}},\cdots, \reallywidetilde{\prescript{k+1}{}{\mathtt{v}}_{\alpha_\ell;\alpha_0 \cdots \alpha_\ell;i_0\cdots i_q}}).$$ 
	Now the element
	\begin{multline*}
	\reallywidetilde{\prescript{k+1}{}{\mathtt{w}}_{\alpha_0 \cdots \alpha_{\ell+1}; i_0 \cdots i_q}} := \reallywidetilde{\prescript{k+1}{}{\mathtt{v}}_{\widehat{\alpha_0} \cdots \alpha_{\ell+1}; i_0 \cdots i_q}} - \reallywidetilde{\prescript{k+1}{}{\mathtt{v}}_{\alpha_0 \widehat{\alpha_1} \cdots \alpha_{\ell+1}; i_0 \cdots i_q}} \\
	+ \cdots + (-1)^j (\reallywidetilde{\prescript{k+1}{}{\mathtt{v}}_{\alpha_0 \cdots \widehat{\alpha_j} \cdots \alpha_{\ell+1}; i_0 \cdots i_q}}) + \cdots + (-1)^{\ell+1} (\reallywidetilde{\prescript{k+1}{}{\mathtt{v}}_{\alpha_0 \cdots \widehat{\alpha_{\ell+1}}; i_0 \cdots i_q}})\\ - \prescript{k+1}{}{\mathtt{w}}_{\alpha_0 \cdots \alpha_{\ell+1}; i_0 \cdots i_q},
	\end{multline*}
	satisfies the condition that $\reallywidetilde{\prescript{k+1}{}{\mathtt{w}}_{\alpha_0 \cdots \alpha_{\ell+1}; i_0 \cdots i_l}} = 0 \ \text{(mod $\mathbf{m}^{k+1}$)}$.
	
	Under the identification \eqref{eqn:locally_free_identification}, we treat $\reallywidetilde{\prescript{k+1}{}{\mathtt{w}}_{\alpha_0 \cdots \alpha_{\ell+1}; i_0 \cdots i_q}} \in \mathcal{A}^q(\simplex_q) \otimes \bva{0}^{p}(U_{i_0\cdots i_q}) \otimes (\mathbf{m}^{k+1}/\mathbf{m}^{k+2})$. So the collection $(\reallywidetilde{\prescript{k+1}{}{\mathtt{w}}_{\alpha_0 \cdots \alpha_{\ell+1}; i_0 \cdots i_q}})_{\alpha_0 \cdots \alpha_{\ell+1}}$ is an $(\ell+1)$-cocycle in the \v{C}ech complex $\check{\mathcal{C}}^{\ell+1}(I_{i_0\cdots i_q}, \bva{0}^p) \otimes \mathcal{A}^q(\simplex_q) \otimes (\mathbf{m}^{k+1}/\mathbf{m}^{k+2})$. By Lemma \ref{lem:cech_complex_of_a_point}, there exists $(\prescript{k+1}{}{\mathtt{c}}_{\alpha_0 \cdots \alpha_{\ell}; i_0 \cdots i_q})_{\alpha_0 \cdots \alpha_{\ell}} \in \check{\mathcal{C}}^{\ell}(I_{i_0\cdots i_q}, \bva{0}^p) \otimes \mathcal{A}^q(\simplex_q) \otimes (\mathbf{m}^{k+1}/\mathbf{m}^{k+2})$ whose image under the \v{C}ech differential is precisely $(\reallywidetilde{\prescript{k+1}{}{\mathtt{w}}_{\alpha_0 \cdots \alpha_{\ell+1}; i_0 \cdots i_q}})_{\alpha_0 \cdots \alpha_{\ell+1}}$. Therefore if we let
	$$
	\prescript{k+1}{}{\mathtt{v}}_{\alpha_0 \cdots \alpha_\ell; i_0\cdots i_q}:= \reallywidetilde{\prescript{k+1}{}{\mathtt{v}}_{\alpha_0 \cdots \alpha_\ell; i_0\cdots i_q}} - \prescript{k+1}{}{\mathtt{c}}_{\alpha_0 \cdots \alpha_{\ell}; i_0 \cdots i_q},
	$$
	then its image under the \v{C}ech differential is $(\prescript{k+1}{}{\mathtt{w}}_{\alpha_0 \cdots \alpha_{\ell+1}; i_0 \cdots i_q})_{\alpha_0 \cdots \alpha_{\ell+1}}$ as desired.
	
	Next we suppose that we are given $\prescript{k+1}{}{\mathtt{v}}_{\alpha_0 \cdots \alpha_\ell; i_0\cdots i_l}$ for some $l \geq q$. Then we need to construct $\prescript{k+1}{}{\mathtt{v}}_{\alpha_0 \cdots \alpha_\ell; i_0\cdots i_{l+1}}$ for any $U_{i_0\cdots i_{l+1}}$ and $V_{\alpha_0\cdots \alpha_\ell}$ such that $U_{i_r} \subset V_{\alpha_0\cdots \alpha_\ell}$ for $r = 1, \ldots, l+1$. We fixed $U_{i_0 \cdots i_{l+1}}$ and consider one such $V_{\alpha_0\cdots \alpha_\ell}$. Letting 
	$$
	\partial(\reallywidetilde{\prescript{k+1}{}{\mathtt{v}}_{\alpha_0 ;\alpha_0 \cdots \alpha_\ell; i_0\cdots i_{l+1}}} ):= \left(\prescript{k+1}{}{\mathtt{v}}_{\alpha_0 ;\alpha_0 \cdots \alpha_\ell; \widehat{i_0}\cdots i_{l+1}},\cdots ,\prescript{k+1}{}{\mathtt{v}}_{\alpha_0 ;\alpha_0 \cdots \alpha_\ell; i_0\cdots \widehat{i_{l+1}}} \right)
	$$
	gives an element in $\mathcal{A}^q(\simplexbdy_{l+1}) \otimes \bva{k+1}^{p}_{\alpha_0}(U_{i_0\cdots i_{l+1}})$. Using Lemma \ref{lem:simplex_lifting_lemma}, we construct
	an element $\reallywidetilde{\prescript{k+1}{}{\mathtt{v}}_{\alpha_0 ;\alpha_0 \cdots \alpha_\ell; i_0\cdots i_{l+1}}} \in \mathcal{A}^q(\simplex_{l+1}) \otimes \bva{k+1}^{p}_{\alpha_0}(U_{i_0\cdots i_{l+1}})$ such that
	\begin{align*}
	{\prescript{k+1}{}{\mathtt{v}}_{\alpha_0 ;\alpha_0 \cdots \alpha_\ell; i_0\cdots i_{l+1}}} & = \prescript{k}{}{\mathtt{v}}_{\alpha_0 ;\alpha_0 \cdots \alpha_\ell; i_0\cdots i_{l+1}} \ \text{(mod $\mathbf{m}^{k+1}$)},\\ \reallywidetilde{\prescript{k+1}{}{\mathtt{v}}_{\alpha_0 ;\alpha_0 \cdots \alpha_\ell; i_0\cdots i_{l+1}}}|_{\simplexbdy_{l+1}} & = \partial(\reallywidetilde{\prescript{k+1}{}{\mathtt{v}}_{\alpha_0 ;\alpha_0 \cdots \alpha_\ell; i_0\cdots i_{l+1}}} ).
	\end{align*}
	We then let $\reallywidetilde{\prescript{k+1}{}{\mathtt{v}}_{\alpha_j ;\alpha_0 \cdots \alpha_\ell; i_0\cdots i_{l+1}}} := \glue{k+1}_{\alpha_0 \alpha_j, i_0 \cdots i_{l+1}} (\reallywidetilde{\prescript{k+1}{}{\mathtt{v}}_{\alpha_0 ;\alpha_0 \cdots \alpha_\ell; i_0\cdots i_{l+1}}})$ for $j = 1, \ldots, \ell$ and set $$\reallywidetilde{\prescript{k+1}{}{\mathtt{v}}_{\alpha_0 \cdots \alpha_\ell; i_0\cdots i_{l+1}}}:=(\reallywidetilde{\prescript{k+1}{}{\mathtt{v}}_{\alpha_0 ;\alpha_0 \cdots \alpha_\ell; i_0\cdots i_{l+1}}},\cdots ,\reallywidetilde{\prescript{k+1}{}{\mathtt{v}}_{\alpha_\ell ;\alpha_0 \cdots \alpha_\ell; i_0\cdots i_{l+1}}}),$$
	The elements $\prescript{k+1}{}{\delta}_\ell (\reallywidetilde{\prescript{k+1}{}{\mathtt{v}}_{\alpha_j ;\alpha_0 \cdots \alpha_\ell; i_0\cdots i_{l+1}}})$ and $\prescript{k+1}{}{\mathtt{w}}_{\alpha_0 \cdots \alpha_{\ell+1}; i_0 \cdots i_{l+1}}$ agree modulo $\mathbf{m}^{k+1}$ and on the boundary $\simplexbdy_{l+1}$ of the simplex $\simplex_{l+1}$, so the rest of the proof of this induction step would be the same as the initial case $l = q$. 
\end{proof}

\begin{corollary}\label{cor:corollary_to_exactness_of_cech_thom_whitney}
	For $k, \ell \in \inte_{\geq 0}$, the map $\rest{k+1,k}: \cech{k+1}^\ell(\twc{}^{p,q},\glue{}) \rightarrow \cech{k}^\ell(\twc{}^{p,q},\glue{})$ and the induced map $\rest{\infty,k}: \cech{}^\ell(\twc{}^{p,q},\glue{}) \rightarrow \cech{k}^\ell(\twc{}^{p,q},\glue{})$ are surjective; in particular, both $\rest{k+1,k} : \polyv{k+1}^{p,q}(\glue{}) \rightarrow \polyv{k}^{p,q}(\glue{})$ and  $\rest{\infty,k} : \polyv{}^{p,q}(\glue{}) \rightarrow \polyv{k}^{p,q}(\glue{})$ are surjective. 
\end{corollary}

\begin{proof}
	It suffices to show that for any $\prescript{k}{}{\mathtt{v}} \in  \cech{k}^{\ell}(\twc{},\glue{})$, there exists $\prescript{k+1}{}{\mathtt{v}} \in  \cech{k+1}^{\ell}(\twc{},\glue{})$ such that $\rest{k+1,k}(\prescript{k+1}{}{\mathtt{v}}) = \prescript{k}{}{\mathtt{v}}$. If  $\cechd{k}_{\ell}(\prescript{k}{}{\mathtt{v}}) = 0$, then applying Lemma \ref{lem:exactness_of_cech_thom_whitney_complex} with $\prescript{k+1}{}{\mathtt{w}} = 0$ gives the desired $\prescript{k+1}{}{\mathtt{v}}$. For a general $\prescript{k}{}{\mathtt{v}}$, we let $\prescript{k}{}{\mathtt{w}} = \prescript{k}{}{\delta}(\prescript{k}{}{\mathtt{v}})$. Since $\cechd{k}_{\ell}(\prescript{k}{}{\mathtt{w}} )= 0$, we can find a lifting $\prescript{k+1}{}{\mathtt{w}}$ such that $\rest{k+1,k}(\prescript{k+1}{}{\mathtt{w}}) = \prescript{k}{}{\mathtt{w}}$. Applying Lemma \ref{lem:exactness_of_cech_thom_whitney_complex} again, we obtain $\prescript{k+1}{}{\mathtt{v}}$ satisfying $\rest{k+1,k}(\prescript{k+1}{}{\mathtt{v}}) = \prescript{k}{}{\mathtt{v}}$. 
	\end{proof}

\begin{definition}\label{def:homotopy_cech_thom_whitney_complex}
	Let $\glue{}(0) = \{\glue{k}_{\alpha\beta}(0)\}$ and $\glue{}(1) = \{\glue{k}_{\alpha\beta}(1)\}$ be two sets of compatible gluing morphisms, and $h = \{\gluehom{k}_{\alpha\beta}\}$ be a compatible homotopy from $\glue{}(0)$ to $\glue{}(1)$. 
	For $\ell \geq 0$, we let $\twc{k}^{p,q}_{\alpha_0 \cdots \alpha_{\ell}}(\gluehom{}) \subset \bigoplus_{i=0}^\ell  \twc{k}^{p,q}_{\alpha_i;\alpha_0 \cdots \alpha_{\ell}}(\simplex_1)$ be the set of elements $(\varphi_0,\cdots, \varphi_{\ell})$ such that $ \varphi_{j} = \gluehom{k}_{\alpha_i \alpha_j} (\varphi_i)$. 
	Then,  for each $k \in \inte_{\geq 0}$, the {\em $k^{\text{th}}$-order homotopy \v{C}ech-Thom-Whitney complex} is defined by setting $\cech{k}^\ell(\twc{}^{p,q},\gluehom{}) := \prod_{\alpha_0 \cdots \alpha_\ell} \twc{k}^{p,q}_{\alpha_0 \cdots \alpha_{\ell}}(\gluehom{})$ and $\cech{k}^\ell(\twc{},\gluehom{}) = \bigoplus_{p,q}  \cech{k}^\ell(\twc{}^{p,q},\gluehom{})$. We have the natural restriction map $ \restmap_{j,\ell} : \cech{k}^{\ell-1}(\twc{},\gluehom{}) \rightarrow \cech{k}^{\ell}(\twc{},\gluehom{})$ as in Definition \ref{def:cech_thom_whitney_complex}.
	
	Let $\cechd{k}_\ell:= \sum_{j =0}^{\ell+1}(-1)^j \restmap_{j,\ell+1} : \cech{k}^{\ell} (\twc{},\gluehom{}) \rightarrow \cech{k}^{\ell+1}(\twc{},\gluehom{})$ be the {\em \v{C}ech differential} acting on $\cech{k}^*(\twc{}, \gluehom{})$. Then the {\em $k^{\text{th}}$ order homotopy polyvector field on $X$} is defined as $\polyv{k}^{*,*}(\gluehom{}):= \text{Ker}(\cechd{k}_{0})$. So we have the following sequences
	\begin{align}
	0 \rightarrow \polyv{k}^{p,q}(\gluehom{}) \rightarrow \cech{k}^0(\twc{}^{p,q},\gluehom{}) \rightarrow \cdots \rightarrow \cech{k}^{\ell}(\twc{}^{p,q},\gluehom{}) \rightarrow \cdots  \label{eqn:homotopy_cech_thom_whitney_complex},\\
	0 \rightarrow  \polyv{}^{p,q}(\gluehom{}) \rightarrow \cech{}^0(\twc{}^{p,q},\gluehom{})\rightarrow \cdots \rightarrow \cech{}^{\ell}(\twc{}^{p,q},\gluehom{}) \rightarrow \cdots \label{eqn:homotopy_cech_thom_whitney_complex_all_order},
	\end{align}
	where \eqref{eqn:homotopy_cech_thom_whitney_complex_all_order} is obtained from \eqref{eqn:homotopy_cech_thom_whitney_complex} by taking the inverse limit. We also write $\cech{}^{\ell} (\twc{},\gluehom{}) := \bigoplus_{p,q} \cech{}^\ell (\twc{}^{p,q},\gluehom{}) $. 
	
	We further let $\prescript{k}{}{\mathtt{r}}_j^* : \cech{k}^\ell(\twc{}^{p,q},\gluehom{}) \rightarrow \cech{k}^{\ell}(\twc{}^{p,q},\glue{}(j))$ and $\prescript{k}{}{\mathtt{r}}_j^* : \polyv{k}^{p,q}(\gluehom{}) \rightarrow \polyv{k}^{p,q}(\glue{}(j))$ be the maps induced by $\mathtt{r}_j^* : \mathcal{A}^*(\hsimplex_l) \rightarrow \mathcal{A}^*(\simplex_l)$ for $j=0,1$, and let $\mathtt{r}_j^* := \varprojlim_k \prescript{k}{}{\mathtt{r}}_j^*$. Then we have the following diagram
	$$
	\xymatrix@1{ 			0  \ar[r] &  \polyv{}^{*,*}(\glue{}(0)) \ar[r]  & \cech{}^0( \twc{},\glue{}(0)) \ar[r] & \cdots \ar[r] &   \cech{}^{\ell}( \twc{},\glue{}(0)) \ar[r] &  \cdots \\
		0  \ar[r] & \polyv{}^{*,*}(\gluehom{}) \ar[r] \ar[u]_{\mathtt{r}_0^*} \ar[d]^{\mathtt{r}_1^*} & \cech{}^0(\twc{},\gluehom{}) \ar[r] \ar[u]_{\mathtt{r}_0^*} \ar[d]^{\mathtt{r}_1^*} & \cdots \ar[r] &   \cech{}^{\ell}(\twc{},\gluehom{}) \ar[r] \ar[u]_{\mathtt{r}_0^*} \ar[d]^{\mathtt{r}_1^*} &  \cdots \\
		0  \ar[r] &  \polyv{}^{*,*}(\glue{}(1)) \ar[r]  & \cech{}^0( \twc{},\glue{}(1)) \ar[r] & \cdots \ar[r] &   \cech{}^{\ell}( \twc{},\glue{}(1)) \ar[r] &  \cdots
	}.
	$$
\end{definition}

Similar proofs as those of Lemma \ref{lem:exactness_of_cech_thom_whitney_complex} and Corollary \ref{cor:corollary_to_exactness_of_cech_thom_whitney} yield the following lemma:
\begin{lemma}\label{lem:homotopy_exactness_of_cech_thom_whitney_complex}
	Given $\prescript{k+1}{}{\mathtt{w}} \in  \cech{k+1}^{\ell+1}(\twc{},\gluehom{})$ with $\cechd{k+1}_{\ell+1} (\prescript{k+1}{}{\mathtt{w}} ) = 0$, $\prescript{k+1}{}{\mathtt{a}}_j \in \cech{k+1}^{\ell}(\twc{},\glue{}(j))$ satisfying $\cechd{k+1}_{\ell}(\prescript{k+1}{}{\mathtt{a}}_j) = \prescript{k+1}{}{\mathtt{r}}_j^*(\prescript{k+1}{}{\mathtt{w}})$ and $\prescript{k}{}{\mathtt{v}} \in \cech{k}^{\ell}(\twc{},\gluehom{})$ such that $\cechd{k}_{\ell}( \prescript{k}{}{\mathtt{v}}) = \prescript{k+1}{}{\mathtt{w}} \ \text{(mod $\mathbf{m}^{k+1}$)}$ and $\prescript{k}{}{\mathtt{r}}_j^*(\prescript{k}{}{\mathtt{v}}) = \prescript{k+1}{}{\mathtt{a}}_j \ \text{(mod $\mathbf{m}^{k+1}$)}$, there exists $\prescript{k+1}{}{\mathtt{v}} \in \cech{k+1}^{\ell}(\twc{},\gluehom{})$ such that\\ $\rest{k+1,k}(\prescript{k+1}{}{\mathtt{v}} ) = \prescript{k}{}{\mathtt{v}} $, $\prescript{k+1}{}{\mathtt{r}}_j^*(\prescript{k+1}{}{\mathtt{v}}) = \prescript{k+1}{}{\mathtt{a}}_j$ and $\cechd{k+1}_{\ell} (\prescript{k+1}{}{\mathtt{v}}) = \prescript{k+1}{}{\mathtt{w}}$. 
	As a consequence, both \eqref{eqn:homotopy_cech_thom_whitney_complex} and \eqref{eqn:homotopy_cech_thom_whitney_complex_all_order} are exact sequences.
	
	Furthermore, the maps $\rest{\infty,k} : \cech{}^{\ell}(\twc{}^{p,q},\gluehom{}) \rightarrow \cech{k}^{\ell}(\twc{}^{p,q},\gluehom{})$ and $\rest{\infty,k} :  \polyv{}^{p,q}(\gluehom{}) \rightarrow \polyv{k}^{p,q}(\gluehom{})$, as well as $\prescript{k}{}{\mathtt{r}}_j^* :  \polyv{k}^{*,*}(\gluehom{}) \rightarrow  \polyv{k}^{*,*}(\glue{}(j))$ and $\mathtt{r}_j^* :  \polyv{}^{*,*}(\gluehom{}) \rightarrow  \polyv{}^{*,*}(\glue{}(j))$ are all surjective.
\end{lemma}

\subsection{An almost dgBV algebra structure}\label{sec:construction_of_differentials}
The complex $\polyv{}^{*,*}(\glue{})$ (as well as $ \polyv{}^{*,*}(\gluehom{})$) constructed in \S \ref{sec:the_Cech_thom_whitney_complex} is only a graded vector space. In this subsection, we equip it with two differential operators $\bar{\partial}$ and $\bvd{}$, turning it into an almost dgBV algebra. 

We fix a set of compatible gluing morphisms $\glue{} = \{\glue{k}_{\alpha \beta}\}$ consisting of isomorphisms $\glue{k}_{\alpha\beta} : \twc{k}^{*,*}_{\alpha,\alpha\beta} \rightarrow \twc{k}^{*,*}_{\beta,\alpha\beta}$ for each $k \in \inte_{\geq 0}$ and pair $V_\alpha, V_\beta \subset \mathcal{V}$. Both $\twc{k}^{*,*}_{\alpha,\alpha\beta}$ and $\twc{k}^{*,*}_{\beta,\alpha\beta}$ are dgBV algebras with differentials and BV operators given by $\prescript{k}{}{\pdb}_{\alpha}$, $\prescript{k}{}{\pdb}_\beta$ and $\bvd{k}_{\alpha}$, $\bvd{k}_{\beta}$ respectively. 

To simplify notations, we introduce the power series
\begin{equation}\label{eqn:power_series_in_proof_gluing_morphism_BV_operator_comparison}
\mathtt{T}(x): = \frac{e^{-x}-1}{x} = \sum_{k=0}^{\infty} \frac{(-1)^{k+1} x^{k}}{(k+1)!}. 
\end{equation}

\begin{lemma}\label{lem:gluing_morphism_partial_bar_comparison}
	For each $k \in \inte_{\geq 0}$ and pair $V_\alpha, V_\beta \subset \mathcal{V}$, there exists $\prescript{k}{}{\mathtt{w}}_{\alpha;\alpha\beta} \in \twc{k}^{-1,1}_{\alpha,\alpha\beta}$ such that $\prescript{k}{}{\mathtt{w}}_{\alpha;\alpha\beta} = 0 \ \text{(mod $\mathbf{m}$)}$, $\prescript{k+1}{}{\mathtt{w}}_{\alpha;\alpha\beta} = \prescript{k}{}{\mathtt{w}}_{\alpha;\alpha\beta}  \ \text{(mod $\mathbf{m}^{k+1}$)}$ and
	$
	\glue{k}_{\beta \alpha} \circ \prescript{k}{}{\pdb}_{\beta} \circ \glue{k}_{\alpha\beta} - \prescript{k}{}{\pdb}_{\alpha} = [\prescript{k}{}{\mathtt{w}}_{\alpha;\alpha\beta}, \cdot].
	$
	Furthermore, if we let $\prescript{k}{}{\mathtt{w}}_{\alpha\beta}:=(\prescript{k}{}{\mathtt{w}}_{\alpha;\alpha\beta}, \glue{k}_{\alpha\beta}(\prescript{k}{}{\mathtt{w}}_{\alpha;\alpha\beta}))$, then $(\prescript{k}{}{\mathtt{w}}_{\alpha\beta})_{\alpha\beta}$ is a \v{C}ech $1$-cocycle in $\cech{k}^1(\twc{}^{-1,1},\glue{})$.
\end{lemma}

\begin{proof}
	Applying Lemma \ref{lem:gauge_action_on_differential} to the dgLa $\mathcal{A}^*(\simplex_l)\otimes \bva{k}^{*}_{\beta}(U_{i_0\cdots i_l})[-1]$ by taking $\xi = 0$, we get
	\begin{multline*}
	\exp(-[\sauto{k}_{\alpha\beta,i_0 \cdots i_l}, \cdot]) \circ d  \circ \exp([\sauto{k}_{\alpha\beta,i_0 \cdots i_l}, \cdot]) =\\d - \Big[\mathtt{T}([\sauto{k}_{\alpha\beta,i_0\cdots i_l}, \cdot ]) (d \sauto{k}_{\alpha\beta,i_0\cdots i_l}), \cdot\Big],
	\end{multline*}
	where $d$ is the de Rham differential acting on $\mathcal{A}^*(\simplex_l)$ (recall that $\prescript{k}{}{\pdb}_{\beta}$ is induced by the de Rham differential on $\mathcal{A}^*(\simplex_l)$).
	Then using \eqref{eqn:alpha_beta_gluing_explicit_form_on_simplex} in Condition \ref{assum:induction_hypothesis} (i.e. $\glue{k}_{\alpha\beta,i_0 \cdots i_l} = \exp([\sauto{k}_{\alpha\beta,i_0 \cdots i_l},\cdot]) \circ (\glue{k}_{\alpha\beta,i_0}|_{U_{i_0 \cdots i_l}})$), we obtain 
	\begin{multline*}
	\glue{k}_{\beta\alpha,i_0 \cdots i_l} \circ d \circ  \glue{k}_{\alpha\beta, i_0 \cdots i_l} = d - \Big[\glue{k}_{\beta \alpha,i_0} \Big( \mathtt{T}([\sauto{k}_{\alpha\beta,i_0\cdots i_l}, \cdot ]) (d \sauto{k}_{\alpha\beta,i_0\cdots i_l}) \Big), \cdot\Big].
	\end{multline*}
	
	Now we put
	$$
	\prescript{k}{}{\mathtt{w}}_{\alpha;\alpha\beta, i_0\cdots i_l} : = -\glue{k}_{\beta \alpha,i_0} \Big( \mathtt{T}([\sauto{k}_{\alpha\beta,i_0\cdots i_l}, \cdot ] ) (d \sauto{k}_{\alpha\beta,i_0\cdots i_l}) \Big).
	$$
	Then $\prescript{k+1}{}{\mathtt{w}}_{\alpha;\alpha\beta, i_0\cdots i_l} = \prescript{k}{}{\mathtt{w}}_{\alpha;\alpha\beta, i_0\cdots i_l} \ \text{(mod $\mathbf{m}^{k+1}$)}$. To check that it is well-defined as an element in $\twc{k}^{-1,1}_{\alpha;\alpha\beta}$, we need to show $\mathtt{d}_{r,l}^*(\prescript{k}{}{\mathtt{w}}_{\alpha;\alpha\beta, i_0\cdots i_l}) = \prescript{k}{}{\mathtt{w}}_{\alpha; \alpha\beta,i_0 \cdots \widehat{i_r} \cdots i_l}$, and by injectivity of $\bva{k}^{-1}_{\alpha} \hookrightarrow \text{Der}(\bva{k}^{0}_\alpha)$ it is enough to show 
	 $
	d + [\mathtt{d}_{r,l}^*(\prescript{k}{}{\mathtt{w}}_{\alpha;\alpha\beta, i_0\cdots i_l}), \cdot ] = d + [\prescript{k}{}{\mathtt{w}}_{\alpha; \alpha\beta,i_0 \cdots \widehat{i_r} \cdots i_l}, \cdot].
	$  We compute the only non-trivial term (i.e. $r=0$) as
	\begin{multline*}
	d + [\mathtt{d}_{0,l}^*(\prescript{k}{}{\mathtt{w}}_{\alpha;\alpha\beta, i_0\cdots i_l}), \cdot ]=\\
	  d  -[\glue{k}_{\beta \alpha,i_0} \big( \mathtt{T}([\sauto{k}_{\alpha\beta,\widehat{i_0}\cdots i_l} \bchprod \autoij{k}_{\alpha\beta, i_0 i_1}, \cdot ] ) d (\sauto{k}_{\alpha\beta,i_0\cdots i_l} \bchprod \autoij{k}_{\alpha\beta, i_0 i_1})  \big) ,\cdot ]\\
	\stackrel{\text{Lemma }\ref{lem:gauge_action_on_differential}}{=}   d - [\glue{k}_{\beta \alpha,i_0} \exp(-[\autoij{k}_{\alpha\beta, i_0 i_1}, \cdot ]) \big( \mathtt{T}([\sauto{k}_{\alpha\beta,\widehat{i_0}\cdots i_l}, \cdot ] ) (d \sauto{k}_{\alpha\beta,\widehat{i_0}\cdots i_l}) \big), \cdot]\\
\stackrel{\eqref{eqn:alpha_beta_gluing_relation_between_i_0_and_i_1}}{=}  d - [\glue{k}_{\beta \alpha,i_1}  \big( \mathtt{T}([\sauto{k}_{\alpha\beta,\widehat{i_0}\cdots i_l}, \cdot ] )(d \sauto{k}_{\alpha\beta,\widehat{i_0}\cdots i_l}) \big), \cdot]
	=  d - [\prescript{k}{}{\mathtt{w}}_{\alpha; \alpha\beta, \widehat{i_0} \cdots i_l}, \cdot],
	\end{multline*}
	where we use the first formula in Lemma \ref{lem:gauge_action_on_differential} and $d (\autoij{k}_{\alpha\beta, i_0 i_1}) =0$ to achieve the equality.

	To see that it is a \v{C}ech cocycle, we deduce from its definition that $[\glue{k}_{\alpha\beta} ( \prescript{k}{}{\mathtt{w}}_{\alpha;\alpha\beta}), \cdot] = \prescript{k}{}{\pdb}_{\beta} - \glue{k}_{\alpha\beta} \circ  \prescript{k}{}{\pdb}_{\alpha} \circ \glue{k}_{\beta \alpha}$, or equivalently 	$
	\glue{k}_{\beta \alpha} \circ \prescript{k}{}{\pdb}_{\beta} \circ \glue{k}_{\alpha\beta} - \prescript{k}{}{\pdb}_{\alpha} = [\prescript{k}{}{\mathtt{w}}_{\alpha;\alpha\beta}, \cdot]
	$. Thus, by a similar computation we have $\glue{k}_{\gamma \alpha} \circ \prescript{k}{}{\pdb}_{\gamma} \circ \glue{k}_{\alpha\gamma} - \prescript{k}{}{\pdb}_{\alpha} = [\prescript{k}{}{\mathtt{w}}_{\alpha;\alpha\gamma},\cdot]$ and $[\glue{k}_{\beta \alpha} ( \prescript{k}{}{\mathtt{w}}_{\beta;\beta\gamma}), \cdot] =\glue{k}_{\gamma \alpha} \circ  \prescript{k}{}{\pdb}_{\gamma} \circ \glue{k}_{\alpha\gamma} - \glue{k}_{\beta \alpha} \circ  \prescript{k}{}{\pdb}_{\beta} \circ \glue{k}_{\alpha\beta} $, we get
	$
	[\prescript{k}{}{\mathtt{w}}_{\alpha;\alpha\beta} - \prescript{k}{}{\mathtt{w}}_{\alpha;\alpha\gamma} + \glue{k}_{\beta \alpha}( \prescript{k}{}{\mathtt{w}}_{\beta;\beta\gamma}),\cdot]
	=  0
	$
	and can conclude that $\prescript{k}{}{\mathtt{w}}_{\alpha;\alpha\beta} - \prescript{k}{}{\mathtt{w}}_{\alpha;\alpha\gamma} + \glue{k}_{\beta \alpha}( \prescript{k}{}{\mathtt{w}}_{\beta;\beta\gamma}) = 0$. 
\end{proof}

We have a similar result concerning the difference between the BV operators $\bvd{k}_{\alpha}$ and $\bvd{k}_{\beta}$:

\begin{lemma}\label{lem:gluing_morphism_BV_operator_comparison}
	For each $k \in \inte_{\geq 0}$ and pair $V_\alpha, V_\beta \subset \mathcal{V}$, there exists $\prescript{k}{}{\mathtt{f}}_{\alpha;\alpha\beta} \in \twc{k}^{0,0}_{\alpha,\alpha\beta}$ such that $\prescript{k}{}{\mathtt{f}}_{\alpha;\alpha\beta} = 0 \ \text{(mod $\mathbf{m}$)}$, $\prescript{k+1}{}{\mathtt{f}}_{\alpha;\alpha\beta} = \prescript{k}{}{\mathtt{f}}_{\alpha;\alpha\beta} \ \text{(mod $\mathbf{m}^{k+1}$)}$, and
	$
	\glue{k}_{\beta \alpha} \circ \bvd{k}_{\beta} \circ \glue{k}_{\alpha\beta} - \bvd{k}_{\alpha} = [\prescript{k}{}{\mathtt{f}}_{\alpha;\alpha\beta}, \cdot].
	$
	Furthermore, if we let $\prescript{k}{}{\mathtt{f}}_{\alpha\beta}:=(\prescript{k}{}{\mathtt{f}}_{\alpha;\alpha\beta}, \glue{k}_{\alpha\beta}(\prescript{k}{}{\mathtt{f}}_{\alpha;\alpha\beta}))$, then $(\prescript{k}{}{\mathtt{f}}_{\alpha\beta})_{\alpha\beta}$ is a \v{C}ech $1$-cocycle in $\cech{k}^1(\twc{}^{0,0},\glue{})$.
\end{lemma}

\begin{proof}
	Similar to the previous proof, we use Lemma \ref{lem:gauge_action_on_differential} for the dgLa $(\bva{k}_{\beta}(U_{i_0\cdots i_l}), \bvd{k}^*_{\beta}[-1],[\cdot,\cdot])$ by taking $\xi = 0$ and get
	\begin{align*}
	& \exp(-[\sauto{k}_{\alpha\beta, i_0\cdots i_l},\cdot]) \circ \bvd{k}_{\beta} \circ \exp([\sauto{k}_{\alpha\beta, i_0\cdots i_l},\cdot])\\
	=& \bvd{k}_{\beta} - [(\mathtt{T}([\sauto{k}_{\alpha\beta, i_0\cdots i_l},\cdot])\circ \bvd{k}_{\beta})(\sauto{k}_{\alpha\beta,i_0\cdots i_l}),\cdot].
	\end{align*}
	 We take $\xi = -(\mathtt{T}([\sauto{k}_{\alpha\beta, i_0\cdots i_l},\cdot])\circ \bvd{k}_{\beta})(\sauto{k}_{\alpha\beta,i_0\cdots i_l})$ and $\vartheta = -\iauto{k}_{\alpha\beta, i_0 }$ in Lemma \ref{lem:gauge_action_on_differential} and get
	\begin{align*}
	& \exp(-[\iauto{k}_{\alpha\beta, i_0 },\cdot])\circ \exp(-[\sauto{k}_{\alpha\beta, i_0\cdots i_l},\cdot]) \circ\\
	& \bvd{k}_{\beta} \circ \exp([\sauto{k}_{\alpha\beta, i_0\cdots i_l},\cdot]) \circ \exp([\iauto{k}_{\alpha\beta, i_0 },\cdot])\\
	= &  \exp(-[\iauto{k}_{\alpha\beta, i_0 },\cdot])\circ \big(\bvd{k}_{\beta} - [(\mathtt{T}([\sauto{k}_{\alpha\beta, i_0\cdots i_l},\cdot])\circ \bvd{k}_{\beta})(\sauto{k}_{\alpha\beta,i_0\cdots i_l}),\cdot]\big)\\
	& \circ \exp([\iauto{k}_{\alpha\beta, i_0 },\cdot])\\
	= & \bvd{k}_{\beta} - [ ( \exp(-[\iauto{k}_{\alpha\beta,i_0},\cdot ]) \circ \mathtt{T}([\sauto{k}_{\alpha\beta,i_0\cdots i_l},\cdot]) \circ \bvd{k}_{\beta}) (\sauto{k}_{\alpha\beta,i_0\cdots i_l}), \cdot] \\
	& \qquad\qquad\qquad\qquad - [(\mathtt{T}([\iauto{k}_{\alpha\beta,i_0},\cdot ]) \circ \bvd{k}_\beta) (\iauto{k}_{\alpha\beta,i_0}),\cdot ].
	\end{align*}
	
	By \eqref{eqn:alpha_beta_gluing_explicit_form_at_point} in Condition \ref{assum:induction_hypothesis}, we can write $\glue{k}_{\alpha\beta, i} = \exp([ \iauto{k}_{\alpha\beta,i}, \cdot ]) \circ \patch{k}_{\alpha\beta,i}$, and from \eqref{eqn:higher_order_bvd_different} in Definition \ref{def:higher_order_patching}, we have $\patch{k}_{\beta\alpha,i} \circ \bvd{k}_{\beta} \circ \patch{k}_{\alpha\beta,i} - \bvd{k}_{\alpha} = [\bvdobs{k}_{\alpha\beta,i}, \cdot]$, so
	\begin{multline*}
	\glue{k}_{\beta\alpha,i_0 \cdots i_l} \circ \bvd{k}_{\beta} \circ  \glue{k}_{\alpha\beta, i_0 \cdots i_l} - \bvd{k}_{\alpha} = \\
	 - [ (\glue{k}_{\beta \alpha, i_0} \circ \mathtt{T}([\sauto{k}_{\alpha\beta,i_0\cdots i_l},\cdot ]) \circ \bvd{k}_{\beta}) (\sauto{k}_{\alpha\beta,i_0\cdots i_l}), \cdot] \\
	 - [(\patch{k}_{\beta \alpha ,i_0} \circ \mathtt{T}([\iauto{k}_{\alpha\beta,i_0},\cdot ]) \circ\bvd{k}_\beta) (\iauto{k}_{\alpha\beta,i_0}),\cdot] + [\bvdobs{k}_{\alpha\beta,i_0}, \cdot].
	\end{multline*}
	
	Now we put
	\begin{equation}\label{eqn:proof_in_lemma_gluing_morphism_BV_operator_comparison}
	\begin{split}
	\prescript{k}{}{\mathtt{f}}_{\alpha;\alpha\beta,i_0 \cdots i_l} := &
	- ( \glue{k}_{\beta \alpha, i_0} \circ \mathtt{T}([\sauto{k}_{\alpha\beta,i_0\cdots i_l},\cdot ]) \circ \bvd{k}_{\beta} ) (\sauto{k}_{\alpha\beta,i_0\cdots i_l})\\ 
	& \qquad - (\patch{k}_{\beta \alpha ,i_0} \circ \mathtt{T}([\iauto{k}_{\alpha\beta,i_0},\cdot ]) \circ \bvd{k}_\beta )(\iauto{k}_{\alpha\beta,i_0}) + \bvdobs{k}_{\alpha\beta,i_0}\\
	= & - ( \glue{k}_{\beta \alpha, i_0} \circ \mathtt{T}([\sauto{k}_{\alpha\beta,i_0\cdots i_l},\cdot ]) \circ \bvd{k}_{\beta} ) (\sauto{k}_{\alpha\beta,i_0\cdots i_l})\\
	&+ 1_{\simplex_l} \otimes \prescript{k}{}{\mathtt{f}}_{\alpha;\alpha\beta,i_0}|_{U_{i_0 \cdots i_l}},
	\end{split}
	\end{equation}
	where $1_{\simplex_l}$ denotes the constant function with value $1$ on $\simplex_l$.
	We need to check the following conditions for the elements $\prescript{k}{}{\mathtt{f}}_{\alpha;\alpha\beta,i_0 \cdots i_l}$'s:
	\begin{enumerate}
		\item $\prescript{k+1}{}{\mathtt{f}}_{\alpha;\alpha\beta,i_0\cdots i_l} = \prescript{k}{}{\mathtt{f}}_{\alpha;\alpha\beta,i_0\cdots i_l} \ \text{(mod $\mathbf{m}^{k+1}$)}$;
		
		\item $\prescript{k}{}{\mathtt{f}}_{\alpha;\alpha\beta} := (\prescript{k}{}{\mathtt{f}}_{\alpha;\alpha\beta,i_0 \cdots i_l})_{(i_0, \cdots, i_l) \in \mathcal{I}} \in \twc{k}^{0,0}_{\alpha;\alpha\beta}$ (see Definition \ref{def:thom_whitney_general});
		
		\item letting $\prescript{k}{}{\mathtt{f}}_{\alpha\beta} := (\prescript{k}{}{\mathtt{f}}_{\alpha; \alpha\beta} , \glue{k}_{\alpha\beta}(\prescript{k}{}{\mathtt{f}}_{\alpha; \alpha\beta}) )$, we have
		\begin{multline}\label{eqn:proof_of_gluing_morphism_BV_operator_comparison_alpha_beta_relation}
		\glue{k}_{\alpha\beta, i_0 \cdots i_l} (\prescript{k}{}{\mathtt{f}}_{\alpha;\alpha\beta,i_0 \cdots i_l }) = \\
		(\glue{k}_{\alpha\beta,i_0} \circ \mathtt{T}([\sauto{k}_{\beta\alpha,i_0 \cdots i_l},\cdot]) \circ \bvd{k}_{\alpha}) (\sauto{k}_{\beta\alpha,i_0 \cdots i_l}) + 1_{\simplex_l} \otimes \prescript{k}{}{\mathtt{f}}_{\beta;\alpha\beta,i_0} ,
		\end{multline}
		where $\prescript{k}{}{\mathtt{f}}_{\beta;\alpha\beta,i_0}  = (\patch{k}_{\alpha\beta,i_0} \circ \mathtt{T}([\iauto{k}_{\beta\alpha,i_0},\cdot]) \circ \bvd{k}_{\alpha}) (\iauto{k}_{\beta\alpha,i_0} ) - \bvdobs{k}_{\beta\alpha,i_0}$; and
		
		\item that $(\prescript{k}{}{\mathtt{f}}_{\alpha\beta})_{\alpha\beta}$ is a \v{C}ech $1$-cocycle in $\cech{k}^1(\twc{}^{0,0},\glue{})$.
	\end{enumerate}
	The properties (1)-(4) are proven by applying the comparison \eqref{eqn:higher_order_volume_form_different} of the volume forms in Definition \ref{def:higher_order_module_patching} (which can be regarded as a more refined piece of information than the comparison of BV operators in \eqref{eqn:higher_order_bvd_different}) and Lemma \ref{lem:BV_gauge_action_computation} in the {\em same} manner, together with some rather tedious (at least notationally) calculations. For simplicity, 
	we shall only present the proof of (1) here.
	
	To prove (1), first notice that the term $(\glue{k}_{\beta \alpha, i_0} \circ \mathtt{T}([\sauto{k}_{\alpha\beta,i_0\cdots i_l},\cdot ]) \circ \bvd{k}_{\beta}) (\sauto{k}_{\alpha\beta,i_0\cdots i_l})$ already satisfies the equality, so we only need to consider the case for $l=0$.
	In the rest of this proof, we shall work $\text{(mod $\mathbf{m}^{k+1}$)}$, meaning that all equalities hold $\text{(mod $\mathbf{m}^{k+1}$)}$.
	First of all, the equation $\rest{k+1,k}_{\beta} \circ \glue{k+1}_{\alpha\beta,i_0} = \glue{k}_{\alpha\beta, i_0} \circ \rest{k+1,k}_{\alpha} \ \text{(mod $\mathbf{m}^{k+1}$)}$ can be rewritten as
	\begin{multline*}
	\exp(-[\iauto{k}_{\alpha\beta,i_0},\cdot]) \circ \exp([\iauto{k+1}_{\alpha\beta,i_0} ,\cdot]) \circ \rest{k+1,k}_{\beta} =\\ \patch{k}_{\alpha\beta,i_0} \circ \rest{k+1,k}_{\alpha} \circ  \patch{k+1}_{\beta\alpha,i_0} = \exp([\resta{k+1,k}_{\beta\alpha,i_0} ,\cdot]) \circ \rest{k+1,k}_{\beta}
	\end{multline*}
	using \eqref{eqn:alpha_beta_gluing_explicit_form_at_point} and \eqref{eqn:bva_different_order_comparison}, so we have
	$$
	-\iauto{k+1}_{\alpha\beta,i_0} =  (-\resta{k+1,k}_{\beta\alpha,i_0}) \bchprod (-\iauto{k}_{\alpha\beta,i_0}) 
	$$
	by the injectivity of $\bva{k}^{-1}_\beta \hookrightarrow \text{Der}(\bva{k}^0_\beta)$.
	
	Applying Lemma \ref{lem:gauge_action_on_differential} to the dgLa $(\bva{k}^*_{ \beta}, [\cdot,\cdot], \bvd{k}_{\beta})$, we get\\ $(\exp(-\resta{k+1,k}_{\beta\alpha,i_0}) \star \exp(-\iauto{k}_{\alpha\beta,i_0} )) \star 0 = \exp(-\iauto{k+1}_{\alpha\beta,i_0}) \star 0 $, which can be expanded as
	\begin{align*}
	0 =& - (\exp(-[\resta{k+1,k}_{\beta\alpha,i_0},\cdot]) \circ \mathtt{T}([\iauto{k}_{\alpha\beta,i_0},\cdot]) \circ \bvd{k}_\beta) (\iauto{k}_{\alpha\beta,i_0} ) \\
	& \qquad - (\mathtt{T}([\resta{k+1,k}_{\beta\alpha,i_0},\cdot]) \circ \bvd{k}_{\beta})(\resta{k+1,k}_{\beta\alpha,i_0}) \\
	&+ (\mathtt{T}([\iauto{k+1}_{\alpha\beta,i_0},\cdot]) \circ \bvd{k}_{\beta}) (\iauto{k+1}_{\alpha\beta,i_0}).
	\end{align*}
	Applying $\patch{k+1}_{\beta\alpha,i_0}$ to both sides (note that $\gamma \in \bva{k}_{\beta}^*$, so when we write $\patch{k+1}_{\beta\alpha,i_0}(\gamma)$, we mean $\patch{k+1}_{\beta\alpha,i_0}(\tilde{\gamma})$ where $\tilde{\gamma} \in \bva{k+1}_{\beta}^*$ is an arbitrary lifting of $\gamma$; as we are working $\text{(mod $\mathbf{m}^{k+1}$)}$, $\patch{k+1}_{\beta\alpha,i_0}(\tilde{\gamma})$ is independent of the choice of the lifting), we obtain
	\begin{align*}
	0 = & - (\patch{k}_{\beta\alpha, i_0} \circ \mathtt{T}([\iauto{k}_{\alpha\beta,i_0},\cdot]) \circ \bvd{k}_\beta) (\iauto{k}_{\alpha\beta,i_0} ) \\
	&- (\patch{k+1}_{\beta\alpha,i_0} \circ \mathtt{T}([\resta{k+1,k}_{\beta\alpha,i_0},\cdot]) \circ \bvd{k}_{\beta}) (\resta{k+1,k}_{\beta\alpha,i_0}) \\
	& \qquad\qquad\qquad + (\patch{k+1}_{\beta\alpha,i_0} \circ \mathtt{T}([\iauto{k+1}_{\alpha\beta,i_0},\cdot]) \circ \bvd{k}_{\beta}) (\iauto{k+1}_{\alpha\beta,i_0}) \\
	= & \prescript{k}{}{\mathtt{f}}_{\alpha;\alpha\beta,i_0} - \bvdobs{k}_{\alpha\beta,i_0}  - \prescript{k+1}{}{\mathtt{f}}_{\alpha;\alpha\beta,i_0} + \bvdobs{k+1}_{\alpha\beta,i_0} \\
	& - 
	(\patch{k}_{\beta\alpha,i_0} \circ \mathtt{T}(-[\resta{k+1,k}_{\beta\alpha,i_0},\cdot]) \circ \bvd{k}_{\beta}) (\resta{k+1,k}_{\beta\alpha,i_0}).
	\end{align*}
	
	From \eqref{eqn:higher_order_volume_form_different}, we learn that $\patch{k}_{\alpha\beta,i_0} ( \bvdobs{k}_{\alpha\beta,i_0}) = - \bvdobs{k}_{\beta\alpha,i_0}$. Hence it remains to show that
	$$
	\bvdobs{k}_{\beta\alpha,i_0} + \patch{k}_{\alpha\beta,i_0}( \bvdobs{k+1}_{\alpha\beta,i_0} )  = (\mathtt{T}(-[\resta{k+1,k}_{\beta\alpha,i_0},\cdot]) \circ \bvd{k}_{\beta})(\resta{k+1,k}_{\beta\alpha,i_0}),
	$$
	which follows from the relation
	\begin{align*}
	&\exp( \bvdobs{k}_{\beta\alpha,i_0} + \patch{k}_{\alpha\beta,i_0}(\bvdobs{k+1}_{\alpha\beta,i_0})) \lrcorner \volf{k}_{\beta}\\
	 =& (\hpatch{k}_{\alpha\beta,i_0} \circ \rest{k+1,k}_{\alpha}\circ \hpatch{k+1}_{\beta\alpha,i_0}) (\volf{k+1}_\beta)  = \exp(\mathcal{L}_{\resta{k+1,k}_{\beta\alpha,i_0}}) (\volf{k}_{\beta})\\
	= &\exp([\dpartial{k}_{\beta},(-\resta{k+1,k}_{\beta\alpha,i_0}) \lrcorner ]) (\volf{k}_{\beta}),
	\end{align*}
	coming from Definition \ref{def:higher_order_module_patching} and using Lemma \ref{lem:BV_gauge_action_computation}.
\end{proof}

The same results hold with the same proofs for the homotopy \v{C}ech-Thom-Whitney complex with gluing morphisms
$
\gluehom{k}_{\alpha\beta} :  \twc{k}^{*,*}_{\alpha;\alpha\beta}(\simplex_1) \rightarrow   \twc{k}^{*,*}_{\beta;\alpha\beta}(\simplex_1),
$
where $ \twc{k}^{*,*}_{\alpha;\alpha\beta}(\simplex_1)$ and $\twc{k}^{*,*}_{\beta;\alpha\beta}(\simplex_1)$ are equipped with the differentials $\prescript{k}{}{\mathbf{D}}_{\alpha}:= d_{\simplex_1} \otimes 1 + 1 \otimes \prescript{k}{}{\pdb}_{\alpha}$ and $\prescript{k}{}{\mathbf{D}}_{\beta}:= d_{\simplex_1} \otimes 1 + 1 \otimes \prescript{k}{}{\pdb}_{\beta} $ and BV operators $\bvd{k}_{\alpha}$ and $\bvd{k}_\beta$ respectively.  Such results are summarized in the following Lemma.
\begin{lemma}\label{lem:homotopy_differentials_comparsion}
	There exist $\prescript{k}{}{\mathtt{W}}_{\alpha;\alpha\beta} \in \twc{k}^{-1,1}_{\alpha,\alpha\beta}(\simplex_1)$ and $\prescript{k}{}{\mathtt{F}}_{\alpha;\alpha\beta} \in  \twc{k}^{0,0}_{\alpha,\alpha\beta}(\simplex_1)$ such that $\prescript{k}{}{\mathtt{W}}_{\alpha;\alpha\beta} = 0 \ \text{(mod $\mathbf{m}$)}$ and $\prescript{k}{}{\mathtt{F}}_{\alpha;\alpha\beta} = 0 \ \text{(mod $\mathbf{m}$)}$, and
	$$
	\gluehom{k}_{\beta \alpha} \circ \prescript{k}{}{\mathbf{D}}_{\beta} \circ \gluehom{k}_{\alpha\beta} - \prescript{k}{}{\mathbf{D}}_{\alpha}  = [\prescript{k}{}{\mathtt{W}}_{\alpha;\alpha\beta}, \cdot], \quad
	\gluehom{k}_{\beta \alpha} \circ \bvd{k}_{\beta} \circ \gluehom{k}_{\alpha\beta} - \bvd{k}_{\alpha} = [\prescript{k}{}{\mathtt{F}}_{\alpha;\alpha\beta}, \cdot]. 
	$$
	Furthermore, if we let 
	$$
	\prescript{k}{}{\mathtt{W}}_{\alpha\beta}  :=(\prescript{k}{}{\mathtt{W}}_{\alpha;\alpha\beta}, \gluehom{k}_{\alpha\beta}(\prescript{k}{}{\mathtt{W}}_{\alpha;\alpha\beta})), \quad
	\prescript{k}{}{\mathtt{F}}_{\alpha\beta}  :=(\prescript{k}{}{\mathtt{F}}_{\alpha;\alpha\beta}, \gluehom{k}_{\alpha\beta}(\prescript{k}{}{\mathtt{F}}_{\alpha;\alpha\beta})), 
	$$
	then $(\prescript{k}{}{\mathtt{W}}_{\alpha\beta})_{\alpha\beta}$ and $(\prescript{k}{}{\mathtt{F}}_{\alpha\beta})_{\alpha\beta}$ are \v{C}ech $1$-cocycles in $ \cech{k}^*(\twc{},\gluehom{})$.
\end{lemma}

We conclude this subsection by the following theorem:
\begin{theorem} \label{thm:construction_of_differentials}
	There exist elements $\dbtwist{}=(\dbtwist{}_{\alpha})_{\alpha} = \varprojlim_{k} (\dbtwist{k}_{\alpha})_{\alpha} \in \cech{}^0(\twc{}^{-1,1},\glue{})$ and $\volftwist{}=(\volftwist{}_{\alpha})_{\alpha} = \varprojlim_{k} (\volftwist{k}_{\alpha})_{\alpha} \in \cech{}^0(\twc{}^{0,0},\glue{})$ such that 
	$$
	\glue{}_{\beta\alpha} \circ (\pdb_\beta + [\dbtwist{}_{\beta}, \cdot])  \circ \glue{}_{\alpha\beta}  = \pdb_{\alpha} + [\dbtwist{}_{\alpha},\cdot],\quad
	\glue{}_{\beta\alpha} \circ (\bvd{}_{\beta} + [\volftwist{}_{\beta}, \cdot]) \circ \glue{}_{\alpha\beta}  = \bvd{}_{\alpha} + [\volftwist{}_{\alpha},\cdot].
	$$
	Also, $(\pdb_{\alpha} + [\dbtwist{}_{\alpha},\cdot])_{\alpha}$ and $(\bvd{}_{\alpha} + [\volftwist{}_{\alpha},\cdot])_{\alpha}$ glue to give operators $\pdb$ and $\bvd{}$ on $\polyv{}^{*,*}(\glue{})$ such that
	\begin{enumerate}
		\item 
		$\pdb$ is a derivation of $[\cdot,\cdot]$ and $\wedge$ in the sense that 
		$$
		\pdb [u,v] = [\pdb u, v] + (-1)^{|u|+1}[u,\pdb v],\quad
		\pdb (u \wedge v) = (\pdb u) \wedge v +(-1)^{|u|} u \wedge (\pdb v),
		$$
		where $|u|$ and $|v|$ denote respectively the total degrees (i.e. $|u| = p+q$ if $u \in \polyv{}^{p,q}(\glue{})$) of the homogeneous elements $u,v \in \polyv{}^{*}(\glue{})$;
		
		\item
		the BV operator $\bvd{}$ satisfies the BV equation and is a derivation for the bracket $[\cdot,\cdot]$, i.e.
		\begin{align*}
		\bvd{}[u,v] &= [\bvd{}u, v] + (-1)^{|u|+1} [ u ,\bvd{}v],\\
		\bvd{}(u \wedge v) & = (\bvd{}u) \wedge v + (-1)^{|u|} u \wedge (\bvd{}v) + (-1)^{|u|} [u,v],
		\end{align*}
		for homogeneous elements $u,v \in \polyv{}^{*}(\glue{})$; and
		
		\item
		we have $\bvd{}^2 = 0$ and
		$$\pdb^2  = 0 = \pdb \bvd{} + \bvd{} \pdb \ \text{(mod $\mathbf{m}$)},$$
		so $(\polyv{}^{*,*} , \wedge, \pdb , \bvd{}) \  \text{(mod $\mathbf{m}$)}$ is an almost dgBV algebra.
	\end{enumerate}
	
	Moreover, if $(\dbtwist{}', \volftwist{}')$ is another pair of such elements defining operators $\pdb'$ and $\bvd{}'$, then we have
	$$
	\pdb' - \pdb = [\mathfrak{v}_1 , \cdot],\quad
	\bvd{}' - \bvd{} = [\mathfrak{v}_2,\cdot],
	$$
	for some $\mathfrak{v}_1 \in \polyv{}^{-1,1}(\glue{})$ and $\mathfrak{v}_2 \in \polyv{}^{0,0}(\glue{})$. 
\end{theorem}

\begin{proof}
	In view of Lemmas \ref{lem:gluing_morphism_partial_bar_comparison} and \ref{lem:gluing_morphism_BV_operator_comparison}, we have a \v{C}ech $1$-cocycle $\mathtt{w} = (\mathtt{w}_{\alpha\beta})_{\alpha\beta} = \varprojlim_{k} ( \prescript{k}{}{\mathtt{w}}_{\alpha\beta})_{\alpha\beta}$ and $\mathtt{f} = (\mathtt{f}_{\alpha\beta})_{\alpha\beta} = \varprojlim_{k} (\prescript{k}{}{\mathtt{f}}_{\alpha\beta})_{\alpha\beta}$. Using the exactness of the \v{C}ech-Thom-Whitney complex in Lemma \ref{lem:exactness_of_cech_thom_whitney_complex}, we obtain $\dbtwist{} \in \cech{}^0(\twc{}^{-1,1},\glue{})$ and $\volftwist{} \in \cech{}^0(\twc{}^{0,0},\glue{})$ such that the images of $-\dbtwist{}$ and $-\volftwist{}$ under the \v{C}ech differential $\cechd{}_0$ are $\mathtt{w}$ and $\mathtt{f}$ respectively, and also $\dbtwist{} = 0 \  \text{(mod $\mathbf{m}$)}$ and $\volftwist{} = 0 \  \text{(mod $\mathbf{m}$)}$. Therefore we obtain the identities 
	$$
	\glue{}_{\beta\alpha} \circ (\pdb_\beta + [\dbtwist{}_{\beta}, \cdot])  \circ \glue{}_{\alpha\beta}  = \pdb_{\alpha} + [\dbtwist{}_{\alpha},\cdot],\quad
	\glue{}_{\beta\alpha} \circ (\bvd{}_{\beta} + [\volftwist{}_{\beta}, \cdot]) \circ \glue{}_{\alpha\beta}  = \bvd{}_{\alpha} + [\volftwist{}_{\alpha},\cdot].
	$$
	Also notice that if we have another choice of $\dbtwist{}'$ and $\volftwist{}'$ such that the images of $-\dbtwist{}'$ and $-\volftwist{}'$ under the \v{C}ech differential $\cechd{}_0$ are $\mathtt{w}$ and $\mathtt{f}$ respectively, then we must have $\dbtwist{} ' - \dbtwist{} = \cechd{}_{-1}(\mathfrak{v}_1)$ and $\volftwist{}' - \volftwist{} = \cechd{}_{-1}(\mathfrak{v}_2)$ for some elements $\mathfrak{v}_1 \in \polyv{}^{-1,1}(\glue{})$ and $\mathfrak{v}_2 \in \polyv{}^{0,0}(\glue{})$. 
	
	It remains to show that the operators $\pdb$ and $\bvd{}$ defined by $\dbtwist{}$ and $\volftwist{}$ satisfying the desired properties. First note that we have an injection
	$
	\cechd{}_{-1}: \polyv{}^{p,q}(\glue{}) \hookrightarrow \cech{}^0(\twc{}^{p,q},\glue{}) = \prod_{\alpha} \twc{}^{p,q}_{\alpha},
	$
	where we write $\twc{}^{p,q}_\alpha := \varprojlim_k \twc{k}^{p,q}_\alpha$. Also the product $\wedge$ and the Lie bracket $[\cdot,\cdot]$ on $\polyv{}^{*,*}(\glue{})$ are induced by those on each $\twc{}^{*,*}_{\alpha}$. 
	Since $\pdb$ and $\bvd{}$ are defined by gluing the operators $(\prescript{}{}{\pdb}_\alpha + [\dbtwist{}_\alpha,\cdot])_\alpha$ and $(\bvd{}_\alpha + [\volftwist{}_\alpha,\cdot])_\alpha$, we only have to check the required identities on each $\twc{}^{*,*}_{\alpha}$, which hold because both $[\dbtwist{}_\alpha,\cdot]$'s and $[\volftwist{}_\alpha,\cdot]$ are derivations of degree $1$ and $\dbtwist{} = 0  = \volftwist{} \  \text{(mod $\mathbf{m}$)}$. Also, $(\bvd{}_\alpha + [\volftwist{}_\alpha,\cdot])^2 = \bvd{}_\alpha^2 + [\bvd{}_\alpha(\volftwist{}_\alpha),\cdot] = 0$ ($\bvd{}_\alpha(\volftwist{}_\alpha) = 0$ for degree reason), so we have $\bvd{}^2 = 0$.
\end{proof}

For the homotopy \v{C}ech-Thom-Whitney complex we have the following proposition which is parallel to Theorem \ref{thm:construction_of_differentials}:

\begin{prop}\label{prop:construction_of_homotopy_thom_whitney_differential}
	There exist elements $\mathfrak{D}=(\mathfrak{D}_{\alpha})_{\alpha} = \varprojlim_{k} (\prescript{k}{}{\mathfrak{D}}_{\alpha})_{\alpha} \in \cech{}^0(\twc{}^{-1,1},\gluehom{})$ and $\mathfrak{F}=(\mathfrak{F}_{\alpha})_{\alpha} = \varprojlim_{k} (\prescript{k}{}{\mathfrak{F}}_{\alpha})_{\alpha} \in \cech{}^0(\twc{}^{0,0},\gluehom{})$ such that 
	\begin{align*}
	\gluehom{}_{\beta\alpha} \circ (\mathbf{D}_\beta + [\mathfrak{D}_{\beta}, \cdot])  \circ \gluehom{}_{\alpha\beta}  &= \mathbf{D}_{\alpha} + [\mathfrak{D}_{\alpha},\cdot],\\
	\gluehom{}_{\beta\alpha} \circ (\bvd{}_{\beta} + [\mathfrak{F}_{\beta}, \cdot]) \circ \gluehom{}_{\alpha\beta}  &= \bvd{}_{\alpha} + [\mathfrak{F}_{\alpha},\cdot].
	\end{align*}
	Furthermore, $(\mathbf{D}_{\alpha} + [\mathfrak{D}_{\alpha},\cdot])_{\alpha}$ and $(\bvd{}_{\alpha} + [\mathfrak{F}_{\alpha},\cdot])_{\alpha}$ glue to give operators $\mathbf{D}$ and $\bvd{}$ on $\polyv{}^{*,*}(\gluehom{})$ so that $(\polyv{}^{*,*}(\gluehom{}), \wedge , \mathbf{D}, \bvd{})$ satisfies $(1)-(3)$ of Theorem \ref{thm:construction_of_differentials} (with $\mathbf{D}$ playing the role of $\pdb$).
\end{prop}
\section{Abstract construction of the de Rham differential complex}\label{sec:differentials}
	
\subsection{The de Rham complex}\label{sec:construction_of_de_rham_complex}

Given a set of compatible gluing morphisms $\glue{} = \{\glue{k}_{\alpha\beta}\}$, the goal of this subsection is to glue the local filtered de Rham modules $\tbva{k}{}_{\alpha}$ over $V_\alpha$ to form a global differential graded algebra over $X$. Similar to \S \ref{sec:local_thom_whitney_complex}, we consider a sheaf of filtered de Rham modules $(\tbva{}{}^*,\tbva{}{\bullet}^*, \wedge, \dpartial{})$ over a sheaf of BV algebras $(\bva{}^* , \wedge ,\bvd{})$ on $V$ and a countable acyclic cover $\mathcal{U} = \{U_i\}_{i \in \inte_+}$ of $V$ which satisfies the condition that $H^{>0}(U_{i_0\cdots i_l},\tbva{}{r}^j) = 0$ for all $j$, $r$ and and all finite intersections $U_{i_0 \cdots i_l}:= U_{i_0}\cap \cdots \cap U_{i_l}$. 
	
\begin{definition}\label{def:local_thom_whitney_de_rham_module}
	We let 
	\begin{multline*}
	\twc{}^{p,q}(\tbva{}{}):= \{ (\eta_{i_0 \cdots i_l})_{(i_0,\dots,  i_l) \in \mathcal{I}} \mid \\ \eta_{i_0 \cdots i_l} \in \mathcal{A}^q(\simplex_l) \otimes_\comp \tbva{}{}^p(U_{i_0\cdots i_l}), \
	\mathtt{d}_{j,l}^* (\eta_{i_0 \cdots i_l}) = \eta_{i_0 \cdots \hat{i}_j \cdots i_l} |_{U_{i_0 \cdots i_l}} \},
	\end{multline*}
	and $\twc{}^{*,*}(\tbva{}{}) := \bigoplus_{p,q} 	\twc{}^{p,q}(\tbva{}{})$. It is equipped with a natural filtration $\twc{}^{*,*}(\tbva{}{\bullet})$ inherited from $\tbva{}{\bullet}$ and the structures $(\wedge,\pdb, \dpartial{})$ defined componentwise by
	\begin{align*}
	(\alpha_{I} \otimes u_I) \wedge ( \beta_I \otimes w_I) &:= (-1)^{|u_I||\beta_I|} (\alpha_I \wedge \beta_I) \otimes (u_I \wedge w_I),\\
	\pdb (\alpha_I \otimes u_I) := (d\alpha_I) \otimes u_I ,&\quad
	\dpartial{}(\alpha_I \otimes u_I) := (-1)^{|\alpha_I|} \alpha_I \otimes (\dpartial{} u_I) ,
	\end{align*}
	for $\alpha_I, \beta_I \in \mathcal{A}^*(\simplex_l)$ and $u_I, w_I \in \tbva{}{}^*(U_I)$ where $l = |I| - 1$. Furthermore, there is an action $\iota_\varphi = \varphi \lrcorner : \twc{}^{*}(\tbva{}{\bullet}) \rightarrow \twc{}^{* + |\varphi|}(\tbva{}{\bullet})$ defined componentwise by
	$$
	(\alpha_I \otimes v_I) \lrcorner (\beta_I \otimes u_I) := (-1)^{|\beta_I| |v_I|} (\alpha_I \wedge \beta_I ) \otimes (v_I \lrcorner u_I),
	$$
	for $\alpha_I, \beta_I \in \mathcal{A}^*(\simplex_l)$, $v_I \in \bva{}^*(U_I)$ and $u_I \in \tbva{}{}^*(U_I)$, where $|\varphi| = p+q $ for $\varphi \in \twc{}^{p,q}(\tbva{}{\bullet})$ and $l = |I| - 1$. 
\end{definition}

Direct computation shows that $(\twc{}^{*}(\tbva{}{\bullet}), \wedge , \dpartial{})$ is a filtered de Rham module over the BV algebra $( \twc{}^{*}(\bva{}), \wedge ,\bvd{})$ with the identity $\mathcal{L}_{\alpha_I \otimes v_I} ( \beta_I \otimes u_I) = (-1)^{(|v_I|+1) |\beta_I|} (\alpha_I \wedge\beta_I) \otimes (\mathcal{L}_{v_I} u_I)$ for $\alpha_I, \beta_I \in \mathcal{A}^*(\simplex_l)$, $v_I \in \bva{}^*(U_I)$ and $u_I \in \tbva{}{}^*(U_I)$ where $l = |I| - 1$. Also, $(\twc{}^{*}(\tbva{}{\bullet}), \wedge , \pdb{})$ is a dga with the relation $\pdb{}(\varphi \lrcorner \eta) = \pdb{}(\varphi) \lrcorner \eta + (-1)^{|\varphi|} \varphi \lrcorner (\pdb{} \eta)$ for $\varphi \in \twc{}^{*}(\bva{})$ and $\eta \in \twc{}^{*}(\tbva{}{\bullet}))$.

\begin{prop}\label{prop:thom_whitney_de_Rham_module_action_properties}
	There is an exact sequence
	$$
	0 \rightarrow \twc{}^{p,q}(\tbva{}{r+1}) \rightarrow \twc{}^{p,q} (\tbva{}{r})  \rightarrow  \twc{}^{p,q} ( \tbva{}{r} / \tbva{}{r +1 }) \rightarrow 0
	$$
	induced naturally by the exact sequence $ 0 \rightarrow \tbva{}{r+1} \rightarrow \tbva{}{r} \rightarrow \tbva{}{r}/ \tbva{}{r+1} \rightarrow 0$. 
\end{prop}
	
\begin{proof}
	The only nontrivial part is the surjectivity of the map $\mathtt{p} : \twc{}^{p,q} (\tbva{}{r})  \rightarrow  \twc{}^{p,q} ( \tbva{}{r} / \tbva{}{r +1 })$, which is induced from surjective maps $\mathtt{p} : \tbva{}{r}^p(U_{i_0\cdots i_l}) \rightarrow (\tbva{}{r}^p/\tbva{}{r+1}^p)(U_{i_0\cdots i_l})$. We fix $\eta = (\eta_{i_0\cdots i_l})_{(i_0, \ldots, i_l)} \in \twc{}^{p,q} ( \tbva{}{r} / \tbva{}{r +1})$ with $\eta_{i_0 \cdots i_l } \in \mathcal{A}^q(\simplex_l) \otimes (\tbva{}{r}^p/\tbva{}{r+1}^p)(U_{i_0 \cdots i_l})$, and show by induction on $l$ that there exists a lifting $ \reallywidetilde{\eta_{i_0 \cdots i_{l'}}} \in \mathcal{A}^q(\simplex_{l'}) \otimes \tbva{}{}^p(U_{i_0 \cdots i_{l'}})$ for any $ l'\leq l$ satisfying $\mathtt{p}(\reallywidetilde{\eta_{i_0 \cdots i_{l'}}}) = \eta_{i_0 \cdots i_{l'}}$ and $\mathtt{d}_{j,l}^* (\reallywidetilde{\eta_{i_0 \cdots i_{l'}}}) = \reallywidetilde{\eta_{i_0 \cdots \hat{i}_j \cdots i_{l'}}} |_{U_{i_0 \cdots i_{l'}}} $ for all $0 \leq j \leq l'$. 
	The initial case $l = q$ follows from the surjectivity of the map $ \mathtt{p} : \tbva{}{r}^p(U_{i_0\cdots i_q}) \rightarrow (\tbva{}{r}^p/\tbva{}{r+1}^p)(U_{i_0 \cdots i_q})$ over the Stein open subset $U_{i_0 \cdots i_q}$. For the induction step, we set
	$
	\partial(\reallywidehat{\eta_{i_0 \cdots i_l}}) := (\reallywidetilde{\eta_{\widehat{i_0} \cdots i_l}},\dots , \reallywidetilde{\eta_{i_0 \cdots \widehat{i_l}}}) \in \mathcal{A}^q(\simplexbdy_l) \otimes \tbva{}{r}^p (U_{i_0 \cdots i_l})
	$.
	Then the Lifting Lemma \ref{lem:simplex_lifting_lemma} gives an extension $\reallywidetilde{\eta_{i_0 \cdots i_l}} \in \mathcal{A}^q(\simplex_l) \otimes \tbva{}{r}^p (U_{i_0 \cdots i_l})$ such that $(\reallywidetilde{\eta_{i_0 \cdots i_l}})|_{\simplexbdy_l} = \partial(\reallywidehat{\eta_{i_0 \cdots i_l}})$ and $\mathtt{p}_{i_0 \cdots i_l}(\reallywidetilde{\eta_{i_0 \cdots i_l}}) = \eta_{i_0 \cdots i_l}$, as desired. 
\end{proof}
	
\begin{notation}\label{not:derham_complex_bullet_notation}
	We will write $\tbva{k}{r_1:r_2}^*_{\alpha}:= \tbva{k}{r_1}^*_\alpha / \tbva{k}{r_2}^*_\alpha$ for any $r_1 \leq r_2$, and extend the notation $\tbva{k}{\bullet}^*_{\alpha}$ by allowing $\bullet = r$ or $r_1:r_2$. We also write $\tbva{k}{}^*_\alpha = \tbva{k}{0}^*_\alpha$ as before.  
\end{notation}
	
Given a set of compatible gluing morphisms $\glue{} = \{\glue{k}_{\alpha\beta}\}$ as in \S \ref{sec:the_Cech_thom_whitney_complex}, we can extend them to gluing morphisms acting on $\ktwc{k}{\bullet}^{*,*}_{\alpha_j;\alpha_0 \cdots \alpha_\ell} := \twc{}^{*,*}(\tbva{k}{\bullet}_{\alpha_j}|_{V_{\alpha_0 \cdots \alpha_\ell}})$. 
	
\begin{definition}\label{def:de_rham_module_induced_gluing_morphism}
	For each pair $V_\alpha, V_\beta \subset \mathcal{V}$, the morphism $\drglue{k}_{\alpha\beta} = (\drglue{k}_{\alpha\beta, i_0 \cdots i_l})_{(i_0, \ldots, i_l)\in \mathcal{I}} : \ktwc{k}{\bullet}^{*,*}_{\alpha;\alpha\beta} \rightarrow \ktwc{k}{\bullet}^{*,*}_{\beta;\alpha\beta}$ is defined componentwise by
	$$
	\drglue{k}_{\alpha\beta,i_0 \cdots i_l}(\eta_{i_0 \cdots i_l}) := (\exp(\mathcal{L}_{\sauto{k}_{\alpha\beta,i_0 \cdots i_l}}) \circ \exp(\mathcal{L}_{\iauto{k}_{\alpha\beta,i_0}})  \circ \hpatch{k}_{\alpha\beta,i_0})(\eta_{i_0 \cdots i_l})
	$$
	for any multi-index $(i_0, \ldots, i_l) \in \mathcal{I}$ such that $U_{i_j} \subset V_{\alpha\beta}$ for every $ 0 \leq j \leq l$.
\end{definition}

From Definition \ref{def:higher_order_module_patching}, we see that the differences between the morphisms $\hpatch{k}_{\alpha\beta,i}$'s are captured by taking Lie derivatives of the same elements $\resta{k,l}_{\alpha\beta,i}$'s, $\patchij{k}_{\alpha\beta,i j}$'s and $\cocyobs{k}_{\alpha\beta\gamma,i}$'s as for the morphisms $\patch{k}_{\alpha\beta,i}$'s. So the morphisms $\drglue{k}_{\alpha\beta}:\ktwc{k}{\bullet}^{*,*}_{\alpha;\alpha\beta} \to \ktwc{k}{\bullet}^{*,*}_{\beta;\alpha\beta}$'s are well-defined and satisfying $\drglue{k+1}_{\alpha\beta} = \drglue{k}_{\alpha\beta} \ (\text{mod \ $\mathbf{m}^{k+1}$})$, $\drglue{k}_{\gamma\alpha} \circ \drglue{k}_{\beta \gamma}  \circ \drglue{k}_{\alpha\beta} = \text{id}$. As a result, we can define the \v{C}ech-Thom-Whitney complex $\cech{}^*(\ktwc{}{\bullet}, \drglue{})$ as in \S \ref{sec:the_Cech_thom_whitney_complex}.

\begin{definition}\label{def:construction_total_de_rham}
	For $\ell \geq 0$, we let $\ktwc{k}{\bullet}^{*,*}_{\alpha_0 \cdots \alpha_{\ell}}(\drglue{}) \subset \bigoplus_{i=0}^\ell \ktwc{k}{\bullet}^{*,*}_{\alpha_i;\alpha_0 \cdots \alpha_{\ell}}$ be the set of elements $(\eta_0,\cdots, \eta_{\ell})$ such that $ \eta_{j} = \drglue{k}_{\alpha_i \alpha_j} (\eta_i)$.Then we set $\cech{k}^\ell(\ktwc{}{\bullet}^{p,q},\drglue{}) := \prod_{\alpha_0 \cdots \alpha_\ell} \ktwc{k}{\bullet}^{p,q}_{\alpha_0 \cdots \alpha_{\ell}}(\drglue{})$ for each $k \in \inte_{\geq 0}$ and $\cech{k}^\ell(\ktwc{}{\bullet},\drglue{}) := \bigoplus_{p,q}\cech{k}^\ell(\ktwc{}{\bullet}^{p,q},\drglue{})$, which is equipped with the natural restriction maps $ \restmap_{j,\ell} : \cech{k}^{\ell-1}(\ktwc{}{\bullet},\drglue{}) \rightarrow \cech{k}^{\ell}(\ktwc{}{\bullet},\drglue{})$.
		
	We let $\cechd{k}_\ell:= \sum_{j =0}^{\ell+1} (-1)^{j}\restmap_{j,\ell+1} : \cech{k}^{\ell} (\ktwc{}{\bullet},\drglue{}) \rightarrow \cech{k}^{\ell+1}(\ktwc{}{\bullet},\drglue{})$ be the \v{C}ech differential. The {\em $k^{\text{th}}$ order total de Rham complex over $X$} is then defined to be $\totaldr{k}{\bullet}^{*,*}(\drglue{}):= \text{Ker}(\cechd{k}_{0})$. Denoting the natural inclusion $\totaldr{k}{\bullet}^{*,*}(\drglue{}) \rightarrow \cech{k}^0(\ktwc{}{\bullet},\drglue{})$ by $\cechd{k}_{-1}$, we obtain the following sequence of maps
	\begin{align}
	\totaldr{k}{\bullet}^{p,q}(\drglue{}) \hookrightarrow \cech{k}^0(\ktwc{}{\bullet}^{p,q},\drglue{}) \rightarrow \cech{k}^1(\ktwc{}{\bullet}^{p,q},\drglue{}) \rightarrow \cdots \rightarrow \cech{k}^{\ell}(\ktwc{}{\bullet}^{p,q},\drglue{}) \rightarrow \cdots, \label{eqn:de_rham_cech_thom_whitney_complex}\\
	 \totaldr{}{\bullet}^{p,q}(\drglue{}) \hookrightarrow \cech{}^0(\ktwc{}{\bullet}^{p,q},\drglue{}) \rightarrow \cech{}^1(\ktwc{}{\bullet}^{p,q},\drglue{}) \rightarrow \cdots \rightarrow \cech{}^{\ell}(\ktwc{}{\bullet}^{p,q},\drglue{}) \rightarrow \cdots,
	\end{align}
	where the second sequence is obtained by taking inverse limits $\varprojlim_k$ of the first sequence.
		
	Furthermore, we let $\prescript{k}{}{\pdb}$ and $\dpartial{k}$ be the operators on $\totaldr{k}{\bullet}^{*,*}(\drglue{})$ obtained by gluing of the operators $(\prescript{k}{}{\pdb}_{\alpha} + \mathcal{L}_{\dbtwist{k}_\alpha})_\alpha$'s (where $(\dbtwist{k}_{\alpha})_\alpha \in \cech{k}^1(\twc{}^{-1,1},\glue{})$ is the element obtained from Theorem \ref{thm:construction_of_differentials}) and $\dpartial{k}_{\alpha}$'s on $\ktwc{k}{\bullet}_{\alpha;\alpha}$, and $\pdb:= \varprojlim_k \prescript{k}{}{\pdb} $ and $\dpartial{}:= \varprojlim_k \dpartial{k}$ be the corresponding inverse limits.
\end{definition}

\begin{prop}\label{prop:de_rham_module_algebraic_structures_summary}
	Let $\totaldr{}{\bullet}^{*,*}(\drglue{})$ be the filtration inherited from that of $\tbva{}{\bullet}$ for $\bullet =0,\dots,s$. Then $(\totaldr{}{\bullet}^{*,*}(\drglue{}), \wedge, \dpartial{})$ is a filtered de Rham module over the BV algebra $(\polyv{}^{*,*}(\glue{}), \wedge, \bvd{})$, and we have $\pdb^2 = 0= \pdb\dpartial{} + \dpartial{} \pdb \ (\text{mod \ $\mathbf{m}$})$ as well as the relations
	$$
	\pdb(\eta \wedge \mu) = \pdb(\eta) \wedge \mu + (-1)^{|\eta|} \eta \wedge(\pdb \mu),\quad
	\pdb(\varphi \lrcorner \eta) = (\pdb \varphi) \lrcorner \eta + (-1)^{|\varphi|} \varphi \lrcorner (\pdb \eta),
	$$
	for $\varphi \in \polyv{}^{*,*}(\glue{})$ and $\eta, \mu \in \totaldr{}{}^{*,*}(\drglue{})$. Furthermore, the filtration $\totaldr{}{\bullet}^{*,*}(\drglue{})$ satisfies the relation
	$$
	\totaldr{}{r}^{*,*}(\drglue{}) / \totaldr{}{r+1}^{*,*}(\drglue{})  = \totaldr{}{r:r+1}^{*,*}(\drglue{}).
	$$
\end{prop}
	
\begin{proof}
	Since $\pdb$ and $\dpartial{}$ are constructed from the operators  $(\prescript{k}{}{\pdb}_{\alpha} + \mathcal{L}_{\dbtwist{k}_\alpha})_\alpha$'s and $\dpartial{k}_{\alpha}$'s on $\ktwc{k}{\bullet}_{\alpha;\alpha}$, we only have to check the relations for each $(\ktwc{k}{\bullet}_{\alpha;\alpha},\wedge,\dpartial{k})$, which is a filtered de Rham module over the BV algebra $(\twc{k}_{\alpha;\alpha}, \wedge, \bvd{k})$. 
	To see the last relation, note that there is an exact sequence of \v{C}ech-Thom-Whitney complexes
	$
	0 \rightarrow \cech{k}^*(\ktwc{}{r+1}^{p,q},\drglue{}) \rightarrow \cech{k}^*(\ktwc{}{r}^{p,q},\drglue{}) \rightarrow \cech{k}^*(\ktwc{}{r:r+1}^{p,q},\drglue{}) \rightarrow 0
	$
	using Proposition \ref{prop:thom_whitney_de_Rham_module_action_properties}. Taking the kernel and inverse limits then gives the exact sequence
	$$
	0\rightarrow \totaldr{}{r+1}^{p,q}(\drglue{}) \rightarrow \totaldr{}{r}^{p,q}(\drglue{}) \rightarrow \totaldr{}{r:r+1}^{p,q}(\drglue{}) \rightarrow 0. 
	$$
	The result follows.
\end{proof}
	
\begin{notation}\label{not:total_de_rham_and_relative_de_rham_simplify_notation}
	We will simplify notations by writing $\totaldr{}{\bullet}^{*,*} = \totaldr{}{\bullet}(\drglue{})$ and $\polyv{}^{*,*} = \polyv{}^{*,*}(\glue{})$ if the gluing morphisms $\drglue{} = \{\drglue{k}_{\alpha\beta}\}$ and $\glue{} = \{\glue{k}_{\alpha\beta}\}$ are clear from the context. We will further denote the relative de Rham complex (over $\text{spf}(\hat{\cfr})$) as $\totaldr{}{\parallel}^{*,*}:= \totaldr{}{0}^{*,*}/ \totaldr{}{1}^{*,*} = \totaldr{}{0:1}^{*,*}$. 
\end{notation}
	
\begin{prop}\label{prop:gluing_volume_form}
	Using the element $(\volftwist{k}_{\alpha})_\alpha \in \cech{k}^0(\twc{}^{0,0},\glue{})$ obtained from Theorem \ref{thm:construction_of_differentials}, we obtain the element $(\exp(\volftwist{k}_{\alpha} \lrcorner) \volf{k}_{\alpha})_{\alpha} \in \cech{k}^{0}(\ktwc{}{0:1}^{d,0},\drglue{})$ whose components glue to form a global element $\volf{k} \in \totaldr{k}{\parallel}^{d,0} $, i.e. we have $(\drglue{k}_{\alpha\beta} \circ \exp(\volftwist{k}_{\alpha} \lrcorner)) (\volf{k}_{\alpha}) = \exp(\volftwist{k}_{\beta} \lrcorner) (\volf{k}_{\beta})$. Furthermore, we have $\volf{k+1} = \volf{k} \ (\text{mod \ $\mathbf{m}^{k+1}$})$. In view of this, we define the {\em relative volume form} to be $\volf{} := \varprojlim_k \volf{k}$.  
\end{prop}
	
\begin{proof}
	Similar to the proof of Lemma \ref{lem:gluing_morphism_BV_operator_comparison}, we use the power series $\mathtt{T}(x)$ in \eqref{eqn:power_series_in_proof_gluing_morphism_BV_operator_comparison} to simplify notations. We fix $V_{\alpha\beta}$ and $U_{i_0 \cdots i_l}$ such that $U_{i_j} \subset V_{\alpha \beta}$ for all $0 \leq j \leq l$. We need to show that 
	$$
	(\exp(\mathcal{L}_{\sauto{k}_{\alpha\beta, i_0 \cdots i_l}}) \circ \glue{k}_{\alpha\beta, i_0 } \circ \exp( \volftwist{k}_{\alpha, i_0 \cdots i_l} \lrcorner)) (\volf{k}_{\alpha}) = \exp( \volftwist{k}_{\beta, i_0 \cdots i_l} \lrcorner) (\volf{k}_{\beta}).
	$$
	We begin with the case $l=0$. Making use of the identities
	\begin{align*}
	\exp([\sauto{}, \cdot ]) (u \wedge v) & = (\exp([\sauto{}, \cdot ])(u)) \wedge (\exp([\sauto{}, \cdot ])(v)),\\
	\exp(\mathcal{L}_{\sauto{}}) (v \lrcorner w) & = (\exp([\sauto{}, \cdot ])(v)) \lrcorner (\exp(\mathcal{L}_{\sauto{}}) (w))
	\end{align*}
	for $\sauto{} \in \bva{k}^{-1}_\beta(U_{i_0 \cdots i_l})$, $u, v \in \bva{k}^{0}_\beta(U_{i_0 \cdots i_l})$ and $w \in \tbva{k}{0:1}^d_{\beta}(U_{i_0 \cdots i_l})$, we have 
	\begin{align*}
	 	& (\drglue{k}_{\alpha\beta,i_0} \circ \exp( \volftwist{k}_{\alpha, i_0} \lrcorner)) (\volf{k}_{\alpha})\\
	 	= & (\exp(\mathcal{L}_{\iauto{k}_{\alpha\beta, i_0}}) \circ \hpatch{k}_{\alpha\beta, i_0 } \circ \exp( \volftwist{k}_{\alpha, i_0} \lrcorner)) (\volf{k}_{\alpha})\\
		= & \exp(\glue{k}_{\alpha\beta,i_0}(\volftwist{k}_{\alpha,i_0})+\exp([\iauto{k}_{\alpha\beta, i_0 }, \cdot])(\bvdobs{k}_{\beta\alpha,i_0}) ) \lrcorner \exp(\mathcal{L}_{\iauto{k}_{\alpha\beta, i_0 }}) (\volf{k}_{\beta})\\
		= &  \exp\Big(\volftwist{k}_{\beta, i_0} + \prescript{k}{}{\mathtt{f}}_{\beta;\alpha\beta,i_0} +\\
		&\glue{k}_{\alpha\beta,i_0} ( - \bvdobs{k}_{\alpha\beta,i_0}  + (\patch{k}_{\beta\alpha,i_0} \circ \mathtt{T}([\iauto{k}_{\alpha\beta,i_0},\cdot]) \circ \bvd{k}_\beta) (\iauto{k}_{\alpha\beta,i_0}) ) \Big) \lrcorner \volf{k}_\beta \\ 
		= & \exp(\volftwist{k}_{\beta, i_0}) \lrcorner \volf{k}_\beta,
	\end{align*}
	where $\prescript{k}{}{\mathtt{f}}_{\beta;\alpha\beta,i_0} = \glue{k}_{\alpha\beta,i_0} (\prescript{k}{}{\mathtt{f}}_{\alpha;\alpha\beta,i_0}) = \glue{k}_{\alpha\beta,i_0}(- (\patch{k}_{\beta\alpha,i_0} \circ \mathtt{T}([\iauto{k}_{\alpha\beta,i_0},\cdot]) \circ \bvd{k}_\beta) (\iauto{k}_{\alpha\beta,i_0}) +  \bvdobs{k}_{\alpha\beta,i_0})$ is the component of the term $\prescript{k}{}{\mathtt{f}}_{\beta;\alpha\beta}$ obtained in Lemma \ref{lem:gluing_morphism_BV_operator_comparison}.
		
	The general case $l \geq 0$ is similar, as we have
	\begin{align*}
		& (\drglue{k}_{\alpha\beta,i_0 \cdots i_l} \circ \exp( \volftwist{k}_{\alpha, i_0 \cdots i_l} \lrcorner)) (\volf{k}_{\alpha})\\
		= & (\exp(\mathcal{L}_{\sauto{k}_{\alpha\beta, i_0 \cdots i_l}}) \circ \glue{k}_{\alpha\beta, i_0 } \circ \exp( \volftwist{k}_{\alpha, i_0 \cdots i_l} \lrcorner)) (\volf{k}_{\alpha})\\
		= & \exp(\volftwist{k}_{\beta,i_0 \cdots i_l} + \prescript{k}{}{\mathtt{f}}_{\beta;\alpha\beta,i_0\cdots i_l} -  \exp([\sauto{k}_{\alpha\beta, i_0\cdots i_l }, \cdot])(\prescript{k}{}{\mathtt{f}}_{\beta;\alpha\beta,i_0})\\ &+(\mathtt{T}(-[\sauto{k}_{\alpha\beta,i_0 \cdots i_l},\cdot]) \circ \bvd{k}_\beta) (\sauto{k}_{\alpha\beta,i_0 \cdots i_l}) )\lrcorner \volf{k}_\beta \\
		= & \exp\Big(\volftwist{k}_{\beta,i_0 \cdots i_l} + \prescript{k}{}{\mathtt{f}}_{\beta;\alpha\beta,i_0\cdots i_l} +\\
		& \glue{k}_{\alpha\beta,i_0 \cdots i_l} ( (\glue{k}_{\beta\alpha,i_0} \circ \mathtt{T}([\sauto{k}_{\alpha\beta,i_0 \cdots i_l},\cdot]) \circ \bvd{k}_\beta) (\sauto{k}_{\alpha\beta,i_0 \cdots i_l}) -\prescript{k}{}{\mathtt{f}}_{\alpha;\alpha\beta,i_0} ) \Big) \lrcorner \volf{k}_\beta \\
		= & \exp(\volftwist{k}_{\beta,i_0 \cdots i_l}) \lrcorner \volf{k}_\beta.
	\end{align*}
\end{proof}
	
\begin{lemma}\label{lem:proving_globalness_for_construction_of_total_de_rham_differential}
	The elements $\dbobs{k}_{\alpha}:= \prescript{k}{}{\pdb}_{\alpha} (\dbtwist{k}_{\alpha}) + \half [\dbtwist{k}_{\alpha}, \dbtwist{k}_{\alpha}]$ and $\dvolfobs{k}_{\alpha} := \bvd{k}_{\alpha}( \dbtwist{k}_{\alpha}) + \prescript{k}{}{\pdb}_{\alpha}(\volftwist{k}_{\alpha}) + [\dbtwist{k}_{\alpha}, \volftwist{k}_{\alpha}]$ glue to give global elements $\dbobs{k} = (\dbobs{k}_{\alpha})_{\alpha} \in \polyv{k}^{-1,2}$ and $\dvolfobs{k}= (\dvolfobs{k}_{\alpha})_{\alpha} \in \polyv{k}^{0,1}$ respectively.
\end{lemma}
	
\begin{proof}
	For the element $\dbobs{k}_{\alpha}$, we have
	$$
	(\prescript{k}{}{\pdb}_{\alpha} + [\dbtwist{k}_{\alpha}, \cdot ])^2 = [ \dbobs{k}_{\alpha}, \cdot]
	$$
	on $\twc{k}^{*,*}_{\alpha}$ and since $\glue{k}_{\beta\alpha} \circ (\prescript{k}{}{\pdb}_{\beta} + [\dbtwist{k}_{\beta}, \cdot ])  \circ \glue{k}_{\alpha\beta} = \prescript{k}{}{\pdb}_{\alpha} + [\dbtwist{k}_{\alpha}, \cdot ]$, we deduce that $[\glue{k}_{\alpha\beta} (\dbobs{k}_{\alpha}), \cdot ] = [\dbobs{k}_{\beta}, \cdot]$ and hence $\glue{k}_{\alpha\beta} (\dbobs{k}_{\alpha}) = \dbobs{k}_{\beta}$ by the injectivity of $\bva{k}^{-1} \hookrightarrow \text{Der}(\bva{k}^{0})$. 
		
	For the element $\dvolfobs{k}_{\alpha}$, we have $\glue{k}_{\alpha\beta} (\dbtwist{k}_{\alpha}) = \dbtwist{k}_{\beta} + \prescript{k}{}{\mathtt{w}}_{\beta;\alpha\beta}$ and $\glue{k}_{\alpha\beta} \volftwist{k}_{\alpha} = \volftwist{k}_{\beta} + \prescript{k}{}{\mathtt{f}}_{\beta;\alpha\beta}$ from the constructions in Theorem \ref{thm:construction_of_differentials}. Making use of the relations $\prescript{k}{}{\pdb}_\beta \circ \glue{k}_{\alpha\beta} = \glue{k}_{\alpha\beta} \circ \prescript{k}{}{\pdb}_{\alpha} + [\prescript{k}{}{\mathtt{w}}_{\beta;\alpha\beta}, \cdot ] \circ \glue{k}_{\alpha\beta}$ and $\bvd{k}_\beta \circ \glue{k}_{\alpha\beta} = \glue{k}_{\alpha\beta} \circ \bvd{k}_{\alpha} + [\prescript{k}{}{\mathtt{f}}_{\beta;\alpha\beta}, \cdot ] \circ \glue{k}_{\alpha\beta}$ from Lemmas \ref{lem:gluing_morphism_partial_bar_comparison} and \ref{lem:gluing_morphism_BV_operator_comparison}, we have
	\begin{align*}
		& \glue{k}_{\alpha\beta} ( \bvd{k}_{\alpha}(\dbtwist{k}_{\alpha}) + \prescript{k}{}{\pdb}_{\alpha} (\volftwist{k}_{\alpha}) + [\dbtwist{k}_{\alpha}, \volftwist{k}_{\alpha}]) \\
		= & \glue{k}_{\alpha\beta}(\bvd{k}_{\alpha} (\dbtwist{k}_{\alpha})) + [\prescript{k}{}{\mathtt{f}}_{\beta;\alpha\beta}, \dbtwist{k}_{\beta}] + \glue{k}_{\alpha\beta}(\prescript{k}{}{\pdb}_{\alpha} (\volftwist{k}_{\alpha}) ) \\
		&+ [\prescript{k}{}{\mathtt{w}}_{\beta;\alpha\beta} , \volftwist{k}_{\beta}] + [\prescript{k}{}{\mathtt{w}}_{\beta;\alpha\beta},\prescript{k}{}{\mathtt{f}}_{\beta;\alpha\beta} ] + [\dbtwist{k}_{\beta}, \volftwist{k}_{\beta}]\\
		= & ( \bvd{k}_{\beta} (\dbtwist{k}_{\beta}) + \prescript{k}{}{\pdb}_{\beta} (\volftwist{k}_{\beta}) + [\dbtwist{k}_{\beta}, \volftwist{k}_{\beta}]) + \bvd{k}_{\beta}(\prescript{k}{}{\mathtt{w}}_{\beta;\alpha\beta}) \\
		&+ \prescript{k}{}{\pdb}_{\beta}( \prescript{k}{}{\mathtt{f}}_{\beta;\alpha\beta}) - [\prescript{k}{}{\mathtt{w}}_{\beta;\alpha\beta},\prescript{k}{}{\mathtt{f}}_{\beta;\alpha\beta} ].
	\end{align*}
	Hence it remains to show $\bvd{k}_{\beta}(\prescript{k}{}{\mathtt{w}}_{\beta;\alpha\beta})+\prescript{k}{}{\pdb}_{\beta}( \prescript{k}{}{\mathtt{f}}_{\beta;\alpha\beta}) - [\prescript{k}{}{\mathtt{w}}_{\beta;\alpha\beta},\prescript{k}{}{\mathtt{f}}_{\beta;\alpha\beta} ] = 0$ in $\twc{k}^{*,*}_{\beta}$. 
	
	We fix a multi-index $(i_0, \ldots, i_l) \in \mathcal{I}$, and recall from the proofs of Lemmas \ref{lem:gluing_morphism_partial_bar_comparison} and \ref{lem:gluing_morphism_BV_operator_comparison} the formulas:
	\begin{align*}
		\prescript{k}{}{\pdb}_{\beta} (\prescript{k}{}{\mathtt{f}}_{\beta;\alpha\beta, i_0\cdots i_l} )  =& - \prescript{k}{}{\pdb}_{\beta} ( (\mathtt{T}(-[\sauto{k}_{\alpha\beta, i_0 \cdots i_l}, \cdot ]) \circ \bvd{k}_{\beta}) (\sauto{k}_{\alpha\beta,i_0 \cdots i_l}) ) \\
		&+ (\prescript{k}{}{\pdb}_{\beta} \circ \glue{k}_{\alpha\beta,i_0 \cdots i_l}) ( \prescript{k}{}{\mathtt{f}}_{\alpha;\alpha\beta,i_0}) \\
		\bvd{k}_{\beta } (\prescript{k}{}{\mathtt{w}}_{\beta;\alpha\beta,i_0 \cdots i_l})  = &- (\bvd{k}_{\beta} \circ \mathtt{T}(-[\sauto{k}_{\alpha\beta,i_0 \cdots i_l},\cdot ]) \circ \prescript{k}{}{\pdb}_\beta) (\sauto{k}_{\alpha\beta, i_0 \cdots i_l}),
	\end{align*}
	where $\mathtt{T}$ is the formal series introduced in \eqref{eqn:power_series_in_proof_gluing_morphism_BV_operator_comparison}. 
	Now we consider the dgLa $(\mathcal{A}^*(\simplex_l) \otimes \bva{k}^*_\beta(U_{i_0 \cdots i_l}), \bvd{k}_{\beta} + \prescript{k}{}{\pdb}_\beta, [\cdot,\cdot])$. We apply Lemma \ref{lem:gauge_action_on_differential} and notice that
	$$
	\mathtt{A}:= \exp(\sauto{k}_{\alpha\beta,i_0 \cdots i_l}) \star 0 = (\mathtt{T}(-[\sauto{k}_{\alpha\beta, i_0 \cdots i_l},\cdot]) \circ (\bvd{k}_\beta + \prescript{k}{}{\pdb}_\beta)) (\sauto{k}_{\alpha\beta, i_0 \cdots i_l}). 
	$$
	Since $\mathtt{A}$ is gauge equivalent to $0$ in the above dgLa, we have the equation $(\bvd{k}_\beta + \prescript{k}{}{\pdb}_\beta) \mathtt{A} + \half [\mathtt{A},\mathtt{A}] = 0$ whose component in $\mathcal{A}^1(\simplex_l) \otimes \bva{k}^0_\beta(U_{i_0 \cdots i_l})$ can be extracted as
	\begin{multline*}
		\prescript{k}{}{\pdb}_{\beta} ( (\mathtt{T}(-[\sauto{k}_{\alpha\beta, i_0 \cdots i_l}, \cdot ]) \circ \bvd{k}_{\beta}) (\sauto{k}_{\alpha\beta,i_0 \cdots i_l}) ) +\\
		 \bvd{k}_{\beta} ( (\mathtt{T}(-[\sauto{k}_{\alpha\beta,i_0 \cdots i_l},\cdot ]) \circ \prescript{k}{}{\pdb}_\beta) (\sauto{k}_{\alpha\beta, i_0 \cdots i_l}) )+ \\
		 [ (\mathtt{T}(-[\sauto{k}_{\alpha\beta, i_0 \cdots i_l}, \cdot ]) \circ \bvd{k}_{\beta}) (\sauto{k}_{\alpha\beta,i_0 \cdots i_l}), (\mathtt{T}(-[\sauto{k}_{\alpha\beta,i_0 \cdots i_l},\cdot ]) \circ \prescript{k}{}{\pdb}_\beta) (\sauto{k}_{\alpha\beta, i_0 \cdots i_l})] \\
		= 0.
	\end{multline*}
	Therefore, the $(i_0,\ldots,i_l)$-component of $\bvd{k}_{\beta}(\prescript{k}{}{\mathtt{w}}_{\beta;\alpha\beta})+\prescript{k}{}{\pdb}_{\beta}( \prescript{k}{}{\mathtt{f}}_{\beta;\alpha\beta}) - [\prescript{k}{}{\mathtt{w}}_{\beta;\alpha\beta},\prescript{k}{}{\mathtt{f}}_{\beta;\alpha\beta} ] $ is given by
	\begin{align*}
		& \bvd{k}_{\beta}(\prescript{k}{}{\mathtt{w}}_{\beta;\alpha\beta, i_0 \cdots i_l})+\prescript{k}{}{\pdb}_{\beta}( \prescript{k}{}{\mathtt{f}}_{\beta;\alpha\beta, i_0 \cdots i_l}) - [\prescript{k}{}{\mathtt{w}}_{\beta;\alpha\beta, i_0 \cdots i_l},\prescript{k}{}{\mathtt{f}}_{\beta;\alpha\beta, i_0 \cdots i_l} ] \\
		= & \prescript{k}{}{\pdb}_{\beta} (\glue{k}_{\alpha\beta,i_0 \cdots i_l} ( \prescript{k}{}{\mathtt{f}}_{\alpha;\alpha\beta,i_0})) - [\prescript{k}{}{\mathtt{w}}_{\beta;\alpha\beta, i_0 \cdots i_l},  \glue{k}_{\alpha\beta,i_0 \cdots i_l} ( \prescript{k}{}{\mathtt{f}}_{\alpha;\alpha\beta,i_0})]\\
		=&  \glue{k}_{\alpha\beta,i_0 \cdots i_l}  (\prescript{k}{}{\pdb}_\alpha (\prescript{k}{}{\mathtt{f}}_{\alpha;\alpha\beta,i_0}) ) = 0.
	\end{align*}
\end{proof}
	
\begin{definition}\label{def:total_de_rham_differential}
	We let $\dbobs{}:= \varprojlim_k \dbobs{k} \in \polyv{}^{-1,2}$ and $\dvolfobs{}:= \varprojlim_k \dvolfobs{k} \in \polyv{}^{0,1}$. The operator $\drd{}{} = \pdb + \dpartial{} + \dbobs{} \lrcorner$, which acts on $\totaldr{}{\bullet}^{*,*}$ preserves the filtration, is called the {\em total de Rham differential}. We also denote the pull back of $\drd{}{}$ to $\polyv{}^{*,*}$ under the isomorphism $\lrcorner \volf{} : \polyv{}^{*,*} \rightarrow \totaldr{}{\parallel}^{d+*,*}$ by $\ptd{}$.
\end{definition}
	
\begin{prop}\label{prop:checking_d_square_equal_zero}
	The pair $(\totaldr{}{\bullet}^{*} , \drd{}{})$ forms a filtered complex, i.e. $\drd{}{}^2 = 0$ and $\drd{}{}$ preserves the filtration. We also have $\ptd{} = \pdb + \bvd{} + (\dbobs{}  + \dvolfobs{} ) \wedge$ on $\polyv{}^{*}$.
\end{prop}
	
\begin{proof}
	From the discussion right before Proposition \ref{prop:thom_whitney_de_Rham_module_action_properties}, we compute $\drd{}{}^2 = (\pdb + \dpartial{} + \dbobs{} \lrcorner)^2 = (\pdb + \dpartial{})^2 - \mathcal{L}_{\dbobs{}} $. If we compute $(\pdb + \dpartial{})^2$ locally on $\ktwc{}{}^{*,*}_{\alpha;\alpha}$, we obtain $(\pdb_\alpha + \mathcal{L}_{\dbtwist{}_\alpha} + \dpartial{}_\alpha)^2 = \mathcal{L}_{\pdb_\alpha(\dbtwist{}_\alpha)} + \mathcal{L}_{\dbtwist{}_\alpha}^2 = \mathcal{L}_{\pdb_\alpha \dbtwist{}_\alpha + \half[\dbtwist{}_\alpha,\dbtwist{}_\alpha]} = \mathcal{L}_{\dbobs{}_\alpha}$. So we get $(\pdb+\dpartial{})^2 = \mathcal{L}_{\dbobs{}}$ and hence $\drd{}{}^2= 0$. 	
	As for $\ptd{}$, we compute locally on $\twc{}_{\alpha;\alpha}^{*,*}$. Taking $\gamma \in \twc{}_{\alpha;\alpha}^{*,*}$, we have 
	\begin{align*}
		& \drd{}{}( \gamma \lrcorner \exp(\volftwist{}_\alpha \lrcorner) (\volf{}_\alpha) )  = (\pdb_\alpha + \mathcal{L}_{\dbtwist{}_\alpha}+ \dpartial{}_\alpha + \dbobs{}_\alpha \lrcorner ) \left( (\gamma \wedge  \exp(\volftwist{}_\alpha) ) \lrcorner \volf{}_\alpha  \right)\\
		 =& \left( \pdb_{\alpha}(\gamma) + [\dbtwist{}_\alpha,\gamma] + \bvd{}_\alpha ( \gamma) + [\volftwist{}_\alpha,\gamma] + \dbobs{}_\alpha \wedge \gamma + \dvolfobs{}_\alpha \wedge \gamma  \right) \lrcorner (\exp(\volftwist{}_\alpha) \lrcorner \volf{}_{\alpha}),
	\end{align*}
	which gives the identity $\ptd{} = \pdb + \bvd{} + (\dbobs{}  + \dvolfobs{} ) \wedge$. 
\end{proof}
	
\subsection{The Gauss-Manin connection}\label{sec:GM_construction}
Using the natural isomorphisms $\gmiso{k}{1}_\alpha : (\tbva{k}{1:2}^*_\alpha, \dpartial{k}_\alpha) \cong \blat \otimes_\inte (\tbva{k}{0:1}^*_\alpha[-1], \dpartial{k}_\alpha)$ from Definition \ref{def:higher_order_data_module}, we obtain isomorphisms 
$$
\gmiso{k}{1}_\alpha : \ktwc{k}{1:2}^{*,*}_{\alpha;\alpha} \rightarrow \blat \otimes_\inte \ktwc{k}{0:1}^{*,*}_{\alpha;\alpha}[-1],
$$
which can be patched together to give an isomorphism of complexes $\gmiso{k}{1} : \totaldr{k}{1:2}^{*,*} =\totaldr{k}{1}^{*,*}/\totaldr{k}{2}^{*,*} \rightarrow \blat \otimes_\inte \totaldr{k}{\parallel}^{*,*}[-1]$ which is equipped with the differential $\drd{k}{}$. This produces an exact sequence of complexes:
\begin{equation}\label{eqn:short_exact_sequence_for_GM_connection}
	0 \rightarrow \blat \otimes_\inte \totaldr{k}{\parallel}^{*,*}[-1] \rightarrow \totaldr{k}{0}^{*,*}/\totaldr{k}{2}^{*,*}  \rightarrow \totaldr{k}{\parallel}^{*,*} \rightarrow 0,
\end{equation}
which we use to define the Gauss-Manin connection (cf. Definition \ref{def:0_th_order_GM_connection}). 
	
\begin{definition}\label{def:Gauss_Manin_connection}
	Taking the long exact sequence associated to \eqref{eqn:short_exact_sequence_for_GM_connection}, we get the map 
	\begin{equation}\label{eqn:GM_connection}
	\gmc{k}: H^*(\totaldr{k}{\parallel},\drd{k}{}) \rightarrow \blat \otimes_\inte H^*(\totaldr{k}{\parallel},\drd{k}{}),
	\end{equation}
	which is called the {\em $k^{\text{th}}$-order Gauss-Manin (abbrev. GM) connection} over $\cfrk{k}$. Taking inverse limit over $k$ gives the {\em formal Gauss-Manin connection} over $\hat{\cfr}$:
	$$
	\gmc{} : \reallywidehat{H^*(\totaldr{}{\parallel},\drd{}{})} \rightarrow \blat \otimes_\inte \reallywidehat{H^*(\totaldr{}{\parallel},\drd{}{})}.
	$$
	The modules $H^*(\totaldr{k}{\parallel},\drd{k}{})$ and $\reallywidehat{H^*(\totaldr{}{\parallel},\drd{}{})}$ over $\cfrk{k}$ and $\hat{\cfr}$ are respectively called the {\em $k^{\text{th}}$-order Hodge bundle} and the {\em formal Hodge bundle}. 
\end{definition}
	
\begin{remark}\label{rem:Gauss_Manin_relation_with_zero_order}
	By its construction, the complex $(\totaldr{0}{\parallel}^{*,*}, \drd{0}{})$ serves as a resolution of the complex $(\rdr{0}{}^*, \dpartial{0})$, and the cohomology $H^*(\totaldr{0}{\parallel}, \drd{0}{})$ computes the hypercohomology $\mathbb{H}^*(\rdr{0}{}, \dpartial{0})$. So the $0^{\text{th}}$-order Gauss-Manin connection $\gmc{0}$ agrees with the one introduced in Definition \ref{def:0_th_order_GM_connection}.  
\end{remark}
	
\begin{prop}\label{prop:flatness_of_gauss_manin}
	The Gauss-Manin connection $\gmc{}$ defined in Definition \ref{def:Gauss_Manin_connection} is a flat connection, i.e. the map $\gmc{}^2 : \reallywidehat{H^*(\totaldr{}{\parallel},\drd{}{})} \rightarrow \wedge^2(\blat_\comp) \otimes_\comp \reallywidehat{H^*(\totaldr{}{\parallel},\drd{}{})}$ is a zero map. 
\end{prop}
	
\begin{proof}
	It suffices to show the $k^{\text{th}}$-order Gauss-Manin connection $\gmc{k}$ is flat for every $k$. Consider the short exact sequence \eqref{eqn:short_exact_sequence_for_GM_connection}, and take a cohomology class $[\eta] \in H^*(\totaldr{k}{\parallel},\drd{k}{})$ represented by an element $\eta \in \totaldr{k}{\parallel}^*$. Then we take a lifting $\widetilde{\eta} \in \totaldr{k}{0}^*$ so that $\gmc{k}([\eta])$ is represented by the element $\drd{k}{}(\widetilde{\eta}) \in \totaldr{k}{0}^*/\totaldr{k}{2}^*$. We write $\gmc{k}([\eta]) = \sum_{i} \alpha_i \otimes [\xi_i]$ for $\alpha_i \in \logsdrk{k}{1}$ and $[\xi_i] \in H^*(\totaldr{k}{\parallel},\drd{k}{})$. 
	Once again we take a representative $\xi_i \in \totaldr{k}{\parallel}^*$ for $[\xi_i]$ and by our construction we have an element $\mathsf{e} \in \totaldr{k}{2}^*$ such that $ \sum_{i} \alpha_i \otimes \xi_i = \drd{k}{}(\widetilde{\eta}) + \mathsf{e}$. Therefore if we consider the exact sequence of complexes
	$$
	0 \rightarrow  \totaldr{k}{2}^{*}/\totaldr{k}{3}^{*} \rightarrow \totaldr{k}{1}^{*}/\totaldr{k}{3}^{*}  \rightarrow \totaldr{k}{1}^{*}/\totaldr{k}{2}^* \rightarrow 0,
	$$
	we have $\drd{k}{} \left ( \sum_{i} \alpha_i \otimes \xi_i \right) = \drd{k}{}(\mathsf{e}) \in \totaldr{k}{2}^{*}/\totaldr{k}{3}^{*}$. Note that $(\gmc{k})^2([\eta])$ is represented by the cohomology class of the element $\drd{k}{} \left( \sum_{i} \alpha_i \otimes \xi_i\right) \in \totaldr{k}{2}^*/\totaldr{k}{3}^* \cong \logsdrk{k}{2}\otimes_{(\cfrk{k})} \totaldr{k}{\parallel}^*[-2]$ using the isomorphism induced by $\gmiso{k}{2}_{\alpha}$'s from Definition \ref{def:higher_order_data_module}. Hence we have $[\drd{k}{} \left( \sum_{i} \alpha_i \otimes \xi_i\right)] = [\drd{k}{}(\mathsf{e})] = 0$ in $\logsdrk{k}{2} \otimes_{(\cfrk{k})} H^*(\totaldr{k}{\parallel},\drd{k}{})[-2]$. 
\end{proof}

\subsection{Freeness of the Hodge bundle from a local criterion}\label{sec:triviality_of_hodge_bundle}

To prove the desired unobstructedness result, we need freeness of the Hodge bundle; in geometric situations, this has been established in various cases \cite{katz1970nilpotent, steenbrink1976limits, kawamata1994logarithmic, Gross-Siebert-logII}. In this subsection, we generalize the techniques in \cite{Gross-Siebert-logII, kawamata1994logarithmic, steenbrink1976limits} to prove the freeness of the $k^{\text{th}}$-order Hodge bundle $H^*(\totaldr{k}{\parallel}^*,\drd{k}{})$ over $\cfrk{k}$ in our abstract setting (Lemma \ref{lem:triviality_of_hodge_bundle}) under a local criterion (Assumption \ref{assum:local_assumption_for_triviality_of_hodge_bundle}).

\subsubsection{A local condition}\label{sec:assumption_on_de_rham_module_for_locally_free}

Recall from Notation \ref{not:universal_monoid} that we have a strictly convex polyhedral cone $Q_\real \subset \blat_\real$, the coefficient ring $\cfr = \comp[Q]$, and the log space $\logs$ (or the formal log space $\logsf$) parametrizing the moduli space near the degenerate Calabi-Yau variety $(X,\mathcal{O}_X)$.
For every primitive element $\mathtt{n} \in \text{int}(Q^\vee_\real) \cap \blat^\vee$, we have a natural ring homomorphism $i_{\mathtt{n}}:\comp[Q] \rightarrow \comp[q]$, $q^{m} \mapsto q^{(m,\mathtt{n})}$, where $(\cdot ,\cdot )$ denotes the natural pairing between $\blat$ and $\blat^\vee$, and then taking spectra gives a map $i_{\mathtt{n}} : \mathbb{A}^{1,\dagger} \rightarrow \logs$ (or $\prescript{k}{}{i}_{\mathtt{n}} : \prescript{k}{}{\mathbb{A}}^{1,\dagger} \rightarrow \logsk{k}$ for each $k \in \inte_{\geq 0}$), where $\mathbb{A}^{1,\dagger}$ is the log space associated to the log ring $\comp[q]^\dagger$ which is equipped with the monoid homomorphism $\mathbb{N}\rightarrow \comp[q]$, $k \mapsto q^k$.

Geometrically, taking base change with the map $\mathbb{A}^{1,\dagger} \rightarrow \logs$ should be viewed as restricting the family to the $1$-dimensional family determined by $\mathtt{n}$. In our abstract setting, we consider the tensor product $\bva{k}_{\mathtt{n}:\alpha}^*:= \bva{k}_{\alpha}^* \otimes_{\cfrk{k}} (\comp[q]/(q^{k+1}))$. Then tensoring the maps $\patch{k}_{\alpha\beta,i}$'s with $\comp[q]/(q^{k+1})$ give patching morphisms for the $\bva{k}_{\mathtt{n}:\alpha}^*$'s which will be denoted as $\patch{k}_{\mathtt{n}:\alpha\beta,i}$. Similarly we use $\resta{k,l}_{\mathtt{n}:\alpha\beta,i}$'s, $\patchij{k}_{\mathtt{n}:\alpha\beta,ij}$'s, $\cocyobs{k}_{\mathtt{n}:\alpha\beta\gamma,i}$'s and $\bvdobs{k}_{\mathtt{n}:\alpha\beta,i}$'s to denote the tensor products of the corresponding terms appearing in Definition \ref{def:higher_order_patching} with $\comp[q]/(q^{k+1})$. Note that all the relations in Definition \ref{def:higher_order_patching} still hold after taking the tensor products.

In view of the isomorphism $\lrcorner \volf{k}: \polyv{k}^{*,*} \rightarrow \totaldr{k}{\parallel}^{d+*,*}$ in Definition \ref{def:total_de_rham_differential} and the fact that the complex $(\polyv{k}^{*,*}, \ptd{k})$ is free over $\cfrk{k}$ (meaning that the differential is $\cfrk{k}$-linear), we see that $(\totaldr{k}{\parallel}^{d+*,*},\drd{k}{})$ is also free over $\cfrk{k}$. Then taking tensor product with $\comp[q]/(q^{k+1})$ (for a fixed $\mathtt{n}$), we obtain the relative de Rham complex $(\totaldr{}{\parallel}^{*} \otimes_{\cfrk{k}} (\comp[q]/(q^{k+1})),\drd{}{})$ over $\comp[q]/(q^{k+1})$. 

Now the filtered de Rham module $\tbva{k}{\bullet}_\alpha^*$ plays the role of the sheaf of holomorphic de Rham complex on the thickening of $V_\alpha$. 
We need to consider restrictions of these holomorphic differential forms to the $1$-dimensional family $\text{Spec}(\comp[q]/(q^{k+1}))$, but na\"ively taking tensor product with $\comp[q]/(q^{k+1})$ does not give the desired answer.
In our abstract setting, the existence of such restrictions can be formulated as the following assumption (which is motivated by the proof of \cite[Theorem 4.1]{Gross-Siebert-logII}):

\begin{assum}\label{assum:local_assumption_for_triviality_of_hodge_bundle}
	For each $\mathtt{n} \in \text{int}(Q_\real^\vee) \cap \blat^\vee$, $k \in \inte_{\geq 0}$ and $V_\alpha \in \mathcal{V}$, we assume there exists a coherent sheaf of dga's 
	$$(\tbva{k}{\bullet}^*_{\mathtt{n}:\alpha}, \wedge, \dpartial{k}_{\mathtt{n}:\alpha})$$
	equipped with a dg module structure over $\prescript{k}{}{\Omega}^*_{\mathbb{A}^{1,\dagger}}$, the natural filtration $$\tbva{k}{}_{\mathtt{n}:\alpha}^*=\tbva{k}{0}^*_{\mathtt{n}:\alpha} \supset \tbva{k}{1}^*_{\mathtt{n}:\alpha} \supset \tbva{k}{2}^*_{\mathtt{n}:\alpha} = \{0\}$$
	where $\tbva{k}{1}^*_{\mathtt{n}:\alpha} = d\log(q) \wedge \tbva{k}{0}^*_{\mathtt{n}:\alpha}[1]$, and a de Rham module structure over $(\bva{k}^*_{\mathtt{n}:\alpha},\wedge, \bvd{k}_{\mathtt{n}:\alpha})$ satisfying all the conditions in Definitions  \ref{def:higher_order_data_module} and \ref{def:higher_order_module_patching} (in particular, we have surjective morphisms $\rest{k,l}_{\mathtt{n}:\alpha}: \tbva{k}{\bullet}^*_{\mathtt{n}:\alpha} \rightarrow \tbva{l}{\bullet}^*_{\mathtt{n}:\alpha}$ for $k \geq l$, a volume element $\volf{k}_{\mathtt{n}:\alpha} \in \tbva{k}{0}_{\mathtt{n}:\alpha}^d/\tbva{k}{1}_{\mathtt{n}:\alpha}^d$, an isomorphism $\gmiso{k}{}_{\mathtt{n}:\alpha}:  (\tbva{k}{1}^*_{\mathtt{n}:\alpha}/ \tbva{k}{2}^*_{\mathtt{n}:\alpha}, \dpartial{k}_{\mathtt{n}:\alpha}) \cong  (\tbva{k}{0}^*_{\mathtt{n}:\alpha}/ \tbva{k}{1}^*_{\mathtt{n}:\alpha}[-1], \dpartial{k}_{\mathtt{n}:\alpha})$, and patching isomorphisms $\hpatch{k}_{\mathtt{n}:\alpha\beta, i} : \tbva{k}{\bullet}_{\mathtt{n}:\alpha}^*|_{U_i} \rightarrow \tbva{k}{\bullet}_{\mathtt{n}:\beta}^*|_{U_i}$ for triples $(U_i; V_\alpha, V_\beta)$ with $U_i \subset V_{\alpha\beta}$ fulfilling all the required conditions).
	We further assume that the complex $(\tbva{k}{}_{\mathtt{n}:\alpha}^*[u], \widetilde{\dpartial{k}_{\mathtt{n}:\alpha}})$, where $$\widetilde{\dpartial{k}_{\mathtt{n}:\alpha}}(\sum_{s=0}^l \nu_s u^s) := \sum_{s} (\dpartial{k}_{\mathtt{n}:\alpha} \nu_s) u^s + s d\log(q) \wedge \nu_s u^{s-1},$$
	satisfies {\em the holomorphic Poincar\'{e} Lemma} in the sense that
	for each Stein open subset $U$ and any $\sum_{s} \nu_s u^s \in \tbva{k}{}_{\mathtt{n}:\alpha}^*(U)[u]$ with $\widetilde{\dpartial{k}_{\mathtt{n}:\alpha}} (\nu) = 0$, we have $\sum_{s}\eta_s u^s \in \tbva{k}{}_{\mathtt{n}:\alpha}^*(U)[u]$ satisfying $\widetilde{\dpartial{k}_{\mathtt{n}:\alpha}}(\sum_{s}\eta_s u^s) = \sum_{s} \nu_s u^s$ on $U$, and if in addition $\rest{k,0}_{\mathtt{n}:\alpha}(\nu_0) = 0$ in $(\tbva{0}{0}_{\mathtt{n}:\alpha}/ \tbva{0}{1}_{\mathtt{n}:\alpha})(U)$, then $\sum_{s}\eta_s u^s$ can be chosen so that $\rest{k,0}_{\mathtt{n}:\alpha}(\eta_0) = 0$ in $(\tbva{0}{0}_{\mathtt{n}:\alpha}/ \tbva{0}{1}_{\mathtt{n}:\alpha})(U)$.
\end{assum}

Assumption \ref{assum:local_assumption_for_triviality_of_hodge_bundle} allows us to construct the total de Rham complex $(\totaldr{}{\bullet}^{*,*}_{\mathtt{n}}, \drd{}{}_{\mathtt{n}} = \pdb_{\mathtt{n}} + \dpartial{}_{\mathtt{n}}+\dbobs{}_{\mathtt{n}} \lrcorner )$ as a dg-module over $\hat{\Omega}^*_{\mathbb{A}^{1,\dagger}}$ such that $\totaldr{}{\parallel}^{*} \otimes_{\hat{\cfr}} (\comp[q]/(q^{k+1})) = \totaldr{k}{0}^{*,*}_{\mathtt{n}}/\totaldr{k}{1}^{*,*}_{\mathtt{n}} =: \totaldr{k}{\parallel}^{*,*}_{\mathtt{n}}$.

\begin{example}\label{ex:log-smooth-VI}
	In the log smooth case (see Examples \ref{ex:log-smooth-I} and \ref{ex:log-smooth-II}), Assumption \ref{assum:local_assumption_for_triviality_of_hodge_bundle}, which is local in nature, can be checked by simply taking base change of the family $\pi: \mathbf{V}_{\alpha}^{\dagger} \rightarrow \logsk{k}$ with $i_{\mathtt{n}}: \mathbb{A}^{1,\dagger} \rightarrow \logs$, and working on $\tbva{k}{}^*_{\mathtt{n}:\alpha}$ using the local computations from \cite[\S 2]{steenbrink1976limits}. Alternatively, and more conveniently, one can use (analytification of) the local computations in \cite[proof of Theorem 4.1]{Gross-Siebert-logII} (or more appropriately \cite[Theorem 1.10]{Felten-Filip-Ruddat}; see footnote \ref{fn:gap-gross-siebert}). Actually Assumption \ref{assum:local_assumption_for_triviality_of_hodge_bundle} is motivated from \cite[proof of Theorem 4.1]{Gross-Siebert-logII}, for which the $d$-semistable log smooth local model $\prescript{k}{}{\mathbb{V}}_{\alpha}$ is included as a special case.
\end{example}

\subsubsection{Freeness of the Hodge bundle}\label{sec:proof_of_triviality_of_hodge_bundle} 
We consider a general monomial ideal $I$ of $\cfr$ such that $\mathbf{m}^{k} \subset I$ for some integer $k \in \inte_+$, and we let $\totaldr{I}{\parallel}^* := \totaldr{k}{\parallel}^* \otimes_{\cfrk{k}} (\cfr/ I)$ equipped with the differential by taking tensor product which is also denoted by $\drd{I}{}$. We should omit the dependence of our notation on the operator $\drd{I}{}$ when it is clear from the contents. We consider two such ideals $I \subset J$ such that $\mathbf{m} \cdot J \subset I$ and the following exact sequence of complexes
\begin{equation}\label{eqn:short_exact_sequence_of_complexes_for_locally_triviality_of_hodge_bundles}
\xymatrix@1{
	0 \ar[r] &  \totaldr{0}{\parallel}^* \otimes_\comp(J/I) \ar[r] & \totaldr{I}{\parallel}^* \ar[r]^{\rest{I,J}} &  \totaldr{J}{\parallel}^* \ar[r] &0}. 
\end{equation}
Then we consider the long exact sequence associated to \eqref{eqn:short_exact_sequence_of_complexes_for_locally_triviality_of_hodge_bundles} and let
\begin{equation}\label{eqn:connecting_homomorphism_for_triviality_of_hodge_bundle}
\prescript{I,J}{}{\delta} : H^*(\totaldr{J}{\parallel}^*) \rightarrow  H^{*}(\totaldr{0}{\parallel}^{*})[1] \otimes_\comp(J/I)
\end{equation}
as in the proof of \cite[Lemma 4.1]{kawamata1994logarithmic}. 

\begin{lemma}\label{lem:ideal_filtration_lemma_for_triviality_of_hodge_bundle}
	Suppose we have a filtration $I = I_l \subset I_{l-1} \subset \cdots \subset I_0 = J$ of monomial ideals. Then the connecting homomorphism $\prescript{I,J}{}{\delta}$ in \eqref{eqn:connecting_homomorphism_for_triviality_of_hodge_bundle} is zero if and only if the corresponding connecting homomorphism $\prescript{I_{j+1},I_{j}}{}{\delta}$ (abbrev. by $\prescript{j+1,j}{}{\delta}$) is zero for each $j = 0,\dots,l-1$. 
\end{lemma}

\begin{proof}
	First assume that the homomorphism $\prescript{I,J}{}{\delta}$ is non-zero. Let $m$ be the minimum of those $j = 1,\dots,l$ such that the composition 
	$$
	\xymatrix@1{
		H^*(\totaldr{J}{\parallel}^*) \ar[rr]^{\prescript{I,J}{}{\delta}} & &  H^{*+1}(\totaldr{0}{\parallel}^{*}) \otimes_\comp(J/I) \ar[r] & H^{*+1}(\totaldr{0}{\parallel}^{*}) \otimes_\comp(J/I_j)
	}
	$$
	is non-zero. Then we consider the commutative diagram 
	$$
	\xymatrix@1{
		H^*(\totaldr{I}{\parallel}^*)  \ar[r]\ar[d]& H^*(\totaldr{J}{\parallel}^*) \ar[d] \ar[rrr]^{\prescript{I,J}{}{\delta}}  & & & H^{*+1}(\totaldr{0}{\parallel}^{*}) \otimes_\comp(J/I) \ar[d]\\
		H^*(\totaldr{I_m}{\parallel}^*)  \ar[r]& H^*(\totaldr{J}{\parallel}^*) \ar[rrr]^{\prescript{I_m,J}{}{\delta}}  & & & H^{*+1}(\totaldr{0}{\parallel}^{*}) \otimes_\comp(J/I_m)\\
		H^*(\totaldr{I_m}{\parallel}^*)  \ar[u] \ar[r]& H^*(\totaldr{I_{m-1}}{\parallel}^*) \ar[rrr]^{\prescript{m,m-1}{}{\delta}} \ar[u] & & & H^{*+1}(\totaldr{0}{\parallel}^{*}) \otimes_\comp(I_{m-1}/I_m) \ar[u].
	}
	$$
	Notice that the connecting homomorphism $\prescript{I_m,J}{}{\delta}$ is non-zero with its image lying in the subspace $H^{*+1}(\totaldr{0}{\parallel}^{*}) \otimes_\comp(I_{m-1}/I_m)$. So $\prescript{m,m-1}{}{\delta}$ is non-zero. 
	
	Conversely, using the above commutative diagram, we observe that if $\prescript{m,m-1}{}{\delta}$ is non-zero for some $ 0\leq m \leq l-1$ then the connecting homomorphism $\prescript{I,J}{}{\delta}$ cannot be zero. 
\end{proof}

To prove triviality of the Hodge bundle, we need a decreasing tower of ideals $\mathbf{m} = J_1 \supset J_2 \supset \cdots $ such that $\mathbf{m}\cdot J_{i} \subset J_{i+1} $ for each $i$, $J_i/ J_{i+1}$ is at most one-dimensional and $\hat{\cfr} = \varprojlim_{i} (\cfr/J_i)$. We should show that the connecting homomorphism 
$$
\prescript{J_{i+1},J_{i}}{}{\delta} : H^*(\totaldr{J_i}{\parallel}^*) \rightarrow  H^{*+1}(\totaldr{0}{\parallel}^{*}) \otimes_\comp(J_i/J_{i+1}) 
$$
is zero for $i = 1,2,\dots$. 

To construct such a tower, we take an element $\mathtt{n}_0 \in \text{int}(Q_\real^\vee)\cap \blat^\vee$ and define the monomial ideal $\tilde{J}_i:= \langle q^{m} \mid m \in Q, \ (m, \mathtt{n}_0) \geq i \rangle$, giving a sequence $\tilde{J}_1 \supset \tilde{J}_2 \supset \cdots$, which should be further refined. For each $i$, notice that the finite dimensional vector space $\tilde{J}_i/\tilde{J}_{i+1}$ has a basis $q^m$ given by the lattice points $m \in Q$ with $(m,\mathtt{n}_0) = i$. We take a generic element $e \in \text{int}(Q_\real^\vee) \cap \blat^{\vee}$ such that $(m_1,e) \neq (m_2, e)$, for all $m_1\neq m_2$ lying in the set $\{m \in Q \ | \ (m,\mathtt{n}_0) = i \} $ (this can be done since there are only finitely many such $m$'s). We further take $L$ large enough so that if we let $\mathtt{n} = L \mathtt{n}_0 + e$, there is an integer $l$ such that $(m,\mathtt{n}) \geq l$ for $m\in Q$ with $(m,\mathtt{n}_0) \geq i+1$, and $ (m',\mathtt{n}) \leq l-1$ for $m' \in Q$ with $(m',\mathtt{n}_0) = i$. We can therefore define the refined filtration $\tilde{J}_{i+1} = I_{j,l} \subset I_{j,l-1} \subset \cdots \subset I_{j,0} = \tilde{J_{i}}$ such that $I_{j,s}$ is the monomial ideal generated by those $q^{m}$ with $m \in Q$, $(m,\mathtt{n}_0) \geq i$ and $(m, \mathtt{n}) \geq s$. 
Such a choice ensures that there is at most one $m \in Q$ such that $(m,\mathtt{n}_0) = i$ and $(m,\mathtt{n}) = s$ for a fixed $s$, and hence $I_{j,s-1}/I_{j,s}$ is at most one-dimensional. Making such a refinement for each $\tilde{J}_i \supset \tilde{J}_{i+1}$ and possibly renumbering the sequence, we obtain the desired sequence $J_i \supset J_{i+1} \supset \cdots $. For each pair $J_j / J_{j+1} \cong \comp$, we notice that there is an $\mathtt{n}$ together with $i_{\mathtt{n}}: \cfr \rightarrow \comp[q]$ and some $k\in \inte_+$ such that $i_{\mathtt{n}}^{-1}(q^{k}) = J_j$ and $i_{\mathtt{n}}^{-1}(q^{k+1}) = J_{j+1}$ with $i_{\mathtt{n}} : J_{j}/J_{j+1} \cong \comp \cdot q^{k}$. 

We have the following commutative diagram of complexes:
$$
\xymatrix@1{
	0 \ar[r] &  \totaldr{0}{\parallel}^* \otimes_\comp(J_{j}/J_{j+1}) \ar[r] \ar[d]^{i_{\mathtt{n}}^*}& \totaldr{J_{j+1}}{\parallel}^* \ar[rr]^{\rest{J_{j+1},J_{j}}} \ar[d]^{i_{\mathtt{n}}^*}& &  \totaldr{J_{j}}{\parallel}^* \ar[r] \ar[d]^{i_{\mathtt{n}}^*}&0 \\
	0 \ar[r] & \totaldr{0}{\parallel}^* \otimes_\comp \comp \cdot q^{k} \ar[r]&  \totaldr{k+1}{\parallel}^*_{\mathtt{n}} \ar[rr] &&\totaldr{k}{\parallel}^*_{\mathtt{n}} \ar[r]& 0
}
$$
such that the induced map $i_{\mathtt{n}}^* : H^*(\totaldr{0}{\parallel}^*) \otimes_\comp(J_{j}/J_{j+1}) \rightarrow H^*(\totaldr{0}{\parallel}^*) \otimes_\comp \comp \cdot q^{k}$ is an isomorphism. Therefore it remains to show that $H^*(\totaldr{k}{\parallel}^*_{\mathtt{n}})$ is a free $\comp[q]/(q^{k+1})$ module for each $k$. 

\begin{lemma}\label{lem:triviality_of_hodge_bundle}
	Under Assumption \ref{assum:local_assumption_for_triviality_of_hodge_bundle}, $H^*(\totaldr{K}{\parallel})$ is a free $\cfr/K$ module for any ideal $K \subset \mathbf{m}$ satisfying $\mathbf{m}^L \subset K$ for some $L$. 
\end{lemma}

\begin{proof}
We first consider the case $K = J_j$ for some $j$.
Similar to \cite[p. 404]{kawamata1994logarithmic} and the proof of \cite[Theorem 4.1]{Gross-Siebert-logII}, it suffices to show that the map
$\rest{k,0}_{\mathtt{n}} : H^*(\totaldr{k}{\parallel}^*_{\mathtt{n}}, \drd{k}{}_{\mathtt{n}}) \rightarrow H^*(\totaldr{0}{\parallel}^*_{\mathtt{n}}, \drd{0}{}_{\mathtt{n}})$,
which is induced by the maps $\rest{k,0}_{\mathtt{n}:\alpha}$'s in Assumption \ref{assum:local_assumption_for_triviality_of_hodge_bundle}, is surjective for all $k \in \inte_{\geq 0}$. 
Following the proof of \cite[Theorem 4.1]{Gross-Siebert-logII}, we consider the complex $(\widetilde{\totaldr{k}{}^*_{\mathtt{n}}}, \widetilde{\drd{k}{}_{\mathtt{n}}})$ constructed from the complexes $(\tbva{k}{}_{\mathtt{n}:\alpha}^*[u], \widetilde{\dpartial{k}_{\mathtt{n}:\alpha}})$'s as in Definition \ref{def:total_de_rham_differential}. There is a natural restriction map $\widetilde{\rest{k,0}_{\mathtt{n}}}: \widetilde{\totaldr{k}{}^*_{\mathtt{n}}}  \rightarrow \totaldr{0}{\parallel}_{\mathtt{n}}$ defined by $\widetilde{\rest{k,0}_{\mathtt{n}:\alpha}} ( \sum_{s=0}^l \eta_s u^s ) = \rest{k,0}_{\mathtt{n}:\alpha}(\eta_0)$ on $\tbva{k}{}_{\mathtt{n}:\alpha}^*[u]$ for each $\alpha$. Since  $\widetilde{\rest{k,0}_{\mathtt{n}}}$ (and hence the induced map on cohomology) factors through $\rest{k,0}_{\mathtt{n}}$, 
we only need to show that the map $\widetilde{\rest{k,0}_{\mathtt{n}}}: H^*(\widetilde{\totaldr{k}{}^*_{\mathtt{n}}}, \widetilde{\drd{k}{}_{\mathtt{n}}}) \rightarrow H^{*}(\totaldr{0}{\parallel}_{\mathtt{n}}, \drd{0}{}_{\mathtt{n}})$ is an isomorphism.

By gluing the sheaves $\tbva{k}{}_{\mathtt{n}:\alpha}^*$'s (resp. $\tbva{k}{}_{\mathtt{n}:\alpha}^*[u]$'s) as in 
Definition \ref{def:construction_total_de_rham}, we can construct the \v{C}ech-Thom-Whitney complexes $\cech{k}^*(\ktwc{}{}^*_{\mathtt{n}},\drglue{}_\mathtt{n})$ (resp. $\cech{k}^*(\prescript{}{}{\mathbf{B}}_{\mathtt{n}}^*,\drglue{}_\mathtt{n})$) and obtain the exact sequences
\begin{align*}
	0 \rightarrow \totaldr{k}{\bullet}_{\mathtt{n}}^{*,*}(\drglue{}_{\mathtt{n}}) \rightarrow \cech{k}^0(\ktwc{}{\bullet}^{*,*}_{\mathtt{n}},\drglue{}_{\mathtt{n}}) \rightarrow  \cdots \rightarrow \cech{k}^{\ell}(\ktwc{}{\bullet}^{*,*}_{\mathtt{n}},\drglue{}_{\mathtt{n}}) \rightarrow \cdots,\\
 0 \rightarrow	\widetilde{\totaldr{k}{\bullet}^{*,*}_{\mathtt{n}}}(\drglue{}_{\mathtt{n}}) \rightarrow \cech{k}^0(\prescript{}{\bullet}{\mathbf{B}}^{*,*}_{\mathtt{n}},\drglue{}_{\mathtt{n}})  \rightarrow \cdots \rightarrow \cech{k}^{\ell}(\prescript{}{\bullet}{\mathbf{B}}^{*,*}_{\mathtt{n}},\drglue{}_{\mathtt{n}}) \rightarrow \cdots.
\end{align*}	
Now we have a commutative diagram
$$
\xymatrix@1{
\widetilde{\totaldr{k}{}^*_\mathtt{n}} \ar[rrr]^{\prescript{k}{}{\delta}_{-1}} \ar[d]^{\widetilde{\rest{k,0}_{\mathtt{n}}}}& & & \cech{k}^*(\mathbf{B}_{\mathtt{n}}^*,\drglue{}_\mathtt{n}) \ar[d]^{\widetilde{\rest{k,0}_{\mathtt{n}}} }\\
\totaldr{0}{\parallel}^*_{\mathtt{n}} \ar[rrr]^{\prescript{0}{}{\delta}_{-1}}& & & \cech{0}^*(\ktwc{}{0:1}^*_{\mathtt{n}},\drglue{}_\mathtt{n}) 
} 
$$
where the horizontal arrows are quasi-isomorphisms. So what we need is to show that
$\widetilde{\rest{k,0}_{\mathtt{n}}}: \cech{k}^*(\prescript{}{}{\mathbf{B}}^*_{\mathtt{n}},\drglue{}_\mathtt{n}) \rightarrow \cech{0}^*(\ktwc{}{0:1}^*_\mathtt{n},\drglue{}_\mathtt{n})$
is a quasi-isomorphism. 
	
The decreasing filtrations 
\begin{align*}
\mathtt{F}^{\geq l} \left(\cech{k}^*(\prescript{}{}{\mathbf{B}}^*_{\mathtt{n}},\drglue{}_\mathtt{n}) \right) &:= \cech{k}^{*}(\prescript{}{}{\mathbf{B}}^{*, \geq l}_{\mathtt{n}},\drglue{}_\mathtt{n}),\\
\mathtt{F}^{\geq l} \left(\cech{0}^*(\ktwc{}{0:1}^*_{\mathtt{n}},\drglue{}_\mathtt{n}) \right) &:= \cech{0}^*(\ktwc{}{0:1}^{*,\geq l}_{\mathtt{n}},\drglue{}_\mathtt{n})
\end{align*}
induce spectral sequences with
\begin{align*}
E_0^{\mathtt{r} \mathtt{q}} \left( \cech{k}^*(\prescript{}{}{\mathbf{B}}^*_{\mathtt{n}},\drglue{}_\mathtt{n})\right) &=  \bigoplus_{ p + \ell = \mathtt{r} } \cech{k}^{\ell}(\prescript{}{}{\mathbf{B}}^{p,  \mathtt{q}}_{\mathtt{n}},\drglue{}_\mathtt{n}), \\
E_0^{\mathtt{r} \mathtt{q} }  \left(\cech{0}^*(\ktwc{}{0:1}^*_{\mathtt{n}},\drglue{}_\mathtt{n}) \right)  &=\bigoplus_{ p + \ell = \mathtt{r} }  \cech{0}^{\ell}(\ktwc{}{0:1}^{p,\mathtt{q}}_{\mathtt{n}},\drglue{}_\mathtt{n})
\end{align*}
respectively converging to their cohomologies.
Therefore it remains to prove that the map 
$$
\widetilde{\rest{k,0}_{\mathtt{n}}} : \bigoplus_{ p + \ell = \mathtt{r} } \cech{k}^{\ell}(\prescript{}{}{\mathbf{B}}^{p,  \mathtt{q}}_{\mathtt{n}},\drglue{}_\mathtt{n}) \rightarrow \bigoplus_{ p + \ell = \mathtt{r} }  \cech{0}^{\ell}(\ktwc{}{0:1}^{p,\mathtt{q}}_{\mathtt{n}},\drglue{}_\mathtt{n})
$$
induces an isomorphism on the cohomology of the $E_0$-page for each fixed $\mathtt{q}$.
	
If we further consider the filtrations
$\displaystyle \bigoplus_{ \substack{p + \ell = \mathtt{r}\\ \ell \geq l}} \cech{k}^{\ell}(\prescript{}{}{\mathbf{B}}^{p,  \mathtt{q}}_{\mathtt{n}},\drglue{}_\mathtt{n})$ and \\
$\displaystyle \bigoplus_{ \substack{ p + \ell = \mathtt{r}\\ \ell \geq l} }  \cech{0}^{\ell}(\ktwc{}{0:1}^{p,\mathtt{q}}_{\mathtt{n}},\drglue{}_\mathtt{n})$,
then we only need to show that the induced map 
$$
\xymatrix@1{
\cech{k}^{\ell}(\prescript{}{}{\mathbf{B}}^{*,  \mathtt{q}}_{\mathtt{n}},\drglue{}_\mathtt{n}) \ar[rr]^{	\widetilde{\rest{k,0}_{\mathtt{n}}} } \ar@{=}[d]& &  \cech{0}^{\ell}(\ktwc{}{0:1}^{*,\mathtt{q}}_{\mathtt{n}},\drglue{}_\mathtt{n}) \ar@{=}[d]\\
\bigoplus_{p \geq 0} \prod_{\alpha_0 \cdots \alpha_\ell} \prescript{k}{}{\mathbf{B}}^{p,\mathtt{q}}_{\mathtt{n}:\alpha_0 \cdots \alpha_\ell} (\drglue{}_{\mathtt{n}}) \ar[rr]^{	\widetilde{\rest{k,0}_{\mathtt{n}}} } & & \bigoplus_{p \geq 0} \prod_{\alpha_0 \cdots \alpha_\ell}  \ktwc{0}{0:1}^{p,\mathtt{q}}_{\mathtt{n}:\alpha_0 \cdots \alpha_\ell} ( \drglue{}_{\mathtt{n}}) }
$$
on the corresponding $E_0$-page is a quasi-isomorphism for any fixed $c$ and $\mathtt{q}$, where $\prescript{k}{}{\mathbf{B}}^{p,\mathtt{q}}_{\mathtt{n}:\alpha_0 \cdots \alpha_\ell} (\drglue{}_{\mathtt{n}})$ is constructed by gluing together the sheaves $\tbva{k}{}^*_{\mathtt{n}:\alpha}[u]$'s as in Definition \ref{def:construction_total_de_rham}.
	
Note that the differential on $\bigoplus_{p \geq 0} \prod_{\alpha_0 \cdots \alpha_\ell} \prescript{k}{}{\mathbf{B}}^{p,\mathtt{q}}_{\mathtt{n}:\alpha_0 \cdots \alpha_\ell} ( \drglue{}_{\mathtt{n}})$ is given componentwise by the differential $\widetilde{\dpartial{k}_{\mathtt{n}:\alpha_j}} : \prescript{k}{}{\mathbf{B}}^{p,\mathtt{q}}_{\mathtt{n}:\alpha_j;\alpha_0 \cdots \alpha_\ell}  \rightarrow \prescript{k}{}{\mathbf{B}}^{p+1,\mathtt{q}}_{\mathtt{n}:\alpha_j;\alpha_0 \cdots \alpha_\ell}$ (where the term $\prescript{k}{}{\mathbf{B}}^{p,\mathtt{q}}_{\mathtt{n}:\alpha_0 \cdots \alpha_\ell} (\drglue{}_{\mathtt{n}}) \subset \bigoplus_{j=0}^\ell \prescript{k}{}{\mathbf{B}}^{p,\mathtt{q}}_{\mathtt{n}:\alpha_j;\alpha_0 \cdots \alpha_\ell} $ is defined as in Definition \ref{def:total_de_rham_differential} using $\tbva{k}{}^*_{\mathtt{n}:\alpha}[u]$'s). Using similar argument as in Lemma \ref{lem:exactness_of_cech_thom_whitney_complex}, we can see that the bottom horizontal map $\widetilde{\rest{k,0}_{\mathtt{n}}}$ in the above diagram is surjective. Finally, the kernel complex of this map is acyclic by the holomorphic Poincar\'{e} Lemma in Assumption \ref{assum:local_assumption_for_triviality_of_hodge_bundle} and arguments similar to Lemma \ref{lem:exactness_of_cech_thom_whitney_complex}. 

For a general ideal $K \subset \mathbf{m}$, one can argue that $H^*(\totaldr{K}{\parallel}^*)$ is a free $R/K$ module as follows. We consider the sequence of ideals $ \mathbf{m} = J_1 + K \supset J_2 + K \supset \cdots J_l +K = K$ for some $l$. Then one can prove that $H^*(\totaldr{J_{j+1}+K}{\parallel}^*) \rightarrow H^*(\totaldr{J_j+K}{\parallel}^*)$ is surjective by induction on $j$. Details are left to the readers.
\end{proof}

\section{An abstract unobstructedness theorem}\label{sec:abstract_theorem}

Theorem \ref{thm:construction_of_differentials} produces an {\em almost} differential graded Batalin-Vilkovisky (abbrev. dgBV) algebra $(\polyv{}^{*,*},\pdb, \bvd{},\wedge)$ (where ``almost'' means $(\pdb + \bvd{})^2$ is zero only at $0^{\text{th}}$-order), together with an {\em almost} de Rham module $(\totaldr{}{\parallel}^{*,*}, \pdb, \dpartial{},\wedge)$ (where ``almost'' means $(\pdb + \dpartial{})^2$ is zero only at $0^{\text{th}}$-order) and the volume element $\volf{} \in \totaldr{}{\parallel}^{d,0}$. From these we can prove an unobstructedness theorem, using the techniques from \cite{Barannikov99, kontsevichgeneralized, KKP08, terilla2008smoothness}.

\subsection{Solving the Maurer-Cartan equation from the almost dgBV algebra structure}\label{sec:algebraic_solving_maurer_cartan}

We first introduce some notations, following Barannikov \cite{Barannikov99}:

\begin{notation}\label{not:formal_variable_t}
	Let $t$ be a formal variable.
	We consider the spaces of formal power series or Laurent series in $t$ or $t^{\half}$ with values in polyvector fields $$\polyv{k}^{p,q}[[t]],\ \polyv{k}^{p,q}[[t^{\half}]],\ \polyv{k}^{p,q}[[t^{\half},t^{-\half}],$$
	together with a scaling morphism $\mathsf{l}_t : \polyv{k}^{p,q}[[t^{\half},t^{-\half}] \rightarrow \polyv{k}^{p,q}[[t^{\half},t^{-\half}]$ induced by $\mathsf{l}_t ( \varphi) = t^{\frac{q-p-2}{2}} \varphi$ for $\varphi \in \polyv{k}^{p,q}[[t^{\half},t^{-\half}]$. We have the identification 
	$\ptd{k}_t := t^{\half }\mathsf{l}_t^{-1}\circ \ptd{k} \circ \mathsf{l}_t  = \prescript{k}{}{\pdb} + t (\bvd{k}) + t^{-1} ( \dbobs{k} + t (\dvolfobs{k})) \wedge $.
	We also consider spaces of formal power series or Laurent series in $t$ or $t^{\half}$ with values in the relative de Rham module $$\totaldr{k}{\parallel}^{p,q}[[t^{\half}]], \quad \totaldr{k}{\parallel}^{p,q}[[t^{\half},t^{-\half}],$$ 
	together with the rescaling $\mathsf{l}_t : \totaldr{k}{}^{p,q}[[t^{\half},t^{-\half}] \rightarrow \totaldr{k}{}^{p,q}[[t^{\half},t^{-\half}]$ given by $\mathsf{l}_t(\alpha) = t^{\frac{d-p+q-2}{2}} \alpha$ which preserves the filtration on $\totaldr{k}{\bullet}$, and gives $\mathsf{l}_t(\varphi) \lrcorner (\volf{k} ) = \mathsf{l}_t ( \varphi \lrcorner \volf{k})$ and $\drd{k}{}_t:= t^{\half} \mathsf{l}_t^{-1} \circ \drd{k}{} \circ \mathsf{l}_t = \prescript{k}{}{\pdb} + t(\dpartial{k}) +t^{-1}(\dbobs{k} \lrcorner)$.
		
	For the purpose of constructing log Frobenius structures in the next section, we consider a finite-dimensional graded vector space $\mathbb{V}^*$ and the associated graded symmetric algebra $\ecfr{} := \text{Sym}^*(\mathbb{V}^\vee)$, equipped with the maximal ideal $\mathbf{I}$ generated by $\mathbb{V}^\vee$. We will abuse notations by using $\mathbf{m}$ and $\mathbf{I}$ again to denote the respective ideals of $\cfr_{\ecfr{}} := \cfr \otimes_\comp \ecfr{}$, where  $\cfr$ is the coefficient ring introduced in Notation \ref{not:universal_monoid}. We also let $\mathcal{I}:= \mathbf{m} + \mathbf{I}$ be the ideal generated by $\mathbf{m} \otimes \ecfr{}+\cfr \otimes \mathbf{I}$ and write $\cfrk{k}_{\ecfr{}}:= (\cfr_{\ecfr{}}/\mathcal{I}^{k+1})$.
	We write $\polyv{k}_{\ecfr{}}:= \polyv{k}\otimes_\comp \ecfr{} \otimes_{\cfr_{\ecfr{}}} (\cfr_{\ecfr{}} / \mathcal{I}^{k+1})$ and $\totaldr{k}{\parallel}_{\ecfr{}} := \totaldr{k}{\parallel} \otimes_\comp \ecfr{} \otimes_{\cfr_{\ecfr{}}} (\cfr_{\ecfr{}} / \mathcal{I}^{k+1})$, and let $ \polyv{k}_{\ecfr{}}^*[[t]] $, $ \polyv{k}_{\ecfr{}}^*[[t^{\half},t^{-\half}]$ and $ \totaldr{k}{\parallel}_{\ecfr{}}^*[[t^{\half},t^{-\half}] $ be the complexes of formal series or Laurent series in $t^{\half}$ or $t$ with values in those coefficient rings. 
	\end{notation}

\begin{remark}\label{rem:triviality_of_hodge_bundle_with_coefficient}
	We can also define the Hodge bundle $\reallywidehat{H^*(\totaldr{}{\parallel}^*,\drd{}{})} \hat{\otimes} \hat{\ecfr{}}$ over the formal power series ring $\hat{\cfr}_{\ecfr{}}:=\varprojlim_k \cfrk{k}_{\ecfr{}}$, which is equipped with the Gauss-Manin connection $\gmc{}$ defined as in Definition \ref{def:Gauss_Manin_connection}. Then Lemma \ref{lem:triviality_of_hodge_bundle} implies that the Hodge bundle $\reallywidehat{H^*(\totaldr{}{\parallel}^*,\drd{}{})} \hat{\otimes} \hat{\ecfr{}}$ is free over $\hat{\cfr}_{\ecfr{}}$, or equivalently, $H^*(\polyv{k}^*_{\ecfr{}}, \ptd{k})$ is free over $\cfrk{k}_{\ecfr{}}$ for each $k \in \inte_{\geq 0}$. 
\end{remark}

\begin{definition}\label{def:Maurer_Cartan_equation_unobstructedness}
	An element $\prescript{k}{}{\varphi} \in \polyv{k}_{\ecfr{}}^0[[t]]$ with $\prescript{k}{}{\varphi} =  0 \ (\text{mod $\mathbf{m} + \mathbf{I}$})$ is called a {\em Maurer-Cartan element} over $\cfr_{\ecfr{}} / \mathcal{I}^{k+1}$ if it satisfies the {\em Maurer-Cartan equation}
	\begin{equation}\label{eqn:Maurer_Cartan_equation_unobstructedness}
	(\prescript{k}{}{\pdb} + t (\bvd{k})) \prescript{k}{}{\varphi}+ \half [\prescript{k}{}{\varphi},\prescript{k}{}{\varphi}] + (\dbobs{k}+ t(\dvolfobs{k}))= 0,
	\end{equation}
	or equivalently, $(\prescript{k}{}{\pdb} + t (\bvd{k}) + [\prescript{k}{}{\varphi}, \cdot ])^2 = 0$. 
\end{definition}

Notice that the MC equation \eqref{eqn:Maurer_Cartan_equation_unobstructedness} is also equivalent to $\drd{k}{} ( e^{\mathsf{l}_{t}(\prescript{k}{}{\varphi})} \lrcorner \volf{}) = 0 $, which can in turn be rewritten as $$\ptd{k}(e^{\mathsf{l}_t(\prescript{k}{}{\varphi})})  =  (\prescript{k}{}{\pdb} + \bvd{k} + (\dbobs{k} + \dvolfobs{k}) \wedge ) (e^{\mathsf{l}_t(\prescript{k}{}{\varphi})}) = 0.$$ In order to solve \eqref{eqn:Maurer_Cartan_equation_unobstructedness} using algebraic techniques as in \cite{KKP08}, we need Assumption \ref{assum:local_assumption_for_triviality_of_hodge_bundle}, which guarantees freeness of the Hodge bundle, as well as a suitable version of the Hodge-to-de Rham degeneracy; recall that these are also the essential conditions for unobstructedness of smoothing of log smooth Calabi-Yau varieties in \cite{kawamata1994logarithmic}. 

Remark \ref{rem:Gauss_Manin_relation_with_zero_order} said that $H^*(\totaldr{0}{\parallel}, \drd{0}{})$ computes the hypercohomology $\mathbb{H}^*(\rdr{0}{}^*,\dpartial{0})$, so the Hodge filtration $\mathcal{F}^{\geq p}\mathbb{H}^* = H^*(\totaldr{0}{\parallel}^{\geq p,*}, \drd{0}{})$ (where $\drd{0}{} = \pdb + \dpartial{0}$) is induced by the filtration $\mathcal{F}^{\geq p} (\totaldr{0}{\parallel}) :=\totaldr{0}{\parallel}^{\geq p,*}$ on the complex $(\totaldr{0}{\parallel}^{*}, \drd{0}{})$.  

\begin{assum}[Hodge-to-de Rham degeneracy]\label{assum:Hodge_de_rham_degeneracy}
	We assume that the spectral sequence associated to the decreasing filtration $\mathcal{F}^{\geq \bullet} (\totaldr{0}{\parallel})$ degenerates at the $E_1$ term. 
\end{assum}

Assumption \ref{assum:Hodge_de_rham_degeneracy} is equivalent to the condition that $H^*\big(\polyv{0}[[t]], \ptd{0} = \pdb + t (\bvd{0}) \big)$ (or equivalently, that $H^*\big(\totaldr{0}{\parallel}[[t]], \pdb + t (\dpartial{0}) \big)$) is a finite rank free $\comp[[t]]$-module (cf. \cite{KKP08}).

\begin{example}\label{ex:log-smooth-VII}
	For the log smooth case (Example \ref{ex:log-smooth-I}), a cohomological mixed Hodge complex of sheaves $( \mathsf{A}_{\inte},(\mathsf{A}^{*}_{\mathbb{Q}},\mathsf{W}), (\mathsf{A}_{\comp}^{*},\mathsf{W},\mathsf{F}) )$, in the sense of \cite[Definition 3.13]{peters2008mixed}, is constructed in \cite[proof of Lemma 4.1, p.406]{kawamata1994logarithmic}, where
	\begin{align*}
		\mathsf{A}_{\comp}^k & = \bigoplus_{p+q = k} \mathsf{A}^{p,q}_{\comp} := \bigoplus_{p+q=k} \big(\Omega^{p+q+1}_{X^{\dagger}/\comp} / W_{q} \Omega^{p+q+1}_{X^{\dagger}/\comp}\big),\\
		\mathsf{F}^{\geq r} \mathsf{A}^*_{\comp} & := \bigoplus_{p}\bigoplus_{q \geq r}\mathsf{A}^{p,q}_{\comp} \text{ and } \mathsf{W}_{\leq r} A^{p,q}_{\comp} :=  W_{r+2p+1}\Omega^{p+q+1}_{X^{\dagger}/\comp} / W_{q} \Omega^{p+q+1}_{X^{\dagger}/\comp};
	\end{align*}
	here $W_{q} $ refers to the subsheaf with at most $q$ log poles.
	There is a natural quasi-isomorphism $\mu : (\Omega^*_{X^{\dagger}/\logsk{0}},\mathsf{F}^{\geq r}:=\Omega^{\geq r}_{X^{\dagger}/\logsk{0}} ) \rightarrow (\mathsf{A}^*_{\comp},\mathsf{F})$ preserving the Hodge filtration $\mathsf{F}$. Applying \cite[Theorem 3.18]{peters2008mixed} gives a mixed Hodge structure $(\mathbb{H}^*(\mathsf{A}^*_{\inte}) , (\mathbb{H}^*(\mathsf{A}^*_{\mathbb{Q}}), \mathsf{W}), (\mathbb{H}^*(\mathsf{A}^*_{\comp}),\mathsf{W},\mathcal{F}))$, as well as the Hodge-to-de Rham degeneracy, i.e. Assumption \ref{assum:Hodge_de_rham_degeneracy}, as in \cite[proof of Lemma 4.1]{peters2008mixed}.
\end{example}

\begin{theorem}\label{thm:unobstructedness_of_MC_equation}
	Suppose Assumptions \ref{assum:local_assumption_for_triviality_of_hodge_bundle} and \ref{assum:Hodge_de_rham_degeneracy} hold. Then for any degree $0$ element $\psi \in \polyv{0}[[t]] \otimes_\comp (\mathbf{I}/\mathbf{I}^2)$ with $\big(\prescript{0}{}{\pdb} + t(\bvd{0})\big) \psi = 0$, there exists a Maurer-Cartan element $\prescript{k}{}{\varphi} \in \polyv{k}_{\ecfr{}}^0[[t]]$ over $\cfr_{T}/\mathcal{I}^{k+1}$ for each $k \in \inte_{\geq 0}$ such that $\prescript{k+1}{}{\varphi} = \prescript{k}{}{\varphi} \ (\text{mod $\mathcal{I}^{k+1}$})$ and $\prescript{k}{}{\varphi} = \psi \ (\text{mod $\mathbf{m} + \mathbf{I}^2$})$. 
\end{theorem}

\begin{proof}

	We will consider the surjective map $\rest{k+1,k} : \polyv{k+1}^{p,q}[[t]] \rightarrow \polyv{k}^{p,q}[[t]]$ obtained from Corollary \ref{cor:corollary_to_exactness_of_cech_thom_whitney}, and inductively solve for $\prescript{k}{}{\varphi}\in \polyv{k}^{0}_{\ecfr{}}[[t]]$ for each $k \in \inte_{\geq 0}$ so that we have $\rest{k+1,k} \left( \prescript{k+1}{}{\varphi} \right) = \prescript{k}{}{\varphi}$, $\prescript{k}{}{\varphi} = \psi \ (\text{mod $\mathbf{m} + \mathbf{I}^2$})$ and $\prescript{k}{}{\varphi}$ satisfies the Maurer-Cartan equation \eqref{eqn:Maurer_Cartan_equation_unobstructedness} in $\polyv{k}_{\ecfr{}}^0[[t]]$. 

	We begin with $\prescript{0}{}{\varphi} = 0$ and try to solve for $\prescript{1}{}{\varphi}$. As the operator $\ptd{} = \pdb +  \bvd{} +  ( \dbobs{} + \dvolfobs{}) \wedge $ satisfies $\ptd{}^2 = 0$, we have $\ptd{}( \dbobs{} + \dvolfobs{}) =(\pdb +  \bvd{} )( \dbobs{} + \dvolfobs{})  = 0 \ \text{(mod $\mathcal{I}^2$)}$ (where $= 0$ (mod $\mathcal{I}^2$) means being mapped to zero under $\rest{\infty,1}$). Together with the fact that $(\dbobs{} + \dvolfobs{}) = 0 \ \text{(mod $\mathcal{I}$)}$, we see that $[(\dbobs{}+\dvolfobs{})]$ represents a cohomology class in $(\polyv{0}^*, \ptd{0}) \otimes_\comp (\mathcal{I}/\mathcal{I}^2)$. Since $\ptd{}(1) = (\dbobs{} + \dvolfobs{})$, we deduce that $[\dbobs{} + \dvolfobs{}] = 0$ in $(\polyv{1}^{*}_{T}, \ptd{1})$. 
	Now applying Lemma \ref{lem:triviality_of_hodge_bundle} or Remark \ref{rem:triviality_of_hodge_bundle_with_coefficient} (freeness of the Hodge bundle) to $(\totaldr{1}{\dagger}^*, \drd{1}{})$ gives the short exact sequence
	$$
	0\rightarrow H^* (\polyv{0}^*)\otimes_\comp (\mathcal{I}/\mathcal{I}^2)  \rightarrow H^*(\polyv{1}_{T}^*)  \rightarrow H^*(\polyv{0}^*) \rightarrow 0
	$$
	under the identification by the volume element $\volf{}$. We conclude that the class $[\dbobs{} + \dvolfobs{}]$ is zero in $H^* (\polyv{0}^*)\otimes (\mathcal{I}/\mathcal{I}^2)$ which means that $(\dbobs{} + \dvolfobs{}) = (\pdb + \bvd{}) (-\tilde{\zeta}) \ \text{(mod $\mathcal{I}^2$)}$ for some $\tilde{\zeta} \in \polyv{0}^0 \otimes (\mathcal{I}/\mathcal{I}^2)$, and we have $(\dbobs{} + t \dvolfobs{}) = (\pdb + t \bvd{})(-\zeta)$ for some $\zeta  \in \polyv{0}^0[[t,t^{-1}] \otimes (\mathcal{I}/\mathcal{I}^2)$.
	
	Applying Assumption \ref{assum:Hodge_de_rham_degeneracy} and using the technique from \cite[Proof of Theorem 2]{terilla2008smoothness}, we can modify $\zeta$ to satisfy $\zeta\in \polyv{0}^0[[t]] \otimes (\mathcal{I}/\mathcal{I}^2)$ (i.e. removing all the negative powers in $t$), and then we can take $\prescript{1}{}{\varphi}$ to be the image of $\zeta$ in $\polyv{1}^0_{\ecfr{}}[[t]]$. 
	We further observe the Maurer-Cartan element $\prescript{1}{}{\varphi}$ can be modified by adding any $\xi \in \polyv{1}^0_{\ecfr{}}[[t]]$ with $\xi = 0 \ \text{(mod $\mathcal{I}$)}$ and $\ptd{1}_t \xi = 0 \ (\text{mod $\mathcal{I}^2$})$. Therefore we can always achieve $\prescript{1}{}{\varphi} + \xi = \psi \ (\text{mod $\mathbf{m} + \mathbf{I}^2$})$ by choosing a suitable $\xi$ and letting $\prescript{1}{}{\varphi} + \xi$ be the new $\prescript{1}{}{\varphi}$. 

	Suppose $\prescript{k-1}{}{\varphi}$ satisfying the Maurer-Cartan equation $\ptd{} \left( e^{\mathsf{l}_t (\prescript{k-1}{}{\varphi})} \right) = 0 \ \text{(mod $\mathcal{I}^k$)}$ up to order $k-1$ has been constructed. Take an arbitrary lifting $\widetilde{\prescript{k-1}{}{\varphi}}$ in $\polyv{k}^0_T[[t]]$ and let 
	$$\prescript{k}{}{\mathsf{O}} := \ptd{k} \left( e^{\mathsf{l}_t (\widetilde{\prescript{k-1}{}{\varphi}})} \right) = t^{\half}\mathsf{l}_t \left( \ptd{k}_t (e^{\widetilde{\prescript{k-1}{}{\varphi}}/t}) \right) \ \text{(mod $\mathcal{I}^{k+1}$)}.$$
	$[\prescript{k}{}{\mathsf{O}}]$ represents a cohomology class in $(\polyv{0}^1[[t^{\half},t^{-\half}] \otimes (\mathcal{I}^{k}/\mathcal{I}^{k+1}), \ptd{0})$. We again apply Lemma \ref{lem:triviality_of_hodge_bundle} to obtain a short exact sequence
	\begin{multline*}
	0 \to H^*(\polyv{0}^*[[t^\half,t^{-\half}] \otimes (\mathcal{I}^{k}/\mathcal{I}^{k+1})) \rightarrow \\ H^*(\polyv{k}^*[[t^\half,t^{-\half}]) \rightarrow H^*(\polyv{k-1}^*[[t^{\half},t^{-\half}]) \to 0,
	\end{multline*}
	which forces $[\prescript{k}{}{\mathsf{O}}]=0$ as in the initial case. Hence, applying Assumption \ref{assum:Hodge_de_rham_degeneracy} and using the technique from \cite[Proof of Theorem 2]{terilla2008smoothness} again, we can find $\zeta \in \polyv{0}^0_T[[t]] \otimes (\mathcal{I}^{k}/\mathcal{I}^{k+1})$ such that $(\prescript{0}{}{\pdb} + t (\bvd{0}))(-\zeta) = \mathsf{l}_t^{-1} (\prescript{k}{}{\mathsf{O}})$ and then set $\prescript{k}{}{\varphi} := \widetilde{\prescript{k-1}{}{\varphi}} + \zeta$ to solve the equation.
	\end{proof}

\subsection{Homotopy between Maurer-Cartan elements for different sets of gluing morphisms}\label{sec:equivalent_between_maurer_cartan_for_homotopy}

Theorem \ref{thm:unobstructedness_of_MC_equation} is proven for a fixed set of compatible gluing morphisms $\glue{} = \{\glue{k}_{\alpha \beta}\}$. In this subsection, we study how Maurer-Cartan elements for two different sets of compatible gluing morphisms $\glue{}(0) = \{\glue{k}_{\alpha\beta}(0)\}$ and $\glue{}(1) = \{\glue{k}_{\alpha\beta}(1)\}$ are related through a fixed homotopy $\gluehom{} = \{\gluehom{k}_{\alpha\beta}\}$.

We begin by assuming that
the data $\mathfrak{D} = (\mathfrak{D}_{\alpha})_{\alpha} \in \cech{}^0(\twc{}^{-1,1},\gluehom{})$ and $\mathfrak{F} = (\mathfrak{F}_{\alpha})_{\alpha}  \in \cech{}^0(\twc{}^{0,0},\gluehom{})$ for the construction of the operators $\mathbf{D}$ and $\bvd{}$ in Proposition \ref{prop:construction_of_homotopy_thom_whitney_differential}
are related to 
the data $\dbtwist{}_j$ and $\volftwist{}_j$ for the construction of the operators $\pdb_{j}$ and $\bvd{}_{j}$ in Theorem \ref{thm:construction_of_differentials}
by the relations
$$
\mathtt{r}_j^* ( \mathfrak{D}) = \dbtwist{}_j,\quad
\mathtt{r}_j^* (\mathfrak{F}) = \volftwist{}_j
$$
for $j=0,1$,
where $\mathtt{r}_j : \polyv{}^{*,*}(\gluehom{}) \rightarrow \polyv{}^{*,*}(\glue{}(j))$ is the map introduced in Definition \ref{def:homotopy_cech_thom_whitney_complex}. 

\begin{notation}
	Similar to Lemma \ref{lem:proving_globalness_for_construction_of_total_de_rham_differential}, we let $\mathfrak{L}_{\alpha} := \mathbf{D}_{\alpha}(\mathfrak{D}_{\alpha}) + \half[\mathfrak{D}_{\alpha}, \mathfrak{D}_{\alpha}]$ and $\mathfrak{E}_{\alpha}:= \bvd{}_{\alpha}(\mathfrak{D}_{\alpha}) + \mathbf{D}_{\alpha}(\mathfrak{F}_{\alpha}) + [\mathfrak{D}_{\alpha}, \mathfrak{F}_{\alpha}]$; $(\mathfrak{L}_{\alpha})_{\alpha}$ and $(\mathfrak{E}_{\alpha})_{\alpha}$ glue to give global terms $\mathfrak{L} \in \polyv{}^{2,-1}(\gluehom{})$ and $\mathfrak{E} \in \polyv{}^{1,0}(\gluehom{})$ respectively. 

	We set $\breve{\mathcal{D}}:= \mathbf{D} + \bvd{} + (\mathfrak{L} + \mathfrak{E}) \wedge$, which defines an operator acting on $\polyv{}^{*}(\gluehom{})$ (and we will use $\prescript{k}{}{ \breve{\mathcal{D}}}$ to denote the corresponding operator acting on $\polyv{k}^{*}(\gluehom{})$). We have $\breve{\mathcal{D}}^2 = 0$ as in Proposition \ref{prop:checking_d_square_equal_zero}. 

	We introduce a scaling $\mathsf{l}_t : \polyv{k}^{p,q}(\gluehom{})[[t^{\half},t^{-\half}] \rightarrow \polyv{k}^{p,q}(\gluehom{})[[t^{\half},t^{-\half}]$ defined by $\mathsf{l}_t (\varphi) = t^{\frac{q-p-2}{2}} \varphi$ for $\varphi \in \polyv{k}^{p,q}(\gluehom{})$. Then we have the identity $\prescript{k}{}{\breve{\mathcal{D}}}_t:= t^{\half} \mathsf{l}_{t}^{-1} \circ \prescript{k}{}{ \breve{\mathcal{D}}} \circ \mathsf{l}_t = \prescript{k}{}{\mathbf{D}} + t (\bvd{k}) + t^{-1} ( \prescript{k}{}{\mathfrak{L}} + t (\prescript{k}{}{\mathfrak{E}})) \wedge$ as in Notation \ref{not:formal_variable_t}. 

	Similar to Notation \ref{not:formal_variable_t}, we consider the complex $\polyv{k}^{*}_T(\gluehom{})[[t]]$ (or formal power series or Laurent series in $t$ or $t^{\half}$) for any graded ring $\ecfr{} = \comp[\mathbb{V}^*]$.
\end{notation}

\begin{lemma}\label{lem:homotopy_complex_cohomology_lemma}
	 The natural restriction map 
	 $\mathtt{r}_j^* : (\polyv{k}^*(\gluehom{}), \prescript{k}{}{\breve{\mathcal{D}}})  \rightarrow (\polyv{k}^*(\glue{}(j)), \ptd{k})$ is a quasi-isomorphism for $j = 0,1$ and all $k \in \inte_{\geq 0}$. 
\end{lemma}

\begin{proof}
	We will only give a proof of the case $j=0$ because the other case is similar. We first consider the following diagram 
	$$
	\xymatrix@1{
	 \polyv{k-1}^*(\gluehom{}) \otimes_\comp (\mathbf{m}/\mathbf{m}^2) \ar@{^{(}->}[r] \ar[d]^{\mathtt{r}_{0}^*} & \polyv{k}^*(\gluehom{}) \ar[r] \ar[d]^{\mathtt{r}_{0}^*} & \polyv{0}^*(\gluehom{}) \ar[r] \ar[d]^{\mathtt{r}_{0}^*}& 0 \\
  	\polyv{k-1}^*(\glue{}(0)) \otimes_\comp (\mathbf{m}/\mathbf{m}^2) \ar@{^{(}->}[r] & \polyv{k}^*(\glue{}(0)) \ar[r]& \polyv{0}^*(\glue{}(0)) \ar[r]& 0
	}
	$$
	with exact horizontal rows. By passing to the corresponding long exact sequence, we see that it suffices to prove that $\mathtt{r}_0^* :(\polyv{0}^*(\gluehom{}), \prescript{0}{}{\mathbf{D}} + \bvd{0}) \rightarrow (\polyv{0}^*(\glue{}(0)), \prescript{0}{}{\pdb} +  \bvd{0})$ is a quasi-isomorphism. 
	In this case we have $\gluehom{0} = id \ (\text{mod $ \mathbf{m}$})$ and $\glue{0}(0) = id \ (\text{mod $ \mathbf{m}$})$, from which we deduce that $\polyv{0}^*(\gluehom{}) = \mathcal{A}^*(\simplex_1) \otimes_\comp \polyv{0}^*(\glue{}(0))$ in which the operators are related by $\prescript{0}{}{\mathbf{D}} + \bvd{0} = \prescript{0}{}{\pdb} +  \bvd{0}  + d_{\mathtt{s}}$,
	where $s$ is the coordinate function on the $1$-simplex $\simplex_1$ and $d_\mathtt{s}$ is the usual de Rham differential acting on $\mathcal{A}^*(\simplex_1)$.
	The quasi-isomorphism is then obtained using the homotopy operator constructed by integration $\int_{0}^s$ along the $1$-simplex. Details are left to the readers. 
\end{proof}

The following proposition relates Maurer-Cartan elements $\varphi_0$ of the almost dgBV $\polyv{}^*_T(\glue{}(0))[[t]]$ and those of $\polyv{}^*_T(\gluehom{})[[t]]$.

\begin{prop}\label{prop:Maurer-Cartan_element_homotopy_relation}
	Given any Maurer-Cartan element \\$\prescript{k}{}{\varphi}_0 \in \polyv{k}^0_T(\glue{}(0))[[t]]$ as in Theorem \ref{thm:unobstructedness_of_MC_equation}, there exists a lifting $\prescript{k}{}{\varphi} \in \polyv{k}^0_T(\gluehom{})[[t]]$ which is a Maurer-Cartan element for each $k$ such that $\prescript{k+1}{}{\varphi} = \prescript{k}{}{\varphi} \ (\text{mod $\mathcal{I}^{k+1}$})$ and $\mathtt{r}_0^* (\prescript{k}{}{\varphi}) = \prescript{k}{}{\varphi}_0$.
	If there are two liftings $(\prescript{k}{}{\varphi})_k$ and $(\prescript{k}{}{\psi})_k$ of $(\prescript{k}{}{\varphi}_0)_k$, then there exists a gauge element $\prescript{k}{}{\vartheta} \in \polyv{k}^{-1}_T(\gluehom{})[[t]]$ for each $k$ such that $\mathtt{r}_0^*(\prescript{k}{}{\vartheta}) =0$, $\prescript{k+1}{}{\vartheta} = \prescript{k}{}{\vartheta} \ (\text{mod $\mathcal{I}^{k+1}$})$ and $e^{\prescript{k}{}{\vartheta}} \star \prescript{k}{}{\varphi} = \prescript{k}{}{\psi}$. 
\end{prop}

\begin{proof}
	We construct $\prescript{k}{}{\varphi}$ by induction on $k$. Given a Maurer-Cartan element $\prescript{k-1}{}{\varphi} \in \polyv{k-1}_T^0[[t]]$ such that $\mathtt{r}_0^*(\prescript{k-1}{}{\varphi}) = \prescript{k-1}{}{\varphi}_0$, our goal is to construct a lifting $\prescript{k}{}{\varphi}$ of $\prescript{k-1}{}{\varphi}$ with $\mathtt{r}_0^*(\prescript{k}{}{\varphi}) = \prescript{k}{}{\varphi}_0$.
	
	By surjectivity of $\rest{k,k-1}: \polyv{k}_T(\gluehom{})[[t]] \rightarrow \polyv{k-1}_T(\gluehom{})[[t]]$, we get a lifting $\widetilde{\prescript{k-1}{}{\varphi}}$ of $\prescript{k-1}{}{\varphi}$.
	From the surjectivity of $\mathtt{r}_0^* : \polyv{k}^0_T(\gluehom{})[[t]] \rightarrow \polyv{k}^0_T(\glue{}(0))[[t]]$ for any $k$ from Lemma \ref{lem:homotopy_exactness_of_cech_thom_whitney_complex}, we further obtain a lifting $\eta$ of $\prescript{k}{}{\varphi}_0 - \mathtt{r}_0^* (\widetilde{\prescript{k-1}{}{\varphi}})$ such that $\eta = 0 \ (\text{mod $\mathcal{I}^k$})$ in $\polyv{k}^0_T(\gluehom{})[[t]]$.
	Then we set $\widehat{\prescript{k}{}{\varphi}} :=  \widetilde{\prescript{k-1}{}{\varphi}} + \eta$ so that $\mathtt{r}_0^* (\widehat{\prescript{k}{}{\varphi}}) =  \prescript{k}{}{\varphi}_0$.
	Similar to the proof of Theorem \ref{thm:unobstructedness_of_MC_equation}, we define the obstruction class 
	$$
	\prescript{k}{}{\mathsf{O}} := \prescript{k}{}{\breve{\mathcal{D}}}_t ( t e^{\widehat{\prescript{k}{}{\varphi}} / t }) = (\prescript{k}{}{\mathbf{D}} + t (\bvd{k})) \widehat{\prescript{k}{}{\varphi}} + \half [\widehat{\prescript{k}{}{\varphi}},\widehat{\prescript{k}{}{\varphi}}] + ( \prescript{k}{}{\mathfrak{L}} + t (\prescript{k}{}{\mathfrak{E}})) 
	$$ 
 	in $\polyv{k}^1_T[[t]]$ which satisfies $\rest{k,k-1}(\prescript{k}{}{\mathsf{O}} ) = 0$ and $ \prescript{k}{}{\breve{\mathcal{D}}}_t (	\prescript{k}{}{\mathsf{O}} ) = 0$.
 	
 	Considering the short exact sequences
	$$
	\xymatrix@1{
	\mathcal{K}^* \ar@{^{(}->}[r]  \ar@{^{(}->}[d] & (\polyv{k}^*_T(\gluehom{})[[t]], \prescript{k}{}{\breve{\mathcal{D}}}_t ) \ar[r]^{\mathtt{r}_0^*} \ar@{^{(}->}[d]^{\mathsf{l}_t^{-1}} & (\polyv{k}^*_T(\glue{}(0))[[t]], \ptd{k}_t)  \ar[r]  \ar@{^{(}->}[d]^{\mathsf{l}_t^{-1}}  &  0\\
	 \widehat{\mathcal{K}}^* \ar@{^{(}->}[r] & (\polyv{k}^*_T(\gluehom{})[[t^{\half},t^{-\half}], \prescript{k}{}{\breve{\mathcal{D}}} ) \ar[r]^{\mathtt{r}_0^*} & (\polyv{k}^*_T(\glue{}(0))[[t^{\half},t^{-\half}], \ptd{k})  \ar[r] & 0
	},
	$$
	and observing that $(\mathcal{K}^* , \prescript{k}{}{\breve{\mathcal{D}}}_t )$ is acyclic, we conclude that $\prescript{k}{}{\mathsf{O}}  \in \mathcal{K}^1$. Hence we can find $\zeta \in \mathcal{K}^0$ such that $\prescript{k}{}{\breve{\mathcal{D}}}_t ( \zeta ) = \prescript{k}{}{\mathsf{O}} $ and $\zeta= 0 \ (\text{mod $\mathcal{I}^{k}$})$. Then $\prescript{k}{}{\varphi} := \widehat{\prescript{k}{}{\varphi}} + \zeta$ is the desired lifting of $\prescript{k-1}{}{\varphi}$. 
	
	The gauge $(\prescript{k}{}{\vartheta})$ can be constructed by a similar inductive process. Given $\prescript{k-1}{}{\vartheta}$, we need to construct a lifting $\prescript{k}{}{\vartheta} \in \polyv{k}^{-1}_T(\gluehom{})[[t]]$ which serves as a homotopy from $\prescript{k}{}{\varphi}$ to $\prescript{k}{}{\psi}$. Again we take a lifting $\widehat{\prescript{k}{}{\vartheta}}$ satisfying $\rest{k,k-1} (\widehat{\prescript{k}{}{\vartheta}}) = \prescript{k-1}{}{\vartheta}$ and $\mathtt{r}_0^*(\widehat{\prescript{k}{}{\vartheta}}) = 0$, and consider the obstruction
	$$
	\prescript{k}{}{\mathtt{O}} := \prescript{k}{}{\psi} - \exp( [\widehat{\prescript{k}{}{\vartheta}}, \cdot ]) (\prescript{k}{}{\varphi}) + \frac{\exp([ \widehat{\prescript{k}{}{\vartheta}}, \cdot]) -1}{[ \widehat{\prescript{k}{}{\vartheta}}, \cdot]} ((\prescript{k}{}{\mathbf{D}} + t (\bvd{k})) \widehat{\prescript{k}{}{\vartheta}},
	$$
	which satisfies $\rest{k,k-1} (\prescript{k}{}{\mathtt{O}}) = 0$ and $\mathtt{r}_0^* (\prescript{k}{}{\mathtt{O}}) = 0$. We can find $\zeta \in \polyv{0}^{-1}_T[[t]] \otimes (\mathcal{I}^k / \mathcal{I}^{k-1})$ with $\mathtt{r}_0^* ( \zeta) = 0$ such that $ - (\prescript{0}{}{\mathbf{D}} + t ( \bvd{0})) \zeta = \prescript{k}{}{\mathtt{O}}$ and letting $\prescript{k}{}{\vartheta} := \widehat{\prescript{k}{}{\vartheta}} + \zeta$ gives the desired gauge element.
\end{proof}

Given a homotopy $\gluehom{}$, we define a map $\mathtt{F}_{\gluehom{}}$ from the set of Maurer-Cartan elements modulo gauge equivalence with respect to $\glue{}(0)$ to that with respect to $\glue{}(1)$ by $\mathtt{F}_{\gluehom{}} \left( (\prescript{k}{}{\varphi}_0)_k \right) := (\mathtt{r}_1^*(\prescript{k}{}{\varphi}))_k$ with $\prescript{k}{}{\varphi} \in \polyv{k}^0_{\ecfr{}}[[t]]$. Proposition \ref{prop:Maurer-Cartan_element_homotopy_relation} says that this map is well-defined, and its inverse $\mathtt{F}_{\gluehom{}}^{-1}$ is given by reversing the roles of $\glue{}(0)$ and $\glue{}(1)$, so $\mathtt{F}_{\gluehom{}}$ is a bijection.

Next we consider the situation where we have a fixed set of compatible gluing morphisms $\glue{} = \{\glue{k}_{\alpha \beta}\}$ but the complex $\polyv{k}^*$ is equipped with two different choices of operators $\prescript{k}{}{\pdb}$ and $\bvd{k}$, $\prescript{k}{}{\pdb}'$ and $\bvd{k}'$,
whose differences are captured by elements $\mathfrak{v}_1 \in \polyv{}^{-1,1}(\glue{})$ and $\mathfrak{v}_2 \in \polyv{}^{0,0}(\glue{})$, as in Theorem \ref{thm:construction_of_differentials}. 
We write $\mathfrak{v} = \mathfrak{v}_1 + \mathfrak{v}_2$ and consider the complex $\mathcal{A}^*(\simplex_1) \otimes_\comp \polyv{k}^*$ equipped with the differential 
$$
\prescript{k}{}{\breve{\mathsf{D}}} := \ptd{k}+ d_{\simplex_1} +   \mathtt{t}_1 [\mathfrak{v}, \cdot ] + (\mathtt{t}_1 (\prescript{k}{}{\pdb} + \bvd{k}) \mathfrak{v} + \frac{\mathtt{t}_1^2}{2} [\mathfrak{v},\mathfrak{v}] )\wedge  + (d\mathtt{t}_1 \wedge  \mathfrak{v} )\wedge, 
$$
where $\mathtt{t}_1$ is the coordinate function on the $1$-simplex $\simplex_1$ and $d_{\simplex_1}$ is the de Rham differential for $\mathcal{A}^*(\simplex_1)$. We let $\prescript{k}{}{\mathsf{O}}_{\mathtt{t}_1 \mathfrak{v}} : = (\mathtt{t}_1 (\prescript{k}{}{\pdb} + \bvd{k}) \mathfrak{v} + \frac{\mathtt{t}_1^2}{2} [\mathfrak{v},\mathfrak{v}] ) + (\dbobs{k} + \dvolfobs{k})$ and compute 
\begin{align*}
	(\prescript{k}{}{\breve{\mathsf{D}}})^2  
	 =& (\prescript{k}{}{\pdb} + \bvd{k} + \mathtt{t}_1 [\mathfrak{v}, \cdot ]  )^2  - [\prescript{k}{}{\mathsf{O}}_{\mathtt{t}_1 \mathfrak{v}} , \cdot  ] +  \\
	&d\mathtt{t}_1 \wedge \dd{\mathtt{t}_1} \big(\prescript{k}{}{\mathsf{O}}_{\mathtt{t}_1 \mathfrak{v}}  \big) \wedge - d\mathtt{t}_1 \wedge ((\prescript{k}{}{\pdb} + \bvd{k} )(\mathfrak{v})+ \mathtt{t}_1 [\mathfrak{v}, \mathfrak{v} ]  ) \wedge   \\
	 =& [\prescript{k}{}{\mathsf{O}}_{\mathtt{t}_1 \mathfrak{v}} , \cdot  ]  - [\prescript{k}{}{\mathsf{O}}_{\mathtt{t}_1 \mathfrak{v}} , \cdot  ]  + d\mathtt{t}_1 \wedge ((\prescript{k}{}{\pdb} + \bvd{k} )(\mathfrak{v})+ \mathtt{t}_1 [\mathfrak{v}, \mathfrak{v} ]  ) \wedge  \\
	 &-d\mathtt{t}_1 \wedge ((\prescript{k}{}{\pdb} + \bvd{k} )(\mathfrak{v})+ \mathtt{t}_1 [\mathfrak{v}, \mathfrak{v} ]  ) \wedge \\
	 =& 0. \\
\end{align*}
Repeating the argument in this subsection but replacing $(\polyv{k}^*(\gluehom{}), \prescript{k}{}{\breve{\mathcal{D}}})$ by $(\mathcal{A}^* \otimes_\comp \polyv{k}^*, \prescript{k}{}{\breve{\mathsf{D}}})$ and arguing as in the proof of Proposition \ref{prop:Maurer-Cartan_element_homotopy_relation} yields the following:

\begin{prop}\label{prop:Maurer-Cartan_element_homotopy_between_operators}
	Given any Maurer-Cartan element\\ $\prescript{k}{}{\varphi}_0 \in \polyv{k}^0_T[[t]]$ with respect to the operators $ \prescript{k}{}{\pdb}$ and $\bvd{k}$ as in Theorem \ref{thm:unobstructedness_of_MC_equation}, there exists a lifting $\prescript{k}{}{\varphi} \in \mathcal{A}^*(\simplex_1) \otimes \polyv{k}^*_T[[t]]$ which is a Maurer-Cartan element with respect to the operators $(\prescript{k}{}{\pdb} + d_{\simplex_1} + \mathtt{t}_1 [ \mathfrak{v}_1 ,\cdot])$ and $\bvd{k} +[ \mathfrak{v}_2 ,\cdot]$ (meaning that $ \left( (\prescript{k}{}{\pdb} + d_{\simplex_1} + \mathtt{t}_1 [ \mathfrak{v}_1 ,\cdot]) + t (\bvd{k} +[ \mathfrak{v}_2 ,\cdot]) + [\prescript{k}{}{\varphi}, \cdot] \right)^2 = 0$) for each $k$ satisfying $\prescript{k+1}{}{\varphi} = \prescript{k}{}{\varphi} \ (\text{mod $\mathcal{I}^{k+1}$})$ and $\mathtt{r}_0^* (\prescript{k}{}{\varphi}) = \prescript{k}{}{\varphi}_0$.
	If there are two liftings $(\prescript{k}{}{\varphi})_k$ and $(\prescript{k}{}{\psi})_k$ of $(\prescript{k}{}{\varphi}_0)_k$, then there exists a gauge element $\prescript{k}{}{\vartheta} \in \mathcal{A}^*(\simplex_1) \otimes \polyv{k}^{*}_T[[t]]$ for each $k$ such that $\mathtt{r}_0^*(\prescript{k}{}{\vartheta}) =0$, $\prescript{k+1}{}{\vartheta} = \prescript{k}{}{\vartheta} \ (\text{mod $\mathcal{I}^{k+1}$})$ and $e^{\prescript{k}{}{\vartheta}} \star \prescript{k}{}{\varphi} = \prescript{k}{}{\psi}$. 
\end{prop}

Propositions \ref{prop:Maurer-Cartan_element_homotopy_relation} and \ref{prop:Maurer-Cartan_element_homotopy_between_operators} together show that the set of gauge equivalence classes of Maurer-Cartan elements is independent of the choice of the gluing morphisms $\glue{} = \{\glue{k}_{\alpha \beta}\}$ and the choices of the operators $\pdb$ and $\bvd{}$ in the construction of $\polyv{k}_{\ecfr{}}^*[[t]]$. 

\subsection{From Maurer-Cartan elements to geometric \v{C}ech gluings}\label{sec:Maurer_Cartan_element_give_consistent_gluing}
In this subsection, we show that a Maurer-Cartan (MC) element $\varphi = (\prescript{k}{}{\varphi})_{k \in \inte{\geq 0}}$ as defined in Definition \ref{def:Maurer_Cartan_equation_unobstructedness} contains the data for gluing the sheaves $\bva{k}^*_\alpha$'s consistently.

We fix a set of gluing morphisms $\glue{} = \{\glue{k}_{\alpha \beta}\}$ and consider a MC element $\varphi = (\prescript{k}{}{\varphi})_{k \in \inte{\geq 0}}$ (where we take $T = \comp$) obtained in Theorem \ref{thm:unobstructedness_of_MC_equation}. Setting $t=0$, we have the element $\prescript{k}{}{\phi}:=\prescript{k}{}{\varphi}|_{t=0}$ which satisfies the following extended MC equation \eqref{eqn:Maurer_Cartan_equation_without_formal_variable_t}.

\begin{definition}\label{def:extended_Maurer_cartan_and_classical_maurer_Cartan}
	An element $\prescript{k}{}{\phi} \in \polyv{k}^0$ is said to be a {\em Maurer-Cartan element in $\polyv{k}^*$} if it satisfies the extended Maurer-Cartan equation:
	\begin{equation}\label{eqn:Maurer_Cartan_equation_without_formal_variable_t}
	\prescript{k}{}{\pdb} (\prescript{k}{}{\phi}) + \half [\prescript{k}{}{\phi},\prescript{k}{}{\phi} ] + \dbobs{k} = 0.
	\end{equation}
	Note that $(\polyv{k}^{-1,*}[-1], \prescript{k}{}{\pdb}, [\cdot,\cdot])$ forms a dgLa, and an element $\prescript{k}{}{\psi} \in  \polyv{k}^{-1,1}$ is called a {\em classical Maurer-Cartan element} if it satisfies \eqref{eqn:Maurer_Cartan_equation_without_formal_variable_t}. 
\end{definition}

\begin{lemma}\label{lem:from_descendant_maurer_cartan_to_classical_maurer_cartan}
	In the proof of Theorem \ref{thm:unobstructedness_of_MC_equation}, the Maurer-Cartan element $\prescript{k}{}{\varphi} = \prescript{k}{}{\phi}_0 + \prescript{k}{}{\phi}_1 t^1 + \cdots + \prescript{k}{}{\phi}_j t^j + \cdots \in \polyv{k}^0[[t]]$, where $\prescript{k}{}{\phi}_0  = \prescript{k}{}{\psi}_0 + \prescript{k}{}{\psi}_1 + \cdots + \prescript{k}{}{\psi}_d$ 
	with $\prescript{k}{}{\psi}_i \in \polyv{k}^{-i,i}$, can be constructed so that $\prescript{k}{}{\psi}_0 = 0$. In particular, $\prescript{k}{}{\psi}_1 \in \polyv{k}^{-1,1}$ is a classical Maurer-Cartan element. 
\end{lemma}

\begin{proof}
	We prove by induction on $k$.
	Recall from the initial step of the inductive proof of Theorem \ref{thm:unobstructedness_of_MC_equation} that $\prescript{1}{}{\varphi} \in \polyv{1}^0[[t]]$ was constructed so that $(\prescript{1}{}{\pdb} + t (\bvd{1})) (\prescript{1}{}{\varphi}) = \dbobs{1} + t \dvolfobs{1}$. As $\dbobs{1} \in \polyv{1}^{-1,2}$ and $\dvolfobs{1} \in \polyv{1}^{0,1}$, we have $\prescript{1}{}{\pdb} ( \prescript{1}{}{\psi}_0) = 0$. Also, we know $\bvd{1} ( \prescript{1}{}{\psi}_0) = 0$ by degree reasons, so we obtain the equation $(\prescript{1}{}{\pdb} + t (\bvd{1})) (\prescript{1}{}{\varphi} - \prescript{1}{}{\psi}_0) = \dbobs{1} + t \dvolfobs{1}$. Hence we can replace $\prescript{1}{}{\varphi}$ by $\prescript{1}{}{\varphi} - \prescript{1}{}{\psi}_0$ in the construction so that the desired condition is satisfied.
	
	For the induction step, suppose that $\prescript{k-1}{}{\varphi} = \prescript{k-1}{}{\phi}_0 +  \prescript{k-1}{}{\phi}_1 t + \cdots  \in \polyv{k-1}^0$ with $\prescript{k-1}{}{\psi}_0 = 0$ has been constructed. Again recall from the construction in Theorem \ref{thm:unobstructedness_of_MC_equation} that we have solved the equation 
	$$
	(\prescript{k}{}{\pdb} + t ( \bvd{k})) (\hat{\eta}) = \ptd{k}_t \left( t e^{\prescript{k-1}{}{\varphi}/t} \right)
	$$
	for $\hat{\eta} \in \polyv{k}^0[[t]]$. We are only interested in the coefficient of $t^0$ of the component lying in $\polyv{k}^{0,1}$ on the RHS of the above equation, which we denote as $\left\lbrack \ptd{k}_t \left( t e^{\prescript{k-1}{}{\varphi}/t} \right) \right\rbrack_0$. By writing $\prescript{k-1}{}{\phi}_0 = \prescript{k-1}{}{\psi}_1 + \cdots + \prescript{k-1}{}{\psi}_d$ using the induction hypothesis, we have
	\begin{multline*}
	\left\lbrack \ptd{k}_t \left( t e^{\prescript{k-1}{}{\varphi}/t} \right) \right\rbrack_0  =\\  \left\lbrack \left(\prescript{k}{}{\pdb} (\prescript{k-1}{}{\varphi}) + t (\bvd{k}) (\prescript{k-1}{}{\varphi}) + \half[\prescript{k-1}{}{\varphi},\prescript{k-1}{}{\varphi}] + \dbobs{k} +  t (\dvolfobs{k} ) \right) \wedge  \exp( \prescript{k-1}{}{\varphi}/t)  \right\rbrack_0 \\
	 = 0.
	\end{multline*}
	Therefore by writing $\hat{\eta} = \zeta_0 + \zeta_0 t^1 + \cdots $, and $\zeta_0 = \xi_0 + \cdots + \xi_d$ with $\xi_i \in \polyv{k}^{-i,i}$, we conclude that $\prescript{k}{}{\pdb}(\xi_0) = 0$ and hence $(\prescript{k}{}{\pdb}+ t(\bvd{k}))(\xi_0) = 0$. As a result, if we replace $\hat{\eta}$ by $\hat{\eta} - \xi_0$ in the construction, we get the desired element $\prescript{k}{}{\varphi}$ for the induction step.
	
	The second statement follows from the first because $\prescript{k}{}{\phi}_0 = \prescript{k}{}{\psi}_1 + \cdots + \prescript{k}{}{\psi}_d$ satisfies the extended MC equation \eqref{eqn:Maurer_Cartan_equation_without_formal_variable_t}. Then by degree reasons, we conclude that
	$
	\prescript{k}{}{\pdb}(\prescript{k}{}{\psi}_1) + \half [\prescript{k}{}{\psi}_1,\prescript{k}{}{\psi}_1] + \dbobs{k} = 0.
	$
\end{proof}

In view of Lemma \ref{lem:from_descendant_maurer_cartan_to_classical_maurer_cartan}, we restrict ourself to the dgLa $(\polyv{k}^{-1,*}[-1])$ and a classical Maurer-Cartan element $\prescript{k}{}{\psi} \in \polyv{k}^{-1,1}$. We write $\prescript{k}{}{\psi} = (\prescript{k}{}{\psi}_\alpha)_{\alpha}$ where $\prescript{k}{}{\psi}_\alpha \in \twc{k}^{-1,1}_{\alpha;\alpha}$ with regard to the \v{C}ech-Thom-Whitney complexes in Definition \ref{def:cech_thom_whitney_complex}.

Since $V_\alpha$ is Stein and $\bva{k}^*_\alpha$ is a coherent sheaf over $V_\alpha$, we have $H^{>0}(\twc{k}^{p,*}_{\alpha;\alpha}[p],\prescript{k}{}{\pdb}_\alpha)= 0$ for any $p$ (here $[p]$ is the degree shift so that $\twc{k}^{p,0}_{\alpha;\alpha}$ is at degree $0$). In particular, the operator $\prescript{k}{}{\pdb}_\alpha +[\dbtwist{k}_\alpha, \cdot ] + [\prescript{k}{}{\psi}_\alpha,\cdot]$ is gauge equivalent to $\prescript{k}{}{\pdb}_{\alpha}$ via a gauge element $\prescript{k}{}{\vartheta}_\alpha \in \twc{k}^{-1,0}_{\alpha;\alpha}$. As $\prescript{k+1}{}{\psi}_{\alpha} = \prescript{k}{}{\psi}_{\alpha} \ (\text{mod $\mathbf{m}^{k+1}$})$, we can further construct $\prescript{k}{}{\vartheta}_\alpha$ via induction on $k$ so that $\prescript{k+1}{}{\vartheta}_{\alpha} = \prescript{k}{}{\vartheta}_{\alpha} \ (\text{mod $\mathbf{m}^{k+1}$})$.

Given any open subset $W \subset V_{\alpha\beta}$, we use the restrictions $\prescript{k}{}{\vartheta}_{\alpha}|_W \in \twc{k}^{-1,0}(\bva{k}_{\alpha}|_{W})$, $\prescript{k}{}{\vartheta}_{\beta} \in \twc{k}^{-1,0}(\bva{k}_{\beta}|_{W})$ to define an isomorphism $\prescript{k}{}{\mathbf{g}}_{\alpha\beta} : \twc{k}^{*,*}(\bva{k}_{\alpha}|_{W})\rightarrow \twc{k}^{*,*}(\bva{k}_{\beta}|_{W})$ which fits into the following commutative diagram 
$$
\xymatrix@1{
 \twc{k}^{*,*}(\bva{k}_{\alpha}|_{W})\ar[rr]^{\glue{k}_{\alpha\beta}} \ar[d]^{\exp([\prescript{k}{}{\vartheta}_\alpha,\cdot])} & &  \twc{k}^{*,*}(\bva{k}_{\beta}|_{W}) 
 \ar[d]^{\exp([\prescript{k}{}{\vartheta}_\beta,\cdot])}\\
( \twc{k}^{*,*}(\bva{k}_{\alpha}|_{W}), \prescript{k}{}{\pdb}_{\alpha}) \ar[rr]^{\prescript{k}{}{\mathbf{g}}_{\alpha\beta}} & &  (\twc{k}^{*,*}(\bva{k}_{\beta}|_{W}),\prescript{k}{}{\pdb}_{\beta})}
$$
here we emphasis that $\prescript{k}{}{\mathbf{g}}_{\alpha\beta}$ identifies the differentials $\prescript{k}{}{\pdb}_{\alpha}$ and $\prescript{k}{}{\pdb}_{\beta}$.

There is an identification $\bva{k}^p_{\alpha}(W)  = H^0 (\twc{k}^{p,*}(\bva{k}_{\alpha}|_{W})[p],\prescript{k}{}{\pdb}_\alpha)$, enabling us to treat $\prescript{k}{}{\mathbf{g}}_{\alpha\beta} : \bva{k}^*_{\alpha}(W) \rightarrow \bva{k}^*_{\beta}(W)$ as an isomorphism of Gerstenhaber algebras.\footnote{We thank Simon Felten for pointing out that this should be an isomorphism of Gerstenhaber algebras, instead of just an isomorphism of graded Lie algebras.} These isomorphisms can then be put together to give an isomorphism of sheaves of Gerstenhaber algebras $\prescript{k}{}{\mathbf{g}}_{\alpha\beta}  : \bva{k}^*_{\alpha}|_{V_{\alpha\beta}} \rightarrow \bva{k}^*_{\beta}|_{V_{\alpha\beta}}$.
Furthermore, the cocycle condition for the gluing morphisms $\glue{k}_{\alpha\beta}$ (see Definition \ref{def:compatible_gluing_morphism}) implies the cocycle condition $\prescript{k}{}{\mathbf{g}}_{\gamma\alpha} \circ \prescript{k}{}{\mathbf{g}}_{\beta \gamma}\circ \prescript{k}{}{\mathbf{g}}_{\alpha\beta} = \text{id}$.

\begin{definition}\label{def:geometric_gluing_data}
	A set of {\em $k$-th order geometric gluing morphisms} $\prescript{k}{}{\mathbf{g}}$ consists of, for any pair $V_\alpha, V_\beta \in \mathcal{V}$, an isomorphism of sheaves of Gerstenhaber algebras $\prescript{k}{}{\mathbf{g}}_{\alpha\beta} : \bva{k}^*_{\alpha}|_{V_{\alpha\beta}} \rightarrow \bva{k}^*_{\beta}|_{V_{\alpha\beta}}$ satisfying $\prescript{k}{}{\mathbf{g}}_{\alpha\beta} = \text{id} \ (\text{mod $\mathbf{m}$})$, and the cocycle condition $\prescript{k}{}{\mathbf{g}}_{\gamma\alpha} \circ \prescript{k}{}{\mathbf{g}}_{\beta \gamma}\circ \prescript{k}{}{\mathbf{g}}_{\alpha\beta} = \text{id}$. 
 	Two such sets of $k$-th order geometric gluing morphisms $\prescript{k}{}{\mathbf{g}}$ and $\prescript{k}{}{\mathbf{h}}$ are said to be {\em equivalent} if there exists a set of isomorphisms of sheaves of Gerstenhaber algebras $\prescript{k}{}{a}_{\alpha} :\bva{k}^*_{\alpha} \rightarrow \bva{k}^*_{\alpha}$ with $\prescript{k}{}{a}_{\alpha} = \text{id} \ (\text{mod $\mathbf{m}$})$ fitting into the following commutative diagram 
	$$
	\xymatrix@1{
	\bva{k}^*_{\alpha}|_{V_{\alpha\beta}} \ar[rr]^{\prescript{k}{}{\mathbf{g}}_{\alpha\beta}} \ar[d]^{\prescript{k}{}{a}_{\alpha}}& &  \bva{k}^*_{\beta}|_{V_{\alpha\beta}} \ar[d]^{\prescript{k}{}{a}_{\beta}}\\
	\bva{k}^*_{\alpha}|_{V_{\alpha\beta}}\ar[rr]^{\prescript{k}{}{\mathbf{h}}_{\alpha\beta}} & &  \bva{k}^*_{\beta}|_{V_{\alpha\beta}}.
	}
	$$
\end{definition}

If we have two classical Maurer-Cartan elements $\prescript{k}{}{\psi}$ and $\prescript{k}{}{\tilde{\psi}}$ which are gauge equivalent via $\prescript{k}{}{\theta} = (\prescript{k}{}{\theta}_\alpha)_{\alpha}$, then we can construct an isomorphism $\exp(-[\prescript{k}{}{\tilde{\vartheta}}_\alpha,\cdot]) \circ \exp([\prescript{k}{}{\theta}_\alpha,\cdot]) \circ \exp([\prescript{k}{}{\vartheta}_\alpha,\cdot]) : (\twc{k}^{*,*}_{\alpha},\prescript{k}{}{\pdb}_\alpha) \rightarrow (\twc{k}^{*,*}_{\alpha},\prescript{k}{}{\pdb}_\alpha)$ inducing an isomorphism $\prescript{k}{}{a}_{\alpha} :\bva{k}^*_{\alpha}(V_{\alpha}) \rightarrow \bva{k}^*_{\alpha}(V_{\alpha})$ by taking $H^0(\twc{k}^{*,*}_{\alpha},\prescript{k}{}{\pdb}_\alpha)$, so that the two sets of $k$-th order geometric gluing morphisms $\prescript{k}{}{\mathbf{g}}$ and $\prescript{k}{}{\tilde{\mathbf{g}}}$ associated to  $\prescript{k}{}{\psi}$ and $\prescript{k}{}{\tilde{\psi}}$ respectively are equivalent via $\prescript{k}{}{a} = (\prescript{k}{}{a}_{\alpha})_\alpha$. This gives the following:

\begin{prop}\label{prop:Maurer_Cartan_give_consistent_gluing}
	Given classical Maurer-Cartan elements $\prescript{k}{}{\psi} \in \polyv{k}^{-1,1}$ such that $\prescript{k+1}{}{\psi} = \prescript{k}{}{\psi} \ (\text{mod $\mathbf{m}^{k+1}$})$ and $\prescript{k}{}{\psi} = 0 \ (\text{mod $\mathbf{m}$})$, there exists an associated set of geometric gluing morphisms $\prescript{k}{}{\mathbf{g}}$ for each $k$ satisfying $\prescript{k+1}{}{\mathbf{g}} = \prescript{k}{}{\mathbf{g}} \ (\text{mod $\mathbf{m}^{k+1}$})$. 
	For two classical Maurer-Cartan elements $\prescript{k}{}{\psi}$ and $\prescript{k}{}{\tilde{\psi}}$ which are gauge equivalent via $\prescript{k}{}{\theta}$ such that $\prescript{k+1}{}{\theta} = \prescript{k}{}{\theta} \ (\text{mod $\mathbf{m}^{k+1}$})$, there exists an equivalence $\prescript{k}{}{a}$, satisfying $\prescript{k+1}{}{a} = \prescript{k}{}{a} \ (\text{mod $\mathbf{m}^{k+1}$})$ between the associated geometric gluing data $\prescript{k}{}{\mathbf{g}}$ and $\prescript{k}{}{\tilde{\mathbf{g}}}$. 
\end{prop}

Lemma \ref{lem:from_descendant_maurer_cartan_to_classical_maurer_cartan} together with Proposition \ref{prop:Maurer_Cartan_give_consistent_gluing} produces a geometric \v{C}ech gluing of the sheaves $\bva{k}^*_\alpha$'s, unique up to equivalence, from a gauge equivalence class of the MC elements obtained in Theorem \ref{thm:unobstructedness_of_MC_equation}.

In the log smooth case, Example \ref{ex:log-smooth-VI}, Example \ref{ex:log-smooth-VII}, Theorem \ref{thm:unobstructedness_of_MC_equation}, Lemma \ref{lem:from_descendant_maurer_cartan_to_classical_maurer_cartan} and Proposition \ref{prop:Maurer_Cartan_give_consistent_gluing} together imply the following:
\begin{corollary}\label{cor:Kawamata-Namikawa}
	In the log smooth case (i.e. in the setting of Example \ref{ex:log-smooth-I}), the complex analytic space $(X,\mathcal{O}_X)$ is smoothable, i.e. there exists a $k^{\text{th}}$-order thickening $(\prescript{k}{}{X},\prescript{k}{}{\mathcal{O}})$ over $\logsk{k}$ locally modeled on $\prescript{k}{}{\mathbf{V}}_{\alpha}$ (which is $d$-semistable) for each $k \in \inte_{\geq 0}$, and these thickenings are compatible.
\end{corollary}
\section{Abstract semi-infinite variation of Hodge structures}\label{sec:abstract_semi_infinite_VHS}

In this section, we apply techniques developed in \cite{barannikov-kontsevich98, Barannikov99, KKP08, li2013variation} to our abstract framework. Under Assumptions \ref{assum:weighted_filtration_assumption} (existence of opposite filtration) and \ref{assum:poincare_pairing_assumption} (nondegeneracy of pairing), this constructs the structure of a logarithmic Frobenius manifold (introduced in \cite{reichelt2009construction}) on the formal neighborhood of the singular Calabi-Yau variety $X$ in the extended moduli space.

\subsection{Brief review of the relevant structures}\label{sec:logarithmic_semi_infinite_VHS}

\begin{notation}\label{not:abstract_germ_of_moduli_space_and_differentials}
Following Notation \ref{not:formal_variable_t}, let $\hat{\cfr}_{\ecfr{}} : =\varprojlim_{k} \cfrk{k}_{\ecfr{}}$ be the completion of $\cfr \otimes \ecfr{}$. We will abuse notations and use $\mathbf{m}$, $\mathbf{I}$ and $\mathcal{I}$ to denote the corresponding ideals in $\hat{\cfr}_{\ecfr{}}$. As in Notation \ref{not:universal_monoid}, we have a monoid homomorphism $Q \rightarrow \cfrk{k}_{\ecfr{}}$ sending $m \mapsto q^m $ \footnote{Here we abuse notations by writing $q^m$ in place of $q^m \otimes 1$.}, which equips $\cfrk{k}_{\ecfr{}}$ with the structure of a (graded) log ring (whose grading is inherited from that on $\ecfr{} =\text{Sym}^*(\mathbb{V}^\vee)$ associated to graded vector space $\mathbb{V}$). We also denote the corresponding formal germ of log spaces by $\logsck{k}$.

We define the {\em module of log differentials} over $\logsc$ by
$$
\logscdrk{}{l}:= \bigoplus_{l_1+l_2 = l}\logsdrk{}{l_1} \otimes_\comp \ecfr{} \otimes_{\comp} \text{Sym}^{l_2}(\mathbb{V}^\vee[-1]) ,
$$
and write element in $\text{Sym}^*(\mathbb{V}^\vee[-1])$ as $dz_{I}$ for some multi-index $I$. It is equipped with a de Rham differential $d$ satisfying the graded Leibniz rule, the relations $d(q^m) = q^m (d\log q^m)$ for $m \in Q$ and $d(z) = dz$ for $z \in \mathbb{V}^\vee$. We also define the $k^{\text{th}}$-order module of log differentials $\logscdrk{k}{*}$ by quotienting out ideal $\mathcal{J}_k$ generated by $\bigcup_{k_1 + k_2 + k_3 \geq k+1} \mathbf{m}^{k_1} \otimes \mathbf{I}^{k_2} \otimes \text{Sym}^{\geq k_3}(\mathbb{V}^\vee[-1])$. In particular we have $\logscdrk{k}{1}= (\cfrk{k}_{\ecfr{}} \otimes_{\comp}  \blat_\comp[-1] )\oplus (\cfrk{k-1}_{\ecfr{}} \otimes_{\comp}\mathbb{V}^\vee [-1])$ as $\cfrk{k}_{\ecfr{}}$ module. 

We also let $\logscvfk{}:= \cfrk{}_{\ecfr{}} \otimes_{\comp}(\blat_\comp^\vee \oplus \mathbb{V} )[1]$ be the space of log derivations and $\logscvfk{k}:= (\cfrk{k}_{\ecfr{}} \otimes_{\comp} \blat_\comp^\vee[1]) \oplus (\cfrk{k-1}_{\ecfr{}}\otimes \mathbb{V}[1]) $ be the {\em space of $k^{\text{th}}$-order log derivations}, which is equipped with a Lie bracket $[\cdot,\cdot]$ and a natural pairing between $X \in \logscvfk{k}$ and $\alpha \in \logscdrk{k}{1}$. 

Similarly we can talk about $\logscdrf{*}$ and $\logscvff$ by taking inverse limits.
\end{notation}

\begin{definition}\label{def:log_semi_infinite_VHS}
	A {\em log semi-infinite variation of Hodge structure (abbrev. $\frac{\infty}{2}$-LVHS) over $\hat{\cfr}_{\ecfr{}}$} consists of triples $(\prescript{k}{}{\mathcal{H}},\prescript{k}{}{\nabla},\prescript{k}{}{\langle} \cdot,\cdot\rangle )$ for each $k\in \inte_{\geq 0}$ and $\cfrk{k}_{\ecfr{}}$-linear maps $\rest{k,l}: \prescript{k}{}{\mathcal{H}} \rightarrow \prescript{l}{}{\mathcal{H}}$ for $k\geq l$, where 
	\begin{enumerate}
		\item $\prescript{k}{}{\mathcal{H}} = \prescript{k}{}{\mathcal{H}}^*$ is a graded free $\cfrk{k}_{\ecfr{}}[[t]]$ module, called the (sections of the) {\em Hodge bundle};
		\item $\prescript{k}{}{\nabla}$ is the {\em Gauss-Manin (partial) connection} of the form
		\begin{equation}\label{eqn:abstract_GM_connection_definition}
		\prescript{k}{}{\nabla} : \prescript{k}{}{\mathcal{H}} \rightarrow \frac{1}{t} (\logscdrk{k}{1}) \otimes_{\cfrk{k}_{\ecfr{}}} \prescript{k}{}{\mathcal{H}}
		\end{equation}
		which is flat and compatible with the maps $\rest{k,l}$'s;
		\item $\prescript{k}{}{\langle} \cdot,\cdot\rangle: \prescript{k}{}{\mathcal{H}}\times \prescript{k}{}{\mathcal{H}} \rightarrow \cfrk{k}_{\ecfr{}}[[t]][-2d]$ is a degree preserving pairing which is compatible with the maps $\rest{k,l}$'s,
		\end{enumerate}
	satisfying the following conditions (when there is no danger of confusion, we will omit the dependence on $k$ and simply write $\nabla$ and $\langle \cdot,\cdot \rangle$ instead of $\prescript{k}{}{\nabla}$ and $\prescript{k}{}{\langle} \cdot,\cdot\rangle$):
	\begin{enumerate}
		\item $\langle s_1,s_2 \rangle(t) = (-1)^{|s_1||s_2|}\langle s_2,s_1\rangle (-t)$,\footnote{Here $f(-t) \in \cfrk{k}_{\ecfr{}}[[t]]$ is the element obtained from $f(t)$ by substituting $t$ with $-t$.} where $|s_i|$ is the degree of the homogeneous element $s_i$;
		\item $\langle f(t)s_1,s_2\rangle =(-1)^{|s_1||f|} \langle s_1,f(-t)s_2\rangle  = f(t)\langle s_1,s_2\rangle $ for $s_i \in \prescript{k}{}{\mathcal{H}}$ and $f(t) \in \cfrk{k}_{\ecfr{}}[[t]]$;
		\item $\nabla_X \langle s_1,s_2\rangle  = \langle \nabla_X s_1,s_2\rangle + (-1)^{|s_1|(|X|+1)} \langle s_1,\nabla_X s_2\rangle$ for $X \in \logscvfk{k}$;
		\item the induced pairing $g(\cdot,\cdot):(\prescript{k}{}{\mathcal{H}}/t\prescript{k}{}{\mathcal{H}}) \times (\prescript{k}{}{\mathcal{H}}/t\prescript{k}{}{\mathcal{H}}) \rightarrow \cfrk{k}_{\ecfr{}}[-2d]$ is non-degenerate. 
	\end{enumerate}
\end{definition}

\begin{definition}\label{def:grading_on_semi_infinite_VHS}
	Given a $\frac{\infty}{2}$-LVHS $(\prescript{k}{}{\mathcal{H}}^*, \nabla, \langle \cdot,\cdot\rangle)$, a {\em grading structure} is an extension of the Gauss-Manin connection $\gmc{}$ along the $t$-coordinate 
	$$
	\gmc{}_{t \dd{t}} : \prescript{k}{}{\mathcal{H}} \rightarrow t^{-1} (\prescript{k}{}{\mathcal{H}}),
	$$
	which is compatible with the maps $\rest{k,l}$'s and such that $[\gmc{}_X,\gmc{}_{t\dd{t}}] = 0$, i.e. it is a flat connection on $\logsck{k} \times (\comp,0)$. We further require the pairing $\langle \cdot,\cdot \rangle$ to be flat with respect to $\gmc{}_{t \dd{t}}$ in the sense that $t \dd{t} \langle s_1, s_2 \rangle = \langle \gmc{}_{t \dd{t}} s_1 ,s_2 \rangle + \langle s_1, \gmc{}_{t\dd{t}} s_2 \rangle$. 
\end{definition}

\begin{notation}\label{not:abstract_hodge_bundle_inverse_limit}
	Let $\prescript{k}{}{\mathcal{H}}_{\pm} := \prescript{k}{}{\mathcal{H}} \otimes_{\comp[[t]]} \comp[[t,t^{-1}]$ be a module over $\cfrk{k}_{\ecfr{}}[[t,t^{-1}]$ equipped with the natural submodule $\prescript{k}{}{\mathcal{H}}_+ := \prescript{k}{}{\mathcal{H}} \subset \prescript{k}{}{\mathcal{H}}_{\pm}$ which is closed under multiplication by $\cfrk{k}_{\ecfr{}}[[t]]$. There is a natural symplectic structure $\prescript{k}{}{\mathbf{w}}(\cdot,\cdot) : \prescript{k}{}{\mathcal{H}}_{\pm} \times\prescript{k}{}{\mathcal{H}}_{\pm} \rightarrow \cfrk{k}_{\ecfr{}}$ defined by $\prescript{k}{}{\mathbf{w}}(\alpha,\beta) = \text{Res}_{t=0} \langle \alpha,\beta\rangle dt$. 
	
	Also let $\mathcal{H}_{\pm}:= \varprojlim_k \prescript{k}{}{\mathcal{H}}_{\pm} $, which is a module over $\reallywidehat{\cfr_{\ecfr{}}[[t,t^{-1}]}:= \varprojlim_k \cfrk{k}_{\ecfr{}}[[t,t^{-1}]$, equipped with a $\reallywidehat{\cfr_{\ecfr{}}[[t]]}$ submodule $\mathcal{H}_{+}:= \varprojlim_k \prescript{k}{}{\mathcal{H}}_{+}$ and the symplectic structure $\mathbf{w}:= \varprojlim_k \prescript{k}{}{\mathbf{w}}$. We also write $\reallywidehat{\cfr_{\ecfr{}}[t^{-1}]} := \varprojlim_k \cfrk{k}_{\ecfr{}}[t^{-1}]$. 
\end{notation}

\begin{definition}\label{def:opposite_filtration}
	An {\em opposite filtration} is a choice of $\cfrk{k}_{\ecfr{}}[t^{-1}]$ submodule $\prescript{k}{}{\mathcal{H}}_- \subset \prescript{k}{}{\mathcal{H}}_{\pm}$ for each $k\in \inte_{\geq 0}$, compatible with the maps $\rest{k,l}$'s and satisfying the following conditions for each $k$:
	\begin{enumerate}
		\item $\prescript{k}{}{\mathcal{H}}_+ \oplus \prescript{k}{}{\mathcal{H}}_- = \prescript{k}{}{\mathcal{H}}_\pm$;
		\item $\nabla_X \prescript{k}{}{\mathcal{H}}_- \subset \prescript{k}{}{\mathcal{H}}_-$ (resp. $\nabla_X \prescript{k}{}{\mathcal{H}}_- \subset \prescript{k-1}{}{\mathcal{H}}_-$) for $X \in \cfrk{k}_{\ecfr{}} \otimes \blat_{\comp}^{\vee}$ (resp. $X \in \cfrk{k-1}_{\ecfr{}} \otimes \mathbb{V}$); 
		\item $\prescript{k}{}{\mathcal{H}}_-$ is isotropic with respect to the symplectic structure $\prescript{k}{}{\mathbf{w}}(\cdot,\cdot)$;
		\item $\prescript{k}{}{\mathcal{H}}_-$ is preserved by the $\gmc{}_{t \dd{t}}$. 
	\end{enumerate}
	We also write $\mathcal{H}_- := \varprojlim_{k}\prescript{k}{}{\mathcal{H}}_- $. 
\end{definition}

Given an opposite filtration $\mathcal{H}_-$, we have a natural isomorphism (cf. \cite{Gross_book}, \cite{li2013primitive})
\begin{equation}\label{eqn:opposite_filtration_isomorphism}
\prescript{k}{}{\mathcal{H}}_+/t(\prescript{k}{}{\mathcal{H}}_+) \cong \prescript{k}{}{\mathcal{H}}_+ \cap t (\prescript{k}{}{\mathcal{H}}_-) \cong t(\prescript{k}{}{\mathcal{H}}_-)/ \prescript{k}{}{\mathcal{H}}_-,
\end{equation}
for each $k$, giving identifications
\begin{align}
\tau_+ &: \prescript{k}{}{\mathcal{H}}_+ \cap t (\prescript{k}{}{\mathcal{H}}_-) \otimes_{\comp} \comp[[t]] \rightarrow \prescript{k}{}{\mathcal{H}}_+, \label{eqn:abstract_hodge_bundle_plus_trivialization}\\
\tau_- &: \prescript{k}{}{\mathcal{H}}_+ \cap t (\prescript{k}{}{\mathcal{H}}_-) \otimes_{\comp} \comp[t^{-1}] \rightarrow t (\prescript{k}{}{\mathcal{H}}_-) \label{eqn:abstract_hodge_bundle_minus_trivialization}.
\end{align}
Using arguments from \cite{li2013primitive}, we see that
$$\langle \prescript{k}{}{\mathcal{H}}_+ \cap t (\prescript{k}{}{\mathcal{H}}_-) , \prescript{k}{}{\mathcal{H}}_+ \cap t (\prescript{k}{}{\mathcal{H}}_-) \rangle  \in \cfrk{k}_{\ecfr{}} \text{ and } \langle \prescript{k}{}{\mathcal{H}}_-,\prescript{k}{}{\mathcal{H}}_- \rangle \in \cfrk{k}_{\ecfr{}}[t^{-1}]t^{-2}.$$

Morally speaking, a choice of an opposite filtration $\mathcal{H}_-$ gives rise to a (trivial) bundle $\mathbf{H}$ over $ \logscf \times \mathbb{P}^1 $, where $t$ is a coordinate on $\mathbb{P}^1$, as follows. We let the sections of germ of $\prescript{k}{}{\mathbf{H}}$ near $\logsck{k} \times (\comp,0)$ be given by $\prescript{k}{}{\mathcal{H}}_+$, and that of $\prescript{k}{}{\mathbf{H}}$ over $\logsck{k} \times (\mathbb{P}^1 \setminus \{0\})$ be given by $t(\prescript{k}{}{\mathcal{H}}_-)$. Then \eqref{eqn:abstract_hodge_bundle_plus_trivialization} and \eqref{eqn:abstract_hodge_bundle_minus_trivialization} give a trivialization of the bundle $\prescript{k}{}{\mathbf{H}}$ over $\logsck{k} \times \mathbb{P}^1$ whose global sections are $\prescript{k}{}{\mathcal{H}}_+ \cap t (\prescript{k}{}{\mathcal{H}}_-)$.

The pairing $\langle \cdot,\cdot\rangle $ can be extended to $\logscf \times \mathbb{P}^1$ using the trivialization in \eqref{eqn:abstract_hodge_bundle_minus_trivialization}, and $\gmc{}$ extends to give a flat connection on $\logscf \times \mathbb{P}^1$ which preserves the pairing $\langle \cdot,\cdot\rangle $, has an order $2$ irregular singularity at $t=0$ and an order $1$ regular pole at $t= \infty$ (besides those coming from the log structure on $\logsf$). The extended pairing and the extended connection give a so-called {\em (logD-trTLEP(w))-structure} \cite{reichelt2009construction}. Finally, let us recall the notion of a miniversal element from \cite{reichelt2009construction}. 

\begin{definition}\label{def:abstract_primitive_form}
	A {\em miniversal section} $\xi = (\prescript{k}{}{\xi})_{k\in \inte_{\geq 0}}$ is an element $\prescript{k}{}{\xi} \in \prescript{k}{}{\mathcal{H}}_+ \cap t (\prescript{k}{}{\mathcal{H}}_-)$ such that
	\begin{enumerate}
		\item $\prescript{k+1}{}{\xi} = \prescript{k}{}{\xi} \ (\text{mod $\mathcal{I}^{k+1}$})$; 
		\item $\nabla_X \prescript{k}{}{\xi} = 0$ in $t (\prescript{k}{}{\mathcal{H}}_-)/\prescript{k}{}{\mathcal{H}}_-$ (resp. $t (\prescript{k-1}{}{\mathcal{H}}_-)/\prescript{k-1}{}{\mathcal{H}}_-$) for $X \in  \cfrk{k}_{\ecfr{}} \otimes \blat_{\comp}^{\vee}$ (resp. $X \in \cfrk{k-1}_{\ecfr{}} \otimes \mathbb{V}$) for each $k$;
		\item $\nabla_{t\dd{t}} \prescript{k}{}{\xi} = r (\prescript{k}{}{\xi}) $ for each $k$ on $t (\prescript{k}{}{\mathcal{H}}_-)/\prescript{k}{}{\mathcal{H}}_-$, with the same $r \in \comp$;
		\item the Kodaira-Spencer map $KS: \blat_{\comp}^{\vee} \oplus \mathbb{V} \rightarrow \prescript{0}{}{\mathcal{H}}_+/t(\prescript{0}{}{\mathcal{H}}_+)$ given by 
		$
		KS(X):= t\nabla_X \xi 
		$
		is an isomorphism.
	\end{enumerate}
\end{definition} 

By \cite[Proposition 1.11]{reichelt2009construction}, an opposite filtration $\mathcal{H}_-$ together with a miniversal element $\xi$ give the structure of a (germ of a) logarithmic Frobenius manifold.

\subsection{Construction of a $\frac{\infty}{2}$-LVHS}\label{sec:construction_of_semi_infinite_VHS}

Following \cite{barannikov-kontsevich98, Barannikov99, KKP08, li2013variation}, we will construct a $\frac{\infty}{2}$-LVHS from the dgBV algebra $\polyv{k}^{*}_{\ecfr{}}[[t]]$ in Notation \ref{not:formal_variable_t} and its unobstructed deformations.

\begin{condition}\label{cond:choice_of_moduli_coefficient_ring}
	For the $0^{\text{th}}$-order Kodaira-Spencer map $\gmc{0}([\volf{0}]) : \blat^\vee_\comp \rightarrow \mathcal{F}^{\geq d-1}\mathbb{H}^0$ defined after Proposition \ref{prop:0_th_order_griffith_transversality}, we assume that the induced map $\gmc{0}([\volf{0}]) : \blat^\vee_\comp \rightarrow \mathcal{F}^{\geq d-1}\mathbb{H}^0/\mathcal{F}^{\geq d-\half} \mathbb{H}^0$ is injective. We fix the choice of the graded vector space $\mathbb{V}^* := \text{Gr}_{\mathcal{F}}(\mathbb{H}^*) / \text{Im}(\gmc{0}([\volf{0}]))$. 
\end{condition}

For example, in the log smooth case (Example \ref{ex:log-smooth-I}), injectivity of $\gmc{0}([\volf{0}])$ can easily be verified using the cohomological mixed Hodge complex $(\mathsf{A}_{\comp},\mathsf{F},\mathsf{W})$ in Example \ref{ex:log-smooth-VII}.

\begin{notation}\label{not:grading_of_element}
From Lemma \ref{lem:triviality_of_hodge_bundle} and Remark \ref{rem:triviality_of_hodge_bundle_with_coefficient}, we define the {\em relative de Rham complex with coefficient in $\ecfr{}$} as $\totaldr{k}{\parallel}_{\ecfr{}}^* := (\totaldr{k}{\parallel} \otimes_{\comp} \ecfr{}) \otimes_{\cfr_{\ecfr{}}} (\cfr_{\ecfr{}}/\mathcal{I}^{k+1})$ and, for each $k$, consider $H^*(\totaldr{k}{\parallel}_{\ecfr{}})[[t^{\half},t^{-\half}]$  which is free over $\cfrk{k}_{\ecfr{}}[[t^{\half},t^{-\half}]$.

Since the ring $\ecfr{}$ is itself graded, for an element $ \varphi \in (\polyv{k}^{p,q} \otimes \ecfr{}) / \mathcal{I}^{k+1} \subset \polyv{k}_{\ecfr{}}$ (resp. $ \alpha \in (\totaldr{k}{\parallel}^{p,q} \otimes \ecfr{}) / \mathcal{I}^{k+1} \subset \totaldr{k}{\parallel}_{\ecfr{}}$), we define the {\em index} of $\varphi$ ($\alpha$ resp.) as $p+q$ and denoted by $\bar{\varphi}$ (resp. $\bar{\alpha}$). 
\end{notation}

\subsubsection{Construction of $\mathcal{H}_+$}

\begin{definition}\label{def:construction_of_mathcalH_+}
	We consider the scaling morphism 
	$$
	\mathsf{l}_t : H^*(\polyv{k}_T[[t,t^{-1}], \ptd{k}_t) \rightarrow H^*(\totaldr{k}{\parallel}_{\ecfr{}})[[t^{\half},t^{-\half}], 
	$$
	and define (the sections of) the Hodge bundle over $\logsk{0}$ to be $\prescript{0}{}{\mathcal{H}}_+ := \mathsf{l}_t (H^*(\polyv{0}[[t]], \ptd{0}_t))$, as a submodule of $\cfrk{0}_{\ecfr{}}[[t]] = \comp[[t]]$. 
	We further take 
\begin{align*}
	\prescript{k}{}{\mathcal{H}}_{\pm} &:= \text{Im}(\mathsf{l}_t)[t^{-1}] \\
	&= \bigoplus \left( H^{d+ev}(\totaldr{k}{\parallel})[[t,t^{-1}] \oplus H^{d+odd}(\totaldr{k}{\parallel})[[t,t^{-1}] t^{\half} \right) \otimes_{\cfrk{k}} \cfrk{k}_{\ecfr{}}
\end{align*}
	(resp. $\mathcal{H}_\pm :=\varprojlim_k \prescript{k}{}{\mathcal{H}}_{\pm}$), as a module over $\cfrk{k}_{\ecfr{}}[[t,t^{-1}]$ (resp. $\reallywidehat{\cfr_{\ecfr{}}[[t,t^{-1}]}$). 

	Given a Maurer-Cartan element $\varphi = (\prescript{k}{}{\varphi})_k$ as in Theorem \ref{thm:unobstructedness_of_MC_equation}, note that the cohomology $H^*(\polyv{k}_{\ecfr{}}[[t]], \prescript{k}{}{\pdb} + t (\bvd{k}) + [\prescript{k}{}{\varphi},\cdot])$ is again a free module over $\cfrk{k}_{\ecfr{}}[[t]]$. We define 
	$$
	\prescript{k}{}{\mathcal{H}}_+ :=\{  ( \mathsf{l}_{t} (\alpha) \wedge e^{\mathsf{l}_t(\prescript{k}{}{\varphi})} ) \lrcorner \volf{k} \ | \ \alpha \in H^*(\polyv{k}_{\ecfr{}}[[t]], \prescript{k}{}{\pdb} + t (\bvd{k}) + [\prescript{k}{}{\varphi},\cdot]) \}
	$$
	as the $\cfrk{k}_{\ecfr{}}[[t]]$ submodule of $\prescript{k}{}{\mathcal{H}}_\pm$. Similarly, we let $\mathcal{H}_+:= \varprojlim_k \prescript{k}{}{\mathcal{H}}_+$ be the $\hat{\cfr}_{\ecfr{}}[[t]]$ module. 
\end{definition}

\begin{remark}\label{rem:hodge_submodule_related_to_hodge_filtration}
	Notice that \begin{align*}
	\prescript{0}{}{\mathcal{H}}^{d + ev}_+ &= \bigoplus_{r = 0}^{d} \left( \mathcal{F}^{\geq r} \cap H^{d+ev}(\totaldr{0}{\parallel}) \right) \comp[[t]]t^{-r + d-1}\\
	\prescript{0}{}{\mathcal{H}}^{d + odd}_+ &= \bigoplus_{r = 0}^{d-1} \left( \mathcal{F}^{\geq r + \half} \cap H^{d+odd}(\totaldr{0}{\parallel}) \right)  \comp[[t]] t^{-(r+\half) + d-1}
	\end{align*} in relation to the Hodge filtration given in Definition \ref{def:hodge_filtration}, and hence $\prescript{0}{}{\mathcal{H}}/t(\prescript{0}{}{\mathcal{H}}) = \text{Gr}_{\mathcal{F}}(H^*(\totaldr{0}{\parallel}))$ as vector spaces. 
\end{remark}

We define the Gauss-Manin connection by taking $\gmc{k}_X$ as in Definition \ref{def:Gauss_Manin_connection} for $X \in \cfrk{k}_{\ecfr{}} \otimes \blat_{\comp}^{\vee}$, and extend it to $\logscvfk{k}$ by setting $\gmc{k}_{f\dd{z}}= f\dd{z}$ for $f\dd{z} \in \cfrk{k-1}_{\ecfr{}}\otimes  \mathbb{V}$. 

\begin{lemma}\label{lem:Hodge_filtration_preserved_by_Gauss_manin}
	The $\cfrk{k}_{\ecfr{}}[[t]]$ submodule $\prescript{k}{}{\mathcal{H}}_+$ is preserved under the operation $t (\gmc{k}_{X})$ for any $X \in \logscvfk{k}$. Therefore, we obtain $\gmc{}: \prescript{k}{}{\mathcal{H}}_+ \rightarrow \frac{1}{t} (\logscdrk{k}{1}) \otimes_{\cfrk{k}_{\ecfr{}}} \prescript{k}{}{\mathcal{H}}_+$. 
\end{lemma}

\begin{proof}
	It suffices to prove the first statement of lemma. We begin by considering the case $\alpha = 1$ and restricting the Maurer-Cartan element $\prescript{k}{}{\varphi}$ to the coefficient ring $\cfrk{k}$ (because the extra coefficient $\ecfr{}$ is not involved in the differential $\drd{k}{}$ for defining the Gauss-Manin connection in Definition \ref{def:Gauss_Manin_connection}). 
	Note that $\mathsf{l}_t(1) \wedge e^{\mathsf{l}_t(\prescript{k}{}{\varphi})} \lrcorner (\volf{k})  = \mathsf{l}_t (  e^{\prescript{k}{}{\varphi}/t} \lrcorner \volf{k})$. Take a lifting $\mathsf{w} \in \totaldr{k}{0}^{d,0}/\totaldr{k}{2}^{d,0}$ of the element $\volf{k}$ for computing the connection $\gmc{k}$ via the sequence \eqref{eqn:short_exact_sequence_for_GM_connection}. Direct computation shows that 
	\begin{align*} &\drd{k}{} \left( \mathsf{l}_t (e^{\prescript{k}{}{\varphi}/t} \lrcorner \mathsf{w}) \right)  = t^{-\half} \mathsf{l}_t \left( \sum_{i} d\log q^{m_i} \otimes  e^{\prescript{k}{}{\varphi}/t} \lrcorner  (\psi_i \lrcorner \volf{k}) \right) \\
	=&\sum_i d\log q^{m_i} \otimes t^{-1} (\mathsf{l}_t ( \psi_i ) \wedge e^{\mathsf{l}_t (\prescript{k}{}{\varphi})}) \lrcorner \volf{k}
	\end{align*} for some $m_i \in \blat$ and $\psi_i \in \polyv{k}^*[[t]]$. Since $(\mathsf{l}_t ( \psi_i ) \wedge e^{\mathsf{l}_t (\prescript{k}{}{\varphi})}) \lrcorner \volf{k} \in \prescript{k}{}{\mathcal{H}}_+$, we have $\gmc{k} \left(e^{\mathsf{l}_t(\prescript{k}{}{\varphi})} \lrcorner (\volf{k}) \right) \in \frac{1}{t} (\logsdrk{k}{1}) \otimes_{(\cfrk{k})} (\prescript{k}{}{\mathcal{H}}_+)$. 
	For the case of $\left( \mathsf{l}_t(\alpha) \wedge e^{\mathsf{l}_t (\prescript{k}{}{\varphi})}\right) \lrcorner \volf{k}$, we may simply introduce a formal parameter $\epsilon$ of degree $-|\alpha|$ such that $\epsilon^2 = 0$, and repeat the above argument for the Maurer-Cartan element $\prescript{k}{}{\varphi} + \epsilon \alpha$ over the ring $\cfrk{k}[\epsilon]/(\epsilon^2)$. 
\end{proof}

\subsubsection{Construction of $\mathcal{H}_-$}
The $0^{\text{th}}$-order Gauss-Manin connection in Definition \ref{def:0_th_order_GM_connection} induces an endomorphism $N_{\nu} := \gmc{0}_{\nu} :  H^*(\totaldr{0}{\parallel}) \rightarrow H^*(\totaldr{0}{\parallel})$ for every element $\nu \in \blat^{\vee}$. The flatness of the Gauss-Manin connection (Proposition \ref{prop:flatness_of_gauss_manin}) then implies that these are commuting operators: $N_{\nu_1} N_{\nu_2} = N_{\nu_2} N_{\nu_1}$.

\begin{assum}\label{assum:weighted_filtration_assumption}
	There is an increasing filtration $\mathcal{W}_{\leq \bullet} H^*(\totaldr{0}{\parallel})$ 
	$$
	\{0\} \subset \mathcal{W}_{\leq 0 } \subset \cdots  \subset \mathcal{W}_{\leq r} \subset \cdots \subset \mathcal{W}_{d} = H^*(\totaldr{0}{\parallel})
	$$
	indexed by $r \in \frac{1}{2}\inte_{\geq 0}$,\footnote{We follow Barannikov \cite{Barannikov99} in using half integers $r \in \frac{1}{2}\inte$ as the weights for the filtration $\mathcal{W}_{\leq r}$; multiplying by $2$ gives the usual indices $\mathcal{W}_{\leq0} \subset \cdots \subset \mathcal{W}_{\leq r} \cdots \subset \mathcal{W}_{\leq 2d}$ with $r \in \inte$.} which is
	\begin{itemize} 
		\item
		preserved by the $0^{\text{th}}$-order Gauss-Manin connection in the sense that $N_{\nu} \mathcal{W}_{r} \subset \mathcal{W}_{r-1}$ for any $\nu \in \blat^{\vee}$, and
		\item
		an opposite filtration to the Hodge filtration $\mathcal{F}^{\geq \bullet}$ in the sense that $\mathcal{F}^{\geq r} \oplus \mathcal{W}_{\leq r-\half} = H^*(\totaldr{0}{\parallel})$.
	\end{itemize}
\end{assum}

\begin{example}\label{ex:log-smooth-VIII}
	In the log smooth case (continuing Example \ref{ex:log-smooth-VII}), we have Deligne's splitting $\mathbb{H}^l(\mathsf{A}^*_{\comp}) = \bigoplus_{s,t} I^{s,t}$ on each $\mathbb{H}^l(\mathsf{A}^*_{\comp})$ satisfying $\mathsf{W}_{\leq r - l} = \bigoplus_{s+t \leq r} I^{s,t}$ and $\mathsf{F}^{\geq r} = \bigoplus_{s \geq r} I^{s,t}$. Since the nilpotent operator $N_\nu$ is defined over $\mathbb{Q}$ such that $N_{\nu} \mathsf{W}_{\leq r} \subset \mathsf{W}_{\leq r-2}$, we deduce that $N_{\nu}I^{s,t} \subset  I^{s-1,t-1}$. As the Hodge filtration $\mathcal{F}^{\geq \bullet} \mathbb{H}^l(\mathsf{A}^*_{\comp})$ in Definition \ref{def:hodge_filtration} is related to $\mathsf{F}^{\geq \bullet}\mathbb{H}^l(\mathsf{A}^*_{\comp})$ by a shift: $\mathcal{F}^{\geq r -\frac{l-d}{2} } = \mathsf{F}^{ \geq r }$, letting $\mathcal{W}_{\leq r - \frac{l-d}{2}} := \bigoplus_{s \leq r} I^{s,t}$ and $\mathcal{W}_{\leq r}(\mathbb{H}^*(\mathsf{A}^*_{\comp})):= \bigoplus_{l} \mathcal{W}_{\leq r}(\mathbb{H}^l(\mathsf{A}^*_{\comp}))$ gives an opposite filtration satisfying Assumption \ref{assum:weighted_filtration_assumption}. 
\end{example}

Under Assumption \ref{assum:weighted_filtration_assumption}, the commuting operators $N_{\nu}$'s are nilpotent. 

\begin{lemma}\label{lem:existence_of_deligne_basis}
	Under Assumption \ref{assum:weighted_filtration_assumption}, there exists an index (introduced in Notation \ref{not:grading_of_element}) and degree preserving trivialization 
	$$
	\varkappa : H^*(\totaldr{0}{\parallel}) \otimes_\comp \hat{\cfr}_{\ecfr{}} \rightarrow \reallywidehat{H^*(\totaldr{}{\parallel}_{\ecfr{}})}
	$$
	which identifies the connection form of the Gauss-Manin connection $\gmc{}$ with the nilpotent operator $N$, i.e. for any $\nu \in \blat_\comp^{\vee}$, we have $\gmc{}_\nu (s \otimes 1) = N_{\nu} (s)\otimes 1$ for $s \in H^*(\totaldr{0}{\parallel})$.
\end{lemma}

\begin{proof}
	Since the extra coefficient ring $\ecfr{}$ does not couple with the differential $\drd{}{}$, we only need to construct inductively a trivialization $\prescript{k}{}{\varkappa} : H^*(\totaldr{0}{\parallel}) \otimes_\comp \cfrk{k} \rightarrow H^*(\totaldr{k}{\parallel})$ for every $k$ such that $\prescript{k+1}{}{\varkappa} = \prescript{k}{}{\varkappa} \ (\text{mod \ $\mathbf{m}^{k+1}$})$ and which identifies $\gmc{k}$ with the nilpotent operator $N$. 
	
	To prove the induction step, we assume that $\prescript{k-1}{}{\varkappa}$ has been constructed and the aim is to construct its lifting $\prescript{k}{}{\varkappa}$. We first choose an arbitary lifting $\widetilde{\prescript{k}{}{\varkappa}}$ and a {\em filtered basis} $e_1, \dots , e_m$ of the finite dimensional vector space $H^*(\totaldr{0}{\parallel})$, meaning that it is a lifting of a basis in the associated quotient $\text{Gr}_{\mathcal{W}}(H^*(\totaldr{0}{\parallel}))$. We also write $\tilde{e}_i$ for $\widetilde{\prescript{k}{}{\varkappa}}(e_i \otimes 1)$. With respect to the frame $\tilde{e}_i$'s of $H^*(\totaldr{k}{\parallel})$, we define a connection $\widehat{\nabla}$ with $\widehat{\nabla}_{\nu} (\tilde{e}_i) =\sum_j (N_{\nu})_i^j (\tilde{e}_j)$ for $\nu \in \blat_\comp^{\vee}$, where $(N_\nu)_i^j$'s are the matrix coefficients of the operator $N_\nu$ with respect to the basis $\{e_i\}$. We may also treat $N = (N_{i}^j)$ as $\blat_\comp$-valued endomorphims on $H^*(\totaldr{0}{\parallel})$.  
	
	From the induction hypothesis, we have $\gmc{k} - \widehat{\nabla} = 0 \ (\text{mod \ $\mathbf{m}^{k}$})$ and hence $(\gmc{k} - \widehat{\nabla}) (\tilde{e}_i) = \sum_{m} \sum_{j} \alpha_{mi}^j  e_j  q^m    \in \logsdrk{0}{1} \otimes_\comp H^*(\totaldr{0}{\parallel}) \otimes_{\comp} (\mathbf{m}^{k}/\mathbf{m}^{k+1})$. From the flatness of both $\gmc{k}$ and $\widehat{\nabla}$, we notice that $(d \log q^m) \wedge \alpha_{mi}^j = 0$ and hence $\alpha_{mi}^j = c_{mi}^j d \log q^m$ for some constant $c_{mi}^j \in \comp$ for every $m$ and $j$. We will use $c_m$ to denote the endomorphism on $H^*(\totaldr{0}{\parallel})$ whose matrix coefficients are given by $c_m = (c_{mi}^j)$ with respect to the basis $e_i$'s.
	
	As a result, if we define a new frame $\tilde{e}^{(0)}_i:= \tilde{e}_i - \sum_{m,j}  c_{mi}^j  e_j  q^m$ and a new connection $\widehat{\nabla}^{(0)}$ by  $\widehat{\nabla}^{(0)}_{\nu}(\tilde{e}_i^{(0)}) = \sum_j (N_{\nu})_i^j (\tilde{e}_j^{(0)})$, then
	$$
	(\gmc{k} - \widehat{\nabla}^{(0)}) (\tilde{e}_i^{(0)}) = \sum_{j,m} [c_m,N]_{i}^j (e_j) q^m,
	$$
	where $[c_m,N]$ is the usual Lie bracket with its $\blat_\comp$-valued matrix coefficient given by $([c_m,N]_i^j) = c_{ml}^{j}N_{i}^l - c_{mi}^l N_{l}^j$. Once again using flatness of both $\gmc{k}$ and $ \widehat{\nabla}^{(0)}$, we get some constant $c^{(1)j}_{mi}$ such that $[c_m,N]_i^j = c^{(1)j}_{mi} d \log q^m$. Taking an element $\nu_m \in \blat_\comp^\vee$ with $(m,\nu_m) \neq 0$, we obtain $c^{(1)}_{m}  = \frac{1}{(m,\nu_m)} [c_m,N_{\nu_m}]$. Now if we define a new frame $\tilde{e}^{(1)}_i:= \tilde{e}^{(0)}_i - \sum_{j,m} c^{(1)j}_{mi} e^i_j q^m$ and a new connection $\nabla^{(1)}(\tilde{e}^{(1)}_i) := \sum_j N^j_i \tilde{e}^{(1)}_j$, then we have $c^{(2)}_{m} = \frac{1}{(m,\nu_m)} [c_m^{(1)},N_{\nu_m}] = \frac{1}{(m,\nu_m)^2} [[c_m,N_{\nu_m}],N_{\nu_m}]$ such that $c^{(2)}_{m} d \log q^m = [c_m^{(1)},N_{\nu_m}] $. 
	
	Repeating this process we get a frame $\{\tilde{e}^{(2d)}_i\}$. Setting $\nabla^{(2d)} (\tilde{e}^{(2d)}_i) =\sum_j N_i^j \tilde{e}^{(2d)}_j$, we get $c_m^{(2d+1)} = \frac{1}{(m,\nu_m)^{2d+1}} (-[N_{\nu_m},\cdot])^{2d+1} (c_m) = 0$. Therefore letting $\prescript{k}{}{\varkappa}(e_i \otimes 1) = \tilde{e}^{(2d)}_i$ gives the desired trivialization for the Hodge bundle. 
\end{proof}

With Assumption \ref{assum:weighted_filtration_assumption}, we can take a filtered basis $\{e_{r;i}\}_{\substack{0 \leq 2r \leq 2d \\ 0\leq i \leq m_r}}$ of the vector space $H^*(\totaldr{0}{\parallel})$ such that $e_{r;i} \in \mathcal{W}_{\leq r} \cap (H^{d+ev}(\totaldr{0}{\parallel}))$ if $r \in \inte$ and $e_{r;i} \in \mathcal{W}_{\leq r} \cap (H^{d+odd}(\totaldr{0}{\parallel}))$ if $r \in \inte + \half$, and $\{e_{r;i}\}_{0\leq i\leq m_r}$ forms a basis of $\mathcal{W}_{\leq r}/\mathcal{W}_{\leq r-\half}$. 

\begin{definition}\label{def:elementary_sections}
	Using the trivialization $\varkappa$ in Lemma \ref{lem:existence_of_deligne_basis}, we let $\mathfrak{e}_{r;i}:= \varkappa(e_{r;i} \otimes 1)$, as a section of the Hodge bundle $\reallywidehat{H^*(\totaldr{0}{\parallel}_{\ecfr{}})}$. The collection $\{\mathfrak{e}_{r;i}\}$, which forms a frame of the Hodge bundle, is called {\em the set of elementary sections} (cf. Deligne's canonical extension \cite{deligne1997local}).
\end{definition}

Note that the index of $\mathfrak{e}_{r;i}$, introduced in Notation \ref{not:grading_of_element}, is the same as that of $e_{r;i}$.

\begin{lemma}\label{lem:uniqueness_of_opposite_filtration}
	If we let 
	\begin{align*}
	\prescript{0}{}{\mathcal{H}}_{-}^{d+ev} & := \bigoplus_{r=0}^d \left(\mathcal{W}_{\leq r} \cap H^{d+ev}(\totaldr{0}{\parallel}) \right) \comp[t^{-1}]t^{-r +d-2} \subset \prescript{0}{}{\mathcal{H}}_{\pm}^{d+ev},\\
	\prescript{0}{}{\mathcal{H}}_{-}^{d+odd} & := \bigoplus_{r=0}^{d-1} \left(\mathcal{W}_{\leq r +\half} \cap H^{d+odd}(\totaldr{0}{\parallel}) \right) \comp[t^{-1}]t^{-(r+\half) +d-2} \subset \prescript{0}{}{\mathcal{H}}_{\pm}^{d+odd},
	\end{align*}
	and use $\prescript{0}{}{\mathcal{H}}_- :=\prescript{0}{}{\mathcal{H}}_{-}^{d+ev} \oplus \prescript{0}{}{\mathcal{H}}_{-}^{d+odd} $ as the $\comp[t^{-1}]$ submodule of $\prescript{0}{}{\mathcal{H}}_{\pm} $, then there exists a unique free $\reallywidehat{\cfr_{\ecfr{}}[t^{-1}]}$ submodule $\mathcal{H}_{-} = \varprojlim_k \prescript{k}{}{\mathcal{H}}_-$ of $\mathcal{H}_{\pm}$, which is preserved by $\gmc{}_X$ for any $X \in \logscvff$ and satisfies $\prescript{k+1}{}{\mathcal{H}}_{-}  = \prescript{k}{}{\mathcal{H}}_- \ (\text{mod $\mathcal{I}^{k+1}$})$. 
\end{lemma}

\begin{proof}
	For existence, we take a set of elementary sections $\{\mathfrak{e}_{r;i}\}_{\substack{0 \leq 2r \leq 2d \\ 0\leq i \leq m_r}}$ as in Definition \ref{def:elementary_sections}. We can take the free submodules $\prescript{k}{}{\mathcal{W}}_{\leq r}^{d+ev} = \bigoplus_{l\leq r} \bigoplus_{0 \leq i \leq m_r}\cfrk{k}_{\ecfr{}} \cdot \mathfrak{e}_{l;i}$, $\prescript{k}{}{\mathcal{W}}_{\leq r+\half}^{d+odd} = \bigoplus_{l\leq r} \bigoplus_{0\leq i \leq m_{r+\half}} \cfrk{k}_{\ecfr{}} \cdot \mathfrak{e}_{l+\half;i}$ of $H^*(\totaldr{k}{\parallel}_{\ecfr{}})$ and let 
	$$
	\prescript{k}{}{\mathcal{H}}_- := \bigoplus_{0\leq r \leq d} \prescript{k}{}{\mathcal{W}}_{\leq r}^{d+ev}   \comp[t^{-1}]t^{-r +d-2}
	\oplus  \bigoplus_{0\leq r \leq d-1} \prescript{k}{}{\mathcal{W}}_{\leq r+\half}^{d+odd} \comp[t^{-1}]  t^{-(r+\half) +d-2} 
	$$
	be the desired $\gmc{}_X$ invariant subspace. 
	
	For uniqueness, we will prove the uniqueness of $\cfrk{k}_{\ecfr{}}[t^{-1}]$ submodule $\prescript{k}{}{\mathcal{H}}_{-}$ of $\prescript{k}{}{\mathcal{H}}_{\pm}$ for each $k$ by induction. Again since the coefficient ring $\ecfr{}$ does not couple with the differential $\drd{}{}$, we only need to consider the corresponding statement for Hodge bundle over $\cfrk{k}$. For the induction step we assume that there is another increasing filtration $\prescript{k}{}{\widetilde{\mathcal{W}}}_{\leq \bullet}$ of $H^*(\totaldr{k}{\parallel})$ satisfying the desired properties such that they agree when passing to $H^*(\totaldr{k-1}{\parallel})$. We should prove that $\prescript{k}{}{\widetilde{\mathcal{W}}}_{\leq r} \subset \prescript{k}{}{\mathcal{W}}_{\leq r}$ for each $r$. 
	
	We take $r\in \frac{1}{2}\inte$ with $\prescript{k}{}{\widetilde{\mathcal{W}}}_{\leq r} \neq 0$. The proof of Lemma \ref{lem:existence_of_deligne_basis} gives a trivialization $\prescript{0}{}{\widetilde{\mathcal{W}}}_{ \leq r} \otimes_\comp \cfrk{k} \rightarrow  \prescript{k}{}{\widetilde{\mathcal{W}}}_{ \leq r}$ using the frame $\{ \tilde{\mathfrak{e}}_{r;i} \}$ which identifies $\gmc{k}$ with $N$; in particular we must have $r \geq 0$. Let $l_{r} \geq r$ be the minimum half integer such that $\prescript{k}{}{\widetilde{\mathcal{W}}}_{\leq r} \subset \prescript{k}{}{\mathcal{W}}_{\leq l_{r}}$, and  take the frame $\{\mathfrak{e}_{l;i}\}_{\substack{ 0\leq 2l \leq 2l_r\\ 0 \leq i \leq m_{l}}}$ for the submodule $\prescript{k}{}{\mathcal{W}}_{\leq l_{r}}$. Then we can write 
	$$
	\tilde{\mathfrak{e}}_{r;i} = \sum_{\substack{0\leq 2l \leq 2r \\ 0 \leq i \leq m_l}} f_{l;i} \mathfrak{e}_{l;i} +\sum_{\substack{2r+1\leq 2l \leq 2l_r \\0\leq i \leq m_{l} }} f_{l;i} \mathfrak{e}_{l;i}
	$$
	for some $f_{l;i} \in \cfrk{k}$ with $f_{l;i} = 0 \ (\text{mod $\mathbf{m}^{k}$})$ for $r+\half \leq l$. 
	
	We start with $r=0$ and assume on the contrary that $l_0 >0$. As $\tilde{\mathfrak{e}}_{0;i} = \sum_{ 0 \leq i \leq m_0} f_{0;i} \mathfrak{e}_{0;i} +\sum_{\substack{1 \leq 2l \leq 2l_0 \\0\leq i \leq m_{l} }} f_{l;i} \mathfrak{e}_{l;i}$, applying the connection $\gmc{k}$ gives
	$$
	0 = \sum_{ 0 \leq i \leq m_0} \partial (f_{0;i})  \mathfrak{e}_{0;i} +\sum_{\substack{1 \leq 2l \leq 2l_0 \\0\leq i \leq m_{l} }} \partial (f_{l;i}) \mathfrak{e}_{l;i} + \sum_{\substack{1 \leq 2l \leq 2l_0 \\0\leq i \leq m_{l} }} f_{l;i} N(\mathfrak{e}_{l;i}).
	$$
	Passing to the quotient $\prescript{k}{}{\mathcal{W}}_{\leq l_0} / \prescript{k}{}{\mathcal{W}}_{\leq l_0 -\half}$ yields $ 0 = \sum_{\substack{1 \leq l \leq l_0 \\0\leq i \leq m_{l} }} \partial (f_{l;i}) \mathfrak{e}_{l;i} $ which implies that $f_{l;i} = 0$ for $ l \geq 1$ and hence $l_0 = 0$. By induction on $r$, we have $\prescript{k}{}{\widetilde{\mathcal{W}}}_{\leq r-\half} \subset  \prescript{k}{}{\mathcal{W}}_{\leq r -\half }$ by the induction hypothesis. We assume on the contrary that $l_r >r$ and consider 
	$$	N(\tilde{\mathfrak{e}}_{r;i}) = \sum_{\substack{0\leq 2l \leq 2r \\ 0 \leq i \leq m_l}} \gmc{k} (f_{l;i} \mathfrak{e}_{l;i}) +\sum_{\substack{2r+1\leq 2l \leq 2l_r \\0\leq i \leq m_{l} }} \partial(f_{l;i}) \mathfrak{e}_{l;i} + \sum_{\substack{2r+1\leq 2l \leq2 l_r \\0\leq i \leq m_{l} }} f_{l;i} N(\mathfrak{e}_{l;i}).
	$$
	This gives $0 = \sum_{\substack{2r+1\leq 2l \leq 2l_r \\0\leq i \leq m_{l} }} \partial(f_{l;i}) \mathfrak{e}_{l;i} $ when passing to the quotient $\prescript{k}{}{\mathcal{W}}_{\leq l_r} / \prescript{k}{}{\mathcal{W}}_{\leq l_r -\half}$ and thus $f_{l;i} = 0$ for $l \geq r+\half$. This gives $l_r = r$, which is a contradiction and hence completes the induction step and the proof of the lemma.
	\end{proof}

\begin{remark}\label{rem:uniqueness_of_opposite_filtration}
	Lemma \ref{lem:uniqueness_of_opposite_filtration} says that the opposite filtration $\mathcal{H}_-$ is determined uniquely by $\prescript{0}{}{\mathcal{H}}_-$ which is given by the opposite filtration in Assumption \ref{assum:weighted_filtration_assumption}.  In the case of maximally degenerate log Calabi-Yau varieties in \S \ref{sec:application_to_gross_siebert_program}, we expect the {\em weight filration} in \cite[Remark 5.7]{Gross-Siebert-logII} determined by the nilpotent operators $N_{\nu}$'s to be the only oppostie filtration satisfying Assumption \ref{assum:weighted_filtration_assumption}.
\end{remark}

\subsubsection{Construction of the pairing $\langle \cdot,\cdot \rangle$}

The next assumption concerns the existence of the pairing $\langle \cdot,\cdot \rangle$ in Definition \ref{def:log_semi_infinite_VHS}. 
\begin{assum}\label{assum:poincare_pairing_assumption}
	We assume that 
	\begin{itemize}
		\item
		$H^*(\totaldr{0}{\parallel},\drd{0}{})$ is nontrivial only when $0 \leq * \leq 2d$;
		\item
		$H^{p,>d}(\totaldr{0}{\parallel},\prescript{0}{}{\pdb}) = 0$ for all $ 0 \leq p \leq d$;
		\item
		there is a trace map $\trace : H^{d,d}(\totaldr{0}{\parallel},\prescript{0}{}{\pdb}) = H^{2d}(\totaldr{0}{\parallel},\drd{0}{})  \rightarrow \comp$, so that we can define a pairing $\prescript{0}{}{\mathsf{p}}(\cdot,\cdot)$ on $H^{*}(\totaldr{0}{\parallel}, \drd{0}{})$ by
		$$
		\prescript{0}{}{\mathsf{p}}(\alpha,\beta):= \trace \left( \alpha \wedge \beta \right)
		$$
		for $\alpha,\beta \in H^{*}(\totaldr{0}{\parallel},\drd{0}{})$;
		\item  the filtration $\mathcal{W}_{\leq \bullet}$ is isotropic with respect to $\prescript{0}{}{\mathsf{p}}(\cdot,\cdot)$, that means $\prescript{0}{}{\mathsf{p}}( \mathcal{W}_{\leq s},\mathcal{W}_{\leq r}) = 0$ when $r+s < d$;
		\item
		the trace map $\trace$ and the corresponding pairing $\prescript{0}{}{\mathsf{p}}(\cdot,\cdot)$, when descended to $\text{Gr}_{\mathcal{F}}(H^*(\totaldr{0}{\parallel}))$, are non-degenerate.
	\end{itemize}
\end{assum}

\begin{example}\label{ex:log-smooth-IX}
	In the log smooth case (continuing Example \ref{ex:log-smooth-VII} and Example \ref{ex:log-smooth-VIII}), we use the trace map $\trace$ defined in \cite[Definition 7.11]{fujisawa2014polarizations}, which induces a pairing $\prescript{0}{}{\mathtt{p}}: \mathbb{H}^*(\Omega^*_{X^{\dagger}/\logsk{0}}) \otimes \mathbb{H}^*(\Omega^*_{X^{\dagger}/\logsk{0}}) \rightarrow \comp$ by the product structure on $\Omega^*_{X^{\dagger}/\logsk{0}}$; this was denoted by $Q_K$ in \cite[Definition 7.13]{fujisawa2014polarizations}. The pairing is compatible with the weight filtration $\mathsf{W}_{\leq r}$ on $\mathbb{H}^*(\Omega^*_{X^{\dagger}/\logsk{0}})$ by \cite[Lemma 7.18]{fujisawa2014polarizations} and is defined over $\mathbb{Q}$ by \cite[Lemma 7.14]{fujisawa2014polarizations}, which implies that the opposite filtration $\mathcal{W}_{\leq \bullet}$ is isotropic with respect to $\prescript{0}{}{\mathtt{p}}$. Furthermore,  non-degeneracy of $\prescript{0}{}{\mathtt{p}}$ follows from that of the induced pairing $\langle \cdot , \cdot \rangle$ on $L_{\comp} = \text{Gr}_{\mathsf{W}}(\mathbb{H}^*(\Omega^*_{X^{\dagger}/\logsk{0}}))$ defined in \cite[Definition 8.10]{fujisawa2014polarizations}, which in turn is a consequence of \cite[Theorem 8.11]{fujisawa2014polarizations} (where projectivity of $X$ was used).
\end{example}

We will denote by $\trace :H^{*}(\totaldr{0}{\parallel})  \rightarrow \comp$ the map which extends $\trace : H^{2d}(\totaldr{0}{\parallel}) \rightarrow \comp$ and is trivial on $H^{<2d}(\totaldr{0}{\parallel})$.
Note that by definition $\mathcal{F}^{\geq \bullet}$ is isotropic with respect to the pairing $\prescript{0}{}{\mathsf{p}}(\cdot,\cdot)$, i.e. $\prescript{0}{}{\mathsf{p}} (\mathcal{F}^{\geq r}, \mathcal{F}^{\geq s})  = 0$ when $r+s >d$.
We have $\trace(N_\nu(\alpha)) = 0$ for any $\alpha \in H^{*}(\totaldr{0}{\parallel})$ and $\nu \in \blat_\comp^\vee$.

\begin{lemma}\label{lem:tuncation_of_dolbeault_cohomology}
	Suppose we take the  subcomplex $\totaldr{k}{\parallel}^{*,>d} \hookrightarrow \totaldr{k}{\parallel}^{*,*}$, then natural induced map between the cohomology group $H^*(\totaldr{k}{\parallel}^{*,>d}) \rightarrow H^*(\totaldr{k}{\parallel}^{*,*})$ is a zero map for each $k$. 
\end{lemma}

\begin{proof}
	We prove by induction on $k$ that the induced map on cohomology is zero. For $k=0$ it follows from second requirement in Assumption \ref{assum:poincare_pairing_assumption} together with the Hodge-de Rham degeneracy Assumption \ref{assum:Hodge_de_rham_degeneracy}. For induction step we consider the commutative diagram
	$$
	\xymatrix@1{
		H^*(\totaldr{0}{\parallel}^{*,>d}) \otimes \mathbf{m}^{k}/\mathbf{m}^{k-1} \ar[r]	\ar[d]	&H^*(\totaldr{k}{\parallel}^{*,>d}) \ar[d] \ar[r] & H^*(\totaldr{k-1}{\parallel}^{*,>d}) \ar[d]\\
	H^*(\totaldr{0}{\parallel}^{*,*})\otimes \mathbf{m}^{k}/\mathbf{m}^{k-1} \ar[r] 	&H^*(\totaldr{k}{\parallel}^{*,*})  \ar[r]&		H^*(\totaldr{k-1}{\parallel}^{*,*}) 
	}.
	$$ Exactly the same argument as in \S \ref{sec:proof_of_triviality_of_hodge_bundle} tells us that the cohomology $H^*( \totaldr{k}{\parallel}^{*,>d})$ is also a free $\cfrk{k}$ module, which means the first row is exact.  Therefore the two rows in the above diagram are exact. We conclude that the vertical map in the center is zero from induction hypothesis that the other two vertical maps being zero. 
\end{proof}

\begin{definition}\label{def:construction_of_higher_residue_pairing}
	Using the elementary sections $\{\mathfrak{e}_{r;i}\}$ in Definition \ref{def:elementary_sections}, we extend the trace map $\trace$ to the Hodge bundle $H^*(\totaldr{k}{\parallel})[[t^{\half},t^{-\half}]$ by the formula
	$$
	\trace(f \mathfrak{e}_{r;i})  := f \trace(e_{r;i})
	$$
	for $f \in \cfrk{k}_{\ecfr{}}[[t^{\half},t^{-\half}]$, and extending linearly. 
	
	We also extend the pairing $\prescript{0}{}{\mathsf{p}}$ to $H^*(\totaldr{k}{\parallel})[[t^{\half},t^{-\half}]$ by the formula
	$$
	\mathsf{p}(f(t)\mathfrak{e}_{r;i},g(t)\mathfrak{e}_{l;j}) : = (-1)^{|g||e_{r;i}|} f(t)g(t) \prescript{0}{}{\mathsf{p}}(e_{r;i}, e_{l;j}),
	$$
	for $f(t) , g(t) \in \cfrk{k}_{\ecfr{}}[[t^{\half},t^{-\half}]$, and extending linearly.
	
	Finally we define a pairing $\langle \cdot,\cdot \rangle$ on $H^*(\totaldr{k}{\parallel})[[t^{\half},t^{-\half}]$ by the formula\footnote{This expression is motivated by \cite{Barannikov99}.}
	$$
	\langle \alpha(t),\beta(t) \rangle := (-1)^{d-1 - d|\beta| + \half (\bar{\beta}-d)} t^{2-d}\mathsf{p}(\alpha(t), \beta(-t) ),
	$$
	for $\alpha(t), \beta(t) \in H^*(\totaldr{k}{\parallel})[[t^{\half},t^{-\half}]$, where $\bar{\beta}$ is the index of \\
	$\beta \in H^*(\totaldr{k}{\parallel})[[t^{\half},t^{-\half}]$ (see Notation \ref{not:grading_of_element}). This naturally induces a pairing on $\prescript{k}{}{\mathcal{H}}_{\pm} \subset H^*(\totaldr{k}{\parallel})[[t^{\half},t^{-\half}]$. 
\end{definition}

\begin{lemma}\label{lem:flatness_of_higher_residue_pairing}
	We have the identification $\mathsf{p}(\alpha,\beta) = \trace(\alpha\wedge \beta)$ between the pairing and the trace map in Definition \ref{def:construction_of_higher_residue_pairing}. Furthermore, the pairing $\mathsf{p}$ is flat, i.e. 
	$$
	X(\mathsf{p}(\alpha,\beta)) = \mathsf{p} ( \gmc{}_X \alpha,\beta)  +\mathsf{p} (\alpha,\gmc{}_X \beta)
	$$
 	for any $\alpha,\beta \in H^*(\totaldr{k}{\parallel})$ and $X \in \logscvfk{k}$.
\end{lemma}

\begin{proof}
	Using the short exact sequence
	$$
	0 \rightarrow \blat_\comp \otimes_\comp \totaldr{k}{\parallel}^*[-1] \rightarrow \totaldr{k}{0}^*/\totaldr{k}{2}^* \rightarrow \totaldr{k}{\parallel}^* \rightarrow 0
	$$
	which defines the $k^{\text{th}}$-order Gauss-Manin connection $\gmc{k}$, we have the flatness of the product:
	$$
	\gmc{k}(\alpha\wedge \beta) = (\gmc{k}\alpha) \wedge \beta + (-1)^{|\alpha|} \alpha \wedge (\gmc{k} \beta). 
	$$
	To prove the identity $ \mathsf{p}(\alpha,\beta) = \trace(\alpha\wedge \beta)$, we choose a basis $\{e_{r;i}\}$ of $H^*(\totaldr{0}{})$ and the corresponding elementary sections $\{\mathfrak{e}_{r;i}\}$ as in Definition \ref{def:elementary_sections}. We claim that the relation $\mathfrak{e}_{r;i} \wedge \mathfrak{e}_{l;j} = \sum_{s;k} c_{r,l;i,j}^{s;k} \mathfrak{e}_{s;k}$ holds for some constant $c_{r,l;i,j}^{s;k} \in \comp$. This can be proved by an induction on the order $k$, followed by an induction on the lexicographical order: $(r,l) < (r',l')$ if $r < r'$, or $r = r'$ and $l<l'$ for each fixed $k$. 
	So we fix $r,l$ and $i,j$, consider the product $\mathfrak{e}_{r;i} \wedge \mathfrak{e}_{l;j} $, and assume that the above relation holds for any $(r',l')< (r,l)$. 
	Writing $\mathfrak{e}_{r;i} \wedge \mathfrak{e}_{l;j} = \sum_{s;k} c^{s;k} \mathfrak{e}_{s;k} + \sum_{s;k} \sum_{q^m \in \mathbf{m}^{k}} b^{s;k}_m e_{s;k} q^{m}$ and applying the Gauss-Manin connection gives $\gmc{k}_{\nu} (\mathfrak{e}_{r;i} \wedge \mathfrak{e}_{l;j} ) = (N_{\nu} \mathfrak{e}_{r;i}) \wedge \mathfrak{e}_{l;j} + \mathfrak{e}_{r;i} \wedge (N_
	{\nu}\mathfrak{e}_{l;j} )$. The induction hypothesis then forces $\sum_{s;k} \sum_{q^m \in \mathbf{m}^{k}} b^{s;k}_m e_{s;k} q^{m} = 0$. As a result we have $ \mathsf{p}(\mathfrak{e}_{r;i},\mathfrak{e}_{l;j}) = \trace(\mathfrak{e}_{r;i} \wedge \mathfrak{e}_{l;j})$ and the general relation $ \mathsf{p}(\alpha,\beta) = \trace(\alpha\wedge \beta)$ follows. Flatness of the pairing $\mathsf{p}$ now follows from that of $\trace$. 
\end{proof}

\begin{lemma}\label{lem:pairing_for_semi_infinite_VHS}
	The pairing in Definition \ref{def:construction_of_higher_residue_pairing} satisfies $\langle s_1,s_2\rangle  \in \cfrk{k}_{\ecfr{}}[[t]]$ for any $s_1, s_2 \in \prescript{k}{}{\mathcal{H}}_+$, and $\langle s_1,s_2 \rangle  \in \cfrk{k}_{\ecfr{}}[t^{-1}]t^{-2}$ for any $s_1,s_2 \in \prescript{k}{}{\mathcal{H}}_-$. 
	Furthermore, it decends to give a non-degenerate pairing $g(\cdot,\cdot) : \prescript{k}{}{\mathcal{H}}_+/t(\prescript{k}{}{\mathcal{H}}_+) \times \prescript{k}{}{\mathcal{H}}_+/t(\prescript{k}{}{\mathcal{H}}_+) \rightarrow \cfrk{k}_{\ecfr{}}[-2d]$. 
	\end{lemma}

\begin{proof}
	From the description in Definition \ref{def:construction_of_mathcalH_+} for $\prescript{k}{}{\mathcal{H}}_{\pm}$ we first notice that the pairing $\langle \cdot,\cdot \rangle$ takes value in $\cfrk{k}_{\ecfr{}}[[t,t^{-1}]$ when restricted to $\prescript{k}{}{\mathcal{H}}_{\pm}$. The statement for $\prescript{k}{}{\mathcal{H}}_-$ follows from flatness of pairing with respect to Gauss-Manin connection in the above Lemma \ref{lem:flatness_of_higher_residue_pairing}, the fourth item in Assumption \ref{assum:poincare_pairing_assumption} and the explicit description of $\prescript{0}{}{\mathcal{H}}_-$ from Lemma \ref{lem:uniqueness_of_opposite_filtration}. 
	
	For the statement on $\prescript{k}{}{\mathcal{H}}_+$, we take two classes of $H^*(\totaldr{k}{\parallel}_{\ecfr{}})[[t^{\half},t^{-\half}]$ represented by $\mathsf{l}_t(\alpha e^{\prescript{k}{}{\varphi}/t}) \lrcorner \volf{k}, \mathsf{l}_t(\beta e^{\prescript{k}{}{\varphi}/t}) \lrcorner \volf{k} \in \totaldr{k}{\parallel}_{\ecfr{}}[[t^{\half},t^{-\half}]$ for some elements $\alpha,\beta \in \polyv{k}^*_{\ecfr{}}[[t]]$. We consider the expression
	$$
		(-1)^{\clubsuit} t^{2-d} \trace \left( (\mathsf{l}_t(\alpha(t)e^{\prescript{k}{}{\varphi}(t)/t})  \lrcorner \volf{k}) \wedge (\mathsf{l}_{-t}(\beta(-t)e^{-\prescript{k}{}{\varphi}(-t)/t}) \lrcorner \volf{k}) \right),
	$$
	where $\clubsuit$ refers to the suitable exponent of $(-1)$ appearing in the third formula in Definition \ref{def:construction_of_higher_residue_pairing}. For computing the trace we can write 
	\begin{multline*}
	(-1)^{\clubsuit}t^{2-d}\left( (\mathsf{l}_t(\alpha(t)e^{\prescript{k}{}{\varphi}(t)/t})  \lrcorner \volf{k}) \wedge (\mathsf{l}_{-t}(\beta(-t)e^{-\prescript{k}{}{\varphi}(-t)/t}) \lrcorner \volf{k}) \right)_{2d} \\= \mu_+(t) + \mu_-(t^{-1}), 
	\end{multline*}
	where the subscript $\left( \cdot \right)_{2d}$ refering to the index $2d$ part, $\mu_+(t) \in \totaldr{k}{\parallel}_{\ecfr{}}[[t^{\half}]]$ and $\mu_-(t) \in \totaldr{k}{\parallel}_{\ecfr{}}[t^{-\half}]$. We claim that $\mu_-(t) \in \totaldr{k}{\parallel}^{*,>d}_{\ecfr{}}[t^{-\half}]$ and it is zero in the cohomology group $H^*(\totaldr{k}{\parallel}_{\ecfr{}})[[t^{\half},t^{-\half}]$ using Lemma \ref{lem:tuncation_of_dolbeault_cohomology}. If the claim is true then we can conclude that 
	$$
	\langle \mathsf{l}_t(\alpha e^{\prescript{k}{}{\varphi}/t}) \lrcorner \volf{k}, \mathsf{l}_t(\beta e^{\prescript{k}{}{\varphi}/t}) \lrcorner \volf{k}\rangle = \trace( \mu_+(t)) \in \cfrk{k}_{\ecfr{}}[[t]]. 
	$$
	
	\cite[Proof of Proposition 5.9.4]{Barannikov99} gives the following formula 
	\begin{multline}\label{eqn:barannikov_trace_map_equation}
		(-1)^{\clubsuit} t^{2-d}\left( (\mathsf{l}_t(\alpha(t)e^{\prescript{k}{}{\varphi}(t)/t})  \lrcorner \volf{k}) \wedge (\mathsf{l}_{-t}(\beta(-t)e^{-\prescript{k}{}{\varphi}(-t)/t}) \lrcorner \volf{k})\right)_{2d} = \\  \left( (\alpha(t)\wedge \beta(-t) e^{(\prescript{k}{}{\varphi}(t) - \prescript{k}{}{\varphi}(-t))/t}) \lrcorner \volf{k} ) \wedge \volf{k} \right)_{2d} + \mathtt{r}(t),
	\end{multline}
	where $\mathtt{r}(t) \in\totaldr{k}{\parallel}^{*,>d}_{\ecfr{}}[[t^{\half},t^{-\half}]$. Since we have 
	$$ \left( (\alpha(t)\wedge \beta(-t) e^{(\prescript{k}{}{\varphi}(t) - \prescript{k}{}{\varphi}(-t))/t}) \lrcorner \volf{k} ) \wedge \volf{k} \right)_{2d} \in \totaldr{k}{\parallel}_{\ecfr{}}[[t]]$$ we prove the claim. 

	For non-degeneracy it suffices to consider the pairing for $\prescript{0}{}{\mathcal{H}}_+$, which follows from the non-degeneracy condition in Assumption \ref{assum:poincare_pairing_assumption}. 
\end{proof}

The above constructions give a $\frac{\infty}{2}$-LVHS $(\mathcal{H}_+, \gmc{}, \langle \cdot,\cdot \rangle)$ together with an opposite filtration $\mathcal{H}_-$ satisfying (1)-(3) in Definition \ref{def:opposite_filtration}. It remains to construct the grading structure.

\subsubsection{Construction of the grading structure}

\begin{definition}
	For each $k$, we define the {\em extended connection} $\gmc{}_{t\dd{t}}$
	acting on $H^*(\totaldr{k}{\parallel}_{\ecfr{}})[[t^{\half},t^{-\half}]$ by the rule that $\gmc{}_{t\dd{t}}(s) = \frac{2-d}{2} s$ for $s \in H^*(\totaldr{k}{\parallel}_{\ecfr{}})$ and $\gmc{}_{t\dd{t}}(fs) = t\dd{t}(f) s + f (\gmc{}_{t\dd{t}}(s))$.
	\end{definition}

\begin{prop}\label{prop:checking_grading_structure}
	The extended connection $\gmc{}$ is a flat connection acting on $\prescript{k}{}{\mathcal{H}}_\pm$, i.e. we have $[\gmc{}_{t\dd{t}},\gmc{}_X] = 0$ for any $X \in \logscvfk{k}$. The submodule $\prescript{k}{}{\mathcal{H}}_-$ is preserved by $\gmc{}_{t\dd{t}}$ and we have $\gmc{}_{t\dd{t}} (\mathcal{H}_+) \subset t^{-1} (\prescript{k}{}{\mathcal{H}}_+)$. Furthermore, the pairing $\langle \cdot,\cdot \rangle$ is flat with respect to $\gmc{}_{t\dd{t}}$.  
	\end{prop}

\begin{proof}
	Beside $\gmc{}_{t\dd{t}} (\mathcal{H}_+) \subset t^{-1} (\prescript{k}{}{\mathcal{H}}_+)$, the other properties simply follow from definitions. Take $\alpha \in \polyv{k}_{\ecfr{}}[[t]]$ and consider $( \mathsf{l}_{t}(\alpha) \wedge e^{\mathsf{l}_t(\prescript{k}{}{\varphi})} ) \lrcorner \volf{k}$. Then
	\begin{multline*}
 	\gmc{}_{t\dd{t}}( \mathsf{l}_{t}(\alpha) \wedge e^{\mathsf{l}_t(\prescript{k}{}{\varphi})} ) \lrcorner \volf{k} =\\ ( (\gmc{}_{t\dd{t}} \mathsf{l}_{t}(\alpha)) \wedge e^{\mathsf{l}_t(\prescript{k}{}{\varphi})} + \mathsf{l}_{t}(\alpha) \wedge (\gmc{}_{t\dd{t}}\mathsf{l}_t(\prescript{k}{}{\varphi}))\wedge e^{\mathsf{l}_t(\prescript{k}{}{\varphi})} ) \lrcorner \volf{k}.
	\end{multline*}
	Since we can write both $\gmc{}_{t\dd{t}} \mathsf{l}_{t}(\alpha) = \mathsf{l}_t(\beta)$ and $\gmc{}_{t\dd{t}}\mathsf{l}_t(\prescript{k}{}{\varphi}) = \mathsf{l}_t(\gamma)$ for some $\beta,\gamma\in \polyv{k}_{\ecfr{}}[[t]]$, we may rewrite 
	$$
	\gmc{}_{t\dd{t}}( \mathsf{l}_{t}(\alpha) \wedge e^{\mathsf{l}_t(\prescript{k}{}{\varphi})} ) \lrcorner \volf{k}  = ( (\mathsf{l}_t(\beta) + t^{-1} \mathsf{l}_t (\alpha\wedge \gamma)) e^{\mathsf{l}_t(\prescript{k}{}{\varphi})} ) \lrcorner \volf{k},
	$$
	which gives the desired result.
	\end{proof}

\subsection{Construction of a miniversal section}\label{sec:construction_of_miniversal_section}

\begin{notation}\label{not:flat_extension_of_volume_form}
Consider the cohomology class $[\volf{0}] \in \mathcal{F}^{\geq d} \cap \mathcal{W}_{\leq d}$. We let $\prescript{k}{}{\mu}$ be the extension of the cohomology classes $[\volf{0}] \in  t(\prescript{0}{}{\mathcal{H}}_+) \cap  t^2(\prescript{0}{}{\mathcal{H}}_-)$ by first expressing it as a linear combination of the filtered basis $[\volf{0}] = \sum_{r;i} c_{r;i} e_{r;i}$, and then extend it by elementary sections in Definition \ref{def:elementary_sections} to $ t^2 ( \prescript{k}{}{\mathcal{H}}_-)$ using the formula $\prescript{k}{}{\mu} = \sum_{r;i} c_{r;i} \mathfrak{e}_{r;i}$ for each $k$.
\end{notation}

\begin{notation}\label{not:choice_of_first_order_in_miniversal_element}
By our choice of the graded vector space $\mathbb{V}^* = \text{Gr}_{\mathcal{F}}(H^*(\totaldr{0}{\parallel})) / \text{Im}(\gmc{0}([\volf{0}]))$ in Condition \ref{cond:choice_of_moduli_coefficient_ring}, we further make a choice of a degree $0$ element $\psi \in \polyv{0}[[t]] \otimes \mathbb{V}^{\vee}$ such that its cohomology class $[\mathsf{l}_{t}(\psi)]|_{t=0} \in \left( \prescript{0}{}{\mathcal{H}}_+/ t(\prescript{0}{}{\mathcal{H}}_+) \right) \otimes \mathbb{V}^\vee = \text{Gr}_{\mathcal{F}}(H^*(\totaldr{0}{\parallel})) \otimes \mathbb{V}^\vee$ maps to the identity element $\text{id} \in \mathbb{V} \otimes \mathbb{V}^\vee$ under the natural quotient $\text{Gr}_{\mathcal{F}}(H^*(\totaldr{0}{\parallel})) \otimes \mathbb{V}^\vee \rightarrow \mathbb{V} \otimes \mathbb{V}^\vee$. 
\end{notation}

\begin{definition}\label{def:miniversal_section}
	For the $\psi$ chosen in Notation \ref{not:choice_of_first_order_in_miniversal_element}, let $\varphi = (\prescript{k}{}{\varphi})_k$ be the corresponding Maurer-Cartan element  constructed in Theorem \ref{thm:unobstructedness_of_MC_equation}. Then $\left( t^{-1} e^{\mathsf{l}_t ( \prescript{k}{}{\varphi})}  \lrcorner \volf{k} \right)_{k}$ is called a {\em primitive section} if it further satisfies the condition
	$$
	 t^{-1} \big(  e^{\mathsf{l}_t ( \prescript{k}{}{\varphi})}  \lrcorner \volf{k} -  \prescript{k}{}{\mu} \big) \in \prescript{k}{}{\mathcal{H}}_-
	$$
	for each $k$, where $\volf{k}$ is the element constructed in Proposition \ref{prop:gluing_volume_form}.
\end{definition}

\begin{prop}\label{prop:existence_of_miniversal_section}
	We can modify the Maurer-Cartan element $\varphi = (\prescript{k}{}{\varphi})_k$ constructed in Theorem \ref{thm:unobstructedness_of_MC_equation} by $\varphi \mapsto \varphi + t \zeta$ for some $\zeta =(\prescript{k}{}{\zeta})_k  \in \varprojlim_{k}\polyv{k}^0_{\ecfr{}}[[t]]$ to get a primitive section. Furthermore, $\mathcal{H}_+$ is unchanged under this modification.  
\end{prop}

\begin{proof}
	The proof is a refinement of that of Theorem \ref{thm:unobstructedness_of_MC_equation} by the same argument as in \cite[Theorem 1]{Barannikov99}.
\end{proof}

The following theorem concludes this section:

\begin{theorem}\label{thm:construction_of_Frobenius_manifold}
	The triple $(\prescript{k}{}{\mathcal{H}}_+,\gmc{},\langle \cdot,\cdot \rangle )$ is a $\frac{\infty}{2}$-LVHS, and $\prescript{k}{}{\mathcal{H}}_-$ is an opposite filtration. Furthermore, the element $\prescript{k}{}{\xi}:=t^{-1} e^{\mathsf{l}_t(\prescript{k}{}{\varphi})} \lrcorner (\volf{k})$ constructed in Proposition \ref{prop:existence_of_miniversal_section} is a miniversal section in the sense of Definition \ref{def:abstract_primitive_form}. 
\end{theorem}

\begin{proof}
	It remains to check that $\xi$ is a miniversal section. We write $\xi = \varprojlim_k \prescript{k}{}{\xi}$ and prove the condition for each $k$. First of all, we have $\prescript{k}{}{\xi} \in \prescript{k}{}{\mathcal{H}}_+ \cap t(\prescript{k}{}{\mathcal{H}}_-)$ from its construction, and that $\prescript{k}{}{\xi}  = t^{-1} (\prescript{k}{}{\mu})$ in $t (\prescript{k}{}{\mathcal{H}}_-)/\prescript{k}{}{\mathcal{H}}_-$. So $\gmc{}_\nu (\prescript{k}{}{\xi}) = t^{-1} (\gmc{}_\nu)(\prescript{k}{}{\mu}) = t^{-1} N_\nu (\prescript{k}{}{\mu}) \in \prescript{k}{}{\mathcal{H}}_-$ for any $\nu \in \blat_\comp^\vee$. We have computed the action of $\gmc{}_{t\dd{t}}$ in the proof of Proposition \ref{prop:checking_grading_structure}, and the formula gives $\gmc{}_{t\dd{t}} ( t^{-1} e^{\mathsf{l}_t(\prescript{k}{}{\varphi})} ) \in (1-d) e^{\mathsf{l}_t(\prescript{k}{}{\varphi})} + (\prescript{k}{}{\mathcal{H}}_-)$. Therefore we have $\gmc{}_{t\dd{t}} (\prescript{k}{}{\xi}) = (1-d) (\prescript{k}{}{\xi})$ in $(\prescript{k}{}{\mathcal{H}}_-)/\prescript{k}{}{\mathcal{H}}_-$. Finally, to check that the Kodaira-Spencer map is an isomorphism, we only need to show this for $\logsck{0}$, which follows from our choice of the input $\psi$ for solving the Maurer-Cartan equation \eqref{eqn:Maurer_Cartan_equation_unobstructedness} in Theorem \ref{thm:unobstructedness_of_MC_equation}.
\end{proof}

Example \ref{ex:log-smooth-VIII}, Example \ref{ex:log-smooth-IX} and Theorem \ref{thm:construction_of_Frobenius_manifold} together gives the following corollary. 
\begin{corollary}\label{cor:Kawamata-Namikawa-2}
	There exists a structure of logarithmic Frobenius manifold on the formal extended moduli $\logscf$ near the log smooth Calabi-Yau variety $(X,\mathcal{O}_X)$ in Example \ref{ex:log-smooth-I}. 
\end{corollary}

\begin{remark}
	Following \cite{Barannikov99,li2013variation}, we can define the semi-infinite period map 
	$
	\Phi : \logscf \rightarrow t\mathcal{H}_-/\mathcal{H}_-
	$
	as $\Phi(s) : = [e^{\mathsf{l}_t(\varphi(s,t))} \lrcorner \volf{} - \mu]$. In the case of maximally degenerate log Calabi-Yau varieties studied in \cite{Gross-Siebert-logII}, we expect this gives the canonical coordinates on the (extended) moduli space. 
\end{remark}
\section{Smoothing of maximally degenerate log Calabi-Yau varieties}\label{sec:application_to_gross_siebert_program}

In this section, we apply our results to the case of maximally degenerate log Calabi-Yau varieties studied by Kontsevich-Soibelman \cite{kontsevich-soibelman04} and Gross-Siebert in \cite{Gross-Siebert-logI,Gross-Siebert-logII, gross2011real}. We will mainly follow \cite{Gross-Siebert-logI,Gross-Siebert-logII} and assume the reader is familiar with these papers.

\begin{notation}\label{not:gross_siebert_affine_manifold_polytope_notation}
	The characteristic 0 algebraically closed field $\Bbbk$ in \cite{Gross-Siebert-logI} is always chosen to be $\comp$. 
	We work with a $d$-dimensional integral affine manifold $B$ with holonomy in $\inte^d \rtimes \text{SL}_d(\inte)$ and codimension $2$ singularities $\Delta$ as in \cite[Definition 1.15]{Gross-Siebert-logI}, together with a toric polyhedral decomposition $\mathcal{P}$ of $B$ into lattice polytopes as in \cite[Definition 1.22]{Gross-Siebert-logI}. 
	Following \cite{Gross-Siebert-logI}, we take $Q = \mathbb{N}$ for simplicity. We also fix an open gluing data $\mathbf{s}$ as in Definition \cite[Definition 2.25]{Gross-Siebert-logI} for the pair $(B,\mathcal{P})$ satisfying the lifting condition in \cite[Proposition 4.25]{Gross-Siebert-logI}.
\end{notation}

\begin{assum}\label{assum:gross_siebert_simple_and_positive}
	We assume that $(B,\mathcal{P})$ satisfies the assumption in \cite[Theorem 3.21]{Gross-Siebert-logII} (in order to get Hodge-to-de Rham degeneracy using results from \cite{Gross-Siebert-logII}). 
\end{assum}

\begin{definition}\label{def:0_order_space_from_gross_siebert}
	Given $(B,\mathcal{P},\mathbf{s})$, we let $(X,\mathcal{O}_X)$ be the $d$-dimensional complex analytic space given by the analytification of the log scheme $X_0(B,\mathcal{P},\mathbf{s})$ constructed in \cite[Theorem 5.2]{Gross-Siebert-logI}. It is  equipped with a log structure over the $Q$-log point $\logsk{0}$.
\end{definition}

We denote the log-space by $X^{\dagger}$ if we want to emphasize the log-structure. Let $Z \subset X$ be the codimension $2$ singular locus of the log-structure (i.e. $X^{\dagger}$ is log-smooth away from $Z$) and $j : X\setminus Z \rightarrow X$ be the inclusion as in \cite{Gross-Siebert-logII}.

\subsection{The $0^{\text{th}}$-order deformation data}\label{sec:0_order_data_from_gross_siebert}

Following the notations from \cite{Gross-Siebert-logI, Gross-Siebert-logII}, the $0^{\text{th}}$-order deformation data in Definition \ref{def:0_order_data} is described as follows: 
\begin{definition}\label{def:0_order_data_from_gross_siebert}
	\begin{itemize}
		\item
		the $0^{\text{th}}$-order complex of polyvector fields is given by the pushforward of the analytic sheaf of relative polyvector fields
		$\bva{0}^* = j_* ( \bigwedge^{-*} \Theta_{X^{\dagger}/\logsk{0}}) $ equipped with the natural wedge product;
		
		\item
		the $0^{\text{th}}$-order de Rham complex is given by the pushforward of the analytic sheaf of de Rham differential forms $\tbva{0}{}^*:= j_* (\Omega^*_{X^{\dagger}/\comp})$, equipped with the de Rham differential $\dpartial{0} = \dpartial{}$ as in \cite[first paragraph of \S 3.2]{Gross-Siebert-logII};
		
		\item
		the volume element $\volf{0}$ is given via the trivialization $j_*(\Omega^d_{X^{\dagger}/\logsk{0}}) \cong \mathcal{O}_X$ by \cite[Theorem 3.23]{Gross-Siebert-logII}, and then the BV operator is defined by $\bvd{0}(\varphi) \lrcorner \volf{0}:= \dpartial{0} (\varphi \lrcorner \volf{0})$.
	\end{itemize}
\end{definition}

The map $\gmiso{0}{r}^{-1} : \logsdrk{0}{r} \otimes_\comp (\tbva{0}{0}^*/ \tbva{0}{1}^*[-r] ) \rightarrow  \tbva{0}{r}^*/ \tbva{0}{r+1}^*$ given by taking wedge product in $ j_* (\Omega^*_{X^{\dagger}/\comp}) $ is a morphism of sheaves of BV modules.

To show that the data in Definition \ref{def:0_order_data_from_gross_siebert} satisfies all the conditions in Definition \ref{def:0_order_data}, we need to verify that $\bva{0}^*$ and $\tbva{0}{}^*$ are coherent, $\gmiso{0}{r}$ is an isomorphism and there is an identification $\rdr{0}{}^* = \tbva{0}{0}^*/ \tbva{0}{1}^* \cong  j_*(\Omega^*_{X^{\dagger}/\logsk{0}})$. Let us briefly explain how to obtain all these from \cite{Gross-Siebert-logII}. 

\begin{notation}\label{not:local_model}
Following \cite[Construction 2.1]{Gross-Siebert-logII}, we consider the monoids $\mathcal{Q}, P$, the corresponding toric varieties $\mathbb{V} = \text{Spec}(\comp[P])$ and $\mathsf{V} = \text{Spec}(\comp[\mathcal{Q}])$ and the associated analytic spaces $\mathbf{V} = \mathbb{V}^{an}$ and $\mathscr{V} = \mathsf{V}^{an}$ respectively. $\mathbb{V}$ is equipped with a divisorial log structure induced from the divisor $\mathsf{V}$, and $\mathsf{V}$ is equipped with the pull-back of the log structure from $\mathbb{V}$.  
\cite[Theorem 2.6]{Gross-Siebert-logII} shows that for every geometric point $\bar{x} \in X_0(B,\mathcal{P},\mathbf{s})$, there is an \'{e}tale neighborhood $\mathsf{W}$ of $\bar{x}$ which can be identified with an \'{e}tale neighborhood of $\mathsf{V}$ as a log scheme in the sense that there are \'{e}tale maps
$$
\xymatrix@1{ & \mathsf{W} \ar[dr] \ar[dl] & \\
	\mathsf{V} & & X_0(B,\mathcal{P},\mathbf{s}).}
$$
Taking analytification of these maps, we can find an open subset (or a sheet) $W \subset \mathsf{W}^{an}$ mapping homeomorphically to both an open subset in $\mathsf{V}^{an}$ and an open subset in $V \subset X$. 
\end{notation}

The desired statements are local on $X$, and we work on the local model $\mathscr{V}$. As in \cite[proof of Proposition 1.12]{Gross-Siebert-logII}, we let $\tilde{\mathbb{V}}$ be the log scheme equipped with the smooth divisorial log structure induced by the boundary toric divisor $\text{Spec}(\comp[\partial P])$. Then we have $j_*(\Omega^{\text{alg},*}_{\mathsf{V}^{\dagger}/\comp}) \hookrightarrow j_* \left( \Omega^{\text{alg},*}_{\tilde{\mathbb{V}}^{\dagger}/\comp} \right)|_{\mathsf{V} \setminus Z} = \left( \Omega^{\text{alg},*}_{\tilde{\mathbb{V}}^{\dagger}/\comp} \right)|_{\mathsf{V}}$,
where the notation $\Omega^{\text{alg},*}$ refers to algebraic sheaves. From these we obtain the identification $\left\lbrack j_*(\Omega^{\text{alg},*}_{\mathsf{V}^{\dagger}/\comp}) \right\rbrack^{an} = j_*(\Omega^*_{\mathscr{V}^{\dagger}/\comp})$ which globalizes to give $\left\lbrack j_*(\Omega^{\text{alg},*}_{X_0(B,\mathcal{P},\mathbf{s})^{\dagger}/\comp}) \right\rbrack^{an} = j_*(\Omega^*_{X^{\dagger}/\comp})$, and similarly, $\left\lbrack j_*(\Omega^{\text{alg},*}_{X_0(B,\mathcal{P},\mathbf{s})^{\dagger}/\logsk{0}}) \right\rbrack^{an} = j_*(\Omega^*_{X^{\dagger}/\logsk{0}})$.  
Because both $ j_*(\Omega^{\text{alg},*}_{X_0(B,\mathcal{P},\mathbf{s})^{\dagger}/\comp}) $ and $ j_*(\Omega^{\text{alg},*}_{X_0(B,\mathcal{P},\mathbf{s})^{\dagger}/\logsk{0}}) $ are coherent sheaves, so are $\bva{0}^*$ and $\tbva{0}{}^*$ via the analytification functor. Taking analytification of the exact sequence
\begin{multline*}
0 \rightarrow \logsdrk{0}{1} \otimes_{\comp} j_*(\Omega^{\text{alg},*}_{X_0(B,\mathcal{P},\mathbf{s})^{\dagger}/\logsk{0}})[-1] \rightarrow \\ j_*(\Omega^{\text{alg},*}_{X_0(B,\mathcal{P},\mathbf{s})^{\dagger}/\comp}) \rightarrow  j_*(\Omega^{\text{alg},*}_{X_0(B,\mathcal{P},\mathbf{s})^{\dagger}/\logsk{0}}) \rightarrow 0
\end{multline*}
in \cite[line 4 in proof of Theorem 5.1]{Gross-Siebert-logII}, we see that $\gmiso{0}{r}$ is an isomorphism and we also obtain the identification $\rdr{0}{}^* =\tbva{0}{0}^*/ \tbva{0}{1}^* \cong  j_*(\Omega^*_{X^{\dagger}/\logsk{0}})$.

\subsection{The higher order deformation data}\label{sec:higher_order_deformation_data_from_Gross_siebert}

From Notation \ref{not:local_model}, we have, at every point $\bar{x} \in X^{\dagger}$, an analytic neighborhood $V$ together with a log space $\mathbf{V}^{\dagger}$ and a log morphism $\pi : \mathbf{V}^{\dagger} \rightarrow \logs$ such that the diagram
\begin{equation}\label{eqn:local_model}
\xymatrix@1{ V \  \ar@{^{(}->}[rr] \ar[d] & & \mathbf{V}^{\dagger} \ar[d]^{\pi}\\
\logsk{0} \  \ar@{^{(}->}[rr] & &\logs 
}
\end{equation}
is a fiber product of log spaces. We fix an open covering $\mathcal{V}$ by Stein open subsets $V_{\alpha}$'s with local thickening $\mathbf{V}^{\dagger}_\alpha$'s given as above, and write $\prescript{k}{}{\mathbf{V}}_{\alpha}^{\dagger}$ for the $k^{\text{th}}$-order thickening over $\logsk{k}$. We also write $j: V_{\alpha} \setminus Z \rightarrow V_{\alpha}$ for the inclusion.

The higher order deformation data in Definitions \ref{def:higher_order_data} and \ref{def:higher_order_data_module} are described as follows:
\begin{definition}\label{def:higher_order_thickening_data_from_gross_siebert}
	For each $k \in \inte_{\geq 0}$,
	\begin{itemize}
		\item
		the $k^{\text{th}}$-order polyvector fields is given by $\bva{k}_{\alpha}^* := j_* ( \bigwedge^{-*} \Theta_{\prescript{k}{}{\mathbf{V}}_{\alpha}^{\dagger}/\logsk{k}})$ (i.e. polyvector fields on $\prescript{k}{}{\mathbf{V}}_\alpha^{\dagger}$);
		
		\item
		the $k^{\text{th}}$-order de Rham complex is given by $\tbva{k}{}^*_\alpha := j_* (\Omega^*_{\prescript{k}{}{\mathbf{V}}_{\alpha}^{\dagger}/\comp})$ (i.e. the space of log de Rham differentials) equipped with the de Rham differential $\dpartial{k}_{\alpha} = \dpartial{}$ which is naturally a dg module over $\logsdrk{k}{*}$; 
		
		\item
		the local $k^{\text{th}}$-order volume element is given by a lifting $\volf{}_\alpha$ of $\volf{0}$ as an element in $j_*(\Omega^d_{\mathbf{V}_{\alpha}^{\dagger}/\logs})$ and taking $\volf{k}_{\alpha} = \volf{}_{\alpha} \ (\text{mod $\mathbf{m}^{k+1}$})$, and then the BV operator is defined by $\bvd{k}_{\alpha}(\varphi) \lrcorner \volf{k} := \dpartial{k}_{\alpha} (\varphi \lrcorner \volf{k})$;
		
		\item
		the morphism $\gmiso{k}{r}^{-1} :\logsdrk{k}{r} \otimes_{\cfrk{k}} (\tbva{k}{0}^*_\alpha/ \tbva{k}{1}^*_\alpha[-r]) \rightarrow \tbva{k}{r}^*_\alpha/ \tbva{k}{r+1}^*_\alpha$ of sheaves of BV modules is given by taking wedge product.
	\end{itemize}
\end{definition}

For both $\tbva{k}{}_{\alpha}^*$'s and $\bva{k}_{\alpha}^*$'s, the natural restriction map $\rest{k+1,k}_{\alpha}$ is given by the isomorphism $\prescript{k}{}{\mathbf{V}}_{\alpha}^{\dagger} \cong \prescript{k+1}{}{\mathbf{V}}_{\alpha}^{\dagger} \times_{\logsk{k+1}} \logsk{k}$.

Similar to the $0^{\text{th}}$-order case, we need to check that $\bva{k}^*_{\alpha}$ and $\tbva{k}{}^*_{\alpha}$ are coherent sheaves which are free over $\cfrk{k}$ for each $k$, and that $\gmiso{k}{r}$ is an isomorphism which induces an identification $\rdr{k}{}^*_{\alpha} \cong j_*(\Omega^*_{(\prescript{k}{}{\mathbf{V}}_{\alpha}^{\dagger}/\logsk{k})})$. Such verification can be done using \cite[Proposition 1.12 and Corollary 1.13]{Gross-Siebert-logII}, using similar argument as in \S \ref{sec:0_order_data_from_gross_siebert}. 

\subsubsection{Higher order patching data}\label{sec:higher_order_patching_data_from_gross_siebert}

To obtain the patching data we again need to take suitable analytification of statements from \cite{Gross-Siebert-logII}. Given $\bar{x} \in V_{\alpha\beta}$, we consider the following diagram of \'{e}tale neighborhoods
$$
\xymatrix@1{ & & \mathsf{W}_{\alpha} \times_{X_0} \mathsf{W}_{\beta} \ar[dr] \ar[dl] \ar[dd]& & \\
&\mathsf{W}_{\alpha} \ar[dr] \ar[dl]& & \mathsf{W}_{\beta} \ar[dr] \ar[dl] & \\
\prescript{k}{}{\mathbb{V}}_{\alpha} \supset \mathsf{V}_\alpha \ \ \	& & X_0 & & \ \ \ \ \mathsf{V}_{\alpha} \subset \prescript{k}{}{\mathbb{V}}_{\beta},}
$$ 
where $X_0 = X_0(B,\mathcal{P},\mathbf{s})$, and $\prescript{k}{}{\mathbb{V}}_{\alpha}$ (resp. $\prescript{k}{}{\mathbb{V}}_{\beta}$) is the $k^{\text{th}}$-order neighborhood of $\mathsf{V}_{\alpha}$ (resp. $\mathsf{V}_{\beta}$). Using \cite[Lemma 2.15]{Gross-Siebert-logII} on local uniqueness of thickening (see also \cite{ruddat2018local} for a more detailed study on local uniqueness), and further passing to an \'{e}tale cover $\mathsf{W}_{\alpha\beta}$ of $\mathsf{W}_{\alpha} \times_{X_0} \mathsf{W}_{\beta}$, we get an isomorphism 
$$
\prescript{k}{}{\Xi}_{\alpha\beta,i} : \mathsf{W}_{\alpha\beta} \times_{\mathsf{V}_{\alpha}} \prescript{k}{}{\mathbb{V}}_{\alpha} \overset{\cong}{\longrightarrow} \mathsf{W}_{\alpha\beta} \times_{\mathsf{V}_{\beta}} \prescript{k}{}{\mathbb{V}}_{\beta}.
$$
Taking analytification, we can find a (small enough) open subset in $(\mathsf{W}_{\alpha\beta})^{an}$ mapping homeomorphically onto a Stein open neighborhood $U_i \subset V_{\alpha\beta}$ of $\bar{x}$.
	
\begin{definition}\label{def:higher_order_patching_data_from_gross_siebert}
	Restriction of the analytification of $\prescript{k}{}{\Xi}_{\alpha\beta,i}$ on $U_i$ gives the gluing map $\prescript{k}{}{\Psi}_{\alpha\beta,i}$:
	$$
	\xymatrix@1{ \prescript{k}{}{\mathbf{V}}_{\alpha}^{\dagger}|_{U_i} \ar[rr]^{\prescript{k}{}{\Psi}_{\alpha\beta,i}} \ar[d]^{\pi_\alpha}& &  \prescript{k}{}{\mathbf{V}}_{\beta}^{\dagger}|_{U_i} \ar[d]^{\pi_{\beta}}\\
	\logsk{k}	\ar@{=}[rr]& & \logsk{k}
	}
	$$
	The patching isomorphisms $$\patch{k}_{\alpha\beta,i} :  j_* ( \bigwedge^{-*} \Theta_{\prescript{k}{}{\mathbf{V}}_{\alpha}^{\dagger}/\logsk{k}})|_{U_i} \rightarrow  j_* ( \bigwedge^{-*} \Theta_{\prescript{k}{}{\mathbf{V}}_\beta^{\dagger}/\logsk{k}})|_{U_i}$$ and $\hpatch{k}_{\alpha\beta,i}$ are then induced by $\prescript{k}{}{\Psi}_{\alpha\beta,i}$. 
\end{definition}
	
The existence of the vector fields $\resta{k,l}_{\alpha\beta,i}$, $\patchij{k}_{\alpha\beta,ij}$ and $\cocyobs{k}_{\alpha\beta\gamma,i}$ in Definition \ref{def:higher_order_patching} follows from the analytic version of \cite[Theorem 2.11]{Gross-Siebert-logII} which says that any log automorphism of the space $\prescript{k}{}{\mathbf{V}}_{\alpha}^{\dagger}|_{U_i} $ (resp. $\prescript{k}{}{\mathbf{V}}_{\alpha}^{\dagger}|_{U_{ij}}$) fixing $X|_{U_i}$ (or $X|_{U_{ij}}$) is obtained by exponentiating the action of a vector field in $\Theta_{\prescript{k}{}{\mathbf{V}}_{\alpha}^{\dagger}/\logsk{k}}(U_i)$ (resp. $\Theta_{\prescript{k}{}{\mathbf{V}}_{\alpha}^{\dagger}/\logsk{k}}(U_{ij})$). The element $\bvdobs{k}_{\alpha\beta,i}$ in Definition \ref{def:higher_order_module_patching} indeed measures the difference between the volume elements, namely, $\prescript{k}{}{\Psi}_{\alpha\beta,i}^*(\volf{k}_{\beta})= \exp(\bvdobs{k}_{\alpha\beta,i} \lrcorner )\volf{k}_{\alpha}$. 
	
\subsubsection{Criterion for freeness of the Hodge bundle} 
To verify Assumption \ref{assum:local_assumption_for_triviality_of_hodge_bundle}, which is needed for proving the freeness of the Hodge bundle in \S \ref{sec:proof_of_triviality_of_hodge_bundle}, notice that by taking $Q = \mathbb{N}$, we are already in the situation of a $1$-parameter family. The holomorphic Poincar\'{e} Lemma in Assumption \ref{assum:local_assumption_for_triviality_of_hodge_bundle} follows by taking the analytification of the results from \cite[proof of Theorem 4.1]{Gross-Siebert-logII}. (As aforementioned, there was a gap in \cite[proof of Theorem 4.1]{Gross-Siebert-logII} as pointed out and filled by \cite{Felten-Filip-Ruddat}; readers may see \cite[Theorem 1.10]{Felten-Filip-Ruddat} for details.) 
	
\subsection{The Hodge theoretic data}\label{sec:0_th_order_hodge_data_from_gross_siebert}
Since we have
$$\left\lbrack j_*(\Omega^{\text{alg},*}_{X_0(B,\mathcal{P},\mathbf{s})^{\dagger}/\logsk{0}}) \right\rbrack^{an} = j_*(\Omega^*_{X^{\dagger}/\logsk{0}}),$$
the Hodge-to-de Rham degeneracy (Assumption \ref{assum:Hodge_de_rham_degeneracy}) follows by applying Serre's GAGA theorems \cite{serre1956geometrie} to \cite[Theorem 3.26]{Gross-Siebert-logII} using the same argument as in the proof of Grothendieck's algebraic de Rham theorem.
Applying Theorem \ref{thm:unobstructedness_of_MC_equation} and Proposition \ref{prop:Maurer_Cartan_give_consistent_gluing}, we obtain an alternative proof of the following unobstructedness result due to Gross-Siebert \cite{gross2011real}:
\begin{corollary}\label{prop:gs_unobstructedness}
  	Under Assumption \ref{assum:gross_siebert_simple_and_positive}, the complex analytic space $(X,\mathcal{O}_X)$ is smoothable, i.e. there exists a $k^{\text{th}}$-order thickening $(\prescript{k}{}{X},\prescript{k}{}{\mathcal{O}})$ over $\logsk{k}$ locally modeled on $\prescript{k}{}{\mathbf{V}}_{\alpha}$ for each 
  	for each $k \in \inte_{\geq 0}$, and these thickenings are compatible.
\end{corollary}
  
\subsection{F-manifold structure near a LCSL}
 
Finally we demonstrate how to apply Theorem \ref{thm:construction_of_Frobenius_manifold} to the Gross-Siebert setting.

\subsubsection{The universal monoid $Q$}
First of all, we consider $(B,\mathcal{P})$ as in Notation \ref{not:gross_siebert_affine_manifold_polytope_notation} and work with the cone picture as in \cite{gross2011real}. We also need the notion of an multivalued integral piecewise affine function on $B$ as described before \cite[Remark 1.15]{gross2011real}. Let $\text{MPA}(B,\mathbb{N})$ be the monoid of multivalued convex integral piecewise affine function on $B$, take $Q = \Hom (\text{MPA}(B,\mathbb{N}),\mathbb{N})$ to be the universal monoid and consider the universal multivalued strictly convex integral piecewise affine function $\varphi : B \rightarrow Q$ as in \cite[equation A.2]{GHKS2016theta} (it was denoted as $\breve{\varphi}$ there). Since we work in the cone picture, we fix an open gluing data $\mathbf{s}$ as in \cite[Definition 1.18]{gross2011real} and replace the monodromy polytopes in Assumption \ref{assum:gross_siebert_simple_and_positive} by the dual monodromy polytopes associated to each $\tau \in \mathcal{P}$. 
  
\subsubsection{Construction of $X^{\dagger}=X_0(B,\mathcal{P},\mathbf{s},\varphi)^{\dagger}$}
  
We now take an element $\mathtt{n} \in \text{int}_{re}(Q^{\vee}_{\real}) \cap \blat^{\vee}$ and define a multivalued strictly convex piecewise affine function $\varphi_{\mathtt{n}}:B \rightarrow \real$. The cone picture construction described in \cite[Construction 1.17]{gross2011real} gives a log scheme $X^{\dagger}_{\mathtt{n}}=X_0(B,\mathscr{P},s,\varphi_{\mathtt{n}})^{\dagger}$ over $\comp^{\dagger}$ (here $\comp^{\dagger}$ is the standard $\inte_+$-log point) which is log smooth away from a codimension $2$ locus $ i : Z \hookrightarrow X$. \cite[Construction A.6]{GHKS2016theta} then gives a log scheme $X^{\dagger}$ (with the same underlying scheme as $X^{\dagger}_{\mathtt{n}}$) over $\logsk{0}$. Definition \ref{def:0_order_data_from_gross_siebert} can be carried through.
  
\subsubsection{Local model on thickening of $X^{\dagger}$}
For each $\tau \in \mathcal{P}$, let $\mathscr{Q}_{\tau}$ be the normal lattice as defined in \cite[Definition 1.33]{Gross-Siebert-logI}. We denote by $\Sigma_{\tau}$ the normal fan of $\tau$ defined in \cite[Definition 1.35]{Gross-Siebert-logI} on $\mathscr{Q}_{\tau,\real}$, equipped with the strictly convex piecewise linear function $\varphi : |\Sigma_{\tau}| = \mathscr{Q}_{\tau,\real}  \rightarrow Q_{\real}^{\text{gp}}$ induced by $\varphi$. We let $\check{\Delta}_1,\dots,\check{\Delta}_r$ be the dual monodromy polytopes associated to $\tau$ as defined in \cite[Definitions 1.58 \& 1.60]{Gross-Siebert-logI}, and $\psi_{i}(m):= -\inf \{ \langle m,n \rangle \ | \ m \in \mathscr{Q}_{\tau,\real},\ n \in \Delta_i \}$ be the integral piecewise linear function on $\mathscr{Q}_{\tau,\real}$. 
  
We define monoids $P_\tau$ and $\mathcal{Q}_\tau$ by
\begin{align*}
	P_{\tau}  := &\{  (m,a_0,\dots,a_r) \ | \ m \in \mathscr{Q}_{\tau}, \ a_0 \in Q^{\text{gp}},\ a_i \in \inte, \ a_0 - \varphi(m) \in Q,\\
	& a_i - \psi_i(m) \geq 0 \ for \ 1 \leq i\leq r  \},\\
	\mathcal{Q}_{\tau} :=& \{  (m,a_0,\dots,a_r) \ | \ m \in \mathscr{Q}_{\tau}, \ a_0 \in Q^{\text{gp}},\ a_i \in \inte,\ a_0 = \varphi(m) \} \cup \{\infty\},
\end{align*}
where the monoid structure on $\mathcal{Q}_\tau$ is given as in \cite[p. 22 in Construction 2.1]{Gross-Siebert-logII}. Also let $\mathbb{V}_\tau = \text{Spec}(\comp[P_{\tau}])$ which comes with a natural family $\pi : \mathbb{V}_\tau \rightarrow \text{Spec}(\comp[Q]) = \logs$, $\mathsf{V}_\tau = \pi^{-1}(0) = \text{Spec}(\comp[\mathcal{Q}_{\tau}])$ and $\prescript{k}{}{\mathbb{V}}_\tau = \pi^{-1}(\logsk{k})$ be the $k$-th order thickening of $\mathsf{V}_\tau$ in $\mathbb{V}_\tau$.
  
For $i = 1,\dots,r$ and a vertex $v \in \check{\Delta}_i$, we define a submonoid $\mathcal{D}_{i,v}:=w_{i,v}^{\perp} \cap P_\tau$, where $w_{i,v} = v + e_i^{\vee}$, and let $D_{i,v}$ be the corresponding toric divisor of $\mathbb{V}_\tau$. To simplify notations, we often omit the dependence on $v$ and write $w_i$, $\mathcal{D}_i$, $D_i$ instead of $w_{i,v}$, $\mathcal{D}_i$, $D_i$. 
Let $v_1,\dots,v_l$ be the generators of $1$-dimensional cones in the dual cone $P_{\tau}^{\vee}$ other than the $w_j$'s, with corresponding toric divisors $\mathscr{D}_1, \dots, \mathscr{D}_l$. Writing $\mathscr{D} = \bigcup_{j} \mathscr{D}_j$, we equip $\mathbb{V}_\tau$ with the divisorial log structure induced by the divisor $\mathscr{D} \hookrightarrow \mathbb{V}_\tau$, which is denoted as $\mathbb{V}_{\tau}^{\dagger}$. Pull back the log structure from $\mathbb{V}_\tau$ give the log schemes $\mathsf{V}_{\tau}^{\dagger}$ and $\prescript{k}{}{\mathbb{V}}^{\dagger}_\tau$. \cite[Theorem 2.6]{Gross-Siebert-logII} holds for this setting as described in Notation \ref{not:local_model}, by taking $P = P_{\tau}$ and $\mathcal{Q} = \mathcal{Q}_{\tau}$ for some $\tau \in \mathcal{P}$.
 
As in \S \ref{sec:higher_order_deformation_data_from_Gross_siebert}, analytification of the log schemes $\mathsf{V}_{\tau}^{\dagger}$ and $\prescript{k}{}{\mathbb{V}}^{\dagger}_\tau$ give the log analytic schemes $\mathscr{V}_{\tau}^{\dagger}$ and $\prescript{k}{}{\mathbf{V}}^{\dagger}_{\tau}$ respectively, and Definition \ref{def:higher_order_thickening_data_from_gross_siebert} can be carried through. We can deduce that $\bva{k}^*_{\alpha}$ and $\tbva{k}{}^*_{\alpha}$ are coherent sheaves which are also sheaves of free modules over $\cfrk{k}$, and that $\gmiso{k}{r}^{-1}$ is an isomorphism, by using the following variant of \cite[Proposition 1.12 \& Corollary 1.13]{Gross-Siebert-logII}.
 
\begin{prop}\label{prop:local_model_derham_complex_computation}
	Let $\mathscr{Z}:= \mathsf{V}_\tau \cap D_{\text{sing}} \hookrightarrow |\prescript{k}{}{\mathbb{V}}_{\tau}| = |\mathsf{V}_{\tau}|$ be the inclusion. Then we have the following decomposition into $P_{\tau}$-homogeneous pieces as 
 	\begin{align*}
 		\Gamma(\mathsf{V}_{\tau} \setminus \mathscr{Z},\Omega_{\prescript{k}{}{\mathbb{V}}^{\dagger}_{\tau}}^r) & = \bigoplus_{p \in P_{\tau} \setminus P_{\tau} + kQ^+} z^{p} \cdot  \bigwedge^r \big( \bigcap_{ \{j| p \in \mathcal{D}_{j} \} } \mathcal{D}_j^{\text{gp}} \otimes_{\inte} \comp \big), \\
 		\Gamma(\mathsf{V}_{\tau} \setminus \mathscr{Z},\Omega_{\prescript{k}{}{\mathbb{V}}_{\tau}^{\dagger}/\logsk{k}}^r) & = \bigoplus_{p \in P_{\tau} \setminus P_{\tau} + kQ^+} z^{p} \cdot \bigwedge^r \big( \bigcap_{ \{j| p \in \mathcal{D}_{j} \} } (\mathcal{D}_j^{\text{gp}}/Q^{\text{gp}}) \otimes_{\inte} \comp \big).
 	\end{align*}
\end{prop}

The construction of the higher order patching data described in \S \ref{sec:higher_order_patching_data_from_gross_siebert} can be carried through because divisorial deformations over $\logsk{0}$ can be defined as in \cite[Definition 2.7]{Gross-Siebert-logII}, and \cite[Theorem 2.11 \& Lemma 2.15]{Gross-Siebert-logII} hold accordingly with the local models $\mathbb{V}_{\tau}$'s. 

\subsubsection{Opposite filtration and pairing} 
The weight filtration in Assumption \ref{assum:weighted_filtration_assumption} is taken to be the filtration described in \cite[Remark 5.7]{Gross-Siebert-logII}, which is opposite to the Hodge filtration and preserved by the nilpotent operators $N_{\nu}$'s. The trace map $\trace$ in Assumption \ref{assum:poincare_pairing_assumption} can be defined via the isomorphism 
$
\trace : H^d(X,j_*(\Omega^d_{X^{\dagger}/\logsk{0}})) \cong H^{2d}(X,j_*(\Omega^*_{X^{\dagger}/\logsk{0}})) \cong \comp. 
$
We conjecture that the induced pairing $\prescript{0}{}{\mathtt{p}}$ is non-degenerate. 

\begin{corollary}\label{cor:Gross-Siebert_log_Frobenius}
	There exists a structure of log F-manifold on the formal extended moduli $\logscf$ near the maximally degenerate log Calabi-Yau variety $(X,\mathcal{O}_X)$. If the pairing $\prescript{0}{}{\mathtt{p}}$ is non-degenerate, it can be enhanced to a logarithmic Frobenius manifold structure.
\end{corollary}


\bibliographystyle{amsplain}
\bibliography{geometry}

\providecommand{\bysame}{\leavevmode\hbox to3em{\hrulefill}\thinspace}
\providecommand{\MR}{\relax\ifhmode\unskip\space\fi MR }
\providecommand{\MRhref}[2]{%
  \href{http://www.ams.org/mathscinet-getitem?mr=#1}{#2}
}
\providecommand{\href}[2]{#2}
\begin{thebibliography}{10}

\bibitem{Barannikov99}
S.~Barannikov, \emph{{G}eneralized periods and mirror symmetry in dimensions
  $n>3$}, preprint,
  \href{http://arxiv.org/abs/math/9903124}{arXiv:math/9903124}.

\bibitem{barannikov-kontsevich98}
S.~Barannikov and M.~Kontsevich, \emph{Frobenius manifolds and formality of
  {L}ie algebras of polyvector fields}, Internat. Math. Res. Notices (1998),
  no.~4, 201--215.

\bibitem{Barrott-Doran21}
L.~Barrott and C.~Doran, \emph{Towards the {D}oran-{H}arder-{T}hompson
  conjecture via the {G}ross-{S}iebert program}, preprint,
  \href{http://arxiv.org/abs/2105.02617}{arXiv:2105.02617}.

\bibitem{bogomolov1978hamiltonian}
F.~A. Bogomolov, \emph{Hamiltonian {K}\"{a}hlerian manifolds}, Dokl. Akad. Nauk
  SSSR \textbf{243} (1978), no.~5, 1101--1104.

\bibitem{Calabi}
E.~Calabi, \emph{The space of {K}\"ahler metrics}, Proceedings of the
  {I}nternational {C}ongress of {M}athematicians, 1954, {A}msterdam, vol. {II},
  pp.~206--207.

\bibitem{cartan1957varietes}
H.~Cartan, \emph{Vari\'{e}t\'{e}s analytiques r\'{e}elles et vari\'{e}t\'{e}s
  analytiques complexes}, Bull. Soc. Math. France \textbf{85} (1957), 77--99.

\bibitem{kwchan-leung-ma}
K.~Chan, N.~C. Leung, and Z.~N. Ma, \emph{Scattering diagrams from asymptotic
  analysis on {M}aurer-{C}artan equations}, J. Eur. Math. Soc. (JEMS), to
  appear, \href{http://arxiv.org/abs/1807.08145}{arXiv:1807.08145}.

\bibitem{Chan-Ma19}
K.~Chan and Z.~N. Ma, \emph{Smoothing pairs over degenerate {C}alabi-{Y}au
  varieties}, Int. Math. Res. Not. IMRN, to appear,
  \href{http://arxiv.org/abs/1910.08256}{arXiv:1910.08256}.

\bibitem{kwchan-ma-p2}
\bysame, \emph{Tropical counting from asymptotic analysis on {M}aurer-{C}artan
  equations}, Trans. Amer. Math. Soc. \textbf{373} (2020), no.~9, 6411--6450.

\bibitem{Chan-Ma-Suen20}
K.~Chan, Z.~N. Ma, and Y.-H. Suen, \emph{Tropical {L}agrangian multi-sections
  and smoothing of locally free sheaves over degenerate {C}alabi-{Y}au
  surfaces}, preprint,
  \href{http://arxiv.org/abs/2004.00523}{arXiv:2004.00523}.

\bibitem{costello2012quantum}
K.~Costello and S.~Li, \emph{{Q}uantum {B}{C}{O}{V} theory on {C}alabi-{Y}au
  manifolds and the higher genus {B}-model}, preprint,
  \href{http://arxiv.org/abs/1201.4501}{arXiv:1201.4501}.

\bibitem{deligne1997local}
P.~Deligne, \emph{Local behavior of {H}odge structures at infinity}, Mirror
  symmetry, {II}, AMS/IP Stud. Adv. Math., vol.~1, Amer. Math. Soc.,
  Providence, RI, 1997, pp.~683--699.

\bibitem{demailly1997complex}
J.-P. Demailly, \emph{Complex analytic and differential geometry}, 2012,
  \href{https://www-fourier.ujf-grenoble.fr/~demailly/manuscripts/agbook.pdf}{https://www-fourier.ujf-grenoble.fr/~demailly/manuscripts/agbook.pdf}.

\bibitem{dupont1976simplicial}
J.~L. Dupont, \emph{Simplicial de {R}ham cohomology and characteristic classes
  of flat bundles}, Topology \textbf{15} (1976), no.~3, 233--245.

\bibitem{Felten}
S.~Felten, \emph{Log smooth deformation theory via {G}erstenhaber algebras},
  Manuscripta Math., to appear,
  \href{http://arxiv.org/abs/2001.02995}{arXiv:2001.02995}.

\bibitem{Felten-Filip-Ruddat}
S.~Felten, M.~Filip, and H.~Ruddat, \emph{Smoothing toroidal crossing spaces},
  preprint, \href{http://arxiv.org/abs/1908.11235}{arXiv:1908.11235}.

\bibitem{felten2020logarithmic}
S.~Felten and A.~Petracci, \emph{The logarithmic {B}ogomolov-{T}ian-{T}odorov
  theorem}, preprint, \href{http://arxiv.org/abs/2010.13656}{arXiv:2010.13656}.

\bibitem{fiorenza2012differential}
D.~Fiorenza, D.~Iacono, and E.~Martinengo, \emph{Differential graded {L}ie
  algebras controlling infinitesimal deformations of coherent sheaves}, J. Eur.
  Math. Soc. (JEMS) \textbf{14} (2012), no.~2, 521--540.

\bibitem{fiorenza2007structures}
D.~Fiorenza and M.~Manetti, \emph{{$L_\infty$} structures on mapping cones},
  Algebra Number Theory \textbf{1} (2007), no.~3, 301--330.

\bibitem{fiorenza2012formality}
\bysame, \emph{Formality of {K}oszul brackets and deformations of holomorphic
  {P}oisson manifolds}, Homology Homotopy Appl. \textbf{14} (2012), no.~2,
  63--75.

\bibitem{fiorenza2012cosimplicial}
D.~Fiorenza, M.~Manetti, and E.~Martinengo, \emph{Cosimplicial {DGLA}s in
  deformation theory}, Comm. Algebra \textbf{40} (2012), no.~6, 2243--2260.

\bibitem{Friedman83}
R.~Friedman, \emph{Global smoothings of varieties with normal crossings}, Ann.
  of Math. (2) \textbf{118} (1983), no.~1, 75--114.

\bibitem{fujisawa2014polarizations}
T.~Fujisawa, \emph{Polarizations on limiting mixed {H}odge structures}, J.
  Singul. \textbf{8} (2014), 146--193.

\bibitem{fukaya05}
K.~Fukaya, \emph{Multivalued {M}orse theory, asymptotic analysis and mirror
  symmetry}, Graphs and patterns in mathematics and theoretical physics, Proc.
  Sympos. Pure Math., vol.~73, Amer. Math. Soc., Providence, RI, 2005,
  pp.~205--278.

\bibitem{griffiths1981rational}
P.~A. Griffiths and J.~W. Morgan, \emph{Rational homotopy theory and
  differential forms}, Progress in Mathematics, vol.~16, Birkh\"{a}user,
  Boston, Mass., 1981.

\bibitem{Gross_book}
M.~Gross, \emph{Tropical geometry and mirror symmetry}, CBMS Regional
  Conference Series in Mathematics, vol. 114, Published for the Conference
  Board of the Mathematical Sciences, Washington, DC; by the American
  Mathematical Society, Providence, RI, 2011.

\bibitem{GHKS2016theta}
M.~Gross, P.~Hacking, and B.~Siebert, \emph{Theta functions on varieties with
  effective anti-canonical class}, Mem. Amer. Math. Soc., to appear,
  \href{http://arxiv.org/abs/1601.07081}{arXiv:1601.07081}.

\bibitem{Gross-Siebert-logI}
M.~Gross and B.~Siebert, \emph{Mirror symmetry via logarithmic degeneration
  data. {I}}, J. Differential Geom. \textbf{72} (2006), no.~2, 169--338.

\bibitem{Gross-Siebert-logII}
\bysame, \emph{Mirror symmetry via logarithmic degeneration data, {II}}, J.
  Algebraic Geom. \textbf{19} (2010), no.~4, 679--780.

\bibitem{gross2011real}
\bysame, \emph{From real affine geometry to complex geometry}, Ann. of Math.
  (2) \textbf{174} (2011), no.~3, 1301--1428.

\bibitem{hinich2001dg}
V.~Hinich, \emph{D{G} coalgebras as formal stacks}, J. Pure Appl. Algebra
  \textbf{162} (2001), no.~2-3, 209--250.

\bibitem{huybrechts2006complex}
D.~Huybrechts, \emph{Complex geometry}, Universitext, Springer-Verlag, Berlin,
  2005, An introduction.

\bibitem{iacono2007differential}
D.~Iacono, \emph{Differential graded {L}ie algebras and deformations of
  holomorphic maps}, PhD Thesis, Rome, (2006),
  \href{http://arxiv.org/abs/math/0701091}{arXiv:math/0701091}.

\bibitem{iacono2015}
\bysame, \emph{Deformations and obstructions of pairs {$(X,D)$}}, Int. Math.
  Res. Not. IMRN (2015), no.~19, 9660--9695.

\bibitem{iacono2017abstract}
\bysame, \emph{On the abstract {B}ogomolov-{T}ian-{T}odorov theorem}, Rend.
  Mat. Appl. (7) \textbf{38} (2017), no.~2, 175--198.

\bibitem{iacono2010algebraic}
D.~Iacono and M.~Manetti, \emph{An algebraic proof of
  {B}ogomolov-{T}ian-{T}odorov theorem}, Deformation spaces, Aspects Math.,
  E40, Vieweg + Teubner, Wiesbaden, 2010, pp.~113--133.

\bibitem{kato1989logarithmic}
K.~Kato, \emph{Logarithmic structures of {F}ontaine-{I}llusie}, Algebraic
  analysis, geometry, and number theory ({B}altimore, {MD}, 1988), Johns
  Hopkins Univ. Press, Baltimore, MD, 1989, pp.~191--224.

\bibitem{katz1970nilpotent}
N.~M. Katz, \emph{Nilpotent connections and the monodromy theorem:
  {A}pplications of a result of {T}urrittin}, Inst. Hautes \'{E}tudes Sci.
  Publ. Math. (1970), no.~39, 175--232.

\bibitem{KKP08}
L.~Katzarkov, M.~Kontsevich, and T.~Pantev, \emph{Hodge theoretic aspects of
  mirror symmetry}, From {H}odge theory to integrability and {TQFT}
  tt*-geometry, Proc. Sympos. Pure Math., vol.~78, Amer. Math. Soc.,
  Providence, RI, 2008, pp.~87--174.

\bibitem{katzarkov2017bogomolov}
\bysame, \emph{Bogomolov-{T}ian-{T}odorov theorems for {L}andau-{G}inzburg
  models}, J. Differential Geom. \textbf{105} (2017), no.~1, 55--117.

\bibitem{kawamata1994logarithmic}
Y.~Kawamata and Y.~Namikawa, \emph{Logarithmic deformations of normal crossing
  varieties and smoothing of degenerate {C}alabi-{Y}au varieties}, Invent.
  Math. \textbf{118} (1994), no.~3, 395--409.

\bibitem{kontsevichgeneralized}
M.~Kontsevich, \emph{Generalized {T}ian-{T}odorov theorems}, talk on Kinosaki
  conference 2008.

\bibitem{kontsevich00}
M.~Kontsevich and Y.~Soibelman, \emph{Homological mirror symmetry and torus
  fibrations}, Symplectic geometry and mirror symmetry ({S}eoul, 2000), World
  Sci. Publ., River Edge, NJ, 2001, pp.~203--263.

\bibitem{KS_deformation_theory}
\bysame, \emph{Deformation theory. {I}}, 2002,
  \href{https://www.math.ksu.edu/~soibel/Book-vol1.ps}{https://www.math.ksu.edu/~soibel/Book-vol1.ps}.

\bibitem{kontsevich-soibelman04}
\bysame, \emph{Affine structures and non-{A}rchimedean analytic spaces}, The
  unity of mathematics, Progr. Math., vol. 244, Birkh\"auser Boston, Boston,
  MA, 2006, pp.~321--385.

\bibitem{kowalzig2015gerstenhaber}
N.~Kowalzig, \emph{Gerstenhaber and {B}atalin-{V}ilkovisky structures on
  modules over operads}, Int. Math. Res. Not. IMRN (2015), no.~22,
  11694--11744.

\bibitem{li2013primitive}
C.-Z. Li, S.~Li, and K.~Saito, \emph{Primitive forms via polyvector fields},
  preprint, \href{http://arxiv.org/abs/1311.1659}{arXiv:1311.1659}.

\bibitem{li2013variation}
S.~Li, \emph{Variation of {H}odge structures, {F}robenius manifolds and gauge
  theory}, preprint, \href{http://arxiv.org/abs/1303.2782}{arXiv:1303.2782}.

\bibitem{Liu-Rao-Wan19}
K.~Liu, S.~Rao, and X.~Wan, \emph{Geometry of logarithmic forms and
  deformations of complex structures}, J. Algebraic Geom. \textbf{28} (2019),
  no.~4, 773--815.

\bibitem{Liu-Rao-Yang15}
K.~Liu, S.~Rao, and X.~Yang, \emph{Quasi-isometry and deformations of
  {C}alabi-{Y}au manifolds}, Invent. Math. \textbf{199} (2015), no.~2,
  423--453.

\bibitem{lurie2011derived}
J.~Lurie, \emph{Derived algebraic geometry {X}: {F}ormal moduli problems},
  preprint (2011), available at \href{http://www.math.harvard.edu/~
  lurie}{http://www.math.harvard.edu/~ lurie}.

\bibitem{manetti2007lie}
M.~Manetti, \emph{Lie description of higher obstructions to deforming
  submanifolds}, Ann. Sc. Norm. Super. Pisa Cl. Sci. (5) \textbf{6} (2007),
  no.~4, 631--659.

\bibitem{manetti2005differential}
\bysame, \emph{Differential graded {L}ie algebras and formal deformation
  theory}, Algebraic geometry---{S}eattle 2005. {P}art 2, Proc. Sympos. Pure
  Math., vol.~80, Amer. Math. Soc., Providence, RI, 2009, pp.~785--810.

\bibitem{manetti2015some}
\bysame, \emph{On some formality criteria for {DG}-{L}ie algebras}, J. Algebra
  \textbf{438} (2015), 90--118.

\bibitem{Morrow-Kodaira_book}
J.~Morrow and K.~Kodaira, \emph{Complex manifolds}, AMS Chelsea Publishing,
  Providence, RI, 2006, Reprint of the 1971 edition with errata.

\bibitem{peters2008mixed}
C.~A.~M. Peters and J.~H.~M. Steenbrink, \emph{Mixed {H}odge structures},
  Ergebnisse der Mathematik und ihrer Grenzgebiete. 3. Folge. A Series of
  Modern Surveys in Mathematics, vol.~52, Springer-Verlag, Berlin, 2008.

\bibitem{pridham2010unifying}
J.~P. Pridham, \emph{Unifying derived deformation theories}, Adv. Math.
  \textbf{224} (2010), no.~3, 772--826.

\bibitem{reichelt2009construction}
T.~Reichelt, \emph{A construction of {F}robenius manifolds with logarithmic
  poles and applications}, Comm. Math. Phys. \textbf{287} (2009), no.~3,
  1145--1187.

\bibitem{ruddat2018local}
H.~Ruddat, \emph{Local uniqueness of approximations and finite determinacy of
  log morphisms}, preprint,
  \href{http://arxiv.org/abs/1812.02195}{arXiv:1812.02195}.

\bibitem{Sano21}
T.~Sano, \emph{Examples of non-{K}\"ahler {C}alabi-{Y}au manifolds with
  arbitrarily large $b_2$}, preprint,
  \href{http://arxiv.org/abs/2102.13437}{arXiv:2102.13437}.

\bibitem{Sernesi_book}
E.~Sernesi, \emph{Deformations of algebraic schemes}, Grundlehren der
  Mathematischen Wissenschaften [Fundamental Principles of Mathematical
  Sciences], vol. 334, Springer-Verlag, Berlin, 2006.

\bibitem{serre1956geometrie}
J.-P. Serre, \emph{G\'{e}om\'{e}trie alg\'{e}brique et g\'{e}om\'{e}trie
  analytique}, Ann. Inst. Fourier, Grenoble \textbf{6} (1955--1956), 1--42.

\bibitem{steenbrink1976limits}
J.~Steenbrink, \emph{Limits of {H}odge structures}, Invent. Math. \textbf{31}
  (1975/76), no.~3, 229--257.

\bibitem{syz96}
A.~Strominger, S.-T. Yau, and E.~Zaslow, \emph{Mirror symmetry is
  {$T$}-duality}, Nuclear Phys. B \textbf{479} (1996), no.~1-2, 243--259.

\bibitem{terilla2008smoothness}
J.~Terilla, \emph{Smoothness theorem for differential {BV} algebras}, J. Topol.
  \textbf{1} (2008), no.~3, 693--702.

\bibitem{tian1987smoothness}
G.~Tian, \emph{Smoothness of the universal deformation space of compact
  {C}alabi-{Y}au manifolds and its {P}etersson-{W}eil metric}, Mathematical
  aspects of string theory ({S}an {D}iego, {C}alif., 1986), Adv. Ser. Math.
  Phys., vol.~1, World Sci. Publishing, Singapore, 1987, pp.~629--646.

\bibitem{todorov1989weil}
A.~N. Todorov, \emph{The {W}eil-{P}etersson geometry of the moduli space of
  {${\rm SU}(n\geq 3)$} ({C}alabi-{Y}au) manifolds. {I}}, Comm. Math. Phys.
  \textbf{126} (1989), no.~2, 325--346.

\bibitem{weibel1995introduction}
C.~A. Weibel, \emph{An introduction to homological algebra}, Cambridge Studies
  in Advanced Mathematics, vol.~38, Cambridge University Press, Cambridge,
  1994.

\bibitem{whitney2012geometric}
H.~Whitney, \emph{Geometric integration theory}, Princeton University Press,
  Princeton, N. J., 1957.

\bibitem{Yau77}
S.-T. Yau, \emph{Calabi's conjecture and some new results in algebraic
  geometry}, Proc. Nat. Acad. Sci. U.S.A. \textbf{74} (1977), no.~5,
  1798--1799.

\bibitem{Yau78}
\bysame, \emph{On the {R}icci curvature of a compact {K}\"ahler manifold and
  the complex {M}onge-{A}mp\`ere equation. {I}}, Comm. Pure Appl. Math.
  \textbf{31} (1978), no.~3, 339--411.

\end{thebibliography}

\end{document}